\documentclass[a4paper,dvipdfmx]{amsart}
\usepackage{amssymb}
\DeclareFontFamily{U}{mathx}{\hyphenchar\font45}
\DeclareFontShape{U}{mathx}{m}{n}{
      <5> <6> <7> <8> <9> <10>
      <10.95> <12> <14.4> <17.28> <20.74> <24.88>
      mathx10
      }{}
\DeclareSymbolFont{mathx}{U}{mathx}{m}{n}
\DeclareFontSubstitution{U}{mathx}{m}{n}
\DeclareMathAccent{\widecheck}{0}{mathx}{"71}
\DeclareMathAccent{\wideparen}{0}{mathx}{"75}
\def\cs#1{\texttt{\char`\\#1}}
\usepackage{mathrsfs}
\usepackage{accents}
\usepackage{verbatim}
\usepackage{epic}
\usepackage{eepic}
\usepackage{url}
\usepackage{cases}
\usepackage{graphicx}
\usepackage{xcolor}
\usepackage{tikz}
\usetikzlibrary{intersections,calc,arrows.meta,patterns,cd}
\theoremstyle{plain}
  \newtheorem{thm}{Theorem}[section]
  \newtheorem{lem}[thm]{Lemma}
  \newtheorem{sublem}[thm]{Sublemma}
  
  \newtheorem{cor}[thm]{Corollary}
  \newtheorem{prop}[thm]{Proposition}
  \newtheorem{conj}[thm]{Conjecture}

  \newtheorem*{obs*}{Observation}
\theoremstyle{definition}
  \newtheorem{defn}[thm]{Definition}

\theoremstyle{remark}
  \newtheorem{rem}[thm]{Remark}

\newcommand{\Z}{\mathbb{Z}}
\newcommand{\C}{\mathbb{C}}
\newcommand{\R}{\mathbb{R}}

\newcommand{\Vol}{\operatorname{Vol}}
\newcommand{\Cl}{\operatorname{Cl}}
\newcommand{\CS}{\operatorname{CS}}
\newcommand{\Li}{\operatorname{Li}}
\newcommand{\Hom}{\operatorname{Hom}}
\renewcommand{\cs}{\operatorname{\mathbf{cs}}}
\newcommand{\J}{\mathcal{J}}
\renewcommand{\L}{\mathcal{L}}

\newcommand{\PSL}{\mathrm{PSL}}

\DeclareMathOperator{\arcosh}{arcosh}
\DeclareMathOperator{\arsinh}{arsinh}

\newcommand{\Int}{\operatorname{Int}}

\newcommand{\floor}[1]{\lfloor#1\rfloor}
\newcommand{\ceil}[1]{\lceil#1\rceil}
\newcommand{\pic}[2]{\raisebox{-0.5\height}{\includegraphics[scale=#1]{#2.eps}}}
\renewcommand{\Re}{\operatorname{Re}}
\renewcommand{\Im}{\operatorname{Im}}

\renewcommand{\i}{\sqrt{-1}}
\newcommand{\FE}{\mathscr{E}}

\renewcommand{\th}{\widetilde{h}}
\newcommand{\tp}{\widetilde{p}}
\newcommand{\tq}{\widetilde{q}}
\newcommand{\tr}{\widetilde{r}}
\newcommand{\ts}{\widetilde{s}}
\newcommand{\tv}{\widetilde{v}}
\newcommand{\tC}{\widetilde{C}}
\newcommand{\tD}{\widetilde{D}}

\newcommand{\tM}{\widetilde{M}}
\newcommand{\tP}{\widetilde{P}}
\newcommand{\tR}{\widetilde{R}}
\newcommand{\tS}{\widetilde{S}}
\newcommand{\tX}{\widetilde{X}}
\newcommand{\tdelta}{\widetilde{\delta}}
\newcommand{\tlambda}{\widetilde{\lambda}}
\newcommand{\tmu}{\widetilde{\mu}}
\newcommand{\tphi}{\widetilde{\varphi}}
\newcommand{\tPhi}{\widetilde{\Phi}}
\newcommand{\tGamma}{\widetilde{\Gamma}}
\newcommand{\hlambda}{\widehat{\lambda}}
\newcommand{\hC}{\widehat{C}}
\newcommand{\hE}{\widehat{E}}
\newcommand{\clambda}{\widecheck{\lambda}}

\makeatletter
\newcommand{\oset}[3][0ex]{%
  \mathrel{\mathop{#3}\limits^{
    \vbox to#1{\kern-1\ex@
    \hbox{$\scriptstyle#2$}\vss}}}}
\newcommand{\uset}[3][0ex]{%
  \mathrel{\mathop{#3}\limits_{
    \vbox to#1{\kern-5\ex@
    \hbox{$\scriptstyle#2$}\vss}}}}
\makeatother
\newcommand{\Rpath}{\oset{\frown}{\mathbb{R}}}

\numberwithin{equation}{section}
\renewcommand{\labelenumi}{(\roman{enumi})}
\allowdisplaybreaks
\begin{document}
\title[The generalized volume conjecture for the figure-eight knot]
{The generalized volume conjecture for \\ the figure-eight knot parametrized \\ by a complex number with small imaginary part}
\author{Hitoshi Murakami}
\address{
Graduate School of Information Sciences,
Tohoku University,
Aramaki-aza-Aoba 6-3-09, Aoba-ku,
Sendai 980-8579, Japan}
\email{hitoshi@tohoku.ac.jp}
\date{\today}
\dedicatory{Dedicated to the memory of my mother, Yachiyo Murakami (1932--2025)}
\begin{abstract}
We study the asymptotic behavior, as $N$ tends to infinity, of the $N$-dimensional colored Jones polynomial of the figure-eight knot, evaluated at $\exp(\xi/N)$ for a complex parameter $\xi$ with $0<\Im\xi<\pi/2$.
We prove that if $\Re{\xi}$ is large the colored Jones polynomial grows exponentially with growth rate expressed by the Chern--Simons invariant, and that if $\Re{\xi}$ is small it converges to the reciprocal of the Alexander polynomial evaluated at $\exp\xi$.
\end{abstract}
\thanks{The author is supported by JSPS KAKENHI Grant Numbers JP20K03601, JP24K06702, JP23K22388, JP20K03931.}
\maketitle
%%%%%%%%%%%%%%%%%%%%%%%%%%%%%%%%%%%%%%%%%%%%%%%%%%%%%%%%%%%%%%%%%%%%
\setcounter{tocdepth}{3}
\tableofcontents
%%%%%%%%%%%%%%%%%%%%%%%%%%%%%%%%%%%%%%%%%%%%%%%%%%%%%%%%%%%%%%%%%%%%
\section{Introduction}
For a knot $K$ in the three-sphere $S^3$, and an integer $N\ge2$, let $J_N(K;q)$ be the colored Jones polynomial associated with the $N$-dimensional irreducible representation of the Lie algebra $\mathfrak{sl}(2;\C)$ with complex parameter $q$.
See \cite[\S~3]{Kirby/Melvin:INVEM1991} for example.
We normalize it so that $J_N(U;q)=1$ for the unknot $U$, and that $J_2(K;q)$ is the celebrated Jones polynomial $V(K;q)$ \cite{Jones:BULAM31985}.
\par
If we replace $q$ with the $N$-th root of unity $e^{2\pi\i/N}$, then $J_N\left(K;e^{2\pi\i/N}\right)$ gives a series $\left\{J_N\left(K;e^{2\pi\i/N}\right)\right\}_{N=2,3,\dots}$ of complex numbers.
The volume conjecture states that this series would determine the simplicial volume of the knot complement $S^3\setminus{K}$.
\begin{conj}[Volume conjecture \cite{Kashaev:LETMP97},\cite{Murakami/Murakami:ACTAM12001}]
For any knot $K$ in $S^3$, we have
\begin{equation}\label{eq:VC}
  \lim_{N\to\infty}\frac{1}{N}\log\left|J_N\left(K;e^{2\pi\i/N}\right)\right|
  =
  \frac{1}{2\pi}\Vol\left(S^3\setminus{K}\right),
\end{equation}
where $\Vol$ denotes the simplicial volume.
Note that if $K$ is hyperbolic, that is, $S^3\setminus{K}$ has a complete hyperbolic structure with finite volume, then $\Vol\left(S^3\setminus{K}\right)$ coincides with the hyperbolic volume.
\end{conj}
If we drop the absolute value symbol from the left hand side of \eqref{eq:VC}, we have a complex number.
\begin{conj}[Complexification of the volume conjecture
\cite{Murakami/Murakami/Okamoto/Takata/Yokota:EXPMA02}]
For a knot $K\subset S^3$, we have
\begin{equation*}
  \lim_{N\to\infty}\frac{1}{N}\log J_N\left(K;e^{2\pi\i/N}\right)
  =
  \frac{1}{2\pi}
  \left(
    \Vol\left(S^3\setminus{K}\right)+\i\CS^{\mathrm{SO}(3)}\left(S^3\setminus{K}\right)
  \right),
\end{equation*}
where $\CS^{\mathrm{SO}(3)}$ is the Chern--Simons invariant associated with the Levi-Civita connection when $K$ is hyperbolic, and otherwise it would be defined by the equation above.
\end{conj}
If we replace $2\pi\i$ with a general complex number $\xi$, we have various series $\left\{J_N\left(K;e^{\xi/N}\right)\right\}_{N=2,3,\ldots}$ depending on $\xi$.
In this paper we consider the figure-eight knot $\FE\subset S^3$, and study the asymptotic behavior of $J_N(\FE;e^{\xi/N})$ as $N\to\infty$, with $\Im\xi$ small.
\par
For a complex number $z$, we define
\begin{align}
  \varphi(z)
  &:=
  \log\left(\cosh{z}-\frac{1}{2}+\frac{1}{2}\sqrt{(2\cosh{z}-3)(2\cosh{z}+1)}\right),
  \label{eq:def_phi}
  \\
  S(z)
  &:=
  \Li_2\left(e^{-z-\varphi(z)}\right)-\Li_2\left(e^{-z+\varphi(z)}\right)+z\varphi(z),
  \label{eq:S}
  \\
  S^{-}(z)
  &:=
  S(z)+2z\pi\i,
  \label{eq:Sminus}
  \\
  S^{+}(z)
  &:=
  -S(z)+2z\pi\i,
  \label{eq:Splus}
  \\
  T(z)
  &:=
  \frac{2}{\sqrt{(2\cosh{z}+1)(2\cosh{z}-3)}},
  \notag
\end{align}
where $\Li_2(w):=-\int_{0}^{w}\frac{\log(1-t)}{t}\,dt$ is the dilogarithm function.
The branch cut of $\Li_2$ is $(1,\infty)$, and we put $\Im\Li_2(x)=-\pi\i\log{x}$ for $x>1$.
We choose $(-\infty,0)$ as the branch cuts of the square root and the logarithm.
For a negative number $x$, we define $\sqrt{x}:=\i\sqrt{|x|}$, and $\log{x}:=\log|x|+\pi\i$.
Note that $\cosh\varphi(z)=\cosh{z}-1/2$.
\par
Put $\kappa:=\arcosh(3/2)=0.962424\ldots$ so that $\varphi(\kappa)=0$.
\begin{rem}\label{rem:kappa}
Since $\varphi(\kappa)$ vanishes we see that $S^{+}(\kappa)=S^{-}(\kappa)=2\kappa\pi\i$.
\end{rem}
\par
For any knot, we have $J_N(K;\overline{q})=\overline{J_N(K;q)}$, where $\overline{z}$ is the complex conjugate of $z$.
Since the figure-eight knot is amphicheiral, we have $J_N(\FE;q^{-1})=J_N(\FE;q)$.
Therefore when we study $J_N(\FE;e^{\xi/N})$, we assume that $\Re\xi\ge0$ and $\Im\xi\ge0$ in this paper.
\par
So far the following results are known.
\begin{enumerate}
\item
$\xi=2p\pi\i+u$, where $u\in\R$ with $0<u<\kappa$ and $p$ is a positive integer.
The author proved the following asymptotic formula in \cite{Murakami:JTOP2013} and \cite[Theorem~1.4]{Murakami:CANJM2023}.
\begin{equation*}
  J_N(\FE;e^{\xi/N})
  \underset{N\to\infty}{\sim}
  J_p(\FE;e^{4N\pi^2/\xi})
  \frac{\sqrt{-\pi}}{2\sinh(u/2)}
  T(\xi)^{1/2}\left(\frac{N}{\xi}\right)^{1/2}
  \exp\left(\frac{S^{+}(u)}{\xi}N\right).
\end{equation*}
Since $J_1(\FE;q)=1$, if $\xi=2\pi\i+u$, then we have
\begin{equation*}
  J_N(\FE;e^{\xi/N})
  \underset{N\to\infty}{\sim}
  \frac{\sqrt{-\pi}}{2\sinh(u/2)}
  T(\xi)^{1/2}\left(\frac{N}{\xi}\right)^{1/2}
  \exp\left(\frac{S^{+}(u)}{\xi}N\right).
\end{equation*}
\item
$\xi=2p\pi\i+\kappa$, where $p$ is a positive integer.
The author proved the following asymptotic formula in \cite[Theorem~1.8]{Murakami:AGT2025}.
\begin{equation*}
  J_N(\FE;e^{\xi/N})
  \underset{N\to\infty}{\sim}
  J_{p}(\FE;e^{4\pi^2N/\xi})
  \frac{\Gamma(1/3)e^{\pi\i/6}}{3^{1/6}}
  \left(\frac{N}{\xi}\right)^{2/3}
  \exp\left(\frac{S^{\pm}(\kappa)}{\xi}N\right),
\end{equation*}
where $\Gamma(z)$ is the Gamma function, and $S^{\pm}(\kappa)$ means either $S^{+}(\kappa)$ or $S^{-}(\kappa)$.
Note that that $S^{+}(\kappa)=S^{-}(\kappa)$ from Remark~\ref{rem:kappa}.
\par
In particular, if $\xi=2\pi\i+\kappa$, we have
\begin{equation*}
  J_N(\FE;e^{\xi/N})
  \underset{N\to\infty}{\sim}
  \frac{\Gamma(1/3)e^{\pi\i/6}}{3^{1/6}}
  \left(\frac{N}{\xi}\right)^{2/3}
  \exp\left(\frac{S^{\pm}(\kappa)}{\xi}N\right).
\end{equation*}
\item
$\xi=2p\pi\i+u$, where $u\in\R$ with $u>\kappa$ and $p$ is a positive integer.
The author proved the following asymptotic formula in \cite[Theorem~1.11]{Murakami:arXiv2023}.
\begin{equation*}
  J_N(\FE;e^{\xi/N})
  \underset{N\to\infty}{\sim}
  J_p(\FE;e^{4N\pi^2/\xi})
  \frac{\sqrt{\pi}}{2\sinh(u/2)}
  T(\xi)^{1/2}
  \left(\frac{N}{\xi}\right)^{1/2}
  \exp\left(\frac{S^{-}(u)}{\xi}N\right).
\end{equation*}
When $\xi=2\pi\i+u$, we have
\begin{equation*}
  J_N(\FE;e^{\xi/N})
  \underset{N\to\infty}{\sim}
  \frac{\sqrt{\pi}}{2\sinh(u/2)}
  T(\xi)^{1/2}
  \left(\frac{N}{\xi}\right)^{1/2}
  \exp\left(\frac{S^{-}(u)}{\xi}N\right).
\end{equation*}
\item
$\xi\in\Omega$, where we put
\begin{equation}\label{eq:Omega_def}
  \Omega
  :=
  \{\xi\in\C\mid\cosh(\Re\xi)-\cos(\Im\xi)<1/2,|\Im\xi|<\pi/3\}.
\end{equation}
The author proved in \cite[Theorem~1.1]{Murakami:JPJGT2007} that the colored Jones polynomial converges as follows.
\begin{equation*}
  \lim_{N\to\infty}
  J_N(\FE;e^{\xi/N})
  =
  \frac{1}{\Delta(e^{\xi})},
\end{equation*}
where $\Delta(t)$ is the Alexander polynomial of the figure-eight knot normalized so that $\Delta(1)=1$ and that $\Delta(t^{-1})=\Delta(t)$.
\item
$\xi=\kappa$.
K.~Hikami and the author proved that the colored Jones polynomial grows polynomially.
In fact we proved the following asymptotic formula in \cite[Theorem~1.1]{Hikami/Murakami:COMCM2008}.
\begin{equation*}
  J_N(\FE;e^{\xi/N})
  \underset{N\to\infty}{\sim}
  \frac{\Gamma(1/3)}{3^{2/3}}
  \left(\frac{N}{\kappa}\right)^{2/3}.
\end{equation*}
\item
$\xi\in\R$ with $\xi>\kappa$.
A.~Tran and the author proved in \cite[Theorem~1.5]{Murakami/Tran:Takata2025} the following asymptotic formula.
\begin{equation*}
  J_N(\FE;e^{\xi/N})
  \underset{N\to\infty}{\sim}
  \frac{\sqrt{\pi}}{2\sinh(\xi/2)}
  T(\xi)^{1/2}
  \left(\frac{N}{\xi}\right)^{1/2}
  \exp\left(\frac{S(\xi)}{\xi}N\right).
\end{equation*}
\end{enumerate}
\par
The following are also known.
\begin{enumerate}
\setcounter{enumi}{6}
\item
$\xi=2\pi\i$.
J.E.~Andersen and S.K.~Hansen proved the following asymptotic formula \cite[Theorem~1]{Andersen/Hansen:JKNOT2006}.
\begin{equation*}
  J_N(\FE;e^{2\pi\i/N})
  \underset{N\to\infty}{\sim}
  -2\pi^{3/2}T(0)^{1/2}\left(\frac{N}{2\pi\i}\right)^{3/2}
  \exp\left(\frac{N}{2\pi\i}S^{+}(0)\right).
\end{equation*}
\item
$\xi=2\pi\i+u$ with $u\in\C\setminus{\i\R}$ and $|u|$ small.
Y.~Yokota and the author \cite{Murakami/Yokota:JREIA2007} proved the following formula.
\begin{equation*}
  \lim_{N\to\infty}\frac{1}{N}\log J_N(\FE;e^{\xi/N})
  =
  \frac{S^{+}(u)}{\xi}.
\end{equation*}
\end{enumerate}
See Figure~\ref{fig:chart}.
\begin{figure}[h]
\pic{0.5}{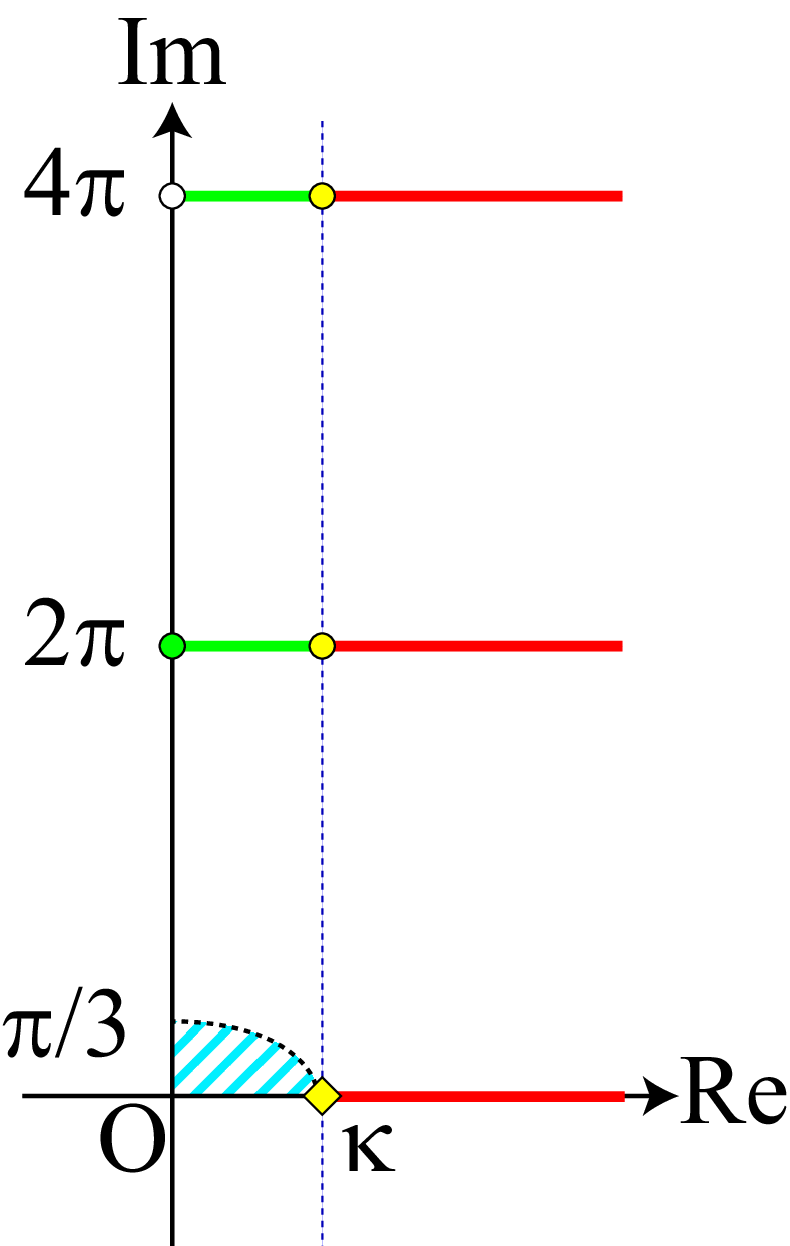}
\caption{The green lines correspond to Case (i), the yellow disks to (ii), the red lines to (iii) and (vi), the cyan striped region to (iv), the yellow diamond to (v), and the green disk to (vii).}
\label{fig:chart}
\end{figure}
\begin{rem}\label{rem:phi_different}
In \cite{Murakami:JTOP2013,Murakami:CANJM2023,Murakami:AGT2025}, the formula for $S^{+}(u)$ for $0<u\le\kappa$ is different because we define $\varphi(u)$ differently.
In fact, there we defined $\varphi(u)$ ($0<u\le\kappa$) as
\begin{equation}\label{eq:def_tphi}
  \log\left(\cosh{u}-\frac{1}{2}-\frac{1}{2}\sqrt{(2\cosh{u}-3)(2\cosh{u}+1)}\right),
\end{equation}
which we denote by $\tphi(u)$ here.
\par
Since both $e^{\varphi(u)}$ and $e^{\tphi(u)}$ satisfy the quadratic equation $x^2-(2\cosh{u}-1)x+1=0$ with discriminant $(2\cosh{u}-1)^2-4=(2\cosh{u}-3)(2\cosh{u}+1)$, we see that $|e^{\varphi(u)}|=|e^{\tphi(u)}|=1$ and $e^{\tphi(u)}=e^{\overline{\varphi(u)}}$.
So, if $0<u<\kappa=\arcosh(3/2)$, then they are purely imaginary with $\tphi(u)+\varphi(u)=0$ because their arguments should be between $-\pi$ and $\pi$.
If $u=\kappa$, then $\varphi(\kappa)=\tphi(\kappa)=0$.
So we conclude that $\varphi(u)=-\tphi(u)$ for $0<u\le\kappa$.
\par
Moreover, $S^{+}(u)$ was defined as
\begin{equation*}
  \Li_2\left(e^{-u-\tphi(u)}\right)
  -
  \Li_2\left(e^{-u+\tphi(u)}\right)
  +u\bigl(\tphi(u)+2\pi\i\bigr)
\end{equation*}
in \cite{Murakami:JTOP2013,Murakami:CANJM2023,Murakami:AGT2025}, which we denote by $\tS^{+}(u)$ here.
From \eqref{eq:Splus}, $S^{+}(u)$ in the current paper equals
\begin{equation*}
  \Li_2\left(e^{-u+\varphi(u)}\right)-\Li\left(e^{-u-\varphi(u)}\right)
  +u\bigl(-\varphi(u)+2\pi\i\bigr),
\end{equation*}
which coincides with $\tS^{+}(u)$.
\end{rem}
\par
In this paper we study the asymptotic behavior of $J_N(\FE;e^{\xi/N})$ as $N\to\infty$ for the case where $0<\Im\xi<\pi/2$ and $\Re\xi>0$.
\par
Throughout this paper, we use the following terminology.
\par
We fix a complex number $\xi$, and write $a:=\Re\xi$ and $b:=\Im\xi$, $c:=\Re\varphi(\xi)$, and $d:=\Im\varphi(\xi)$, where $\varphi(\xi)$ is defined in \eqref{eq:def_phi}.
We often drop $\xi$ in $\varphi(\xi)$, and  write $\varphi$ for short.
We also put
\begin{align}
  \alpha
  &:=
  \cosh{a}\cos{b},\label{eq:alpha}
  \\
  \beta
  &:=
  \sinh{a}\sin{b}\label{eq:beta}
\end{align}
so that $\cosh\xi=\alpha+\beta\i$.
\par
In the definition of $\varphi$, we have a square root.
Since we calculate
\begin{equation}\label{eq:sqrt}
  (2\cosh\xi-3)(2\cosh\xi+1)
  =
  (2\alpha-3)(2\alpha+1)-4\beta^2+4\beta(2\alpha-1)\i,
\end{equation}
the square root $\sqrt{(2\cosh\xi-3)(2\cosh\xi+1)}$ has a branch cut along the curve $2\alpha=1$.
So it is natural to restrict ourselves to the following region:
\begin{equation}\label{eq:Xi}
  \Xi
  :=
  \{\xi\in\C\mid a>0,0<b<\pi/2,\cosh{a}\cos{b}>1/2\}.
\end{equation}
However, by some technical reason, we restrict ourselves to the following smaller region $\Gamma\subset\Xi$.
\begin{equation*}
  \Gamma
  :=
  \Xi\cap\{\xi\in\C\mid a\tanh{c}-b\tan{d}\ge0,\cosh{a}-\cos{b}>1/2\}.
\end{equation*}
If we put
\begin{equation*}
  \tGamma
  :=
  \Xi\cap\{\xi\in\C\mid a\tanh{c}-b\tan{d}<0,\cosh{a}-\cos{b}>1/2\},
\end{equation*}
then $\Xi=\Gamma\sqcup\tGamma\sqcup\bigl(\Cl(\Omega)\cap\Xi\bigr)$ because $\Cl(\Omega)\cap\Xi=\{z\in\C\mid a>0,0<b<\pi/2,\cosh{a}\cos{b}>1/2,\cosh{a}-\cos{b}\le1/2\}$ from Remark~\ref{rem:Omega_appendix}, where $\Cl$ means the closure and $\sqcup$ is the disjoint union.
We also introduce the following regions:
\begin{equation*}
\begin{split}
  \Gamma_{+}
  :=&
  \Gamma\cap\{\xi\in\C\mid\Re\bigl(S(\xi)/\xi\bigr)>0\},
  \\
  \Gamma_{-}
  :=&
  \Gamma\cap\{\xi\in\C\mid\Re\bigl(S(\xi)/\xi\bigr)<0\},
  \\
  \Gamma_{0}
  :=&
  \Gamma\cap\{\xi\in\C\mid\Re\bigl(S(\xi)/\xi\bigr)=0\},
  \\
  \tGamma_{+}
  :=&
  \tGamma\cap\{\xi\in\C\mid\Re\bigl(S(\xi)/\xi\bigr)>0\},
  \\
  \tGamma_{-}
  :=&
  \tGamma\cap\{\xi\in\C\mid\Re\bigl(S(\xi)/\xi\bigr)<0\},
  \\
  \tGamma_{0}
  :=&
  \tGamma\cap\{\xi\in\C\mid\Re\bigl(S(\xi)/\xi\bigr)=0\}.
\end{split}
\end{equation*}
See Figure~\ref{fig:Gamma}.
\begin{figure}
\pic{0.8}{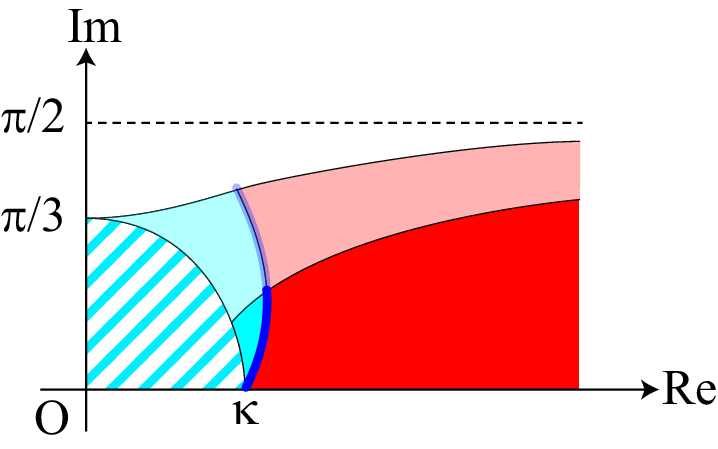}
\caption{The region $\Xi$ is partitioned into the five regions $\Gamma_{+}$ (red), $\Gamma_{-}$ (cyan), $\tGamma_{+}$ (pink), $\tGamma_{-}$ (light cyan), and $\Omega\cap\Xi$ (striped cyan).
The blue arc and the light blue arc indicate $\Gamma_{0}$ and $\tGamma_{0}$, respectively.}
\label{fig:Gamma}
\end{figure}
\par
The following is our main theorem, whose proof is given in Subsection~\ref{subsec:main}
\begin{thm}[Main Theorem]\label{thm:main}
If $\xi\in\Gamma_{+}$, then we have
\begin{equation}\label{eq:main_pos}
  J_N(\FE;e^{\xi/N})
  =
  \frac{\sqrt{\pi}}{2\sinh(\xi/2)}
  T(\xi)^{1/2}\left(\frac{N}{\xi}\right)^{1/2}
  \exp\left(\frac{S(\xi)}{\xi}{N}\right)
  \bigl(1+O(N^{-1})\bigr)
\end{equation}
as $N\to\infty$.
\par
If $\xi\in\Gamma_{0}$, then we have
\begin{equation}\label{eq:main_zero}
  J_N(\FE;e^{\xi/N})
  =
  \frac{\sqrt{\pi}}{2\sinh(\xi/2)}
  T(\xi)^{1/2}\left(\frac{N}{\xi}\right)^{1/2}
  \exp\left(\frac{S(\xi)}{\xi}N\right)
  +
  \frac{1}{\Delta(e^{\xi})}
  +O(N^{-1/2})
\end{equation}
as $N\to\infty$.
Note that in this case the absolute value $\left|J_N(\FE;e^{\xi/N})\right|$ grows polynomially since $\Re\bigl(S(\xi)/\xi\bigr)=0$.
\par
If $\xi\in\Gamma_{-}$, then we have the following.
\begin{equation}\label{eq:main_neg}
  J_N(\FE;e^{\xi/N})
  =
  \frac{1}{\Delta(e^{\xi})}
  +O(N^{-2})
\end{equation}
as $N\to\infty$.
\end{thm}
\begin{rem}\label{rem:Omega}
From Remark~\ref{rem:Omega}, we have $\Xi\cap\Omega=\{z\in\C\mid a>0,0<b<\pi/2,\cosh{a}-\cos{b}<1/2\}$.
So from Lemma~\ref{lem:Omega} we know that if $\xi\in\Xi\cap\Omega$, then $\Re\bigl(S(\xi)/\xi\bigr)<0$.
It follows that $\Xi\cap\Omega\subset\Gamma_{-}\cup\tGamma_{-}$.
On the other hand, by Mathematica, for $\xi=1+0.5\i$ we have $a\tanh{c}-b\tan{d}=0.0661743>0$, $\cosh{a}-\cos{b}=0.665498\ldots>1/2$ and $\Re\bigl(S(\xi)/\xi\bigr)=-0.166996\ldots<0$, which implies that $\xi\in\Gamma_{-}$ but that $\xi\not\in\Omega$.
So \eqref{eq:main_neg} is new.
\end{rem}
\begin{rem}\label{rem:main_neg_0}
The author does not know whether \eqref{eq:main_neg} is also true for $\xi\in\Xi\cap\partial\Omega=\{z\in\C\mid a>0,0<b<\pi/2,\cosh{a}-\cos{b}=1/2\}$, where $\partial$ means the boundary.
See Remark~\ref{rem:cosha_cosb_1_2}.
\end{rem}
\begin{rem}
If $\Re\frac{S(\xi)}{\xi}<0$, the term containing $\exp\left(\frac{S(\xi)}{\xi}N\right)$ may be hidden behind the term $1/\Delta(e^{\xi})$ because $\exp\left(\frac{S(\xi)}{\xi}N\right)$ decays exponentially.
Similarly, if $\Re\frac{S(\xi)}{\xi}>0$, $1/\Delta(e^{\xi})$ may be hidden behind $\exp\left(\frac{S(\xi)}{\xi}N\right)$ because $\exp\left(\frac{S(\xi)}{\xi}N\right)$ grows exponentially.
\end{rem}
\begin{rem}
The condition $a\tanh{c}-b\tan{d}\ge0$ in Theorem~\ref{thm:main} is just technical.
It is used in Lemmas~\ref{lem:chi_increasing}, \ref{lem:Re_sigma}, \ref{lem:h_F}, \ref{lem:Poisson_iii}, \ref{lem:Poisson_iv}, \ref{lem:Poisson_negative}, \ref{lem:saddle_i_ii}, and \ref{lem:ReF_1}, and Corollaries~\ref{cor:F1} and \ref{cor:sum_delta1}.
\par
The author thinks that Theorem~\ref{thm:main} holds for any $\xi\in\Xi$, that is, if $\xi\in\tGamma_{+}$ ($\tGamma_{0}$, and $\tGamma_{-}$, respectively), then \eqref{eq:main_pos} (\eqref{eq:main_zero}, and \eqref{eq:main_neg}, respectively) also hold.
\end{rem}
\par
Topological interpretations of the above formulas are in order.
See Section~\ref{sec:CSR} for details.
\par
Put $X:=S^3\setminus\Int{N(\FE)}$, where $N(\FE)$ is the regular neighborhood of the figure-eight knot $\FE\subset S^3$, and $\Int$ denotes the interior.
Then the fundamental group $\pi_1(\partial{X})$ of its boundary $\partial{X}$ is isomorphic to $\Z^2$ generated by the meridian $m$ and the preferred longitude $l$, where $m$ is null-homotopic in $N(\FE)$, and $l$ is parallel to $\FE$ and null-homologous in $X$.
\par
Given $\xi\in\C$ one can define two non-Abelian representations $\rho^{\pm}_{\xi}\colon\pi_1(X)\to\PSL(2;\C)$ such that
\begin{align*}
  \rho^{\pm}_{\xi}(m)
  &=
  \begin{pmatrix}e^{\xi/2}&\ast\\0&e^{-\xi/2}\end{pmatrix},
  \\
  \rho^{\pm}_{\xi}(l)
  &=
  \begin{pmatrix}\ell(\xi)^{\pm}&\ast\\0&\ell(\xi)^{\mp}\end{pmatrix}.
\end{align*}
\par
Putting $v^{\pm}(\xi):=d\,S^{\pm}(\xi)/d\,\xi-2\pi\i$, we can prove that $\ell(\xi)=-e^{\pm v^{\pm}(\xi)}$.
We denote by $\CS_{\left(\xi,v^{\pm}(\xi)\right)}\left(\rho^{\pm}_{\xi}\right)$ the Chern--Simons invariant of $\rho^{\pm}_{\xi}$ associated with $\left(u,-v^{\pm}(u)\right)$.
It can be proved that $\CS_{\bigl(u,v^{+}(u)\bigr)}\left(\rho^{+}_{u}\right)=S^{+}(u)-u\pi\i-\frac{1}{4}uv^{+}(u)$ in Cases (i), (ii), (vii), and (viii).
Similarly, we have $\CS_{\bigl(u,v^{-}(u)\bigr)}\left(\rho^{-}_{u}\right)=S^{-}(u)-u\pi\i-\frac{1}{4}uv^{-}(u)$ in Case (iii).
Case (iv) corresponds to an Abelian representation, and Case (v) corresponds to the affine representation $\rho^{+}_{\kappa}=\rho^{-}_{\kappa}$ with $\kappa:=\arcosh(3/2)$.
For Case (vi), see Remark~\ref{rem:big_real}.
\par
Similarly, we put $v(\xi):=d\,S(\xi)/d\,\xi-2\pi\i$, and let $\CS_{\left(u,v(\xi)\right)}\left(\rho^{-}_{\xi}\right)$ be the Chern--Simons invariant of $\rho^{-}_{\xi}$ associated with $\bigl(\xi,v(\xi)\bigr)$.
Then we can prove that $\CS_{\bigl(\xi,v(\xi)\bigr)}(\rho^{-}_{\xi})=S(\xi)-\xi\pi\i-\frac{1}{4}\xi v(\xi)$ (Lemma~\ref{lem:CS_minus}).
\par
One can also define the cohomological adjoint Reidemeister torsion of $\rho^{\pm}_{\xi}$ as the torsion of the cochain complex $\Hom_{\Z(\pi_1(X))}(C_{\ast}(\tX;\Z),\mathfrak{sl}(2;\C))$, where $\tX$ is the universal cover of $X$.
We can prove that it coincides with $T(\xi)\in\C\setminus\{0\}$.
\par
Figure~\ref{fig:chart_plus_minus} shows how $\xi$ is related to representation $\rho^{\pm}_{\xi}$.
\begin{figure}[h]
\pic{0.5}{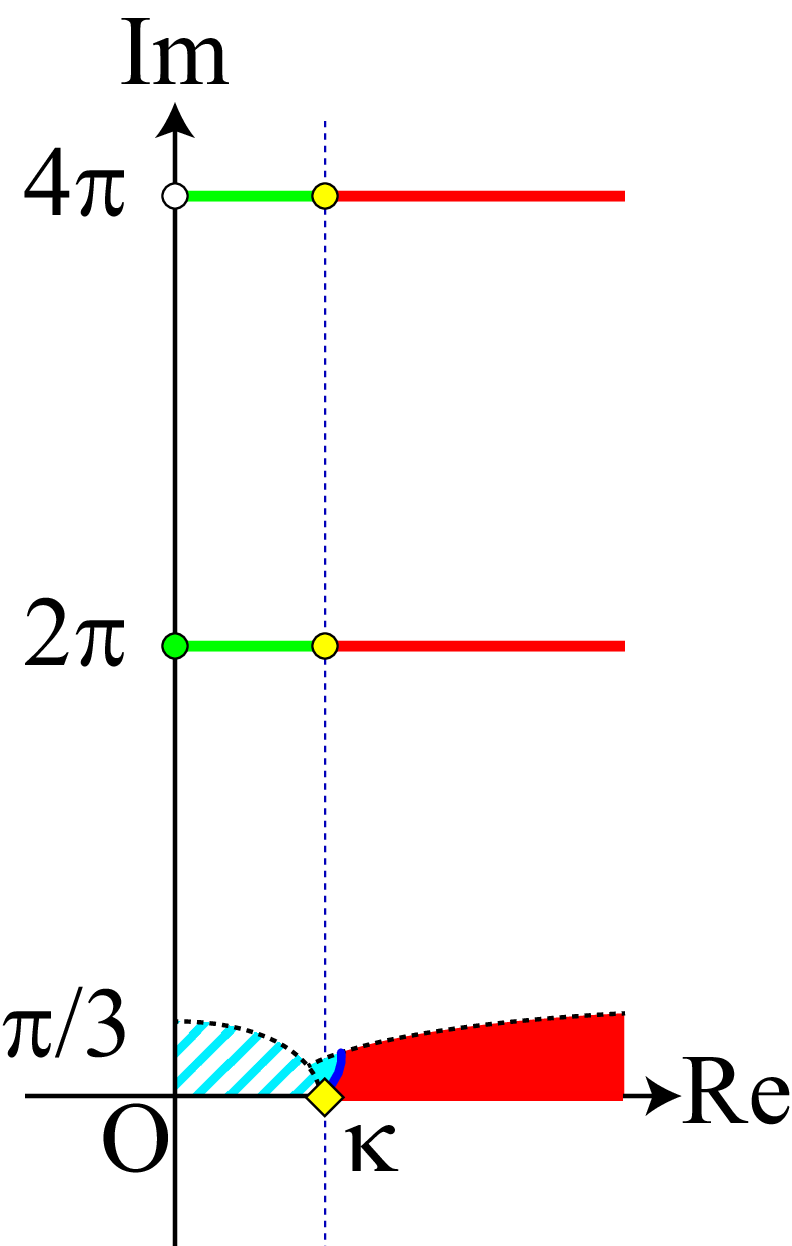}
\caption{The red lines and the red region correspond to the representation $\rho^{+}_{\xi}$, and the green lines and the green disk to $\rho^{-}_{\xi}$, where $\rho^{+}_{\xi}$ ($\rho^{-}_{\xi}$, respectively) is a non-Abelian representation defining a (possibly incomplete) hyperbolic metric with positive (negative, respectively) volume.
The yellow disks and the yellow diamond correspond to the affine representation $\rho^{\pm}_{\kappa}$.
The cyan region and the cyan striped region correspond to the Abelian representation $\rho^{\rm{Abel}}_{\xi}$.}
\label{fig:chart_plus_minus}
\end{figure}
The paper is organized as follows.
\par
In Section~\ref{sec:sum}, we express the colored Jones polynomial $J_N(\FE;e^{\xi/N})$ as a sum of the form $\sum_{k=0}^{N-1}e^{Nf_N\bigl((2k+1)/(2N)\bigr)}$, where $f_N(z)$ is a holomorphic function depending on $N$.
The series of functions $\{f_N(z)\}_{N=2,3,\dots}$ uniformly converges to a function $F(z)$, and so the colored Jones polynomial is approximated by the sum $\sum_{k=0}^{N-1}e^{NF\bigl((2k+1)/(2N)\bigr)}$.
In Section~\ref{sec:F_G} we prepare several results used in Sections~\ref{sec:Poisson} and \ref{sec:saddle}.
We will use the Poisson summation formula to change the sum into the integral of the form $\int_{0}^{1}e^{NF(z)}\,dz$ in Section~\ref{sec:Poisson}.
We then use the saddle point method in Section~\ref{sec:saddle} to obtain an asymptotic formula of the integral, which proves the main theorem.
We give topological interpretation of the main theorem in Section~\ref{sec:CSR}.
Appendices~\ref{sec:curvature} and \ref{sec:lemmas} contain technical lemmas used in the previous sections.

\section{Summation formula}\label{sec:sum}
Put $\gamma:=\frac{\xi}{2\pi\i}$ with $a>0$ and $0<b<\pi/2$, where we put $\xi:=a+b\i$.
For a complex number $z$ with $-\frac{\Re\gamma}{2N}<\Re{z}<\frac{\Re\gamma}{2N}+1$ and a positive integer $N$, we define
\begin{equation}\label{eq:TN_defn}
  T_N(z)
  :=
  \frac{1}{4}
  \int_{\Rpath}\frac{e^{(2z-1)t}}{t\sinh{t}\sinh(\gamma t/N)}\,dt
\end{equation}
following \cite{Faddeev:LETMP1995}, where $\Rpath$ is a path $(-\infty,-1]\cup[1,\infty)\cup\{z\in\C\mid|z|=1,\Im{z}\ge0\}$ oriented from left to right.
Note that $\Re\gamma=\frac{b}{2\pi}>0$.
The convergence of the right hand side is given in Lemma~\ref{lem:T_converge}.
\par
We also define the following two functions $\L_i(z)$ ($i=1,2$) for $0<\Re{z}<1$ \cite[Definition~2.4]{Murakami/Tran:Takata2025}:
\begin{equation}\label{eq:L1_L2_defn}
\begin{split}
  \L_1(z)
  &:=
  -\frac{1}{2}
  \int_{\Rpath}\frac{e^{(2z-1)t}}{t\sinh{t}}\,dt,
  \\
  \L_2(z)
  &:=
  \frac{\pi\i}{2}\int_{\Rpath}\frac{e^{(2z-1)t}}{t^2\sinh{t}}\,dt.
\end{split}
\end{equation}
The integrals in the right hand sides converge similarly, so we omit proofs.
These functions can be expressed in terms of well-known functions \cite[Lemma~2.5]{Murakami/Tran:Takata2025}.
\begin{equation}\label{eq:L1_L2}
\begin{split}
  \L_1(z)
  &=
  \begin{cases}
    \log\left(1-e^{2\pi\i z}\right)
    &\quad\text{if $\Im{z}\ge0$},
    \\
    2\pi\i z-\pi\i+\log\left(1-e^{-2\pi\i z}\right)
    &\quad\text{if $\Im{z}<0$},
  \end{cases}
  \\
  \L_2(z)
  &=
  \begin{cases}
    \Li_2\left(e^{2\pi\i z}\right)
    &\quad\text{if $\Im{z}\ge0$},
    \\
    2\pi^2z^2-2\pi^2z+\frac{\pi^2}{3}
    -\Li_2\left(e^{-2\pi\i z}\right)
    &\quad\text{if $\Im{z}<0$},
  \end{cases}
\end{split}
\end{equation}
where $\Li_2(z):=-\int_{0}^{z}\frac{\log(1-t)}{t}\,dt$ is the dilogarithm function.
We can extend them to holomorphic functions in $\C\setminus\bigl((-\infty,0]\cup[1,\infty)\bigr)$ by using the right hand sides of the equalities above.
The derivative of $\L_2(z)$ is given in terms of $\L_1(z)$ \cite[Lemma~2.9]{Murakami/Tran:Takata2025}.
\begin{equation}\label{eq:L2_der}
\begin{split}
  \frac{d}{d\,z}\L_2(z)
  &=
  -2\pi\i\L_1(z),
\end{split}
\end{equation}
\par
In \cite[\S~2]{Murakami:AGT2025}, we proved that $T_N(z)$ is extended to the following region:
\begin{equation}\label{eq:Sigma}
  \Sigma
  :=
  \left\{
    z\in\C\Bigm|
    -1+\frac{\Re\gamma}{2N}<\Re{z}<2-\frac{\Re\gamma}{2N}
  \right\}
  \setminus\left(\rhd\cup\lhd\right),
\end{equation}
where we put
\begin{align*}
  \rhd
  &:=
  \left\{
    z\in\C\Bigm|
    \Im{z}\ge0,-1+\frac{\Re\gamma}{2N}<\Re{z}\le0,\Im\left(\frac{z}{\gamma}\right)\le0
  \right\},
  \\
  \lhd
  &:=
  \left\{
    z\in\C\Bigm|
    \Im{z}\le0,1\le\Re{z}<2-\frac{\Re\gamma}{2N},\Im\left(\frac{z-1}{\gamma}\right)\ge0
  \right\}.
\end{align*}
Moreover, the function $T_N(z)$ is holomorphic in $\Sigma$.
See Figure~\ref{fig:domain_T} (left).
\par
The series of functions $\{\frac{1}{N}T_N(z)\}_{N=2,3,\dots}$ uniformly converges to $\frac{\L_2(z)}{\xi}$.
More precisely, we have the following lemma (\cite[Proposition~2.25]{Murakami:AGT2025}).
\begin{lem}\label{lem:converge_T}
Let $\nu<1/4$ be a positive number.
Then we have
\begin{equation*}
  T_N(z)
  =
  \frac{N}{\xi}\L_2(z)
  +
  O(N^{-1})
\end{equation*}
as $N\to\infty$ in the region $\Sigma_\nu$.
Here we put
\begin{equation*}
  \Sigma_{\nu}
  :=
  \left\{
    z\in\C\bigm|
    -1+\nu\le\Re{z}\le2-\nu,|\Im{z}|\le M
  \right\}
  \setminus\left(\rhd_{\nu}\cup\lhd_{\nu}\right)
\end{equation*}
for a real number $M>0$ with $M>\max\left\{\nu,\left|\frac{\Im\gamma}{\Re\gamma}\right|\right\}$, and
\begin{align*}
  \rhd_{\nu}
  &:=
  \left\{
    z\in\C\mid
    \Im{z}>-\nu,-1+\nu\le\Re{z}<\nu,\Im\left(\frac{z-\nu}{\gamma}\right)<0
  \right\},
  \\
  \lhd_{\nu}
  &:=
  \left\{
    z\in\C\mid
    \Im{z}<\nu,1-\nu<\Re{z}\le2-\nu,\Im\left(\frac{z-1+\nu}{\gamma}\right)>0
  \right\}.
\end{align*}
Note that $\Sigma_{\nu}\subset\Sigma$ if $N>\frac{\Re\gamma}{2\nu}$.
See Figure~\ref{fig:domain_T} \rm{(}right\rm{)}.
\begin{figure}[h]
\pic{0.25}{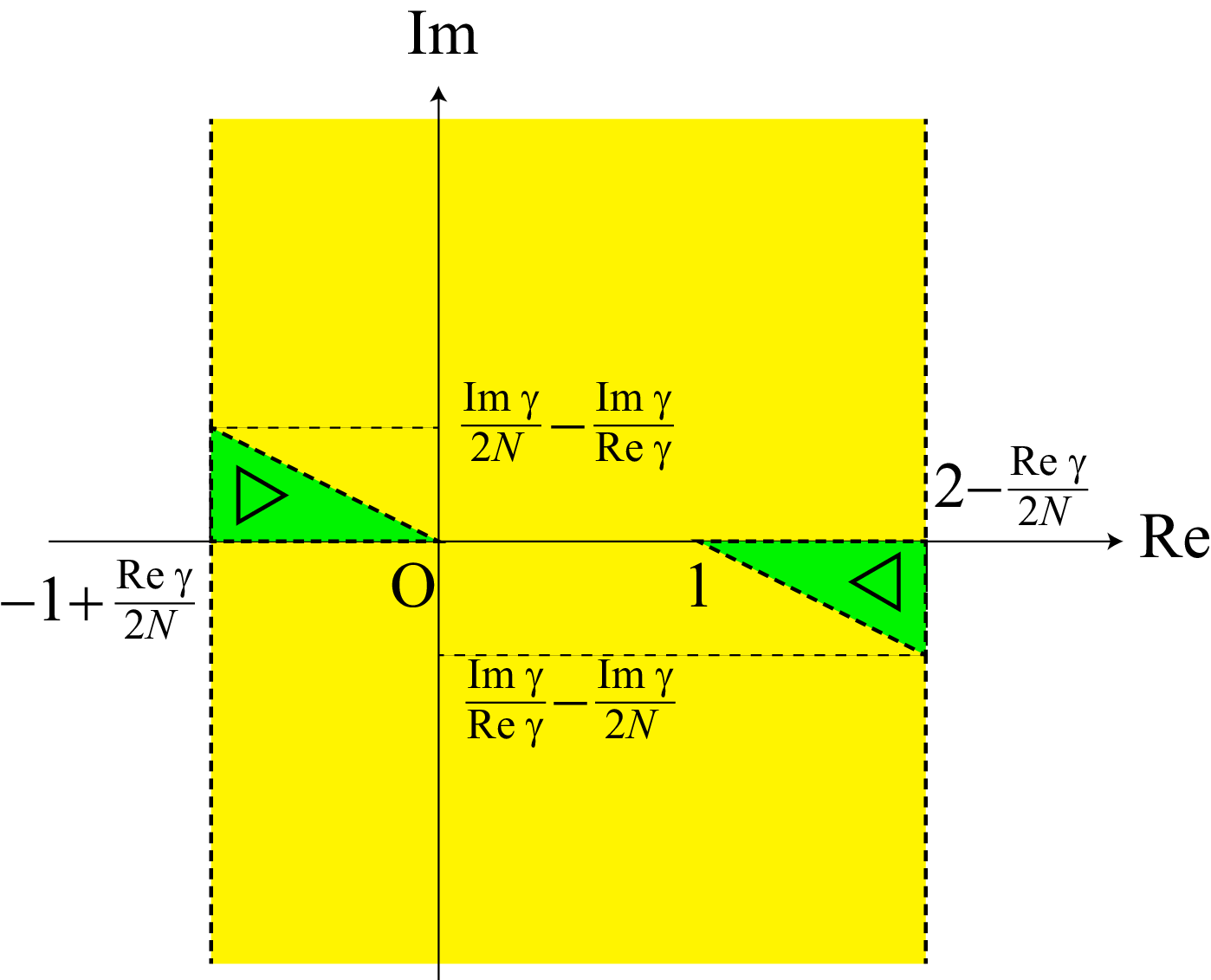}\quad\pic{0.25}{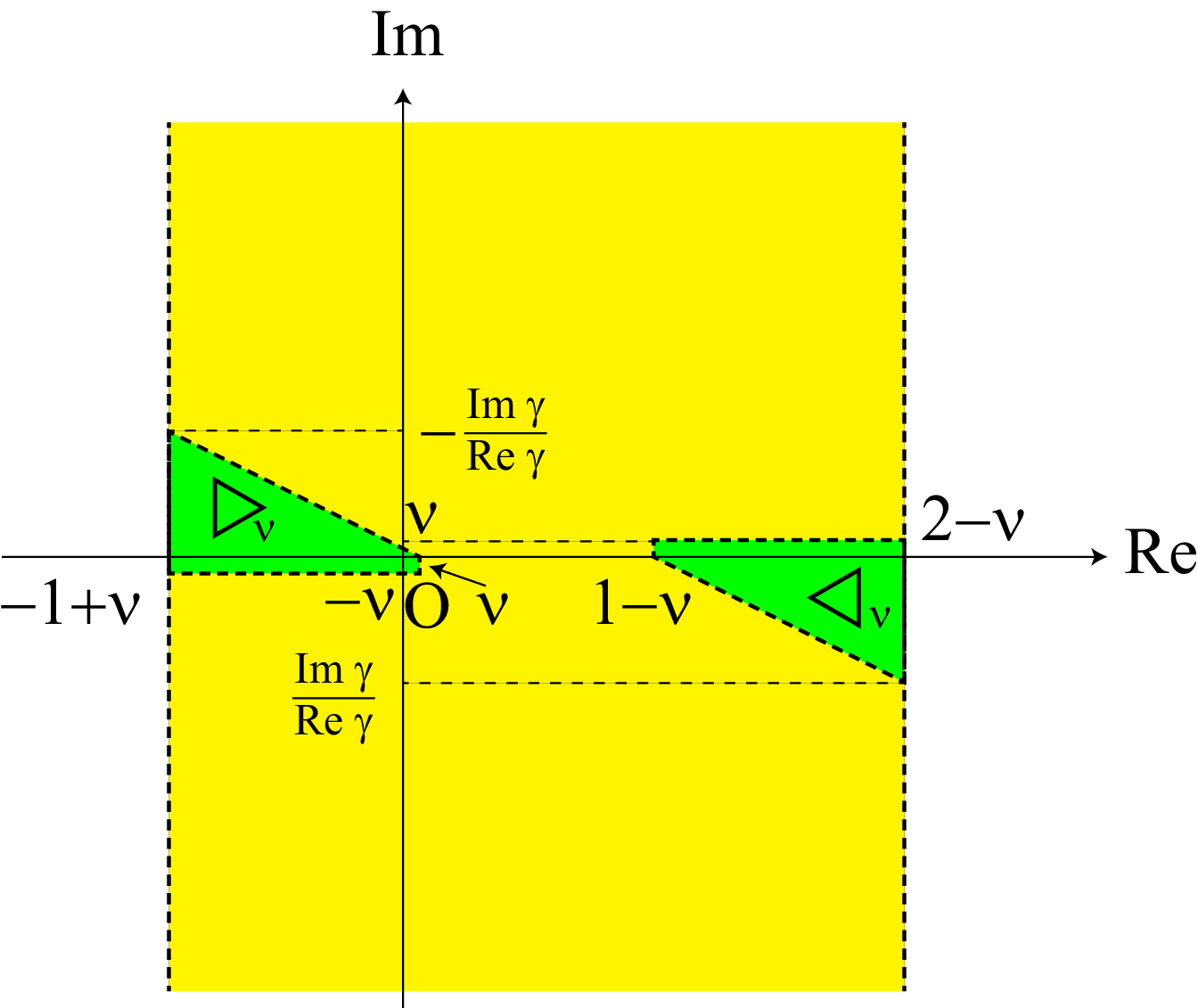}
\caption{$\Sigma$ (left) and $\Sigma_{\nu}$ (right).}
\label{fig:domain_T}
\end{figure}
\end{lem}
Let us consider the following integrals $\J_m$ for $m=0,1,2,\dots$:
\begin{equation*}
  \J_m(z)
  :=
  \int_{\Rpath}\frac{t^{m}e^{(2z-1)t}}{\sinh{t}}\,dt,
\end{equation*}
Then the derivatives of $\L_1(z)$ can be expressed in terms of $\J_m(z)$.
\begin{lem}
We have
\begin{equation*}
  \frac{d}{d\,z}\L_1(z)
  =
  -\J_0(z),
\end{equation*}
and
\begin{equation*}
  \frac{d}{d\,z}\J_m(z)
  =
  2\J_{m-1}(z).
\end{equation*}
\end{lem}
\begin{proof}
We have
\begin{equation*}
  \frac{d}{d\,z}\J_{m}(z)
  =
  2\int_{\Rpath}\frac{t^{m+1}e^{(2z-1)t}}{\sinh{t}}\,dt
  =
  2\J_m(z).
\end{equation*}
\end{proof}
We also have the following lemma.
\begin{lem}\label{lem:TN'}
For $0<\nu<1/4$, we have
\begin{equation*}
  T'_N(z)
  =
  -\frac{N}{\gamma}\L_1(z)+O(N^{-1})
\end{equation*}
as $N\to\infty$  in $\Sigma_{\nu}$, where $T'_N(z)$ means the derivative.
\end{lem}
\begin{proof}
We follow \cite[Proof of Lemma~2.4]{Murakami/Tran:Takata2025}.
Since
\begin{equation*}
  T'_N(z)
  =
  \frac{1}{2}\int_{\Rpath}\frac{e^{(2z-1)t}}{\sinh{t}\sinh(\gamma t/N)}\,dt,
\end{equation*}
we have
\begin{equation*}
\begin{split}
  \left|T'_N(z)+\frac{N}{\gamma}\L_1(z)\right|
  =&
  \frac{1}{2}
  \left|
    \int_{\Rpath}
    \frac{Ne^{(2z-1)t}}{\gamma t\sinh{t}}
    \left(\frac{\gamma t/N}{\sinh(\gamma t/N)}-1\right)\,dt
  \right|
  \\
  \le&
  \frac{N}{2|\gamma|}
  \int_{\Rpath}
  \left|\frac{e^{(2z-1)t}}{t\sinh{t}}\right|
  \left|\frac{\gamma t/N}{\sinh(\gamma t/N)}-1\right|\,dt.
\end{split}
\end{equation*}
Since $x/\sinh{x}=1-x^2/6+O(x^4)$ as $x\to0$, $\left|\frac{\gamma t/N}{\sinh(\gamma t/N)}-1\right|\le\frac{c|t|^2}{N^2}$ for a constant $c>0$.
Therefore we have
\begin{equation*}
  \left|T'_N(z)+\frac{N}{\gamma}\L_1(z)\right|
  \le
  \frac{c'}{N}
  \int_{\Rpath}
  \left|\frac{te^{(2z-1)t}}{\sinh{t}}\right|\,dt,
\end{equation*}
where $c':=\frac{c}{2|\gamma|}$.
\par
We will estimate the integral $\int_{\Rpath}\left|\frac{te^{(2z-1)t}}{\sinh{t}}\right|\,dt$.
Put
\begin{align*}
  I_{+}
  &:=
  \int_{1}^{\infty}\left|\frac{te^{(2z-1)t}}{\sinh{t}}\right|\,dt,
  \\
  I_{-}
  &:=
  \int_{-\infty}^{-1}\left|\frac{te^{(2z-1)t}}{\sinh{t}}\right|\,dt,
  \\
  I_{0}
  &:=
  \int_{|t|=1,\Im{t}\ge0}\left|\frac{te^{(2z-1)t}}{\sinh{t}}\right|\,dt.
\end{align*}
Since $\Re{z}\le1-\nu$, we have
\begin{equation*}
\begin{split}
  I_{+}
  =&
  \int_{1}^{\infty}\frac{2te^{(2\Re{z}-1)t}}{e^t-e^{-t}}\,dt
  =
  \int_{1}^{\infty}\frac{2te^{2(\Re{z}-1)t}}{1-e^{-2t}}\,dt
  \le
  \frac{2}{1-e^{-2}}\int_{1}^{\infty}te^{-2\nu t}\,dt
  \\
  =&
  \frac{2e^{-2\nu}(1+2\nu)}{4\nu^2(1-e^{-2})}.
\end{split}
\end{equation*}
Similarly, from $\Re{z}\ge\nu$, we have
\begin{equation*}
\begin{split}
  I_{-}
  =&
  -\int_{-\infty}^{-1}\frac{2te^{(2\Re{z}-1)t}}{e^{-t}-e^{t}}\,dt
  =
  -\int_{-\infty}^{-1}\frac{2te^{2t\Re{z}}}{1-e^{2t}}\,dt
  \le
  \frac{-2}{1-e^{-2}}
  \int_{-\infty}^{-1}te^{2\nu t}\,dt
  \\
  =&
  \frac{2e^{-2\nu}(1+2\nu)}{4\nu^2(1-e^{-2})}.
\end{split}
\end{equation*}
As for $I_{0}$, putting $t:=e^{s\i}$ for $0\le s\le\pi$, and $L:=\min_{|z|=1,\Im{z}\ge0}|\sinh{z}|$, we have
\begin{equation*}
\begin{split}
  I_{0}
  =&
  \int_{0}^{\pi}
  \left|
    \frac{e^{s\i}e^{(2z-1)e^{s\i}}}{\sinh\left(e^{s\i}\right)}
  \right|
  \times
  \left|\i e^{s\i}\right|
  \,ds
  \\
  \le&
  \frac{1}{L}
  \int_{0}^{\pi}e^{(2\Re{z}-1)\cos{s}-2\Im{z}\sin{s}}\,ds,
\end{split}
\end{equation*}
which is bounded since $z$ is bounded.
\par
Therefore we conclude that $I_{+}+I_{-}+I_{0}$ is bounded, and the lemma follows.
\end{proof}
\par
If $0<\Re{z}<1$ then we have
\begin{equation}\label{eq:TN_gamma}
  \frac{\exp\left(T_N\left(z-\frac{\gamma}{2N}\right)\right)}
       {\exp\left(T_N\left(z+\frac{\gamma}{2N}\right)\right)}
  =
  1-e^{2\pi\i z}
\end{equation}
from \cite[Lemma~2.5]{Murakami:CANJM2023}.
For an integer $j$ with $0<j<2N$, putting $z=j\gamma/N$ in \eqref{eq:TN_gamma} we have
\begin{equation}\label{eq:TN_j}
  \frac{\exp\left(T_N\bigl((j-1/2)\gamma/N\bigr)\right)}
       {\exp\left(T_N\bigl((j+1/2)\gamma/N\bigr)\right)}
  =
  1-e^{j\xi/N}
\end{equation}
since $\Re(j\gamma/N)=\frac{jb}{2N\pi}$, which is between $0$ and $1/2$.
\par
The colored Jones polynomial $J_N(\FE;q)$ of the figure-eight knot $\FE$ has the following simple formula due to K.~Habiro \cite{Habiro:SURIK2000} and T.~Le \cite{Le:TOPOA2003}.
\begin{equation}\label{eq:JN}
  J_N(\FE;q)
  =
  \sum_{k=0}^{N-1}q^{-kN}\prod_{l=1}^{k}\left(1-q^{N-l}\right)\left(1-q^{N+l}\right).
\end{equation}
For a complex number $\xi=a+b\i$ with $a>0$ and $0<b<\pi/2$, we can express $J_N(\FE;e^{\xi/N})$ by using $T_N(z)$.
From \eqref{eq:TN_j} we have
\begin{equation*}
\begin{split}
  &\prod_{l=1}^{k}\left(1-q^{N-l}\right)\left(1-q^{N+l}\right)
  \\
  =&
  \prod_{l=1}^{k}
  \left(1-e^{(N-l)\xi/N}\right)
  \left(1-e^{(N+l)\xi/N}\right)
  \\
  =&
  \prod_{l=1}^{k}
  \frac{\exp\left(T_N\bigl((N-l-1/2)\gamma/N\bigr)\right)}
       {\exp\left(T_N\bigl((N-l+1/2)\gamma/N\bigr)\right)}
  \frac{\exp\left(T_N\bigl((N+l-1/2)\gamma/N\bigr)\right)}
       {\exp\left(T_N\bigl((N+l+1/2)\gamma/N\bigr)\right)}
  \\
  =&
  \frac{\exp\left(T_N\bigl((N-k-1/2)\gamma/N\bigr)\right)}
       {\exp\left(T_N\bigl((N-1/2)\gamma/N\bigr)\right)}
  \frac{\exp\left(T_N\bigl((N+1/2)\gamma/N\bigr)\right)}
       {\exp\left(T_N\bigl((N+k+1/2)\gamma/N\bigr)\right)}.
\end{split}
\end{equation*}
Therefore we obtain
\begin{equation*}
\begin{split}
  &J_N\left(\FE;e^{\xi/N}\right)
  \\
  =&
  \frac{\exp\left(T_N\left(\left(1+\frac{1}{2N}\right)\gamma\right)\right)}
       {\exp\left(T_N\left(\left(1-\frac{1}{2N}\right)\gamma\right)\right)}
  \sum_{k=0}^{N-1}e^{-k\xi}
  \frac{\exp\left(T_N\left(\left(1-\frac{2k+1}{2N}\right)\gamma\right)\right)}
       {\exp\left(T_N\left(\left(1+\frac{2k+1}{2N}\right)\gamma\right)\right)}
  \\
  =&
  \frac{1}{1-e^{\xi}}
  \sum_{k=0}^{N-1}e^{-k\xi}
  \frac{\exp\left(T_N\left(\left(1-\frac{2k+1}{2N}\right)\gamma\right)\right)}
       {\exp\left(T_N\left(\left(1+\frac{2k+1}{2N}\right)\gamma\right)\right)},
\end{split}
\end{equation*}
where we use \eqref{eq:TN_gamma} with $z=\gamma$ in the second equality, noting that $\Re\gamma=\frac{b}{2\pi}$.
Putting
\begin{equation}\label{eq:fN_defn}
  f_N(z)
  :=
  \frac{1}{N}T_N\bigl(\gamma(1-z)\bigr)
  -
  \frac{1}{N}T_N\bigl(\gamma(1+z)\bigr)
  -\xi z+2\pi\i z,
\end{equation}
we have
\begin{equation}\label{eq:JN_fN}
  J_N\left(\FE;e^{\xi/N}\right)
  =
  \frac{1}{2\sinh(\xi/2)}
  \sum_{k=0}^{N-1}e^{Nf_N\bigl((2k+1)(2N)\bigr)}.
\end{equation}
\par
\begin{lem}\label{lem:domain_f}
The function $f_N(z)$ is defined in the following region:
\begin{equation*}
  \Theta
  :=
  \left\{
    z\in\C\Bigm|
    \bigl|\Im(\xi z)\bigr|<2\pi+\left(1-\frac{1}{2N}\right)b
  \right\}
  \setminus
  \left(
    \bigtriangleup^{+}
    \cup
    \bigtriangledown^{+}
    \cup
    \bigtriangleup^{-}
    \cup
    \bigtriangledown^{-}
  \right),
\end{equation*}
where
\begin{align*}
  \bigtriangleup^{+}
  :=&
  \left\{
    z\in\C\Bigm|
    \Re(\xi z)\ge a,\Im(\xi z)<2\pi+\left(1-\frac{1}{2N}\right)b,\Im{z}\ge0
  \right\},
  \\
  \bigtriangledown^{-}
  :=&
  \left\{
    z\in\C\Bigm|
    \Re(\xi z)\le a,
    -2\pi-\left(1-\frac{1}{2N}\right)b<\Im(\xi z),
    \Im{z}\le-\frac{2\pi a}{|\xi|^2}
  \right\},
  \\
  \bigtriangledown^{+}
  :=&
  \left\{
    z\in\C\Bigm|
    \Re(\xi z)\le-a,-2\pi-\left(1-\frac{1}{2N}\right)b<\Im(\xi z),\Im{z}\le0
  \right\},
  \\
  \bigtriangleup^{-}
  :=&
  \left\{
    z\in\C\Bigm|
    \Re(\xi z)\ge-a,
    \Im(\xi z)<2\pi+\left(1-\frac{1}{2N}\right)b,
    \Im{z}\ge\frac{2\pi a}{|\xi|^2}
  \right\}.
\end{align*}
See Figure~\ref{fig:domain_fN}.
Note that $\bigtriangleup^{\pm}$ is congruent to $\bigtriangledown^{\pm}$, and that $\bigtriangleup^{+}$ and $\bigtriangledown^{+}$ are similar to $\bigtriangleup^{-}$ and $\bigtriangledown^{-}$ respectively with similarity ratio $\pi/b$.
\begin{figure}[h]
\pic{0.2}{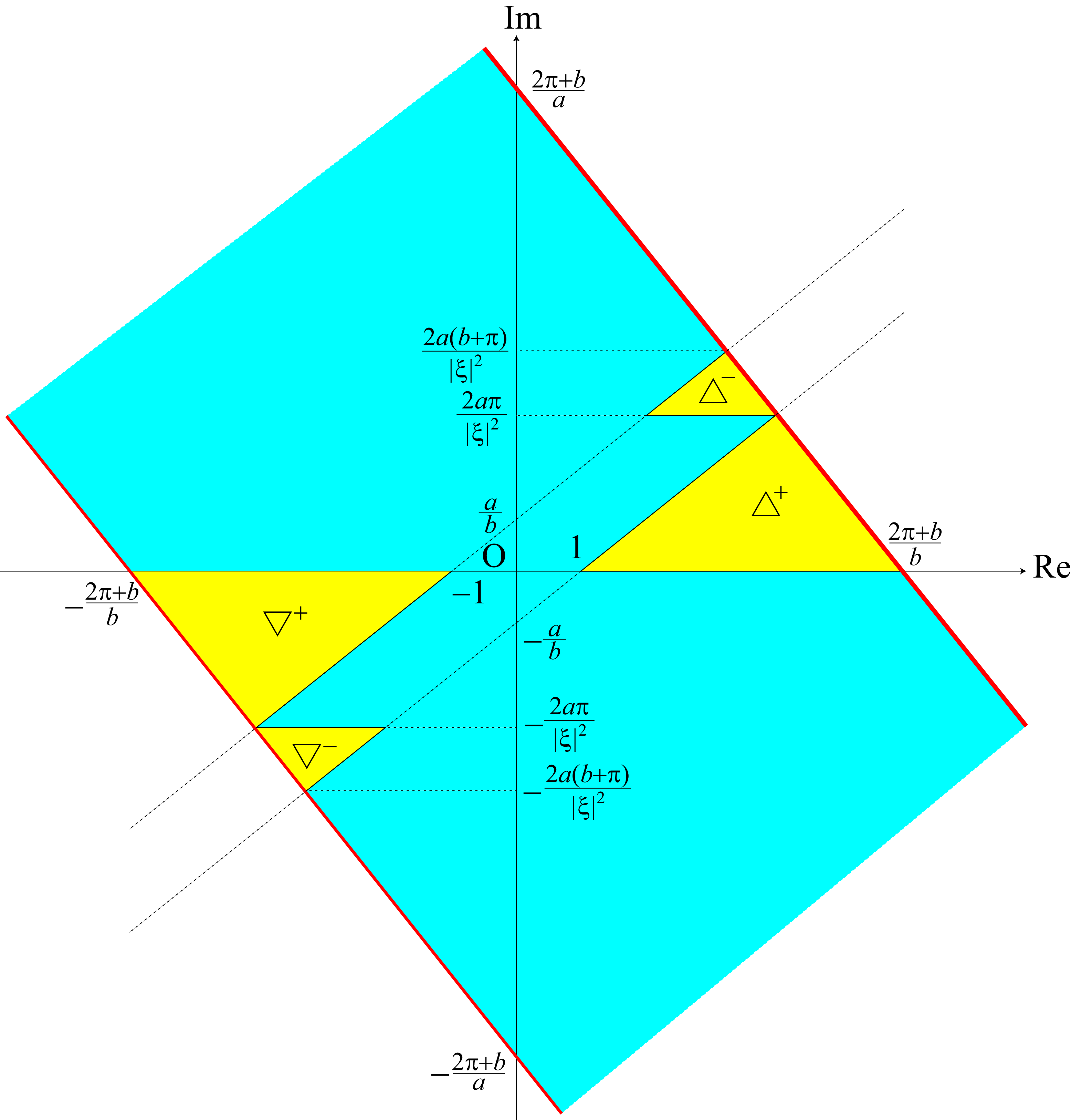}
\caption{The function $f_N(z)$ is defined in the region $\Theta$, which is between the two red lines except for the four yellow triangles.}
\label{fig:domain_fN}
\end{figure}
See Figure~\ref{fig:domain_fN}.
\end{lem}
\begin{proof}
We put $x:=\Re{z}$ and $y:=\Im{z}$.
\par
Since $\Re\gamma=\frac{b}{2\pi}$, $\Re\bigl(\gamma(1-z)\bigr)=\frac{b}{2\pi}-\frac{\Im(\xi z)}{2\pi}$, and $\Im\bigl(\gamma(1-z)\bigr)=-\frac{a}{2\pi}+\frac{\Re(\xi z)}{2\pi}$ we see that $\gamma(1-z)$ is in the region \eqref{eq:Sigma} if and only if
\begin{equation}\label{eq:region_minus}
\begin{split}
  &\left\{
    z\in\C\Bigm|
    -4\pi+\left(1+\frac{1}{2N}\right)b<\Im(\xi z)<2\pi+\left(1-\frac{1}{2N}\right)b
  \right\}
  \\
  &\quad
  \setminus
  \left\{
    z\in\C\Bigm|
    \Re(\xi z)\ge a,b\le\Im(\xi z)<2\pi+\left(1-\frac{1}{2N}\right)b,\Im{z}\ge0
  \right\}
  \\
  &\quad
  \setminus
  \left\{
    z\in\C\Bigm|
    \Re(\xi z)\le a,
    -4\pi+\left(1+\frac{1}{2N}\right)b<\Im(\xi z)\le-2\pi+b,
  \right.
  \\
  &\phantom{\setminus\left\{z\in\C\Bigm|\vphantom{\frac{1}{2N}}\right.}\quad
  \left.
    \Im{z}\le-\frac{2\pi a}{|\xi|^2}
  \right\}.
\end{split}
\end{equation}
By the symmetry $z\leftrightarrow-z$, we also see that $\gamma(1+z)$ is in the region \ref{eq:Sigma} if and only if
\begin{equation}\label{eq:region_plus}
\begin{split}
  &\left\{
    z\in\C\Bigm|
    -2\pi-\left(1-\frac{1}{2N}\right)b<\Im(\xi z)<4\pi-\left(1+\frac{1}{2N}\right)b
  \right\}
  \\
  &\quad
  \setminus
  \left\{
    z\in\C\Bigm|
    \Re(\xi z)\le-a,-2\pi-\left(1-\frac{1}{2N}\right)b<\Im(\xi z)\le-b,\Im{z}\le0
  \right\}
  \\
  &\quad
  \setminus
  \left\{
    z\in\C\Bigm|
    \Re(\xi z)\ge-a,
    2\pi-b\le\Im(\xi z)<4\pi-\left(1+\frac{1}{2N}\right)b,
  \right.
  \\
  &\phantom{\setminus\left\{z\in\C\Bigm|\vphantom{\frac{1}{2N}}\right.}\quad
  \left.
    \Im{z}\ge\frac{2\pi a}{|\xi|^2}
  \right\}.
\end{split}
\end{equation}
\par
Therefore the domain of $f_N(z)$ is $\left\{z\in\C\Bigm|\bigl|\Im(\xi z)\bigr|<2\pi+\left(1-\frac{1}{2N}\right)b\right\}$ minus the sets in the second, third and fourth lines in \eqref{eq:region_minus} and \eqref{eq:region_plus}.
\par
Since $\Re(\xi z)\ge a$ implies that $ax-by\ge a$, we have $x-1\ge by/a$.
So if $y\ge0$, we have
\begin{equation*}
  \Im(\xi z)-b
  =
  b(x-1)+ay
  \ge
  \frac{|\xi|^2}{a}y
  \ge0.
\end{equation*}
Therefore we do not need $\Im(\xi z)\ge b$ in the second line in \eqref{eq:region_minus}.
\par
Similarly, if $\Re(\xi z)\le a$, we have $x-1\le by/a$.
So if $y\le-2\pi a/|\xi|^2$, we have
\begin{equation*}
  \Im(\xi z)-(-2\pi+b)
  =
  b(x-1)+ay+2\pi
  \le
  \frac{b^2y}{a}+ay+2\pi
  =
  \frac{|\xi|^2}{a}y+2\pi
  \le
  0,
\end{equation*}
and we do not need  $\Im(\xi z)\le-2\pi+b$ in the third and fourth lines in \eqref{eq:region_minus}.
Since $-4\pi+(1+1/(2N))b$ is less than $-2\pi-(1-1/(2N))b$, we can replace the set in the third and fourth lines in \eqref{eq:region_minus} with $\bigtriangledown^{-}$.
\par
Similarly, we can replace the set in the second, third, and fourth lines in \eqref{eq:region_plus} with $\bigtriangledown^{+}$ and $\bigtriangleup^{-}$.
\par
From the symmetry, it is clear that $\bigtriangleup^{\pm}$ is congruent to $\bigtriangledown^{\pm}$.
The triangles $\bigtriangleup^{+}$ is similar to $\bigtriangleup^{-}$, because they share the upper red line in Figure~\ref{fig:domain_fN}, and the other edges are parallel.
Since the height of $\bigtriangleup^{+}$ is $2a\pi/|\xi|^2$, and that of $\bigtriangleup^{-}$ is $2a(b+\pi)/|\xi|^2-2a\pi/|\xi|^2=2ab/|\xi|^2$, the similarity ratio is $\pi/b$.
\par
This completes the proof.
\end{proof}
\par
From Lemma~\ref{lem:converge_T}, the series $\{\frac{1}{N}T_N(w)\}_{N=2,3,\ldots}$ uniformly converges to $\frac{1}{\xi}\L_2(w)$ in the compact region $\Sigma_{\nu}$.
So we can show the following lemma, whose proof is omitted since it is similar way to that of Lemma~\ref{lem:domain_f}.
\begin{lem}\label{lem:converge_F}
The series of functions $\{f_N(z)\}_{N=2,3,4,\dots}$ uniformly converges to
\begin{equation}\label{eq:F_defn}
  F(z)
  :=
  \frac{1}{\xi}\L_2\bigl(\gamma(1-z)\bigr)
  -
  \frac{1}{\xi}\L_2\bigl(\gamma(1+z)\bigr)
  -\xi z+2\pi\i z
\end{equation}
in the region
\begin{multline}\label{eq:Theta_nu}
  \Theta_{\nu}
  :=
  \{z\in\C\mid|\Im(\xi z)|\le b+2\pi(1-\nu),|\Re(\xi z)|\le2M\pi-a\}
  \\
  \setminus
  \left(
    \bigtriangleup^{+}_{\nu}
    \cup
    \bigtriangledown^{+}_{\nu}
    \cup
    \bigtriangleup^{-}_{\nu}
    \cup
    \bigtriangledown^{-}_{\nu}
  \right),
\end{multline}
where $M$ and $\nu$ are positive numbers given in Lemma~\ref{lem:converge_T}, and
\begin{align*}
  \bigtriangleup^{+}_{\nu}
  :=&
  \{z\in\C\mid
    \Re(\xi z)>a-2\pi\nu,b-2\pi\nu<\Im(\xi z)\le b+2\pi(1-\nu),
  \\
  &\phantom{\{z\in\C\mid}\quad
    \Im{z}>-2\pi\nu a/|\xi|^2
  \},
  \\
  \bigtriangledown^{-}_{\nu}
  :=&
  \{
    z\in\C\mid
    \Re(\xi z)<a+2\pi\nu,
    -b-2\pi(1-\nu)\le\Im(\xi z)<b-2\pi(1-\nu),
  \\
  &\phantom{\{z\in\C\mid}\quad
    \Im{z}<-2\pi(1-\nu)a/|\xi|^2
  \},
  \\
  \bigtriangledown^{+}_{\nu}
  :=&
  \{
    z\in\C\mid
    \Re(\xi z)<-a+2\pi\nu,-b-2\pi(1-\nu)\le\Im(\xi z)<-b+2\pi\nu,
  \\
  &\phantom{\{z\in\C\mid}\quad
    \Im{z}<2\pi\nu a/|\xi|^2
  \},
  \\
  \bigtriangleup^{-}_{\nu}
  :=&
  \{
    z\in\C\mid
    \Re(\xi z)>-a-2\pi\nu,
    -b+2\pi(1-\nu)<\Im(\xi z)\le b+2\pi(1-\nu),
  \\
  &\phantom{\{z\in\C\mid}\quad
    \Im{z}>2\pi(1-\nu)a/|\xi|^2
  \}.
\end{align*}
\begin{figure}[h]
\pic{0.2}{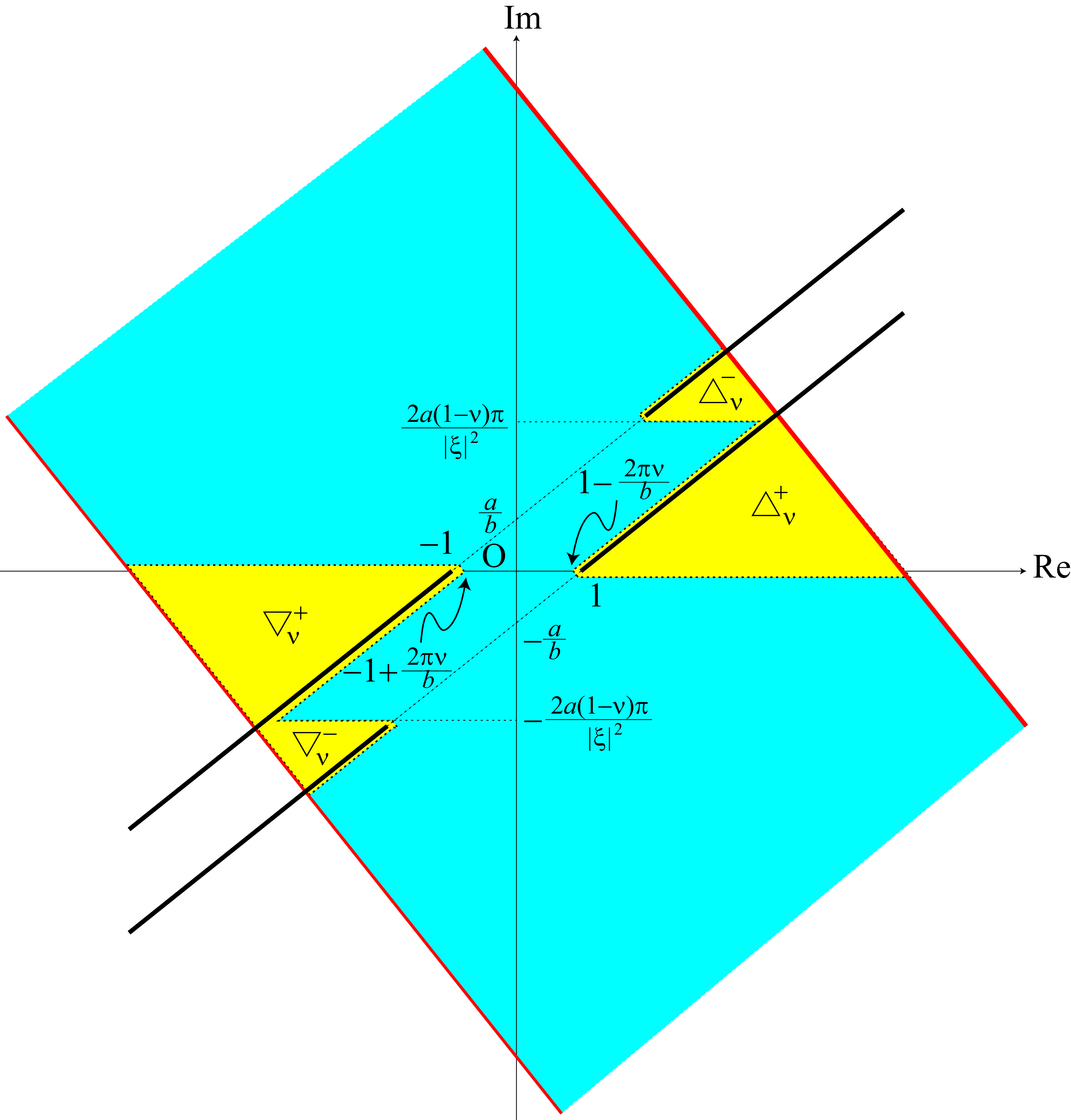}
\caption{The region $\Theta_{\nu}$ is between the two red lines except for the yellow parallelograms, which are neighborhoods of yellow triangles in Figure~\ref{fig:domain_fN}.
Note that $\bigtriangleup_{\nu}^{+}$ and $\bigtriangleup_{\nu}^{-}$ overlap, and $\bigtriangledown_{\nu}^{+}$ and $\bigtriangledown_{\nu}^{-}$ also overlap.
The function $F(z)$ is holomorphic in $\C$ minus the four black half lines.}
\label{fig:converge_F}
\end{figure}
See Figure~\ref{fig:converge_F}.
Compare it with Figure~\ref{fig:domain_fN} noting that $\bigtriangledown^{\pm}_{\nu}$ and $\bigtriangleup^{\pm}_{\nu}$ are neighborhoods of $\bigtriangledown^{\pm}$ and $\bigtriangleup^{\pm}$ respectively.
\end{lem}
\begin{rem}
Since $\L_2(z)$ is holomorphic in the region $\C\setminus\bigl((-\infty,0]\cup[1,\infty)\bigr)$, the function $F(z)$ is holomorphic in the region $\C\setminus\{t/\gamma-1\mid \text{$t\le0$ or $t\ge1$}\}\cup\{-t/\gamma+1\mid\text{$t\le0$ or $t\ge1$}\}$, which is the complement of the four thick lines in Figure~\ref{fig:converge_F}.
\end{rem}
From Lemmas~\ref{lem:converge_T} and \ref{lem:TN'}, we have the following lemma.
\begin{lem}\label{lem:fN_F}
We have
\begin{align*}
  f_N(z)
  =
  F(z)+O(N^{-2}),
  \\
  f'_N(z)
  =
  F'(z)+O(N^{-2})
\end{align*}
as $N\to\infty$.
\end{lem}
\begin{proof}
The first asymptotic formula follows immediately from Lemma~\ref{lem:converge_T}.
\par
As for the second, on one hand from \eqref{eq:fN_defn}, we have
\begin{equation*}
\begin{split}
  f'_N(z)
  =&
  -\frac{\gamma}{N}T'_N\bigl(\gamma(1-z)\bigr)
  -\frac{\gamma}{N}T'_N\bigl(\gamma(1+z)\bigr)
  -\xi+2\pi\i
  \\
  =&
  \L_1\bigl(\gamma(1-z)\bigr)
  +
  \L_1\bigl(\gamma(1+z)\bigr)
  -\xi+2\pi\i+O(N^{-2}),
\end{split}
\end{equation*}
where we use Lemma~\ref{lem:TN'} in the second equality.
On the other hand, we have from \eqref{eq:F_defn} and \eqref{eq:L2_der}
\begin{equation}\label{eq:F'}
\begin{split}
  F'(z)
  =&
  -\frac{1}{2\pi\i}\L'_2\bigl(\gamma(1-z)\bigr)
  -\frac{1}{2\pi\i}\L'_2\bigl(\gamma(1+z)\bigr)
  -\xi+2\pi\i
  \\
  =&
  \L_1\bigl(\gamma(1-z)\bigr)
  +
  \L_1\bigl(\gamma(1+z)\bigr)
  -\xi+2\pi\i.
\end{split}
\end{equation}
Therefore we have $f'_N(z)=F'(z)+O(N^{-2})$, and the proof is complete.
\end{proof}
\par
It is sometimes convenient to use another variable $Z:=\xi z$.
We denote the corresponding function as $g_N(Z):=f_N(Z/\xi)$, that is,
\begin{equation*}
  g_N(Z)
  :=
  \frac{1}{N}T_N\left(\frac{\xi-Z}{2\pi\i}\right)
  -
  \frac{1}{N}T_N\left(\frac{\xi+Z}{2\pi\i}\right)
  -Z+\frac{2\pi\i Z}{\xi},
\end{equation*}
The function $g_N(Z)$ is defined in $\xi\Theta:=\{\xi z\in\C\mid z\in\Theta\}$.
\par
Since the series of functions $\{f_N(z)\}_{N=2,3,4,\dots}$ uniformly converges to $F(z)$, the series of functions $\{g_N(Z)\}_{N=2,3,4,\dots}$ uniformly converges to
\begin{equation*}
  G(Z)
  :=
  \frac{1}{\xi}\L_2\left(\frac{\xi-Z}{2\pi\i}\right)
  -
  \frac{1}{\xi}\L_2\left(\frac{\xi+Z}{2\pi\i}\right)
  -Z+\frac{2\pi\i Z}{\xi}
\end{equation*}
in the region $\xi\Theta_{\nu}$.
\par
We can express $G(Z)$ in terms of $\Li_2(w)$ in some cases.
\begin{lem}\label{lem:G}
If $|\Re{Z}|<a$, then we have
\begin{equation*}
  G(Z)
  =
  \frac{1}{\xi}\Li_2\left(e^{-\xi-Z}\right)
  -
  \frac{1}{\xi}\Li_2\left(e^{-\xi+Z}\right)
  +Z.
\end{equation*}
\end{lem}
\begin{proof}
If $|\Re{Z}|<a$, then $\Im\frac{\xi\pm Z}{2\pi\i}=-\frac{1}{2\pi}(a\pm\Re{Z})<0$.
So, from \eqref{eq:L1_L2}, we have
\begin{equation*}
\begin{split}
  G(Z)
  =&
  -\frac{1}{\xi}\Li_2(e^{-\xi+Z})+\frac{1}{\xi}\Li_2(e^{-\xi-Z})
  -Z+\frac{2\pi\i Z}{\xi}
  \\
  &+
  \frac{2\pi^2}{\xi}
  \left(
    \left(\frac{\xi-Z}{2\pi\i}\right)^2
    -
    \left(\frac{\xi+Z}{2\pi\i}\right)^2
  \right)
  -\frac{2\pi^2}{\xi}
  \left(
    \frac{\xi-Z}{2\pi\i}
    -
    \frac{\xi+Z}{2\pi\i}
  \right)
  \\
  =&
  \frac{1}{\xi}\Li_2\left(e^{-\xi-Z}\right)
  -
  \frac{1}{\xi}\Li_2\left(e^{-\xi+Z}\right)
  +Z,
\end{split}
\end{equation*}
and the proof is complete
\end{proof}
As for the derivative, we have the following formulas when $\Re{z}>-a$.
\begin{lem}\label{lem:G_der}
If $|\Re{Z}|<a$, then we have
\begin{equation*}
  \frac{d}{d\,Z}G(Z)
  =
  \frac{1}{\xi}\log\bigl(2\cosh\xi-2\cosh{Z}\bigr).
\end{equation*}
If $\Re{Z}\ge a$, then we have
\begin{equation*}
  \frac{d}{d\,Z}G(Z)
  =
  \frac{1}{\xi}
  \bigl(
    \log\left(1-e^{-\xi-Z}\right)
    +
    \log\left(1-e^{\xi-Z}\right)
    +Z+\pi\i
  \bigr).
\end{equation*}
\end{lem}
\begin{proof}
If $|\Re{Z}|<a$, from Lemma~\ref{lem:G}, we have
\begin{equation*}
  \frac{d}{d\,Z}G(Z)
  =
  \frac{1}{\xi}
  \bigl(
    \log\left(1-e^{-\xi-Z}\right)
    +
    \log\left(1-e^{-\xi+Z}\right)
    +\xi
  \bigr)
  =
  \frac{1}{\xi}\log\bigl(2\cosh\xi-2\cosh{Z}\bigr),
\end{equation*}
there the second equality follows from Lemma~\ref{lem:arg} below.
\par
Now, we consider the case where $\Re{Z}\ge a$.
From \eqref{eq:L2_der}, we have
\begin{equation*}
  \frac{d}{d\,Z}G(Z)
  =
  \frac{1}{\xi}
  \left(
    \L_1\left(\frac{\xi+Z}{2\pi\i}\right)
    +
    \L_1\left(\frac{\xi-Z}{2\pi\i}\right)
    -\xi+2\pi\i
  \right)
\end{equation*}
\par
Since $\Re{Z}\ge a$, we have $\Im\left(\frac{\xi+Z}{2\pi\i}\right)=\frac{-1}{2\pi}(a+\Re{Z})<0$ and $\Im\left(\frac{\xi-Z}{2\pi\i}\right)=\frac{-1}{2\pi}(a-\Re{Z})\ge0$.
So from \eqref{eq:L1_L2} we obtain the desired formula.
\end{proof}
Since $F(z)=G(\xi z)$, we have the following corollary to Lemmas~\ref{lem:G} and \ref{lem:G_der}.
\begin{cor}\label{cor:F}
If $|\Re(\xi z)|<a$, then we have
\begin{equation*}
  F(z)
  =
  \frac{1}{\xi}\Li_2\left(e^{-\xi-\xi z}\right)
  -
  \frac{1}{\xi}\Li_2\left(e^{-\xi+\xi z}\right)
  +
  \xi z,
\end{equation*}
and
\begin{equation*}
  \frac{d}{d\,z}F(z)
  =
  \log\bigl(2\cosh{\xi}-2\cosh(\xi z)\bigr).
\end{equation*}
\par
If $\Re(\xi z)\ge a$, then we have
\begin{equation*}
  \frac{d}{d\,z}F(z)
  =
  \log\left(1-e^{-\xi-\xi z}\right)
  +
  \log\left(1-e^{\xi-\xi z}\right)
  +\xi z+\pi\i.
\end{equation*}
\end{cor}
The following lemma is used in the proof of Lemma~\ref{lem:G_der}.
\begin{lem}\label{lem:arg}
Put
\begin{align*}
  D_1(Z)
  &:=\arg(1-e^{-\xi-Z})+\arg(1-e^{-\xi+Z})+b,
  \\
  D_2(Z)
  &:=\arg(1-e^{-\xi-Z})+\arg(1-e^{\xi-Z})+\Im{Z}+\pi.
\end{align*}
If $|\Re{Z}|<a$, then we have $|D_1(Z)|<\pi$.
If $\Re{Z}\ge a$ and $\Im{Z}\le b$, then we have $D_2(Z)<3\pi/2$.
\end{lem}
The proof of the lemma is elementary but complicated.
So we postpone it to Appendix~\ref{sec:lemmas}.
\par
Lemma~\ref{lem:arg} has another application.
Together with Corollary~\ref{cor:F}, we have the following result, which will be used in the proof of Lemma~\ref{lem:Poisson_iv}.
\begin{cor}\label{cor:arg}
Put $x:=\Re{z}$, and $y:=\Im{z}$.
If $0\le x\le1$ and $y\le0$, then $\frac{\partial}{\partial\,y}\Re F(z)>-3\pi/2$.
\end{cor}
\begin{proof}
From Corollary~\ref{cor:F}, the Cauchy--Riemann equations implies
\begin{equation*}
\begin{split}
  &\frac{\partial}{\partial\,y}\Re F(z)
  =
  -\frac{\partial}{\partial\,x}\Im F(z)
  \\
  =&
  \begin{cases}
    -\arg\bigl(2\cosh\xi-2\cosh(\xi z)\bigr)
    &\quad\text{(if $|\Re(\xi z)|<a$),}
    \\
    -D_2(\xi z)
    &\quad\text{(if $\Re(\xi z)\ge a$).}
  \end{cases}
\end{split}
\end{equation*}
\par
If $|\Re(\xi z)|<a$, then it is clear that $\frac{\partial}{\partial\,y}\Re F(z)>-3\pi/2$ from the formula above.
\par
If $\Re(\xi z)\le-a$, then since $\Re(\xi z)=ax-by\le-a$ and $y\le0$, we have $0\ge y\ge\frac{a}{b}(x+1)$.
This is impossible since we are assuming $x\ge0$.
\par
If $\Re(\xi z)\ge a$, then since $ax-by\ge a$, we have $y\le\frac{a}{b}(x-1)$.
The assumption $x\le1$ implies $\Im(\xi z)=bx+ay\le b$, and we have $\frac{\partial}{\partial\,y}\Re F(z)=-D_2(\xi z)>-3\pi/2$ from Lemma~\ref{lem:arg}.
\par
So, it is clear that $\frac{\partial}{\partial\,y}\Re F(z)>-3\pi/2$ if $|\Re(\xi z)|<a$.
\par
Suppose that $\Re(\xi z)\ge a$ and $\Im(\xi z)\le b$, then $\frac{\partial}{\partial\,y}\Re F(z)=-D_2(\xi z)>-3\pi/2$ from Lemma~\ref{lem:arg}.
\par
\end{proof}

\section{Regions and curves}\label{sec:F_G}
In this section, we define several regions and curves both in the $z$-plane and the $Z$-plane, and study their properties, where $z$ is used as the arguments of $f_N(z)$ and $F(z)$, and $Z$ is used as the arguments of $g_N(Z)$ and $G(Z)$.
\par
Recall that we put $\xi=a+b\i$, $\varphi=c+d\i$, $\alpha=\cosh{a}\cos{b}$, $\beta=\sinh{a}\sin{b}$, where $\varphi:=\varphi(\xi)$ is defined as \eqref{eq:def_phi}.
We will often use $\varphi$ instead of $\varphi(\xi)$ as above, if the dependence of $\varphi$ on $\xi$ is clear.
Throughout the section we assume that $\xi\in\Xi$ (see \eqref{eq:Xi}), that is, $a>0$, $0<b<\pi/2$, and $\alpha=\cosh{a}\cos{b}>1/2$.
%%%%%%%%%%%%%%%%%%%%%%%%%%%%%%%%%%%%%%%%%%%%%%%%%%%%%%%%%%%%%%%%%%%%%%%%%%%%%%%
\subsection{Locations of $\varphi$ and $\sigma$}
We put
\begin{equation}\label{eq:def_sigma}
  \sigma(\xi):=\varphi(\xi)/\xi.
\end{equation}
We will often use $\sigma$ instead of $\sigma(\xi)$.
We first locate $\varphi$ in the $Z$-plane and $\sigma$ in the $z$-plane.
We write $z=x+y\i$ and $Z=X+Y\i$ with $x,y,X,Y\in\R$.
\par
From \eqref{eq:sqrt}, we have $\Im\bigl((2\cosh\xi-3)(2\cosh\xi+1)\bigr)=4\beta(2\alpha-1)>0$ since we assume $\alpha>1/2$.
So the square root in \eqref{eq:def_phi} is in the first quadrant.
Since we also have $\Im(\cosh\xi-1/2)=\sinh{a}\sin{b}>0$, the imaginary part of the argument of $\log$ in \eqref{eq:def_phi} is positive when $z=\xi$.
Therefore we have the following inequalities.
\begin{equation}\label{eq:d_pi}
  0<d<\pi.
\end{equation}
\par
Taking the real parts and the imaginary parts of the equality $\cosh\varphi=\cosh\xi-1/2$, we have
\begin{align}
  \cosh{c}\cos{d}&=\alpha-1/2,
  \label{eq:Re_phi}
  \\
  \sinh{c}\sin{d}&=\beta.
  \label{eq:Im_phi}
\end{align}
Now, we have the following lemma.
\begin{lem}\label{lem:phi_xi}
We have $\arsinh\beta<c<a$ and $b<d<\pi/2$.
\end{lem}
\begin{proof}
Since $\alpha>1/2$, we have $\cos{d}>0$ from \eqref{eq:Re_phi}.
From \eqref{eq:d_pi} we conclude that $d<\pi/2$.
\par
Since $a>0$ and $0<b<\pi/2$, we have $\beta=\sinh{a}\sin{b}>0$.
Moreover, we have $0<\sin{d}<1$ since $0<d<\pi/2$, and so we see that $\sinh{c}>\beta$ from \eqref{eq:Im_phi}, that is, $c>\arsinh\beta$.
\par
Since $\cosh{a}=\sqrt{1+\sinh^2{a}}$ and $\cos{b}=\sqrt{1-\sin^2{b}}$, we have $\alpha=\sqrt{1+\sinh^2{a}}\sqrt{1-\sin^2{b}}$ and $\beta=\sinh{a}\sin{b}$.
Eliminating $\sin{b}$ from these equations, we obtain
\begin{equation*}
  \sinh{a}
  =
  \sqrt{\frac{\alpha^2+\beta^2-1+\sqrt{(\alpha^2+\beta^2-1)^2+4\beta^2}}{2}}.
\end{equation*}
Similarly, from \eqref{eq:Re_phi} and \eqref{eq:Im_phi}, we have
\begin{equation*}
  \sinh{c}
  =
  \sqrt{\frac{(\alpha-1/2)^2+\beta^2-1+\sqrt{\bigl((\alpha-1/2)^2+\beta^2-1\bigr)^2+4\beta^2}}{2}}.
\end{equation*}
These two formulas imply that $\sinh{a}>\sinh{c}$, that is, $a>c$.
\par
From \eqref{eq:Im_phi}, we have $\sin{d}=\beta/\sinh{c}>\beta/\sinh{a}=\sin{b}$.
Since both $b$ and $d$ are between $0$ and $\pi/2$, we conclude that $b<d$, completing the proof.
\end{proof}
From Corollary~\ref{cor:F} and Lemma~\ref{lem:G_der}, we have $\frac{d}{d\,z}F(\sigma)=\frac{d}{d\,Z}G(\varphi)=0$ because $\Re\varphi=c<a$ and $\cosh\varphi=\cosh\xi-1/2$.
\begin{rem}
Since $\cos{b}=\alpha/\cosh{a}<\alpha$, we have $b>\arccos\alpha$.
\end{rem}
\par
Since $\sigma=\varphi/\xi$ we have the following corollary.
\begin{cor}
We have $\arsinh\beta<\Re(\xi\sigma)<a$ and $b<\Im(\xi\sigma)<\pi/2$, that is, $\sigma$ is in the rectangle $\{z\in\C\mid\arsinh\beta<\Re(\xi\sigma)<a,b<\Im(\xi\sigma)<\pi/2\}$.
In particular, $\sigma$ is in the first quadrant of the $z$-plane.
See Figure~\ref{fig:sigma}.
\begin{figure}[h]
\pic{0.3}{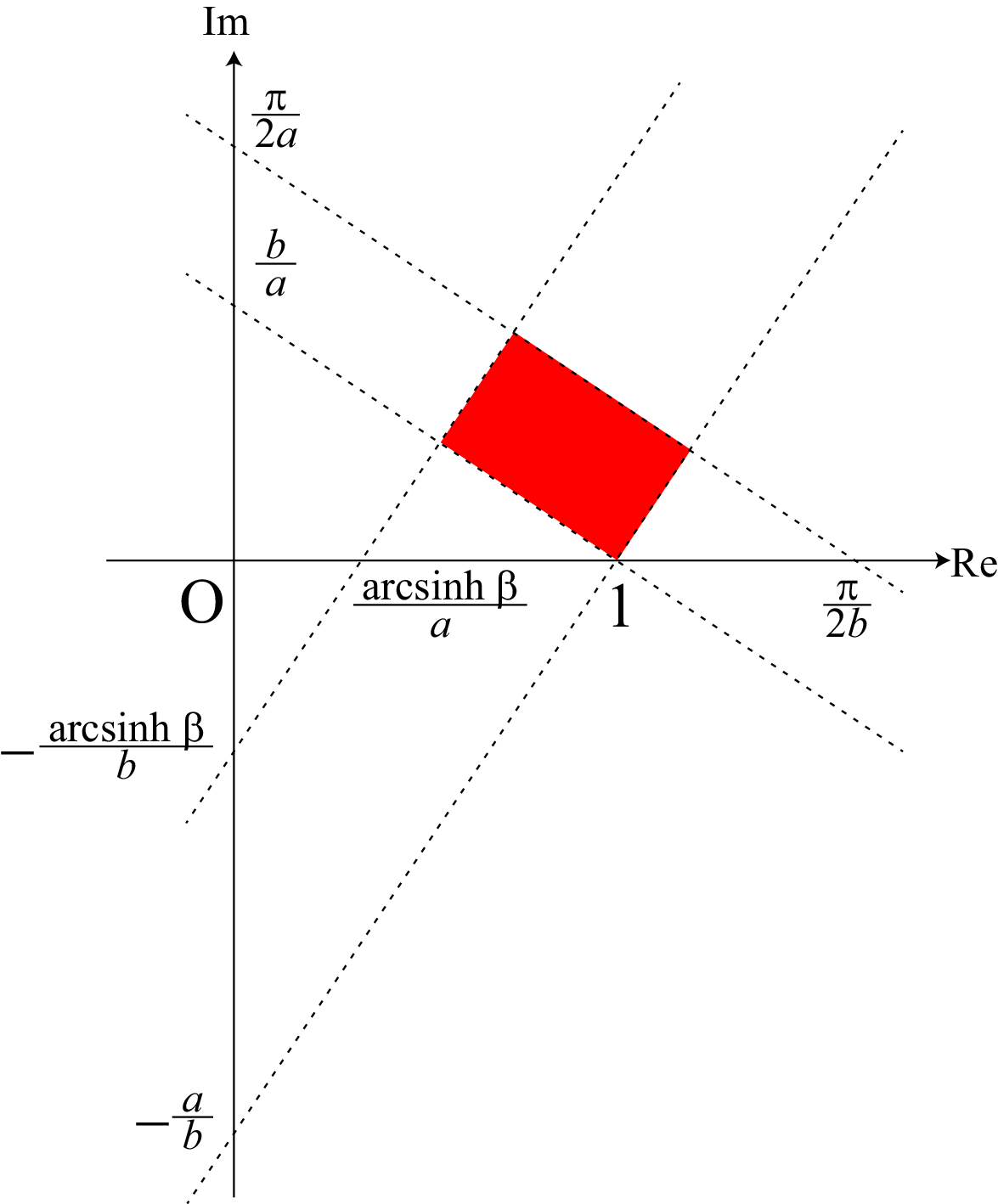}
\caption{$\sigma$ is in the red rectangle.}
\label{fig:sigma}
\end{figure}
\end{cor}
%%%%%%%%%%%%%%%%%%%%%%%%%%%%%%%%%%%%%%%%%%%%%%%%%%%%%%%%%%%%%%%%%%%%%
\subsection{Partial derivatives of $\Re F(z)$}
In this subsection, we study the partial derivatives of $\Re F(z)$ with respect to $x:=\Re{z}$ and $y:=\Im{z}$.
We are assuming that $a>0$, $0<b<\pi/2$, and $\alpha=\cosh{a}\cos{b}>1/2$.
\par
Define four regions $H^{\pm}_G$ and $V^{\pm}_G$ in the $Z$-plane as follows:
\begin{equation*}
\begin{split}
  H^{+}_G
  &:=
  \{Z\in\C\mid|\cosh{Z}-\cosh\xi|>1/2,-\pi<\Im{Z}\le\pi\},
  \\
  H^{-}_G
  &:=
  \{Z\in\C\mid|\cosh{Z}-\cosh\xi|<1/2,-\pi<\Im{Z}\le\pi\},
  \\
  V^{+}_G
  &:=
  \{Z\in\C\mid\arg(\cosh{Z}-\cosh\xi)>0,-\pi<\Im{Z}\le\pi\},
  \\
  V^{-}_G
  &:=
  \{Z\in\C\mid\arg(\cosh{Z}-\cosh\xi)<0,-\pi<\Im{Z}\le\pi\}.
\end{split}
\end{equation*}
Their boundaries are denoted by $H^{0}_G$ and $V^{0}_G$, that is, we put
\begin{equation*}
\begin{split}
  H^{0}_G
  &:=
  \{Z\in\C\mid|\cosh{Z}-\cosh\xi|=1/2,-\pi<\Im{Z}\le\pi\},
  \\
  V^{0}_G
  &:=
  \{Z\in\C\mid\arg(\cosh{Z}-\cosh\xi)=0,-\pi<\Im{Z}\le\pi\}.
\end{split}
\end{equation*}
Note that $\xi\in H^{-}_{G}\cap V^{0}_{G}$, and that $\varphi\in H^{0}_{G}\cap V^{0}_{G}$.
See Figure~\ref{fig:HG_VG}.
\begin{figure}[h]
\begin{minipage}{39mm}
\begin{center}
  \pic{0.5}{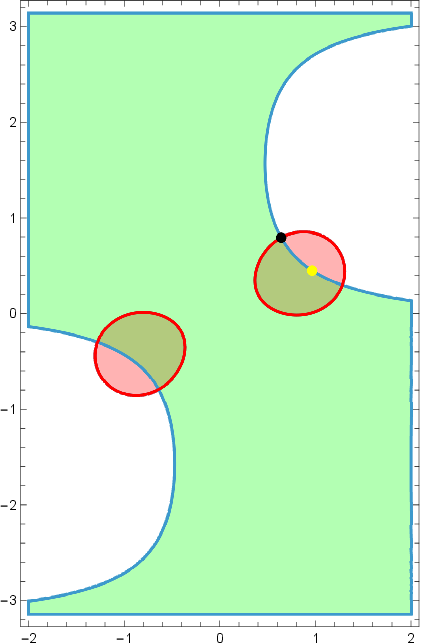}
  \\
  $(\cosh{a},\cos{b})=(1.5,0.9)$
\end{center}
\end{minipage}
\quad
\begin{minipage}{39mm}
\begin{center}
  \pic{0.5}{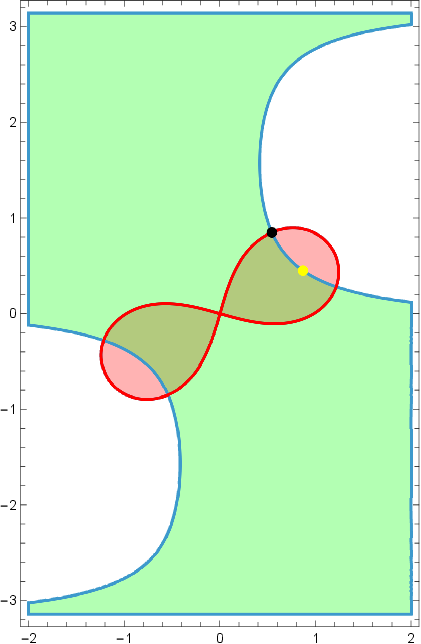}
  \\
  $(\cosh{a},\cos{b})=(1.4,0.9)$
\end{center}
\end{minipage}
\quad
\begin{minipage}{39mm}
\begin{center}
  \pic{0.5}{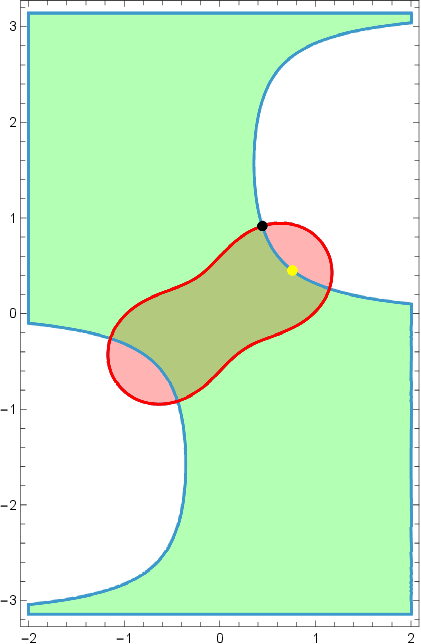}
  \\
  $(\cosh{a},\cos{b})=(1.3,0.9)$
\end{center}
\end{minipage}
\caption{The regions $H^{-}_G$ and $V^{-}_G$ are indicated by red and green respectively.
The black points are $\varphi$, and the yellow points are $\xi$.}
\label{fig:HG_VG}
\end{figure}
\par
We also define the regions $H^{\pm}_F$ and $V^{\pm}_F$ in the $z$-plane as $\xi^{-1}H^{\pm}_G$ and $\xi^{-1}V^{\pm}_G$, respectively.
So we have
\begin{equation}\label{eq:HVF_def}
\begin{split}
  H^{+}_F
  &=
  \{z\in\C\mid|\cosh(\xi z)-\cosh\xi|>1/2,-\pi<\Im(\xi z)\le\pi\},
  \\
  H^{-}_F
  &=
  \{z\in\C\mid|\cosh(\xi z)-\cosh\xi|<1/2,-\pi<\Im(\xi z)\le\pi\},
  \\
  V^{+}_F
  &=
  \{z\in\C\mid\arg(\cosh(\xi z)-\cosh\xi)>0,-\pi<\Im(\xi z)\le\pi\},
  \\
  V^{-}_F
  &=
  \{z\in\C\mid\arg(\cosh(\xi z)-\cosh\xi)<0,-\pi<\Im(\xi z)\le\pi\}.
\end{split}
\end{equation}
Their boundaries are denoted by $H^{0}_F:=\xi^{-1}H^{0}_G$ and $V^{0}_F:=\xi^{-1}V^{0}_G$, that is, we put
\begin{align*}
  H^{0}_F
  &=
  \{z\in\C\mid|\cosh(\xi z)-\cosh\xi|=1/2,-\pi<\Im(\xi z)\le\pi\},
  \\
  V^{0}_F
  &=
  \{z\in\C\mid\arg(\cosh(\xi z)-\cosh\xi)=0,-\pi<\Im(\xi z)\le\pi\}.
\end{align*}
\par
Putting $Z=X+Y\i$, $\alpha=\cos{a}\cos{b}$, and $\beta=\sinh{a}\sin{b}$ (see \eqref{eq:alpha} and \eqref{eq:beta}), the regions $H^{\pm}_G$ and $V^{\pm}_G$, and the curves $H^{0}_G$ and $V^{0}_G$ can be expressed as
\begin{equation}\label{eq:HV}
\begin{split}
  H^{+}_G
  =&
  \bigl\{(X,Y)\in\R^2\mid(\cosh{X}\cos{Y}-\alpha)^2+(\sinh{X}\sin{Y}-\beta)^2>1/4,
  \\
  &\phantom{\{(X,Y)\in\R^2\mid}
  -\pi<Y\le\pi\bigr\},
  \\
  H^{-}_G
  =&
  \bigl\{(X,Y)\in\R^2\mid(\cosh{X}\cos{Y}-\alpha)^2+(\sinh{X}\sin{Y}-\beta)^2<1/4,
  \\
  &\phantom{\{(X,Y)\in\R^2\mid}
  -\pi<Y\le\pi\bigr\},
  \\
  H^{0}_G
  =&
  \bigl\{(X,Y)\in\R^2\mid(\cosh{X}\cos{Y}-\alpha)^2+(\sinh{X}\sin{Y}-\beta)^2=1/4,
  \\
  &\phantom{\{(X,Y)\in\R^2\mid}
  -\pi<Y\le\pi\bigr\}
  \\
  V^{+}_G
  =&
  \bigl\{(X,Y)\in\R^2\mid\sinh{X}\sin{Y}>\beta,-\pi<Y\le\pi\bigr\},
  \\
  V^{-}_G
  =&
  \bigl\{(X,Y)\in\R^2\mid\sinh{X}\sin{Y}<\beta,-\pi<Y\le\pi\bigr\},
  \\
  V^{0}_G
  =&
  \bigl\{(X,Y)\in\R^2\mid\sinh{X}\sin{Y}=\beta,-\pi<Y\le\pi\bigr\},
\end{split}
\end{equation}
since we have
\begin{equation*}
  \cosh{Z}-\cosh\xi
  =
  \cosh{X}\cos{Y}-\alpha
  +
  \i(\sinh{X}\sin{Y}-\beta).
\end{equation*}
\par
If we put
\begin{equation}\label{eq:chi}
  \chi(Y)
  :=
  \arsinh\left(\frac{\beta}{\sin{Y}}\right)
\end{equation}
for $0<|Y|<\pi$, from \eqref{eq:HV}, the region $V^{-}_G$ is also expressed as
\begin{equation*}
\begin{split}
  V^{-}_G
  =&
  \left\{(X,Y)\in\R^2\mid-\pi<Y<0,X>\chi(Y)\right\}
  \\
  &\cup
  \left\{(X,Y)\in\R^2\mid Y=0\right\}\cup\left\{(X,Y)\in\R^2\mid Y=\pi\right\}
  \\
  &\cup
  \left\{(X,Y)\in\R^2\mid 0<Y<\pi,X<\chi(Y)\right\}.
\end{split}
\end{equation*}
Using the function $\chi$, we define a curve as follows.
\begin{defn}\label{defn:chi_F_G}
Let $\chi_G$ be the curve $\{\chi(t)+t\i\in\C\mid b\le t\le d\}$ in the $Z$-plane.
Note that $\chi_G$ is a part of the boundary $V^{0}_{G}$ of $V^{-}_G$ since $0<b<d<\pi/2$ from Lemma~\ref{lem:phi_xi}.
\par
Let also $\chi_F$ be the curve  $\xi^{-1}\chi_G:=\{\xi^{-t}\bigl(\xi(t)+t\i\bigr)\mid b\le t\le d\}$ in the $z$-plane.
It is a part of the boundary $V^{0}_{F}$ of $V^{-}_F$.
\end{defn}
\begin{lem}
The curve $\chi_G$ connects $\xi$ to $\varphi$.
Moreover, $\chi_G\in H^{-}_G$ except for $\varphi$.
\end{lem}
\begin{proof}
Since $\chi(b)=a$ and $\chi(d)=c$ from \eqref{eq:Im_phi}, we conclude that $\chi_G$ connects $\xi$ to $\varphi$.
\par
Since $\chi(t)$ is decreasing for $0<t<\pi/2$ from \eqref{eq:chi}, and $0<b<d<\pi/2$ from Lemma~\ref{lem:phi_xi}, we see that $c\le\chi(t)\le a$ for $b\le t\le d$, and that $\alpha=\cosh{a}\cos{b}\ge\cosh\chi(t)\cos{t}$ if $t\ge b$.
So we have
\begin{equation*}
\begin{split}
  &\bigl(\cosh\chi(t)\cos{t}-\alpha\bigr)^2+\bigl(\sinh\chi(t)\sin{t}-\beta\bigr)^2
  \\
  =&
  \bigl(\alpha-\cosh\chi(t)\cos{t}\bigr)^2
  \le
  \bigl(\alpha-\cosh{c}\cos{d}\bigr)^2
  =1/4,
\end{split}
\end{equation*}
where the inequality follows since $\cosh\chi(t)\ge\cosh{c}$ and $\cos{t}\ge\cos{d}$, and the last equality follows from \eqref{eq:Re_phi}.
Note that the equality holds when $t=d$.
From \eqref{eq:HV}, we conclude that the curve $\chi_G$ is in the region $H^{-}_{G}$ except for $\varphi$.
\par
This completes the proof.
\end{proof}
Since $\sigma=\varphi/\xi$, we have the following corollary.
\begin{cor}
The curve $\chi_F$ connects $1$ to $\sigma$, and $\chi_F$ is contained in $H^{-}_F$ except for $\sigma$.
\end{cor}
The following lemma shows why we name these regions $H^{\pm}_F$ and $V^{\pm}_F$.
\begin{lem}\label{lem:ReF}
We assume that $|\Re(\xi z)|<a$ and $-\pi<\Im(\xi z)\le\pi$.
\par
Writing $z=x+y\i$ with $x,y\in\R$, the function $\Re{F(z)}$ is increasing {\em(}decreasing, respectively{\rm)} with respect to $x$ if and only if $z\in H^{+}_F$ {\rm(}$z\in H^{-}_F$, respectively{\rm)}, and increasing {\rm(}decreasing, respectively{\rm)} with respect to $y$ if and only if $z\in V^{+}_F$ {\rm(}$z\in V^{-}_F$, respectively{\rm)}.
\end{lem}
\begin{proof}
Since $|\Re(\xi z)|<a$, we can use Corollary~\ref{cor:F}.
We have
\begin{align*}
  \frac{\partial}{\partial\,x}\Re F(z)
  =&
  \log|2\cosh\xi-2\cosh(\xi z)|,
  \\
  \frac{\partial}{\partial\,y}\Re F(z)
  =&
  -\Im\log(2\cosh\xi-2\cosh(\xi z))
  \\
  =&
  \arg\bigl(\cosh(\xi z)-\cosh\xi\bigr).
\end{align*}
Now, the lemma follows from \eqref{eq:HVF_def}.
\end{proof}
%%%%%%%%%%%%%%%%%%%%%%%%%%%%%%%%%%%%%%%%%%%%%%%%%%%%%%%%%%%%%%%%%%%%%%%%%%%%%%%
\subsection{Shape of the curve $H^{0}_G$}
In this subsection, we study the shape of the curve $H^{0}_G$.
Recall that we are assuming that $a>0$, $0<b<\pi/2$, and $\alpha:=\cosh{a}\cos{b}>1/2$.
\par
\begin{lem}\label{lem:alpha_beta_a_b}
The equalities $\sqrt{(\alpha-1)^2+\beta^2}=|\cosh\xi-1|=\cosh{a}-\cos{b}$ hold.
\end{lem}
\begin{proof}
Since $|\cosh\xi-1|^2=(\alpha-1)^2+\beta^2$, the first equality follows.
\par
As for the second equality, we have
\begin{equation*}
\begin{split}
  |\cosh\xi-1|^2
  =&
  (\cosh{a}\cos{b}-1)^2+\sinh^2{a}\sin^2{b}.
  \\
  =&
  \cosh^2{a}\cos^2{b}+(\cosh^2{a}-1)(1-\cos^2{b})-2\cosh{a}\cos{b}+1
  \\
  =&
  (\cosh{a}-\cos{b})^2.
\end{split}
\end{equation*}
Since $\cosh{a}>1>\cos{b}$, we have $\cosh{a}-\cos{b}>0$, and the second equality follows.
\end{proof}
\begin{lem}\label{lem:H0G}
If $\cosh{a}-\cos{b}>1/2$, then $H^{0}_G$ consists of two simple closed curves.
If $\cosh{a}-\cos{b}=1/2$, then $H^{0}_G$ is a curve with a double point singularity at $(0,0)$.
If $\cosh{a}-\cos{b}<1/2$, then $H^{0}_G$ is a simple closed curve.
\end{lem}
\begin{proof}
We put
\begin{equation}\label{eq:Phi}
  \Phi(X,Y)
  :=
  (\alpha-\cosh{X}\cos{Y})^2+(\beta-\sinh{X}\sin{Y})^2
\end{equation}
for $-\pi<Y\le\pi$.
\par
We will study the critical points of the function $\Phi(X,Y)$ and use Morse theory to determine the inverse image $C_{h}:=\Phi^{-1}(h)\cap\{(X,Y)\in\R^2\mid-\pi<Y\le\pi\}$, noting that $H^{0}_G=C_{1/4}$ from the third equality in \eqref{eq:HV}.
We assume that $h\ge0$ since $C_{h}=\emptyset$ when $h<0$.
\par
Since $\Phi(X,\pi)=(\alpha+\cosh{X})^2+\beta^2$, the intersection $\Phi^{-1}(h)\cap\{(X,Y)\in\R^2\mid Y=\pi\}$ is given as
\begin{equation*}
\begin{split}
  &\{(X,Y)\in\R^2\mid(\alpha+\cosh{X})^2+\beta^2=h,Y=\pm\pi\}
  \\
  =&
  \{(X,Y)\in\R^2\mid\cosh{X}=\sqrt{h-\beta^2}-\alpha,Y=\pm\pi\},
\end{split}
\end{equation*}
which is empty when $h<(\alpha+1)^2+\beta^2$.
It follows that $C_{1/4}$ does not cross the lines $Y=\pm\pi$ because $(\alpha+1)^2+\beta^2>9/4>1/4$ from $\alpha>1/2$.
Since we are interested in $C_{1/4}$, we may assume that $h<(\alpha+1)^2+\beta^2$.
\par
Denoting by $\Phi_X(X,Y)$ and $\Phi_Y(X,Y)$ the partial derivatives of $\Phi$ with respect to $X$ and $Y$ respectively, we have
\begin{equation}\label{eq:Phi_partial}
\begin{split}
  \Phi_{X}(X,Y)
  &=
  2(\sinh{X}\cosh{X}-\beta\cosh{X}\sin{Y}-\alpha\sinh{X}\cos{Y}),
  \\
  \Phi_{Y}(X,Y)
  &=
  2(-\sin{Y}\cos{Y}-\beta\sinh{X}\cos{Y}+\alpha\cosh{X}\sin{Y}).
\end{split}
\end{equation}
Then from Sublemma~\ref{sublem:Phi_critical} below the critical points of $\Phi(X,Y)$  are $(0,0)$, $(X_{+},Y_{+})$, and $(X_{-},Y_{-})$, where $(X_{\pm},Y_{\pm})$ are given by
\begin{align*}
  X_{\pm}
  &:=
  \pm\arsinh
  \sqrt{\frac{1}{2}\left(\alpha^2+\beta^2-1+\sqrt{(\alpha^2+\beta^2-1)^2+4\beta^2}\right)},
  \\
  Y_{\pm}
  &:=
  \pm\arcsin
  \sqrt{\frac{1}{2}\left(1-\alpha^2-\beta^2+\sqrt{(\alpha^2+\beta^2-1)^2+4\beta^2}\right)}.
\end{align*}
\par
The second derivatives are given as follows.
\begin{equation}\label{eq:Phi_partial2}
\begin{split}
  \Phi_{XX}(X,Y)
  &=
  2\left(
    \cosh^2{X}+\sinh^2{X}
    -\beta\sinh{X}\sin{Y}
    -\alpha\cosh{X}\cos{Y}
  \right),
  \\
  \Phi_{YY}(X,Y)
  &=
  2\left(
    -\cos^2{Y}+\sin^2{Y}
    +\beta\sinh{X}\sin{Y}
    +\alpha\cosh{X}\cos{Y}
  \right),
  \\
  \Phi_{XY}(X,Y)
  &=
  2\left(
    -\beta\cosh{X}\cos{Y}
    +\alpha\sinh{X}\sin{Y}
  \right).
\end{split}
\end{equation}
Thus we have
\begin{equation}\label{eq:Phi_partial2_00}
\begin{split}
  \Phi_{XX}(0,0)
  &=
  2(1-\alpha),
  \\
  \Phi_{YY}(0,0)
  &=
  2(\alpha-1),
  \\
  \Phi_{XY}(0,0)
  &=
  -2\beta,
\end{split}
\end{equation}
and
\begin{align*}
  \Phi_{XX}(X_{\pm},Y_{\pm})
  &=
  2\sqrt{(\alpha^2+\beta^2-1)^2+4\beta^2},
  \\
  \Phi_{YY}(X_{\pm},Y_{\pm})
  &=
  2\sqrt{(\alpha^2+\beta^2-1)^2+4\beta^2},
  \\
  \Phi_{XY}(X_{\pm},Y_{\pm})
  &=
  0.
\end{align*}
Let $J_{\Phi}(X,Y)$ be the Jacobian matrix of $\Phi(X,Y)$.
Then we have $J_{\Phi}(0,0)=2\begin{pmatrix}1-\alpha&-\beta\\-\beta&\alpha-1\end{pmatrix}$ with $\det J_{\Phi}(0,0)=-4(\alpha-1)^2-4\beta^2<0$.
We also have $J_{\Phi}(X_{\pm},Y_{\pm})=2\sqrt{(\alpha^2+\beta^2-1)^2+4\beta^2}\times\begin{pmatrix}1&0\\0&1\end{pmatrix}$, which is clearly positive definite.
So the points $(X_{\pm},Y_{\pm})$ are minima, and $(0,0)$ is a saddle point.
\par
Since we calculate $\Phi(X_{\pm},Y_{\pm})=0$ and $\Phi(0,0)=(\alpha-1)^2+\beta^2=(\cosh{a}-\cos{b})^2$, we have the following.
\begin{itemize}
\item
$C_{h}=\{(X_{+},Y_{+}),(X_{-},Y_{-})\}$ if $h=0$,
\item
$C_{h}$ consists of two simple closed curves if $0<h<(\cosh{a}-\cos{b})^2$,
\item
$C_{h}$ is a closed curve with double point singularity at $(0,0)$ if $h=(\cosh{a}-\cos{b})^2$,
\item
$C_{h}$ is a simple closed curve if $(\cos{a}-\cos{b})^2<h<(\alpha+1)^2+\beta^2$.
\end{itemize}
by Morse Theory.
See Figure~\ref{fig:Phi_3D}
\begin{figure}[h]
\pic{0.43}{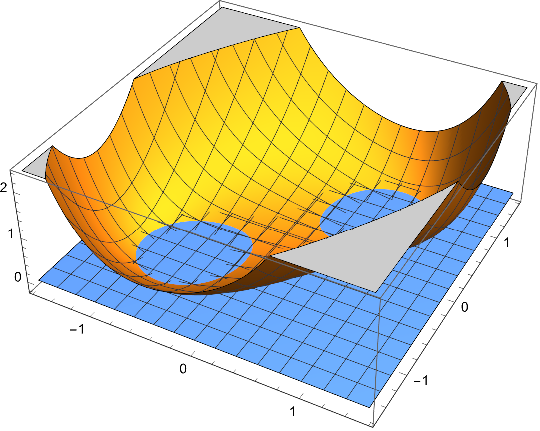}\quad\pic{0.43}{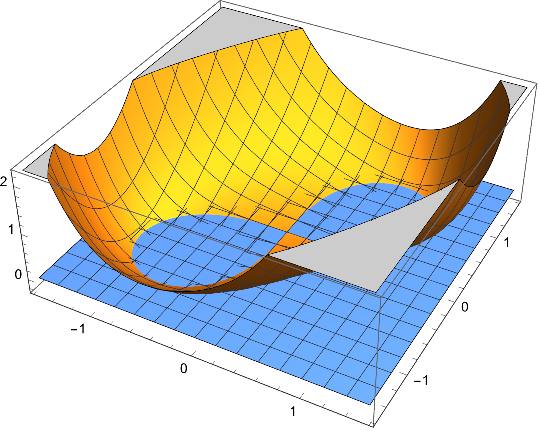}\quad\pic{0.43}{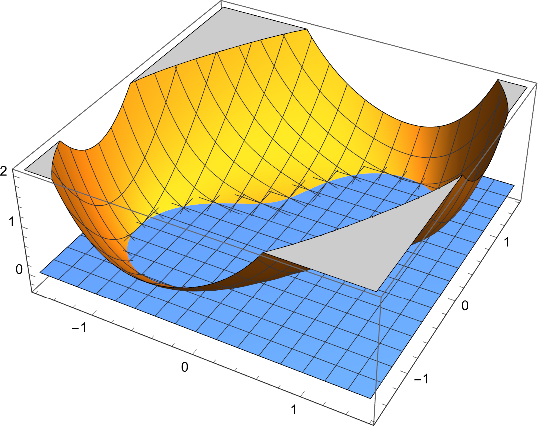}
\caption{Graphs of $C_h$ with $h=1/5$ (left), $1/4$ (middle), and $1/3$ (right), where $(\alpha,\beta)=(1,1/2)$.}
\label{fig:Phi_3D}.
\end{figure}
\par
Since $H^{0}_{G}=C_{1/4}$, the lemma follows.
\end{proof}
Now, we prove the sublemma used in the previous lemma.
\begin{sublem}\label{sublem:Phi_critical}
The critical points of $\Phi(X,Y)$ with $-\pi<Y\le\pi$ are $(0,0)$ and $(X_{\pm},Y_{\pm})$.
\end{sublem}
\begin{proof}
First of all, we may assume that $|Y|<\pi$ because we know that $C_{h}$ does not cross the line $Y=\pi$ (see the proof of Lemma~\ref{lem:H0G}).
\par
From \eqref{eq:Phi_partial}, $\Phi_{X}(X,Y)=\Phi_{Y}(X,Y)=0$ if and only if
\begin{align}
  \sinh{X}\cosh{X}-\beta\cosh{X}\sin{Y}-\alpha\sinh{X}\cos{Y}&=0,
  \label{eq:sing1}
  \\
  -\sin{Y}\cos{Y}-\beta\sinh{X}\cos{Y}+\alpha\cosh{X}\sin{Y}&=0.
  \label{eq:sing2}
\end{align}
Multiplying \eqref{eq:sing1} by $\sinh{X}\cos{Y}$ and \eqref{eq:sing2} by $\cosh{X}\sin{Y}$, and taking their difference, we obtain
\begin{equation*}
\begin{split}
  0
  =&
  \sinh^2{X}\cosh{X}\cos{Y}-\beta\sinh{X}\cosh{X}\sin{Y}\cos{Y}-\alpha\sinh^2{X}\cos^2{Y}
  \\
  &
  +\cosh{X}\sin^2{Y}\cos{Y}+\beta\sinh{X}\cosh{X}\sin{Y}\cos{Y}-\alpha\cosh^2{X}\sin^2{Y}
  \\
  =&
  \cosh{X}\cos{Y}\left(\sinh^2{X}+\sin^2{Y}\right)
  -
  \alpha\left(\sinh^2{X}\cos^2{Y}+\cosh^2{X}\sin^2{Y}\right)
  \\
  =&
  \cosh{X}\cos{Y}(\cosh^2{X}-\cos^2{Y})
  \\
  &-
  \alpha\left(\cosh^2{X}\cos^2{Y}-\cos^2{Y}+\cosh^2{X}-\cosh^2{X}\cos^2{Y}\right)
  \\
  =&
  \left(\cosh^2{X}-\cos^2{Y}\right)(\cosh{X}\cos{Y}-\alpha).
\end{split}
\end{equation*}
Since $\cosh{X}\ge1$, $|\cos{Y}|\le1$, and $|Y|<\pi$, we have $\cosh{X}\cos{Y}=\alpha$, or $X=Y=0$.
\par
When $(X,Y)=(0,0)$, both partial derivatives certainly vanish.
\par
Now, consider the case where $\cosh{X}\cos{Y}=\alpha$.
\par
Multiplying \eqref{eq:sing1} by $\cos{Y}$, and using $\cosh{X}\cos{Y}=\alpha$, we obtain
\begin{equation*}
\begin{split}
  0
  =&
  \alpha\sinh{X}-\alpha\beta\sin{Y}-\alpha\sinh{X}\cos^2{Y}
  \\
  =&
  \alpha\sin{Y}\bigl(\sinh{X}\sin{Y}-\beta\bigr).
\end{split}
\end{equation*}
Since $|Y|<\pi$, we obtain $Y=0$, or $\sinh{X}\sin{Y}=\beta$.
If $Y=0$, from \eqref{eq:sing2} we have $\beta\sinh{X}=0$, which implies $X=0$.
\par
So the remaining case is $\sinh{X}\sin{Y}=\beta$ with $\cosh{X}\cos{Y}=\alpha$ and $(X,Y)\ne(0,0)$.
\par
Since $\sinh{X}\ne0$ and $\cosh{X}\ne0$, we have $\beta/\sinh{X}=\sin{Y}$ and $\alpha/\cosh{Y}=\cos{Y}$, which implies $(\beta/\sinh{X})^2+(\alpha/\cosh{Y})^2=1$.
Therefore we have $\sinh^4{X}-(\alpha^2+\beta^2-1)\sinh^2{X}-\beta^2=0$.
So we have $\sinh{X}=\pm\sqrt{\frac{1}{2}\left(\alpha^2+\beta^2-1+\sqrt{(\alpha^2+\beta^2-1)^2+4\beta^2}\right)}$.
Note that the expression in the bigger square root is positive since $|\alpha^2+\beta^2-1|<\sqrt{(\alpha^2+\beta^2-1)^2+4\beta^2}$.
Similarly, we have $\sin{Y}=\beta/\sinh{X}=\pm\sqrt{\frac{1}{2}\left(1-\alpha^2-\beta^2+\sqrt{(\alpha^2+\beta^2-1)^2+4\beta^2}\right)}$.
The expression in the bigger square root is positive from the same reason above.
It is less than $1$ since $(\alpha^2+\beta^2-1)^2+4\beta^2-(\alpha^2+\beta^2+1)^2=-4\alpha^2<0$.
\par
Thus, we conclude that the critical points of $\Phi(X,Y)$ with $-\pi<Y\le\pi$ are $(0,0)$ and $(X_{\pm},Y_{\pm})$
\end{proof}
\begin{rem}
If $(X,Y)\in C_{1/4}=H^{0}_{G}$, then we have $|\alpha-\cosh{X}\cos{Y}|\le1/2$.
So we have $\cosh{X}\cos{Y}\ge\alpha-1/2>0$ since $\alpha>1/2$ from the assumption.
It follows that $\cos{Y}>0$ and so in fact we have $|Y|<\pi/2$.
\end{rem}
\begin{rem}\label{rem:double_point}
If $\cosh{a}-\cos{b}=1/2$, that is, $(\alpha-1)^2+\beta^2=1/4$ from Lemma~\ref{lem:alpha_beta_a_b}, then the point $(0,0)$ is a unique double point of $H^{0}_G$.
From \eqref{eq:Phi_partial2_00}, the Taylor expansion of $\Phi(X,Y)$ at $(0,0)$ is
\begin{equation*}
  \Phi(X,Y)
  =
  1/4
  +
  (1-\alpha)(X^2-Y^2)-2\beta XY
  +
  \text{(terms with degree three or more)}.
\end{equation*}
Since the degree two terms of the right hand side becomes
\begin{equation*}
  -\frac{1}{2(2\beta+1)}
  \bigl(2(\alpha-1)X+(2\beta+1)Y\bigr)
  \bigl((2\beta+1)X+2(1-\alpha)Y\bigr),
\end{equation*}
the tangents of $H^{0}_G$ at $(0,0)$ are given as follows:
\begin{align*}
  2(\alpha-1)X+(2\beta+1)Y&=0,
  \\
  (2\beta+1)X-2(\alpha-1)Y&=0.
\end{align*}
\par
If we put
\begin{equation}\label{eq:rho_def}
\begin{split}
  \rho_1
  &:=
  (2\beta+1)-2(\alpha-1)\i,
  \\
  \rho_2
  &:=
  2(\alpha-1)+(2\beta+1)\i,
\end{split}
\end{equation}
then, the lines $\{\rho_1t\mid t\in\R\}$ and $\{\rho_2t\mid t\in\R\}$ are parallel to the tangents.
Note that $\rho_1\times\i=\rho_2$, and so we conclude that a $\pi/2$-rotation of the vector corresponding to $\rho_1$ gives the one corresponding to $\rho_2$.
\end{rem}
From \eqref{eq:HV}, it is clear that $H^{0}_G$ is symmetric with respect to the origin $O$.
\par
If $\cosh{a}-\cos{b}\ge1/2$, the curve $H^{0}_G$ separates the $Z$-plane into three open, connected components.
In this case, we define two simple closed curves $\overline{C}_{G}$ and $\underline{C}_{G}$ as follows.
Note that $\varphi=c+d\i$ is on $H^{0}_G$ from \eqref{eq:HV}, \eqref{eq:Re_phi}, and \eqref{eq:Im_phi}.
\begin{defn}\label{defn:C}
When $\cosh{a}-\cos{b}>1/2$, we define $\overline{C}_{G}$ and $\underline{C}_{G}$ as smooth, simple closed curves so that $H^{0}_G=\overline{C}_{G}\cup \underline{C}_{G}$ and $\overline{C}_{G}\cap \underline{C}_{G}=\emptyset$, where we choose $\overline{C}_{G}$ so that $\varphi\in\overline{C}_{G}$.
See the left picture in Figure~\ref{fig:HG_VG}.
\par
If $\cosh{a}-\cos{b}=1/2$, then we define simple closed curves $\overline{C}_{G}$ and $\underline{C}_{G}$ so that $H^{0}_G=\overline{C}_{G}\cup \underline{C}_{G}$ with $\overline{C}_{G}\cap \underline{C}_{G}=\{O\}$, and that $\overline{C}_{G}$ contains the point $\varphi$.
Then both $\overline{C}_{G}$ and $\underline{C}_{G}$ are are smooth except at $O$.
See the middle picture in Figure~\ref{fig:HG_VG}.
Note that $\overline{C}_{G}$ and $\underline{C}_{G}$ are symmetric in both cases.
\end{defn}
\par
We can prove that the curve $\overline{C}_{G}$ is convex.
In fact, we show the following proposition.
\begin{prop}\label{prop:curvature}
If $\cosh{a}-\cos{b}>1/2$, then the curvature of $\overline{C}_{G}$ is positive if we orient it appropriately.
If $\cosh{a}-\cos{b}=1/2$, then the curvature of $\overline{C}_{G}$ is positive if we orient it appropriately, except for $(0,0)$.
\end{prop}
The proof is lengthy with many technical calculations, and so we postpone to Appendix~\ref{sec:curvature}.
\par
If $\cosh{a}-\cos{b}=1/2$, from Remark~\ref{rem:double_point}, the two tangents of $\overline{C}_{G}$ at $(0,0)$ are orthogonal.
So its interior angle at $(0,0)$ is either $\pi/2$ or $3\pi/2$.
If it is $3\pi/2$, then from the symmetry of $\overline{C}_{G}$ and $\underline{C}_{G}$, they should intersect.
Therefore the interior angle should be $\pi/2<\pi$, and we have the following corollary.
\begin{cor}\label{cor:convex}
If $\cosh{a}-\cos{b}\ge1/2$, then the curve $\overline{C}_{G}$ is convex.
\end{cor}
\par
We can prove that $\overline{C}_{G}$ and $\underline{C}_{G}$ are separated by the line $aX+bY=0$.
\begin{lem}\label{lem:separate_C}
If $\cosh{a}-\cos{b}>1/2$, then the line $aX+bY=0$ separates $H^{0}_G$ into $\overline{C}_{G}$ and $\underline{C}_{G}$.
If $\cosh{a}-\cos{b}=1/2$, then the line $aX+bY=0$ separates $H^{0}_G\setminus\{O\}$ into $\overline{C}_{G}\setminus\{O\}$ and $\underline{C}_{G}\setminus\{O\}$.
\end{lem}
\begin{proof}
Recall that we assume $-\pi<Y\le\pi$, $a>0$, $0<b\le\pi/2$, and $\alpha=\cosh{a}\cos{b}>1/2$.
Note also that $(\alpha-1)^2+\beta^2=(\cosh{a}-\cos{b})^2$ from Lemma~\ref{lem:alpha_beta_a_b}.
\par
We first consider the case where $\cosh{a}-\cos{b}=1/2$, that is, $(\alpha-1)^2+\beta^2=1/4$.
Note that $1/2<\alpha<3/2$ and $0<\beta<1/2$.
\par
We will show that $\overline{C}_{G}$ is in the circular sector $S:=\{Z\in\C\mid Z=s\rho_1+t\rho_2,s,t\ge0\}$, where $\rho_{i}$ ($i=1,2$) is defined in \eqref{eq:rho_def}.
\par
We will first show that $\xi$ is in $S$.
Since $\rho_2=\rho_1\times\i$ from Remark~\ref{rem:double_point}, it suffices to prove that $\xi/\rho_1$ is in the first quadrant, that is, $\Im\bigl(\xi/\rho_1\bigr)>0$ and $\Re\bigl(\xi/\rho_1\bigr)>0$.
\par
We have
\begin{align}
  |\rho_1|^2\Im\frac{\xi}{\rho_1}
  &=
  b(2\beta+1)+2a(\alpha-1),
  \label{eq:xi/rho_Im}
  \\
  |\rho_1|^2\Re\frac{\xi}{\rho_1}
  &=
  a(2\beta+1)-2b(\alpha-1).
  \label{eq:xi/rho_Re}
\end{align}
Denoting the right hand side of \eqref{eq:xi/rho_Im} (\eqref{eq:xi/rho_Re}, respectively) by $R^{\rm{Im}}(a,b)$ ($R^{\rm{Re}}(a,b)$, respectively), we have
\begin{align*}
  \frac{\partial\,R^{\rm{Im}}}{\partial\,a}(a,b)
  &:=
  2\bigl(a\sinh{a}\cos{b}+\cosh{a}(b\sin{b}+\cos{b})-1\bigr),
  \\
  \frac{\partial\,R^{\rm{Re}}}{\partial\,b}(a,b)
  &:=
  2\bigl(a\sinh{a}\cos{b}+\cosh{a}(b\sin{b}-\cos{b})+1\bigr)
\end{align*}
\par
Since the derivative of $b\sin{b}+\cos{b}$ equals $b\cos{b}>0$ for $0<b<\pi/2$, $b\sin{b}+\cos{b}>1$.
Therefore $\frac{\partial\,R^{\rm{Im}}}{\partial\,a}(a,b)>0$, and so we have $R^{\rm{Im}}(a,b)>R^{\rm{Im}}(0,b)=b>0$.
\par
Similarly the derivative of $b\sin{b}-\cos{b}$ equals $b\cos{b}+2{b}>0$ for $0<b<\pi/2$, which implies $b\sin{b}-\cos{b}>-1$ and $\frac{\partial\,R^{\rm{Re}}}{\partial\,b}(a,b)>0$.
Therefore, we conclude that $R^{\rm{Re}}(a,b)>R^{\rm{Re}}(a,0)=a>0$.
\par
Thus we have proved that $\xi\in S$.
\par
From Corollary~\ref{cor:convex}, we know that $\overline{C}_{G}$ is convex.
Since $\xi\in S$ and $\xi\in H^{-}_{G}$, it follows that the vector $\rho_1$ is to the right of $H^{-}_{G}$ and that the vector $\rho_2$ is to the left of $H^{-}_{G}$.
So we conclude that $\overline{C}_{G}\setminus\{O\}\subset S$.
Since the line $aX+bY=0$ is parallel to $\xi\i$ that is perpendicular to $\xi$, it is not in $S$ except for $O$.
So we finally see that the line $aX+bY=0$ separates $C\setminus\{O\}$ into $\overline{C}_{G}\setminus\{O\}$ and $\underline{C}_{G}\setminus\{O\}$.
\par
Note that this means $\Phi(-bt,at)\ge1/4$ if $\cosh{a}-\cos{b}=1/2$.
\par
Next, we will show that if $(\alpha-1)^2+\beta^2>1/4$ then the line $aX+bY=0$ separates $H^{0}_{G}$ into $\overline{C}_{G}$ and $\underline{C}_{G}$.
To do this, we will show that the line $aX+bY=0$ is in the region $H^{+}_{G}$.
From \eqref{eq:HV} and \eqref{eq:Phi}, it is enough to show that $\Phi(-bt,at)>1/4$ for $-\pi/a<t\le\pi/a$ since $-\pi<Y\le\pi$.
\par
Since $\alpha=\cosh{a}\cos{b}$ and $\beta=\sinh{a}\sin{b}$, from \eqref{eq:Phi} we have
\begin{equation*}
\begin{split}
  \tPhi(a,b,t)
  :=&
  \Phi(-bt,at)
  \\
  =&
  \bigl(\alpha-\cosh(bt)\cos(at)\bigr)^2
  +
  \bigl(\beta+\sinh(bt)\sin(at)\bigr)^2
  \\
  =&
  \bigl(\cosh{a}\cos{b}-\cosh(bt)\cos(at)\bigr)^2
  +
  \bigl(\sinh{a}\sin{b}+\sinh(bt)\sin(at)\bigr)^2.
\end{split}
\end{equation*}
We may assume that $0\le t\le\pi/a$ since $\tPhi(a,b,t)=\tPhi(a,b,-t)$.
\par
If $\pi/(2a)\le t\le\pi/a$, then since $\cos(at)\le0$, $\sinh(bt)>0$, $\sin(at)\ge0$, and $\alpha>1/2$, we have
\begin{equation*}
  \tPhi(a,b,t)
  \ge
  \alpha^2+\beta^2
  >1/4.
\end{equation*}
\par
Let us consider the case where $0\le t<\frac{\pi}{2a}$.
The partial derivative of $\tPhi(a,b,t)$ with respect to $a$ is
\begin{equation*}
\begin{split}
  &\frac{\partial\tPhi}{\partial a}(a,b,t)
  \\
  =&
  2\cosh{a}\sinh{a}-2t\cos(at)\sin(at)
  \\
  &+
  2\cosh{a}\sin(at)\bigl(t\cosh(bt)\cos{b}+\sinh(bt)\sin{b}\bigr)
  \\
  &+
  2\sinh{a}\cos(at)\bigl(t\sinh(bt)\sin{b}-\cosh(bt)\cos{b}\bigr).
\end{split}
\end{equation*}
Taking its partial derivative with respect to $b$, we obtain
\begin{equation*}
\begin{split}
  &\frac{\partial^2\tPhi}{\partial b\,\partial a}(a,b,t)
  \\
  =&
  2(1+t^2)
  \bigl(\cosh{a}\sinh(bt)\cos{b}\sin(at)+\cosh(bt)\sinh{a}\cos(at)\sin{b}\bigr)
  \ge0.
\end{split}
\end{equation*}
So the partial derivative $\partial\tPhi(a,b,t)/\partial a$ is increasing with respect to $b$, and we have
\begin{equation*}
\begin{split}
  \frac{\partial\tPhi}{\partial a}(a,b,t)
  \ge&
  \frac{\partial\tPhi}{\partial a}(a,0,t)
  \\
  =&
  2\bigl(\cosh{a}-\cos(at)\bigr)\bigl(\sinh{a}+t\sin(at)\bigr)
  >0.
\end{split}
\end{equation*}
Therefore we see that $\tPhi(a,b,t)$ is strictly increasing with respect to $a$.
\par
So, if $\pi/3<b<\pi/2$, then $\tPhi(a,b,t)>\tPhi(0,b,t)=\bigl(\cosh(bt)-\cos{b}\bigr)^2$, which is strictly increasing with respect to $t$ because $\cosh(bt)>\cos{b}$.
Thus we conclude that $\tPhi(a,b,t)>(1-\cos{b})^2>1/4$.
\par
Suppose that $0<b\le\pi/3$.
From $\cosh{a}-\cos{b}>1/2$ and $\cos{b}>1/2$, we have $a>\arcosh(\cos{b}+1/2)$.
Since $\tPhi(a,b,t)$ is strictly increasing with respect to $a$, we see that $\tPhi(a,b,t)$ is greater than the value of $\tPhi(a,b,t)$ evaluated at $a=\arcosh(\cos{b}+1/2)$.
However, if $a=\arcosh(\cos{b}+1/2)$, then $\cosh{a}-\cos{b}=1/2$.
Now, we already know that if $\cosh{a}-\cos{b}=1/2$, then $\tPhi(a,b,t)\ge1/4$.
So we finally conclude that $\tPhi(a,b,t)>1/4$.
\par
This completes the proof.
\end{proof}
Putting $\overline{C}_{F}:=\xi^{-1}\overline{C}_{G}$ and $\underline{C}_{F}:=\xi^{-1}\underline{C}_{G}$, we have the following corollary because the map $Z\mapsto \xi^{-1}Z$ sends the line $aX+bY=0$ in the $Z$-plane to the imaginary axis in the $z$-plane.
\begin{cor}
If $\cosh{a}-\cos{b}>1/2$, then the line $x=0$ separates $H^{0}_{F}$ into $\overline{C}_{F}$ and $\underline{C}_{F}$.
If $\cosh{a}-\cos{b}=1/2$, then the line $x=0$ separates $H^{0}_{F}\setminus\{O\}$ into $\overline{C}_{F}\setminus\{O\}$ and $\underline{C}_{F}\setminus\{O\}$.
\end{cor}
It follows that when $\cosh{a}-\cos{b}=1/2$, then the half-open interval $(0,1]$ is in $H^{-}_{F}$.
Similarly, if $\cosh{a}-\cos{b}<1/2$, then the closed interval $[0,1]$ is in $H^{-}_{F}$.
So we have another corollary.
\begin{cor}\label{cor:cosha_cosb_1_2}
If $\cosh{a}-\cos{b}\le1/2$, then $\Re F(x)$ is decreasing if $x\in(0,1]$, and so $\Re F(x)<0$ for $0<x\le1$.
\end{cor}
%%%%%%%%%%%%%%%%%%%%%%%%%%%%%%%%%%%%%%%%%%%%%%%%%%%%%%%%%%%%%%%%%%%%%%%%%%%%%%%
\subsection{The curves $\chi_G$ and $\chi_F$}
In this subsection, we study the curves $\chi_G$ and $\chi_F$ defined in Definition~\ref{defn:chi_F_G}.
Recall that we assume $a>0$, $0<b<\pi/2$, and $\alpha=\cosh{a}\cos{b}>1/2$.
\par
\begin{lem}\label{lem:chi_increasing}
Assume that $a\tanh{c}-b\tan{d}\ge0$.
Then along the curve $\chi_G$, $\Re{G(Z)}$ is strictly increasing.
In particular, we have $\Re{G(\xi)}<\Re{G(\varphi)}$.
\end{lem}
\begin{proof}
Since $\cosh\bigl(\chi(t)+t\i\bigr)=\cosh\chi(t)\cos{t}+\i\sinh\chi(t)\sin{t}=\cosh\chi(t)\cos{t}+\i\beta$, we have $\cosh\xi-\cosh\bigl(\chi(t)+t\i\bigr)=\alpha-\cosh\chi(t)\cos{t}\in\R$ from \eqref{eq:alpha} and \eqref{eq:beta}.
Moreover, since $\chi(t)$ is strictly decreasing for $b\le t\le d$, we conclude that
\begin{equation}\label{eq:chi_decrease}
  0<2\cosh\xi-2\cosh\bigl(\chi(t)+t\i\bigr)<2\alpha-2\cosh{c}\cos{d}=1
\end{equation}
for $b<t<d$, where we use \eqref{eq:Re_phi}.
Using the first equality in Lemma~\ref{lem:G_der} since $c<\chi(t)<a$, we have
\begin{equation*}
\begin{split}
  &\frac{d}{d\,t}\Re{G\bigl(\chi(t)+t\i\bigr)}
  \\
  =&
  \frac{\partial}{\partial\,X}\Re G\bigl(\chi(t)+t\i\bigr)\times\chi'(t)
  +
  \frac{\partial}{\partial\,Y}\Re G\bigl(\chi(t)+t\i\bigr)
  \\
  =&
  \Re\left(\frac{1}{\xi}\log\Bigr(2\cosh\xi-2\cosh\bigl(\chi(t)+t\i\bigr)\Bigr)\right)
  \times\left(\frac{-\beta\cot{t}}{\sqrt{\beta^2+\sin^2t}}\right)
  \\
  &-
  \Im\left(\frac{1}{\xi}\log\Bigr(2\cosh\xi-2\cosh\bigl(\chi(t)+t\i\bigr)\Bigr)\right)
  \\
  &\text{(since $2\cosh\xi-2\cosh\bigl(\chi(t)+t\i\bigr)\in\R$)}
  \\
  =&
  \frac{\log\Bigr(2\cosh\xi-2\cosh\bigl(\chi(t)+t\i\bigr)\Bigr)}{|\xi|^2}
  \left(
    b-\frac{a\beta\cot{t}}{\sqrt{\beta^2+\sin^2{t}}}
  \right).
\end{split}
\end{equation*}
Since $\cot{t}$ is decreasing and $\sin^2{t}$ is increasing for $0<b\le t\le d<\pi/2$, and $\beta=\sinh{c}\sin{d}$ from \eqref{eq:Im_phi}, we have
\begin{equation*}
\begin{split}
  b-\frac{a\beta\cot{t}}{\sqrt{\beta^2+\sin^2{t}}}
  \le&
  b-\frac{a\sinh{c}\sin{d}\cot{d}}{\sqrt{\sinh^2{c}\sin^2{d}+\sin^2{d}}}
  \\
  =&
  b-\frac{a\tanh{c}}{\tan{d}}\le0,
\end{split}
\end{equation*}
where the second inequality follow from the assumption $a\tanh{c}-b\tan{d}\ge0$.
Note that the first equality holds when $t=d$.
So from \eqref{eq:chi_decrease}, we conclude that $\frac{d}{d\,t}\Re{G\bigl(\chi(t)+t\i\bigr)}>0$ for $b<t<d$.
This shows that $\Re{G(Z)}$ is strictly increasing along the curve $\chi_G$.
\par
This completes the proof.
\end{proof}
In the $z$-plane, we have the following corollary.
\begin{cor}\label{cor:F1}
If $a\tanh{c}-b\tan{d}\ge0$, then $\Re{F(z)}$ is increasing along the curve $\chi_F$.
In particular, we have $\Re{F(1)}<\Re{F(\sigma)}$.
\end{cor}
We can also prove the following lemma by using the curve $\chi_G$.
\begin{lem}\label{lem:Re_sigma}
If $a\tanh{c}-b\tan{d}\ge0$, then $\Re\sigma<1$.
\end{lem}
\begin{proof}
The curve $\chi_G$ is parametrized as $X=\chi(Y)$ ($b\le Y\le d$) in the $XY$-plane.
Its derivatives become
\begin{align*}
  \chi'(Y)
  &=
  -\frac{\beta\cos{Y}}{\sin{Y}\sqrt{\sin^2{Y}+\beta^2}}<0,
  \\
  \chi''(Y)
  &=
  \frac{\beta\left(\sin^2{Y(1+\cos^2{Y})}+\beta^2\right)}
       {\sin^2{Y}\bigl(\sin^2{Y}+\beta^2\bigr)^{3/2}}>0,
\end{align*}
where the first inequality follows since $0<b\le Y\le d<\pi/2$.
So the graph of $\chi(Y)$ is decreasing and convex to the left in the $XY$-plane.
It follows that the tangent $T_{V}$ of $\chi_G$ at $\varphi$ is to the left of $\chi_G$ except for $\varphi$.
Since $\chi(Y)$ is decreasing, the tangent $T_{V}$ is up to the left.
Recalling that $\xi$ is the other end of $\chi_G$ corresponding to $\chi(b)$, we conclude that the point $\xi$ is on the right of $T_{V}$.
See the left picture of Figure~\ref{fig:location_sigma}.
\begin{figure}[h]
\pic{0.3}{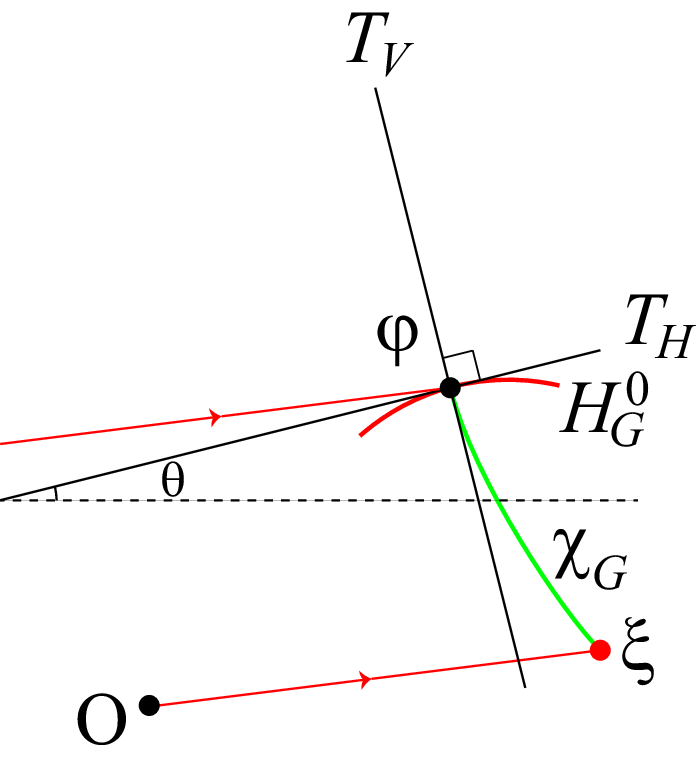}
\quad$\xrightarrow{\text{$-\theta$-rotation}}$\quad
\pic{0.3}{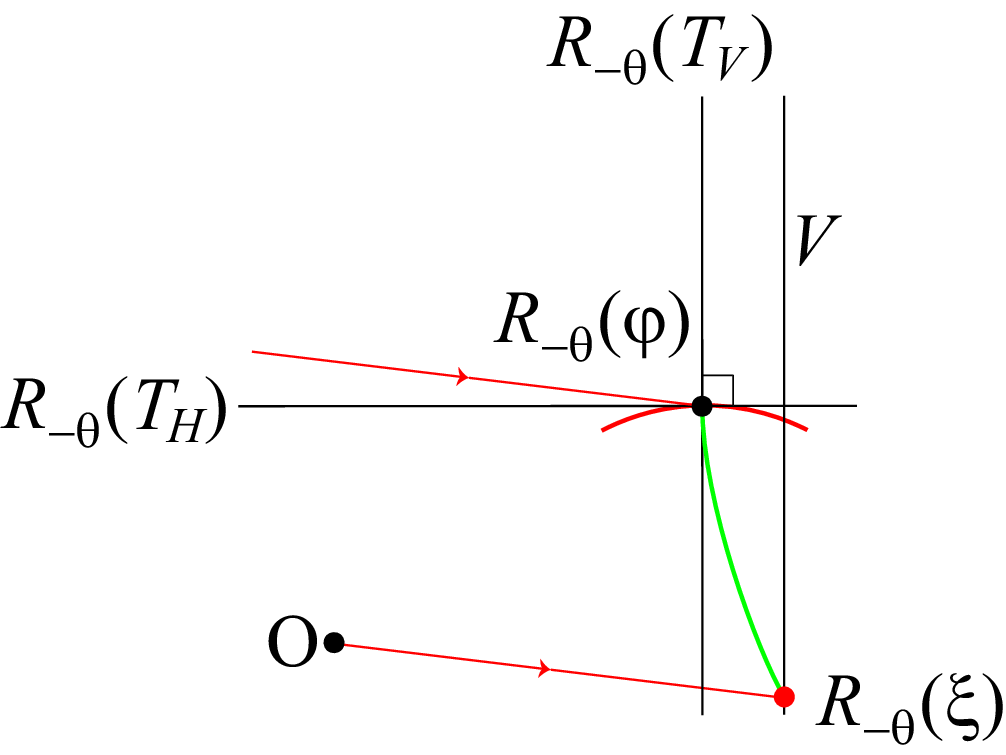}
\caption{A $-\theta$-rotation sends $T_{V}$ and $T_H$ to the vertical and horizontal lines respectively.
In the left picture, he green curve is $\chi_G$, and the red curve is a part of $H^{0}_G$.}
\label{fig:location_sigma}
\end{figure}
\par
Since $\beta=\sinh{c}\sin{d}$ from \eqref{eq:Im_phi}, the slope of $T_{V}$ equals the inverse of
\begin{equation*}
  -\frac{\beta\cos{d}}{\sin{d}\sqrt{\sin^2{d}+\beta^2}}
  =
  -\frac{\sinh{c}\sin{d}\cos{d}}{\sin{d}\sqrt{\sin^2{d}+\sinh^2{c}\sin^2{d}}}
  =
  -\tanh{c}\cot{d},
\end{equation*}
that is, $-\coth{c}\tan{d}$.
\par
By the implicit function theorem, the slope of the tangent $T_{H}$ of $H^{0}_{G}$ at $\varphi=c+d\i$ equals $-\frac{\Phi_X(c,d)}{\Phi_Y(c,d)}$.
From \eqref{eq:Phi_partial}, \eqref{eq:Re_phi}, and \eqref{eq:Im_phi}, we have
\begin{align*}
  \Phi_X(c,d)
  &=
  -\sinh{c}\cos{d},
  \\
  \Phi_Y(c,d)
  &=
  \cosh{c}\sin{d}.
\end{align*}
Therefore the slope of $T_{H}$ is $\tanh{c}\cot{d}$, and so we conclude that $T_{V}$ is perpendicular to $T_{H}$.
\par
We put $\theta:=\arctan(\tanh{c}\cot{d})$ so that $\theta$ is the angle from the $X$-axis to $T_H$.
Note that $0<\theta<\pi/2$ since $c>0$ and $0<d<\pi/2$ from Lemma~\ref{lem:phi_xi}.
Denoting by $R_{\upsilon}$ the $\upsilon$-rotation around the origin, the line $R_{-\theta}(T_H)$ ($R_{-\theta}(T_V)$, respectively) is horizontal (vertical, respectively).
Since $\xi$ is on the right of $T_V$, the point $R_{-\theta}(\xi)$ is on the right of the line $R_{-\theta}(T_V)$, which means that the point $R_{-\theta}(\varphi)$ is on the left of the vertical line $V$ passing through $R_{-\theta}(\xi)$.
Note that it is above the line connecting $O$ and $R_{-\theta}(\xi)$.
See the right picture of Figure~\ref{fig:location_sigma}.
\par
From the assumption, we have $\tanh{c}\cot{d}\ge b/a$, and so we have $\arg{\xi}\le\theta$.
Since the map $z\mapsto\xi z$ equals $|\xi|R_{\arg{\xi}}$, its inverse $|\xi|^{-1}R_{-\arg\xi}$ sends $\xi$ to $1$, and $\varphi$ to $\sigma$, that is, we have $|\xi|^{-1}R_{-\arg{\xi}}(\varphi)=\sigma$, and $|\xi|^{-1}R_{-\arg{\xi}}(\xi)=1$.
Therefore, we see that $\sigma=|\xi|^{-1}\left(R_{\theta-\arg{\xi}}\circ R_{-\theta}\right)(\varphi)$ and $1=|\xi|^{-1}\left(R_{\theta-\arg{\xi}}\circ R_{-\theta}\right)(\xi)$ are obtained from $R_{-\theta}(\varphi)$ and $R_{-\theta}(\xi)$ by the $\theta-\arg{\xi}\ge0$ rotation around the origin (and enlargement), respectively.
\par
If we rotate the right picture of Figure~\ref{fig:location_sigma} by $\theta-\arg{\xi}$ around the origin, we conclude that $\sigma$ is on the left of the line $R_{\theta-\arg{\xi}}(V)$, and above the real axis.
Since the line $R_{\theta-\arg{\xi}}(V)$ is down to the right and passes through the point $1$, we conclude that $\Re\sigma<1$.
\end{proof}
%%%%%%%%%%%%%%%%%%%%%%%%%%%%%%%%%%%%%%%%%%%%%%%%%%%%%%%%%%%%%%%%%%%%%%%%%%%%%%%
\subsection{Line segments $h_F$ and $v_F$}
In this subsection we define line segments $v_F\subset V^{-}_F$ and $h_F\subset H^{+}_F$ in the $z$-plane, connecting $O$ to $\i\Im\sigma$, and $\sigma$ to $\i\Im\sigma$, respectively.
We are assuming that $a>0$, $0<b<\pi/2$, and $\alpha=\cosh{a}\cos{b}>1/2$.
\par
\begin{defn}\label{defn:h_v_F_G}
Let $v_F$ be the line segment $v(t):=t\i$ in the $z$-plane with $0\le t\le(ad-bc)/|\xi|^2$, and $h_F$ be the line segment $h(t):=\sigma-t$ with $0\le t\le(ac+bd)/|\xi|^2$ in the $z$-plane.
Note that $ad-bc>0$ from Lemma~\ref{lem:phi_xi}.
\par
In the $Z$-plane we define $v_G$ and $h_G$ as the line segments $\xi v_F$ and $\xi h_F$ respectively.
\end{defn}
Since we have $v\bigl((ad-bc)/|\xi|^2\bigr)=h\bigl((ac+bd)/|\xi|^2\bigr)=\frac{(ad-bc)\i}{|\xi|^2}=\i\Im\sigma$, $v_F$ connects $O$ to $\i\Im\sigma$, and $h_F$ connects $\sigma$ to $\i\Im\sigma$.
It follows that $v_G$ connects $O$ to $\xi\i\Im\sigma$, and that $h_G$ connects $\varphi$ to $\xi\i\Im\sigma$.
\begin{lem}
Suppose that $\cosh{a}-\cos{b}\ge1/2$.
Then the line segments $v_{G}$ and $h_{G}$ are in the rectangle $\{Z\in\C\mid|\Re{Z}|<a,-\pi<\Im{Z}\le\pi\}$.
\end{lem}
\begin{proof}
Since $\varphi=c+d\i$ with $0<c<a$ and $0<d<\pi/2$ from Lemma~\ref{lem:phi_xi}, it is enough to show $|\Re(\xi\i\Im\sigma)|<a$ and $-\pi<|\Im(\xi\i\Im\sigma)|\le\pi$.
\par
Since we have $\Im(\xi\i\Im\sigma)=\frac{a(ad-bc)}{|\xi|^2}>0$ and $\Re(\xi\i\Im\sigma)=-\frac{b(ad-bc)}{|\xi|^2}<0$ from Lemma~\ref{lem:phi_xi}, we need to show $\pi(a^2+b^2)-a(ad-bc)\ge0$ and $a(a^2+b^2)-b(ad-bc)>0$.
\par
Since $d<\pi/2$ from Lemma~\ref{lem:phi_xi}, we have $\pi(a^2+b^2)-a(ad-bc)=a^2(\pi-d)+\pi b^2+abc>a^2\pi/2+\pi b^2+abc>0$.
\par
To prove $a(a^2+b^2)-b(ad-bc)>0$, we first show that $\cosh{a}-\cos{b}\ge1/2$ implies $a^2+b^2\ge\kappa^2$, where $\kappa:=\arcosh(3/2)$.
In fact we will show its contraposition; that $a^2+b^2<\kappa^2$ implies $\cosh{a}-\cos{b}<1/2$.
\par
From $a^2+b^2<\kappa^2$ and $a>0$, we have $a<\sqrt{\kappa^2-b^2}$ and $b<\kappa$.
Since $\cosh{a}$ is increasing for $a>0$ it is sufficient to prove that $\cosh(\sqrt{\kappa^2-b^2})-\cos{b}<1/2$ for $0<b<\kappa$.
Since we have
\begin{equation*}
\begin{split}
  \frac{d}{d\,b}\bigl(\cosh(\sqrt{\kappa^2-b^2})-\cos{b}\bigr)
  =&
  \sin{b}-\frac{b\sinh(\sqrt{\kappa^2-b^2})}{\sqrt{\kappa^2-b^2}}
  \\
  &\quad\text{(since $\sinh{x}>x$ for $x>0$)}
  \\
  <&
  \sin{b}-b
  <0
\end{split}
\end{equation*}
for $0<b<\kappa$.
Therefore we conclude that $\cosh(\sqrt{\kappa^2-b^2})-\cos{b}<\bigl(\cosh(\sqrt{\kappa^2-b^2})-\cos{b}\bigr)\Bigm|_{b=0}=\cosh\kappa-1=1/2$.
\par
Next, we prove that if $a^2+b^2\ge\kappa^2$, then $a(a^2+b^2)-b(ad-bc)>0$.
\par
From Lemma~\ref{lem:phi_xi}, we have $d<\pi/2$ and $c>\arsinh\beta$, where $\beta=\sinh{a}\sin{b}$ (see \eqref{eq:beta}).
Together with the assumption $a^2+b^2\ge\kappa^2$, we have
\begin{equation*}
  a(a^2+b^2)-b(ad-bc)
  >
  \kappa^2a+b^2\arsinh(\sinh{a}\sin{b})-ab\pi/2.
\end{equation*}
Putting $h(a,b):=\kappa^2a+b^2\arsinh(\sinh{a}\sin{b})-ab\pi/2$, we have
\begin{align*}
  \frac{\partial\,h}{\partial\,a}(a,b)
  =&
  \kappa^2-\frac{\pi}{2}b+\frac{b^2\cosh{a}\sin{b}}{\sqrt{1+\sinh^2{a}\sin^2{b}}},
  \\
  \frac{\partial^2\,h}{\partial\,a^2}(a,b)
  =&
  \frac{b^2\sinh{a}\sin{b}\cos^2{b}}{\left(1+\sinh^2{a}\sin^2{b}\right)^{3/2}}.
\end{align*}
Since $\frac{\partial^2\,h}{\partial\,a^2}(a,b)>0$ for $a>0$ and $0<b<\pi/2$, we conclude that $\frac{\partial\,h}{\partial\,a}(a,b)$ is increasing with respect to $a$.
Therefore $\frac{\partial\,h}{\partial\,a}(a,b)>\frac{\partial\,h}{\partial\,a}(0,b)=\kappa^2-\frac{\pi}{2}b+b^2\sin{b}$.
\par
We will show that $h_0(b):=\frac{\partial\,h}{\partial\,a}(0,b)=\kappa^2-\frac{\pi}{2}b+b^2\sin{b}>0$ for $0<b<\pi/2$.
To do that we consider the function $\th_0(b):=\left(\frac{4\kappa}{\pi}\right)^2\left(b-\frac{\pi}{4}\right)^2\ge0$.
We will prove that $h_0(b)>\th_0(b)$ ($0<b<\pi/2$).
We have
\begin{equation*}
  h_0(b)-\th_0(b)
  =
  b\left(
    \left(\sin{b}-\frac{16\kappa^2}{\pi^2}\right)b
    +
    \frac{16\kappa^2-\pi^2}{2\pi}
  \right).
\end{equation*}
\par
We will show that $k(b):=\left(\sin{b}-\frac{16\kappa^2}{\pi^2}\right)b+\frac{16\kappa^2-\pi^2}{2\pi}$ is positive for $0<b<\pi/2$.
We calculate $k'(b)=\sin{b}+b\cos{b}-\frac{16\kappa^2}{\pi^2}$ and $k''(b)=2\cos{b}-b\sin{b}$.
Since $k''(b)$ decreases from $2$ to $-\pi/2$ for $0<b<\pi/2$, we conclude that there exists $0<b_0<\pi/2$ such that $k''(b)>0$ for $0<b<b_0$, $k''(b_0)=0$ for $b=b_0$, and $k''(b)<0$ for $b_0<b<\pi/2$.
It follows that $k'(b)$ takes its maximum for $0<b<\pi/2$ at $b_0$.
Since Mathematica tells us that $b_0=1.07687\ldots$ and $k'(b_0)=-0.110587\ldots$, we conclude that $k'(b)<0$ for $0<b<\pi/2$.
Thus, $k(b)$ is decreasing and since $k(\pi/2)=0$, we conclude that $k(b)>0$ for $0<b<\pi/2$.
\par
It follows that $h(a,b)$ is decreasing with respect to $a>0$, and so we have $h(a,b)>h(0,b)=0$, which implies $a(a^2+b^2)-b(ad-bc)>0$ and the proof is complete.
\end{proof}
\begin{cor}
If $\cosh{a}-\cos{b}\ge1/2$, then the line segments $h_F$ and $v_F$ are in the region $\{z\in\C\mid|\Re(\xi z)|<a,-\pi<\Im(\xi z)\le\pi\}$.
\end{cor}
Now, we can prove the following two lemmas.
\begin{lem}\label{lem:h_F}
If $\cosh{a}-\cos{b}\ge1/2$ and $a\tanh{c}-b\tan{d}\ge0$, then the line segment $h_F$ is in $H^{+}_F$ except for $\sigma$.
\end{lem}
\begin{proof}
It is sufficient to prove that the line segment $h_G$ is in $H^{+}_G$ except for $\varphi$.
\par
Recall that $h_G$ connects $\varphi$ to $\xi\i\Im\sigma$.
Since the point $\xi\i\Im\sigma=(ad-bc)\xi\i/|\xi|^2$ is on the line $aX+bY=0$ and $h_G$ contains $\varphi\in\overline{C}_G$ (see Definition~\ref{defn:C}), we see that $h_G$ does not touch $\underline{C}_{G}$ from Lemma~\ref{lem:separate_C}.
\par
Since $\cosh{a}-\cos{b}\ge1/2$, from Lemma~\ref{prop:curvature}, the simple closed curve $\overline{C}_G$ is convex.
It follows that every tangent line lies in $H^{+}_G$ except for the point of tangency.
\par
From the proof of Lemma~\ref{lem:Re_sigma}, the slope of tangent at $\varphi$ is $\tanh{c}\cot{d}$, which is greater than or equal to $b/a$ since $a\tanh{c}-b\tan{d}\ge0$.
So the tangent line at $\varphi$ is steeper than the vector $\xi$, which means that the ray $\varphi+s\xi$ with $s<0$ is outside of $\overline{C}_{G}$.
Since $h_G$ does not touch $\underline{C}_{G}$, and is presented as $\varphi+s\xi$ with $-(ac+bd)/|\xi|^2\le t\le0$, we conclude that $h_G$ is in $H^{+}_{G}$ except for $\varphi$, completing the proof.
\end{proof}
\begin{lem}\label{lem:v_F}
If $\cosh{a}-\cos{b}\ge1/2$, then the line segment $v_F$ is in $V^{-}_F$.
\end{lem}
\begin{proof}
We will show that the line segment $v_G$ is in $V^{-}_G$.
\par
Since $v_G$ is parametrized as $t\xi\i$ with $0\le t\le(ad-bc)/|\xi|^2$, from \eqref{eq:HV}, it suffices to show $\sinh\bigl(\Re(t\xi\i)\bigr)\sin\bigl(\Im(t\xi\i)\bigr)-\beta<0$, that is,
\begin{equation*}
  \sinh(bt)\sin(at)+\sinh{a}\sin{b}>0.
\end{equation*}
We have $\sinh{a}\sin{b}>0$ and $\sinh(bt)\ge0$ from $0<b<\pi/2$ and $t\ge0$, respectively.
Moreover, since $t\le(ad-bc)/|\xi|^2$, we have
\begin{equation*}
  at
  \le
  \frac{a(ad-bc)}{|\xi|^2}
  =
  d-\frac{b(ac+bd)}{|\xi|^2}
  <
  \frac{\pi}{2}
\end{equation*}
from Lemma~\ref{lem:phi_xi}.
It follows that $\sin(at)>0$, completing the proof.
\end{proof}

\section{Poisson summation formula}\label{sec:Poisson}
In this section, we will use the Poisson summation formula to change the sum in \eqref{eq:JN_fN} into an integral.
%%%%%%%%%%%%%%%%%%%%%%%%%%%%%%%%%%%%%%%%%%%%%%%%%%%%%%%%%%%%%%%%%%%%%%%%%%%%%%%
\subsection{Statement of the Poisson summation formula}
We use the following version of the Poisson summation formula, a proof of which is similar to that of \cite{Ohtsuki:QT2016}.
See also \cite[Appendix~A]{Murakami:AGT2025}.
\begin{prop}[Poisson summation formula]\label{prop:Poisson}
Let $\{\psi_N(z)\}_{N=1,2,3\dots}$ be a series of holomorphic functions in a domain $D\subset\C$.
We assume that $\psi_N(z)$ uniformly converges to a holomorphic function $\psi(z)$ in $D$.
We put
\begin{align*}
  R_{+}
  &:=
  \{z\in D\mid\Im{z}\ge0,\Re\psi(z)<2\pi\Im{z}\},
  \\
  R_{-}
  &:=
  \{z\in D\mid\Im{z}\le0,\Re\psi(z)<-2\pi\Im{z}\}.
\end{align*}
We also assume the following:
\begin{enumerate}
\item
there exists a closed interval $[p_0,p_1]$ \rm{(}$p_0<p_1$\rm{)} contained in $D$,
\item
$\Re\psi(p_0)<0$ and $\Re\psi(p_1)<0$,
\item
there exists a path $C_{+}$ in $R_{+}$ connecting $p_0$ and $p_1$, which is homotopic to $[p_0,p_1]$ in $D$,
\item
there exists a path $C_{-}$ in $R_{-}$ connecting $p_0$ and $p_1$, which is homotopic to $[p_0,p_1]$ in $D$.
\end{enumerate}
\par
Then we have
\begin{equation*}
  \frac{1}{N}\sum_{p_0\le(2k+1)/(2N)\le p_1}e^{N\psi_N\bigl((2k+1)/(2N)\bigr)}
  =
  \int_{p_0}^{p_1}e^{N\psi_N(z)}\,dz+O\left(e^{-\varepsilon N}\right)
\end{equation*}
for some $\varepsilon>0$ independent of $N$.
\end{prop}
%%%%%%%%%%%%%%%%%%%%%%%%%%%%%%%%%%%%%%%%%%%%%%%%%%%%%%%%%%%%%%%%%%%%%%%%%%%%%%%
\subsection{Assumptions of the Poisson summation formula}
We will show that the assumptions (i)--(iv) of Proposition~\ref{prop:Poisson} are satisfied.
We assume that $\xi\in\Xi$, that is, $a>0$, $0<b<\pi/2$, and $\alpha=\cosh{a}\cos{b}>1/2$, as usual.
\par
We put
\begin{equation*}
  D:=\{z\in\C\mid-\delta<\Re{z}<1-\delta,-4\Im\sigma<\Im{z}<2\Im\sigma\}
\end{equation*}
for a sufficiently small $\delta>0$.
\begin{lem}\label{lem:D_Theta}
The region $D$ is contained in $\Theta_{\nu}$ if we choose $\nu$ sufficiently small, and $M$ sufficiently large.
\end{lem}
\begin{proof}
Recall the definition \eqref{eq:Theta_nu} of $\Theta_{\nu}$.
\par
We first show that $D$ is in the parallelogram (in fact it is a rectangle) $\{z\in\C\mid|\Im(\xi z)|\le b+2\pi(1-\nu),|\Re(\xi z)|\le2M\pi-a\}$ if we choose $0<\nu<1/4$ sufficiently small and $M>0$ sufficiently large.
\par
If $z\in D$, then we have $-\delta<x<1-\delta$ and $-4\Im\sigma<y<2\Im\sigma$, where we put $x:=\Re{z}$ and $y:=\Im{z}$ as usual.
Since $|\Re(\xi z)|=|ax-by|\le a|x|+b|y|<a(1-\delta)+4b\Im\sigma$, if we choose $M$ so that $M>\bigl(a+a(1-\delta)+4b\Im\sigma\bigr)/(2\pi)$, then we have $|\Re(\xi z)|<2M\pi-a$.
Recalling that $\sigma=\varphi/\xi$, we have
\begin{equation*}
\begin{split}
  b+2\pi(1-\nu)-|\Im(\xi z)|
  =&
  b+2\pi(1-\nu)-|bx+ay|
  \\
  \ge&
  b+2\pi(1-\nu)-b|x|-a|y|
  \\
  >&
  b+2\pi(1-\nu)-b(1-\delta)-4a\Im\sigma
  \\
  =&
  \delta b+2\pi(1-\nu)-\frac{4a(ad-bc)}{|\xi|^2}
  \\
  =&
  \delta b-2\pi\nu+\frac{4abc}{|\xi|^2}+\frac{1}{|\xi|^2}\left(2\pi(a^2+b^2)-4a^2d\right)
  \\
  >&
  \delta b-2\pi\nu+\frac{4abc}{|\xi|^2}+\frac{2\pi b^2}{|\xi|^2}
  >0
\end{split}
\end{equation*}
if we choose $\nu>0$ sufficiently small, where we use $d<\pi/2$ from Lemma~\ref{lem:phi_xi} in the last inequality.
This proves $|\Im(\xi z)|<b+2\pi(1-\nu)$.
\par
It remains to show that $D$ avoids $\Delta^{\pm}_{\nu}$ and $\nabla^{\pm}_{\nu}$.
\par
The $x$-coordinate of the left-most point of $\Delta^{+}_{\nu}$ is $1-2\pi\nu/b$ (see Figure~\ref{fig:converge_F}).
Since $x<1-\delta$ if $z\in D$, and
\begin{equation*}
  1-\frac{2\pi\nu}{b}-(1-\delta)
  =
  \delta-\frac{2\pi\nu}{b}>0
\end{equation*}
if $\nu$ is sufficiently small.
So we conclude that $D\cap\Delta^{+}_{\nu}=\emptyset$.
Similarly, the $x$-coordinate of the right-most point of $\nabla^{+}_{\nu}$ is $-1+2\pi\nu/b$ (see Figure~\ref{fig:converge_F}).
Since $x>-\delta$ if $z\in D$, and
\begin{equation*}
  (-\delta)-\left(-1+\frac{2\pi\nu}{b}\right)
  =
  1-\delta-\frac{2\pi\nu}{b}>0
\end{equation*}
if $\nu$ is sufficiently small, which implies that $D\cap\nabla^{+}_{\nu}=\emptyset$.
\par
Since the $y$-coordinate of the bottom side of $\Delta^{-}_{\nu}$ is $2\pi(1-\nu)a/|\xi|^2$ (see Figure~\ref{fig:converge_F}), $\Im\sigma=(ad-bc)/|\xi|^2$, and $d<\pi/2$ from Lemma~\ref{lem:phi_xi}, we have
\begin{equation*}
\begin{split}
  \frac{2\pi(1-\nu)a}{|\xi|^2}-2\Im\sigma
  =&
  \frac{1}{|\xi|^2}\bigl(2a\pi-2a\pi\nu-2(ad-bc)\bigr)
  \\
  >&
  \frac{1}{|\xi|^2}(a\pi+2bc-2a\pi\nu),
\end{split}
\end{equation*}
which is positive if $\nu$ is sufficiently small.
This shows that $D$ is below $\Delta^{-}_{\nu}$.
Similarly, the $y$-coordinate of the top side of $\nabla^{-}_{\nu}$ is $-2\pi(1-\nu)a/|\xi|^2$ (see Figure~\ref{fig:converge_F}).
Since $d<\pi/2$, we have
\begin{equation*}
\begin{split}
  -4\Im\sigma-\left(-\frac{2\pi(1-\nu)a}{|\xi|^2}\right)
  =&
  \frac{1}{|\xi|^2}\bigl(2a\pi-2a\pi\nu-4(ad-bc)\bigr)
  \\
  >&
  \frac{1}{|\xi|^2}(-2a\pi\nu+4bc)
\end{split}
\end{equation*}
which is positive if $\nu$ is sufficiently small.
So we conclude that that $D$ is above $\nabla^{-}_{\nu}$.
\par
So we conclude that $D\subset\Theta_{\nu}$ if we choose $\nu>0$ sufficiently small.
\end{proof}
%%%%%%%%%%%%%%%%%%%%%%%%%%%%%%%%%%%%%%%%%%%%%%%%%%%%%%%%%%%%%%%%%%%%%%%%%%%%%%%
Now, we will prove that Conditions~(i)--(iv) hold, assuming that $a\tanh{c}-b\tan{d}\ge0$, and that $\cosh{a}-\cos{b}>1/2$.
We consider two cases $\Re F(\sigma)>0$ and $\Re F(\sigma)\le0$ separately.
%%%%%%%%%%%%%%%%%%%%%%%%%%%%%%%%%%%%%%%%%%%%%%%%%%%%%%%%%%%%%%%%%%%%%%%%%%%%%%%
\subsubsection{The case  $\Re F(\sigma)>0$}\label{subsubsec:positive}
First, we assume that $\Re F(\sigma)>0$.
\par
We put
\begin{itemize}
\item
$\psi_N(z):=f_N(z)-f_N(\sigma)$ and $\psi(z):=F(z)-F(\sigma)$,
\item
$p_0:=0$ and $p_1=1-\delta_1$ for sufficiently small $\delta_1>0$ with $\delta_1>\delta$.
\end{itemize}
\par
From Lemmas~\ref{lem:converge_F} and \ref{lem:D_Theta}, the series of functions $\{\psi_N(z)\}$ uniformly converges to $\psi(z)$ in $D$.
\par
Clearly, Condition (i) holds if we choose $\delta_1$ so that $\delta_1<1$.
\par
As for Condition (ii), we have $\Re\psi(p_0)=\Re F(0)-\Re F(\sigma)=-\Re F(\sigma)<0$ from the assumption.
We also have $\Re\psi(p_1)=\Re F(1-\delta_1)-F(\sigma)$, which is negative if $\delta_1>0$ is sufficiently small, because $\Re F(1)<\Re F(\sigma)$ if $a\tanh{c}-b\tan{d}\ge0$ from Corollary~\ref{cor:F1}.
\begin{rem}
Since the point $1$ is in $H^{-}_F$ from \eqref{eq:HVF_def}, $\Re F(x)$ is decreasing if $x$ is close to $1$ and $x<1$ from Lemma~\ref{lem:ReF}.
So we can choose $\delta_1$ so that $\Re F(x)<\Re F(\sigma)$  for $x\in[1-\delta_1,1]$.
\end{rem}
\par
So it remains to show that Conditions (iii) and (iv) are satisfied.
\par
We note that $R_{\pm}$ are given as follows.
\begin{align*}
  R_{+}
  &=
  \{z\in\C\mid0\le\Im{z}<2\Im\sigma,-\delta<\Re{z}<1-\delta,
  \\
  &\phantom{R_{+}=\{z\in\C\mid}
  \Re{F(z)}<\Re{F(\sigma)}+2\pi\Im{z}\},
  \\
  R_{-}
  &=
  \{z\in\C\mid-4\Im\sigma<\Im{z}\le0,-\delta<\Re{z}<1-\delta,
  \\
  &\phantom{R_{-}=\{z\in\C\mid}
  \Re{F(z)}<\Re{F(\sigma)}-2\pi\Im{z}\}.
\end{align*}
Note also that since $D$ is simply connected, it suffices to show the existence of paths $C_{\pm}\subset R_{\pm}$ connecting $p_0$ and $p_1$.
See Figure~\ref{fig:contour_Re_F_pos}.
\begin{figure}[h]
\pic{0.67}{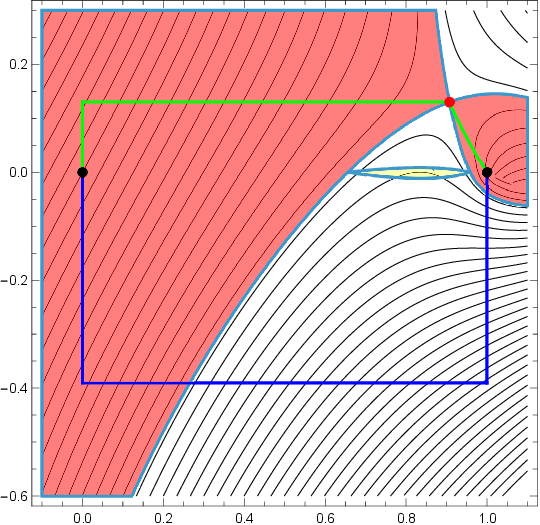}
\caption{A contour plot of $\Re F(z)$ with $\xi=1.5+0.5\i$.
Note that $\Re F(\sigma)>0$.
The red point is $\sigma$, and the black points are $0$ and $1$.
The region $\{z\in\C\mid\Re F(z)<0\}$ is indicated by pink.
The yellow region indicates $D\setminus\left(R_{+}\cup R_{-}\right)$.
The green (blue, respectively) curve is $\overline{\lambda}$ ($\underline{\lambda}$, respectively).}
\label{fig:contour_Re_F_pos}
\end{figure}
\begin{lem}\label{lem:Poisson_iii}
If $\Re F(\sigma)>0$, then Condition {\rm(iii)} holds.
\end{lem}
\begin{proof}
Put $\overline{\lambda}:=v_F\cup h_F\cup\chi_F$, where $v_F$ ($h_F$, respectively) is the line segment connecting $O$ to $\i\Im\sigma$ ($\sigma$ to $\i\Im\sigma$, respectively), and $\chi_F$ is the curve connecting $1$ to $\sigma$.
See Definitions~\ref{defn:chi_F_G} and \ref{defn:h_v_F_G}.
Then the piecewise smooth curve $\overline{\lambda}$ connects $O$ to $1$ via $\sigma$.
From the assumption $a\tanh{c}-b\tan{d}\ge0$, Lemma~\ref{lem:Re_sigma} and its proof, we see that the curve $\chi_F$ is on the left of the line $\Re{z}=1$ except for $1$ (see Figure~\ref{fig:location_sigma}).
\par
Since $a\tanh{c}-b\tan{d}\ge0$, from Corollary~\ref{cor:F1}, $\Re{F(z)}$ is decreasing from $\sigma$ to $1$ along the curve $\chi_F$, and so if $z\in\chi_F$ and $z\ne\sigma$ then $\Re F(z)<\Re F(\sigma)$.
From Lemma~\ref{lem:h_F}, $\Re F(x+y\i)$ is increasing with respect to $x$ along the line segment $h_F$.
So if $z\in h_F$, then $\Re F(z)<\Re F(\sigma)$ except for $z=\sigma$.
From Lemma~\ref{lem:v_F}, $\Re F(x+y\i)$ is decreasing with respect to $y$ along the line segment $v_F$ since $\cosh{a}-\cos{b}>1/2$.
So if $z\in v_F$, then $\Re F(z)<\Re F(0)=0<\Re F(\sigma)$.
\par
Therefore if $z\in\overline{\lambda}\setminus\{\sigma\}$, then $\Re F(z)<\Re F(\sigma)$.
\par
Since the curve $\overline{\lambda}$ is in the rectangle $\{z\in\C\mid0\le\Im{z}\le\Im\sigma,0\le\Re{z}\le1\}$, we can obtain $C_{+}$ by pushing $v_F$ a little to the right near $O$, and $\chi_F$ a little to the left near $1$.
\end{proof}
\begin{rem}\label{rem:Poisson_iii}
The curve $C_{+}$ defined above in contained in a region smaller than $R_{+}$:
\begin{equation*}
  \tR_{+}
  =
  \{z\in\C\mid0\le\Im{z}<2\Im\sigma,-\delta<\Re{z}<1-\delta,
  \Re{F(z)}\le\Re{F(\sigma)}\}.
\end{equation*}
\end{rem}
\begin{lem}\label{lem:Poisson_iv}
If $\Re F(\sigma)>0$, then Condition {\rm(iv)} holds.
\end{lem}
\begin{proof}
Let $\lambda_1$, $\lambda_2$, and $\lambda_3$ be the line segments connecting $0$ to $-3\i\Im\sigma$, $-3\i\Im\sigma$ to $1-3\i\Im\sigma$, and $1-3\i\Im\sigma$ to $1$, respectively.
Put $\underline{\lambda}:=\lambda_1\cup\lambda_2\cup\lambda_3$, which connects $0$ to $1$.
\par
Define $r_x(y):=\Re{F(x+y\i)}-\Re{F(\sigma)}+2\pi y$ as a function of $y$.
Since $\underline{\lambda}$ is in the region $\{z\in\C\mid0\le\Re{z}\le1,\Im{z}\le0\}$, from Corollary~\ref{cor:arg} we have $r'_{x}(y)=\frac{\partial}{\partial y}\Re F(x+y\i)+2\pi>\pi/2$.
\par
We will show that $r_x(y)<0$ when $(x,y)$ is on $\underline{\lambda}$.
\par
\begin{itemize}
\item
$\lambda_1\subset R_{-}$:
Since $r_0'(y)>\pi/2$, we obtain
\begin{equation*}
\begin{split}
  r_0(y)
  =&
  \int_{0}^{y}r'_0(s)\,ds+r_0(0)
  =
  \int_{0}^{y}r'_0(s)\,ds-\Re F(\sigma)
  \\
  <&
  y\pi/2\le0
\end{split}
\end{equation*}
since $y\le0$ and $\Re F(\sigma)>0$.
Thus we have $r_x(y)<0$ if $z\in\lambda_1$.
\item
$\lambda_2\subset R_{-}$:
For each point $x-3\i\Im\sigma\in\lambda_2$ ($0\le x\le1$), we choose $y_0$ with $0\le y_0\le\Im\sigma$ so that the point $(x,y_0)$ is contained in $\overline{\lambda}$.
\par
Since the curve $\overline{\lambda}$ is in $\tR_{+}$ from Remark~\ref{rem:Poisson_iii}, we have $r_x(y_0)=\Re F(x+y_0\i)-\Re F(\sigma)+2\pi y_0\le2\pi y_0$.
The inequality $r_x'(y)>\pi/2$ implies
\begin{equation*}
\begin{split}
  r_x(-3\Im\sigma)
  =&
  \int_{y_0}^{-3\Im\sigma}r_x'(s)\,ds+r_x(y_0)
  \\
  <&
  \frac{1}{2}\pi(-3\Im\sigma-y_0)+2\pi y_0
  =
  \frac{3\pi}{2}(y_0-\Im\sigma)
  \le0.
\end{split}
\end{equation*}
So we have proved that $r_x(y)<0$ if $(x,y)\in\lambda_2$.
\item
$\lambda_3\subset R_{-}$:
Since $r_1(0)=\Re{F(1)}-\Re{F(\sigma)}<0$ from Corollary~\ref{cor:F1}, and $r_1'(s)>\pi/2$ for $-3\Im\sigma\le s\le0$ as above, we have
\begin{equation*}
  r_1(y)
  =
  \int_{0}^{y}r_1'(s)\,ds+r_1(0)
  <0,
\end{equation*}
proving that $r_x(y)<0$ if $(x,y)\in\lambda_3$.
\end{itemize}
Since the polygonal line $\underline{\lambda}$ is in the rectangle $\{z\in\C\mid-3\Im\sigma\le\Im{z}\le0,0\le\Re{z}\le1\}$, the curve $C_{-}$ is obtained by pushing it a little inside.
\par
The proof is complete.
\end{proof}
Therefore from Proposition~\ref{prop:Poisson}, we have
\begin{multline}\label{eq:sum_int_positive}
  \frac{1}{N}
  e^{-Nf_N(\sigma)}
  \sum_{0\le(2k+1)/(2N)\le1-\delta_1}e^{Nf_N\bigl((2k+1)/(2N)\bigr)}
  \\
  =
  e^{-Nf_N(\sigma)}
  \int_{0}^{1-\delta_1}e^{Nf_N(z)}\,dz
  +
  O\left(e^{-\varepsilon N}\right)
\end{multline}
as $N\to\infty$ for some $\varepsilon>0$ independent of $N$.
%%%%%%%%%%%%%%%%%%%%%%%%%%%%%%%%%%%%%%%%%%%%%%%%%%%%%%%%%%%%%%%%%%%%%%%%%%%%%%%
\subsubsection{The case $\Re F(\sigma)\le0$}\label{subsubsec:negative}
Next, we assume that $\Re F(\sigma)\le0$.
\par
We put
\begin{itemize}
\item
$\psi_N(z):=f_N(z)$ and $\psi(z):=F(z)$,
\item
$p_0:=-\delta_0$ and $p_1=1-\delta_1$,
\end{itemize}
where we choose $\delta_0>0$ and $\delta_1>0$ sufficiently small with $\delta_0<\delta$ and $\delta_1>\delta$.
\par
\begin{lem}\label{lem:Poisson_negative}
Suppose that $a\tanh{c}-b\tan{d}\ge0$ and that $\cosh{a}-\cos{b}>1/2$.
If $\Re F(\sigma)\le0$, then Conditions {\rm(i)}--{\rm(iv)} hold.
\end{lem}
\begin{proof}
Condition~(i) holds because $\delta_0<\delta$ and $\delta_1>\delta$.
\par
Since $|1-\cosh{\xi}|=\cosh{a}-\cos{b}$ from Lemma~\ref{lem:alpha_beta_a_b}, the origin is in $H^{+}_{F}$ from the definition \eqref{eq:HVF_def}.
So we can choose $\delta_0$ so that $\Re F(x)$ is increasing for $-\delta_0<x<0$.
Since $F(0)=0$, we have $\Re F(p_0)=\Re F(-\delta_0)<0$.
\par
Since $\Re F(1)<\Re F(\sigma)\le0$ from Corollary~\ref{cor:F1}, we can choose $\delta_1>0$ with $\delta_1<\delta$ so small that $\Re F(p_1)=\Re F(1-\delta_1)<0$.
Therefore we conclude that Condition (ii) holds.
\par
Conditions (iii) and (iv) also hold as above if we assume that $a\tanh{c}-b\tan{d}\ge0$, and $\cosh{a}-\cos{b}\ge1/2$.
See Figure~\ref{fig:contour_Re_F_0_neg}.
\begin{figure}[h]
\pic{0.67}{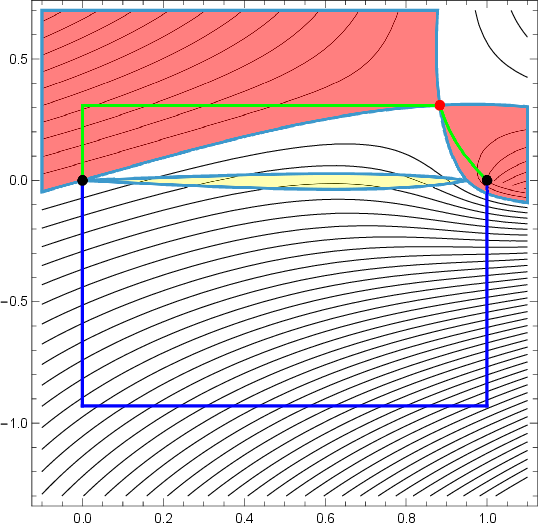}\quad
\pic{0.67}{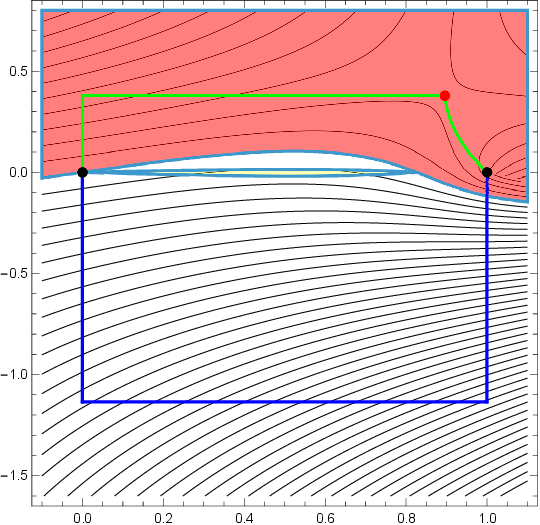}
\caption{Contour plots of $\Re F(z)$ with $\xi=1.0943\ldots+0.5\i$ (left) and $\xi=1+0.5\i$ (right).
Note that $\Re F(\sigma)=0$ (left) and that $\Re F(\sigma)<0$ (right).
The red points are $\sigma$, and the black points are $0$ and $1$.
The regions with $\Re F(z)<0$ are indicated by pink.
The region $D\setminus\left(R_{+}\cup R_{-}\right)$ is indicated by yellow.
The green (blue, respectively) curve is $\overline{\lambda}$ ($\underline{\lambda}$, respectively).}
\label{fig:contour_Re_F_0_neg}
\end{figure}
\end{proof}
Therefore if $a\tanh{c}-b\tan{d}\ge0$, $\cosh{a}-\cos{b}>1/2$, and $\Re F(\sigma)\le0$, we have
\begin{equation}\label{eq:sum_int_negative}
  \frac{1}{N}
  \sum_{-\delta_0\le(2k+1)/(2N)\le1-\delta_1}e^{Nf_N\bigl((2k+1)/(2N)\bigr)}
  =
  \int_{-\delta_0}^{1-\delta_1}e^{Nf_N(z)}\,dz
  +
  O\left(e^{-\varepsilon N}\right)
\end{equation}
as $N\to\infty$ for some $\varepsilon>0$ independent of $N$.
%%%%%%%%%%%%%%%%%%%%%%%%%%%%%%%%%%%%%%%%%%%%%%%%%%%%%%%%%%%%%%%%%%%%%%%%%%%%%%%

\section{Proof of the main theorem}\label{sec:saddle}
%%%%%%%%%%%%%%%%%%%%%%%%%%%%%%%%%%%%%%%%%%%%%%%%%%%%%%%%%%%%%%%%%%%%%%%%%%%%%%%
%%%%%%%%%%%%%%%%%%%%%%%%%%%%%%%%%%%%%%%%%%%%%%%%%%%%%%%%%%%%%%%%%%%%%%%%%%%%%%%
In this section, we apply the saddle point method to the integrals appearing in \eqref{eq:sum_int_positive} and \eqref{eq:sum_int_negative}, and prove the main theorem.
\par
We assume $a>0$, $0<b<\pi/2$, and $\alpha=\cosh{a}\cos{b}>1/2$.
\subsection{Saddle point method}
We use the saddle point method as in \cite[Proposition~5.3]{Murakami:arXiv2023}.
See also \cite[Proposition~3.2 and Remark~3.3]{Ohtsuki:QT2016} and \cite[Proposition~5.2]{Murakami:AGT2025}.
\begin{prop}[Saddle point method]\label{prop:saddle}
Let $\eta(w)$ be a holomorphic function in a domain $E$ containing $0$, with $\eta(0)=\eta'(0)=0$ and $\eta''(0)\ne0$.
We put $V:=\{w\in E\mid\Re\eta(w)<0\}$.
Let $C$ be a path in $E$ from $q_0$ to $q_1$, where the two end-points are in $V$.
\par
We assume the following:
\begin{enumerate}
\item
there exists a neighborhood $\hE\subset E$ of $0$ such that $V\cap\hE$ has two connected components,
\item
there exists a path $\hC\subset V\cup\{0\}$ connecting $q_0$ and $q_1$ via $0$ that is homotopic to $C$ in $E$ fixing $q_0$ and $q_1$, such that $(\hC\cap\hE)\setminus\{0\}$ splits into two connected components that are in distinct components of $V$.
\end{enumerate}
Let $\{h_N(w)\}_{N=1,2,3,\dots}$ be a series of holomorphic functions in $E$ that uniformly converges to a holomorphic function $h(w)$ with $h(0)\ne0$.
We also assume that $|h_N(w)|$ is bounded irrelevant to $N$.
\par
Then, we have
\begin{equation*}
  \int_{C}h_N(w)e^{N\eta(w)}\,dw
  =
  \frac{h(0)\sqrt{2\pi}}{\sqrt{-\eta''(0)}\sqrt{N}}\bigl(1+O(N^{-1})\bigr)
\end{equation*}
as $N\to\infty$, where we choose the sign of $\sqrt{-\eta''(0)}$ so that $\Re\left(q_1\sqrt{-\eta''(0)}\right)>0$.
\end{prop}
Putting $\eta(w):=F(w+\sigma)-F(\sigma)$, $h_N(w):=e^{N(f_N(w+\sigma)-F(w+\sigma))}$, $h(w):=1$, and $E:=D-\sigma=\{w\in\C\mid-\delta<\Re(w+\sigma)<1-\delta,-4\Im\sigma<\Im(w+\sigma)<2\Im\sigma\}$.
We will show that the assumptions of Proposition~\ref{prop:saddle} hold for certain points $q_0$ and $q_1$, and certain paths $C$ and $\hC$.
\par
First of all, we see that $\eta(0)=0$, and that $\eta'(0)=F'(0)=0$ from the definition of $\sigma$ (see \eqref{eq:def_sigma}).
From Corollary~\ref{cor:F} we have
\begin{equation*}
  F''(z)
  =
  \frac{-\xi\sinh(\xi z)}{\cosh\xi-\cosh(\xi z)}
\end{equation*}
if $|\Re(\xi z)|<a$.
Since $|\Re(\xi\sigma)|=|\Re\varphi|=c<a$ from Lemma~\ref{lem:phi_xi}, we have
\begin{equation*}
\begin{split}
  \eta''(0)
  =&
  F''(\sigma)
  =
  \frac{-\xi\sinh\varphi}{\cosh\xi-\cosh\varphi}
  \\
  =&
  -2\xi\sinh\varphi
  =
  -2\xi\sqrt{\cosh^2\varphi-1}
  \\
  =&
  -\xi\sqrt{(2\cosh\xi+1)(2\cosh\xi-3)},
\end{split}
\end{equation*}
which does not vanish because $\varphi\ne0$.
Here the third equality holds since $\cosh\varphi=\cosh\xi-1/2$, and we choose the square root so that its argument is between $0$ and $\pi/2$ because $\sinh\varphi=\sinh{c}\cos{d}+\i\cosh{c}\sin{d}$ is in the first quadrant.
\par
So, it remains to prove the following lemma.
\begin{lem}\label{lem:saddle_i_ii}
Suppose that $a\tanh{c}-b\tan{d}\ge0$ and $\cosh{a}-\cos{b}>1/2$.
Then Conditions (i) and (ii) hold if we choose $q_0$, $q_1$, $C$ and $\hC$ appropriately.
\end{lem}
\begin{proof}
We first prove (i).
From the Taylor expansion around $0$, we have
\begin{equation*}
\begin{split}
  \eta(w)
  =&
  \frac{1}{2}\eta''(0)w^2+c_3w^3+c_4w^4+\dots.
\end{split}
\end{equation*}
So if we choose $\hE\ni0$ sufficiently small, then (i) follows.
\par
To prove (ii), we will consider the two cases where $\Re F(\sigma)>0$ and $\Re F(\sigma)\le0$.
\begin{itemize}
\item
The case $\Re F(\sigma)>0$.
\par
We put $q_0:=-\sigma$ and $q_1:=1-\delta_1-\sigma$, and let $C$ be the line segment connecting $q_0$ to $q_1$, where we choose $\delta_1>0$ as Subsubsection~\ref{subsubsec:positive}.
Now, we put $\hC:=C_{+}-\sigma$, where $C_{+}$ is defined in the proof of Lemma~\ref{lem:Poisson_iii}.
Then the path $\hC$ connects $-\sigma$ to $1-\delta_1-\sigma$ via $0$ in the region $\{w\in E\mid \Re\eta(w)<0\}$ except for $0$, since $C_{+}\subset\{z\in D\mid\Re{F(z)}-\Re{F(\sigma)}<0\}$ except for $\sigma$.
Since near $\sigma$, $C_{+}$ is split into $h_{F}$ and $\chi_{F}$, we may assume that $(\hC\cap\hE)\setminus\{0\}$ has two connected components taking $\hE$ smaller if necessary.
Clearly, $\hC$ is homotopic to $C$ in $E$.
\item
The case $\Re F(\sigma)\le0$.
\par
We put $q_0:=-\Re\sigma$ and $q_1:=1-\delta_1-\sigma$, and let $C$ be the polygonal chain connecting $q_0$, $-\sigma$, and $q_1$, where we choose $\delta_1>0$ as Subsubsection~\ref{subsubsec:negative}.
Then $C$ is homotopic to the path $\hC:=h_F\cup\chi_F-\sigma=\{z-\sigma\mid z\in h_F\cup\chi_F\}$.
The path $\hC$ connects $-\Re\sigma$ to $1-\delta_1-\sigma$ via $0$ and it is in the region $\{w\in E\mid\Re\eta(w)<0\}$ except for $O$ since $h_F\cup\chi_F$ is in the region $\{z\in D\mid\Re F(z)<\Re F(\sigma)\}$ except for $\sigma$ from Lemma~\ref{lem:h_F} and Corollary~\ref{cor:F1}.
It is clear that $\hC$ is homotopic to $C$.
See Figure~\ref{fig:saddle_C}.
\begin{figure}[h]
\pic{0.67}{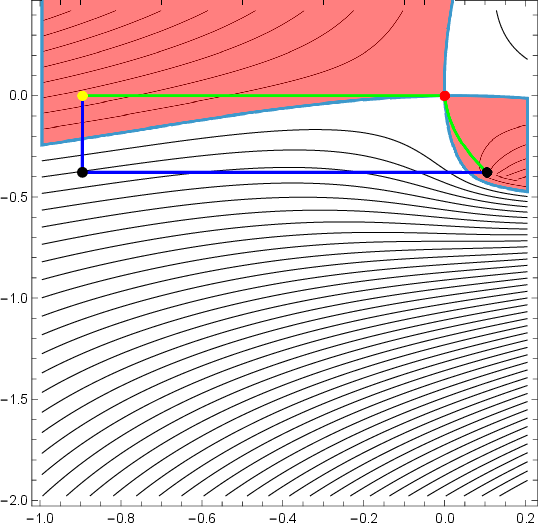}
\caption{A contour plot of $\Re\eta(w)$ with $\xi=1+0.5\i$.
The pink region indicates $\{w\in\C\mid\Re\eta(w)<0\}$.
The green curve indicates $\hC$, and the blue curve indicates $C$.
The red point is $O$, the yellow point is $-\Re\sigma$, and the black points are $-\sigma$ and $1-\delta_1-\sigma$.
Compare it with the right picture of Figure~\ref{fig:contour_Re_F_0_neg}.}
\label{fig:saddle_C}
\end{figure}
\end{itemize}
So we conclude that Condition (ii) holds.
\end{proof}
Therefore we can use Proposition~\ref{prop:saddle} to estimate the integrals in the right hand sides of \eqref{eq:sum_int_positive} and \eqref{eq:sum_int_negative}.
\par
If $\Re F(\sigma)>0$, we obtain
\begin{equation}\label{eq:saddle_point_C}
\begin{split}
  &\int_{C}e^{N\bigl(f_N(w+\sigma)-F(\sigma)\bigr)}\,dw
  \\
  =&
  \int_{C}
  e^{N\bigl(f_N(w+\sigma)-F(w+\sigma)\bigr)}
  e^{N\bigl(F(w+\sigma)-F(\sigma)\bigr)}
  \,dw
  \\
  =&
  \frac{\sqrt{2\pi}}
  {\sqrt{\xi\sqrt{(2\cosh\xi+1)(2\cosh\xi-3)}}\sqrt{N}}
  \times\bigl(1+O(N^{-1})\bigr),
\end{split}
\end{equation}
where we choose the square root so that $\Re\sqrt{\xi\sqrt{(2\cosh\xi+1)(2\cosh\xi-3)}}>0$.
Putting $z:=w+\sigma$, we have
\begin{equation*}
  \int_{0}^{1-\delta_1}e^{Nf_N(z)}\,dz
  =
  \frac{\sqrt{2\pi}e^{NF(\sigma)}}
  {\sqrt{\xi\sqrt{(2\cosh\xi+1)(2\cosh\xi-3)}}\sqrt{N}}
  \times\bigl(1+O(N^{-1})\bigr).
\end{equation*}
\par
Together with \eqref{eq:sum_int_positive} we have
\begin{equation*}
\begin{split}
  &\frac{1}{N}e^{-Nf_N(\sigma)}\sum_{0\le(2k+1)/(2N)\le1-\delta_1}e^{Nf_N\bigl((2k+1)/(2N)\bigr)}
  \\
  =&
  \frac{\sqrt{2\pi}e^{N\bigl(F(\sigma)-f_N(\sigma)\bigr)}}
       {\sqrt{\xi\sqrt{(2\cosh\xi+1)(2\cosh\xi-3)}}\sqrt{N}}
  \times\bigl(1+O(N^{-1})\bigr)
  +O\left(e^{-\varepsilon N}\right).
\end{split}
\end{equation*}
Therefore we have
\begin{equation}\label{eq:saddle_positive}
\begin{split}
  &\sum_{0\le(2k+1)/(2N)\le1-\delta_1}e^{Nf_N\bigl((2k+1)/(2N)\bigr)}
  \\
  =&
  \frac{\sqrt{2\pi N}e^{NF(\sigma)}}{\sqrt{\xi\sqrt{(2\cosh\xi+1)(2\cosh\xi-3)}}}
  \left(
    1+O\bigl(N^{-1}\bigr)
     +O\Bigl(N^{1/2}e^{N\bigl(f_N(\sigma)-F(\sigma)-\varepsilon\bigr)}\Bigr)
  \right)
  \\
  =&
  \frac{\sqrt{2\pi N}e^{NF(\sigma)}}{\sqrt{\xi\sqrt{(2\cosh\xi+1)(2\cosh\xi-3)}}}
  \left(1+O\bigl(N^{-1}\bigr)\right),
\end{split}
\end{equation}
where the second equality holds because $f_N(\sigma)$ converges to $F(\sigma)$.
\par
If $\Re F(\sigma)\le0$, then \eqref{eq:saddle_point_C} also holds as in the case $\Re F(\sigma)>0$.
Putting $z:=w+\sigma$, we also have
\begin{equation*}
\begin{split}
  \int_{C}e^{Nf_N(w)}\,dw
  =&
  \int_{\i\Im\sigma}^{0}e^{Nf_N(z)}\,dz+\int_{0}^{1-\delta_1}e^{Nf_N(z)}\,dz
  \\
  =&
  \frac{\sqrt{2\pi}e^{NF(\sigma)}}
  {\sqrt{\xi\sqrt{(2\cosh\xi+1)(2\cosh\xi-3)}}\sqrt{N}}
  \times\bigl(1+O(N^{-1})\bigr)
\end{split}
\end{equation*}
since $C$ is a polygonal line connecting $-\Re\sigma$, $-\sigma$, and $1-\delta_1-\sigma$.
So we have
\begin{equation*}
\begin{split}
  &\int_{0}^{1-\delta_1}e^{Nf_N(z)}\,dz
  \\
  =&
  \frac{\sqrt{2\pi}e^{NF(\sigma)}}{\sqrt{\xi\sqrt{(2\cosh\xi+1)(2\cosh\xi-3)}}\sqrt{N}}
  \bigl(1+O(N^{-1})\bigr)
  -
  \int_{\i\Im\sigma}^{0}
  e^{Nf_N(z)}\,dz.
\end{split}
\end{equation*}
From \eqref{eq:sum_int_negative} we have
\begin{equation}\label{eq:saddle_negative}
\begin{split}
  &\sum_{0\le(2k+1)/(2N)\le1-\delta_1}e^{Nf_N\bigl((2k+1)/(2N)\bigr)}
  \\
  =&
  \sum_{-\delta_0\le(2k+1)/(2N)\le1-\delta_1}e^{Nf_N\bigl((2k+1)/(2N)\bigr)}
  -
  \sum_{-\delta_0\le(2k+1)/(2N)<0}e^{Nf_N\bigl((2k+1)/(2N)\bigr)}
  \\
  =&
  N\int_{-\delta_0}^{1-\delta_1}e^{Nf_N(z)}\,dz
  -
  \sum_{-\delta_0\le(2k+1)/(2N)<0}e^{Nf_N\bigl((2k+1)/(2N)\bigr)}
  +O\left(Ne^{-\varepsilon N}\right)
  \\
  =&
  N\int_{0}^{1-\delta_1}e^{Nf_N(z)}\,dz
  +N\int_{-\delta_0}^{0}e^{Nf_N(z)}\,dz
  -\sum_{-\delta_0\le(2k+1)/(2N)<0}e^{Nf_N\bigl((2k+1)/(2N)\bigr)}
  \\
  &+O\left(Ne^{-\varepsilon N}\right)
  \\
  =&
  \frac{\sqrt{2\pi N}e^{NF(\sigma)}}{\sqrt{\xi\sqrt{(2\cosh\xi+1)(2\cosh\xi-3)}}}
  \bigl(1+O(N^{-1})\bigr)
  -N\int_{\i\Im\sigma}^{0}e^{Nf_N(z)}\,dz
  \\
  &
  +N\int_{-\delta_0}^{0}e^{Nf_N(z)}\,dz
  -\sum_{-\delta_0\le(2k+1)/(2N)<0}e^{Nf_N\bigl((2k+1)/(2N)\bigr)}
  +O\left(Ne^{-\varepsilon N}\right).
\end{split}
\end{equation}
%%%%%%%%%%%%%%%%%%%%%%%%%%%%%%%%%%%%%%%%%%%%%%%%%%%%%%%%%%%%%%%%%%%%%%%%%%%%%%%
\subsection{Proof of the main theorem}\label{subsec:main}
To prove the main theorem, we need to estimate the sums $\sum_{1-\delta_1<(2k+1)/(2N)<1}e^{Nf_N\bigl((2k+1)/(2N)\bigr)}$ (for both cases $\Re F(\sigma)>$ and $\Re F(\sigma)\le0$, Corollary~\ref{cor:sum_delta1}) and $\sum_{-\delta_0\le(2k+1)/(2N)<0}e^{Nf_N\bigl((2k+1)/(2N)\bigr)}$ (for the case $\Re F(\sigma)\le0$ and $\cosh{a}-\cos{b}>1/2$, Lemma~\ref{lem:fN_DN}), and the difference of the integrals $\int_{-\delta_0}^{0}e^{Nf_N(z)}\,dz-\int_{\i\Im\sigma}^{0}e^{Nf_N(z)}\,dz$ (for the case $\Re F(\sigma)\le0$ and $\cosh{a}-\cos{b}>1/2$, Lemma~\ref{lem:int_delta_sigma}).
\par
\begin{lem}\label{lem:ReF_1}
Assume that $a\tanh{c}-b\tan{d}\ge0$ so that $\Re F(\sigma)>\Re F(1)$ from Corollary~\ref{cor:F1}.
For any $\varepsilon_1$ with $0<2\varepsilon_1<\Re F(\sigma)-\Re F(1)$, there exists $\delta_2>0$ such that
\begin{equation*}
  \Re f_N\bigl((2k+1)/(2N)\bigr)<\Re F(\sigma)-\varepsilon_1
\end{equation*}
if $1-\delta_2<(2k+1)/(2N)<1$ and $N$ is sufficiently large.
\end{lem}
\begin{proof}
We follow the proof of \cite[Lemma~5.4]{Murakami:arXiv2023}.
\par
From the continuity of $\Re F(x)$, there exists $\delta'_2>0$ such that $\left|\Re F\bigl((2k'+1)/(2N)\bigr)-\Re F(1)\right|<\Re F(\sigma)-\Re F(1)-2\varepsilon_1$ if $1-\delta'_2<(2k'+1)/(2N)<1$.
It follows that
\begin{equation}\label{eq:F_sigma}
  \Re F\bigl((2k'+1)/(2N)\bigr)
  <\Re F(\sigma)-2\varepsilon_1
\end{equation}
if $1-\delta'_2<(2k'+1)/(2N)<1$.
\par
From Lemma~\ref{lem:converge_F}, the function $f_N(x)$ for $x\in\R$ converges to $F(x)$ if $x\not\in\Delta^{+}_\nu\cup\nabla^{+}_\nu$, that is, if $|x|\le 1-2\pi\nu/b$.
Therefore, we have
\begin{equation}\label{eq:fN_F}
  \left|\Re f_N\bigl((2k'+1)/(2N)\bigr)-\Re F\bigl((2k'+1)/(2N)\bigr)\right|
  <
  \varepsilon_1
\end{equation}
if $0<(2k'+1)/(2N)\le1-2\pi\nu/b$ and $N$ is sufficiently large.
\par
So from \eqref{eq:fN_F} and \eqref{eq:F_sigma}, we have
\begin{equation}\label{eq:fN_F_sigma}
  \Re f_N\bigl((2k'+1)/(2N)\bigr)
  <
  \Re F\bigl((2k'+1)/(2N)\bigr)+\varepsilon_1
  <
  \Re F(\sigma)-\varepsilon_1
\end{equation}
if $1-\delta'_2<(2k'+1)/(2N)<1-2\pi\nu/b$ and $N$ is sufficiently large.
\par
Put $r_N(k):=\prod_{l=1}^{k}\left(4\sinh\left(\frac{\xi}{2}(1-l/N)\right)\sinh\left(\frac{\xi}{2}(1+l/N)\right)\right)$.
Note that $J_N(\FE;q)=\sum_{k=0}^{N-1}r_N(k)$ from \eqref{eq:JN}.
Since the function $4\sinh\left(\frac{\xi}{2}(1-x)\right)\sinh\left(\frac{\xi}{2}(1+x)\right)$ vanishes when $x=1$, its absolute values is less than $1$ when $1-\delta''_2<x<1$ for some $0<\delta''_2<1$.
It follows that $\left|r_N(k)\right|<\left|r_N(k')\right|$ when $1-\delta''_2<k'/N<k/N<1$.
\par
Now, since we have
\begin{equation*}
  r_N(k)
  =
  \frac{1}{2\sinh(\xi/2)}
  \exp\left(Nf_N\bigl((2k+1)/(2N)\bigr)\right)
\end{equation*}
from \eqref{eq:JN_fN}, we obtain
\begin{equation*}
  \Re f_N\bigl((2k+1)/(2N)\bigr)
  =
  \frac{1}{N}\log\bigl|2\sinh(\xi/2)r_N(k)\bigr|.
\end{equation*}
It follows that $\Re f_N\bigl((2k+1)/(2N)\bigr)<\Re f_N\bigl((2k'+1)/(2N)\bigr)$ if $1-\delta''_2<k'/N<k/N<1$ because $|r_N(k)|<|r_N(k')|$.
\par
Putting $\delta_2:=\min\{\delta'_2,\delta''_2\}$, we choose $k$ and $k'$ so that $k'<k$ and $1-\delta_2<k'/N<(2k'+1)/(2N)<1-2\pi\nu/b$ and that $k'/N<k/N<1$.
Then from \eqref{eq:fN_F_sigma} we have
\begin{equation*}
  \Re f_N\bigl((2k+1)/(2N)\bigr)
  <
  \Re f_N\bigl((2k'+1)/(2N)\bigr)
  <
  \Re F(\sigma)-\varepsilon_1.
\end{equation*}
\par
The proof is now complete.
\end{proof}
We have the following corollary.
\begin{cor}\label{cor:sum_delta1}
If $a\tanh{c}-b\tan{d}\ge0$ and $0<\delta_1<\delta_2$, then we have
\begin{equation*}
  \sum_{1-\delta_1<(2k+1)/(2N)<1}e^{Nf_N\bigl((2k+1)/(2N)\bigr)}
  =
  O\left(Ne^{N\bigl(\Re F(\sigma)-\varepsilon_1\bigr)}\right)
\end{equation*}
for some $\varepsilon_1>0$ as $N\to\infty$.
\end{cor}
\par
As for the difference of the integrals $\int_{-\delta_0}^{0}e^{Nf_N(z)}\,dz-\int_{\i\Im\sigma}^{0}e^{Nf_N(z)}\,dz$, we have the following lemma.
\begin{lem}\label{lem:int_delta_sigma}
Assume that $\cosh{a}-\cos{b}>1/2$.
If $\Re F(\sigma)\le0$, then we have
\begin{equation*}
  \int_{-\delta_0}^{0}e^{Nf_N(z)}\,dz-\int_{\i\Im\sigma}^{0}e^{Nf_N(z)}\,dz
  =
  O(e^{-\varepsilon_0N})
\end{equation*}
for some $\varepsilon_0>0$ as $N\to\infty$.
\end{lem}
\begin{proof}
From the implicit function theorem, the slope of the tangent of the curve $\Re F(z)=0$ at $O$ is given by
\begin{equation*}
  -\frac{\partial\,\Re F(x+y\i)/\partial\,x}{\partial\,\Re F(x+y\i)/\partial\,y}\Biggm|_{x=y=0}
  =
  \frac{\log|2\cosh\xi-2|}{\arg(2\cosh\xi-2)}
  =
  \frac{\log\sqrt{4(\alpha-1)^2+4\beta^2}}{\arg(2\alpha-2+2\beta\i)}
\end{equation*}
from Corollary~\ref{cor:F}.
Since the numerator is positive from Lemma~\ref{lem:alpha_beta_a_b}, and the denominator is also positive because $\beta>0$, we conclude that the slope is positive.
From $F(0)=0$, it follows that $\Re F(z)<0$ for any $z$ satisfying $\Re{z}\le0$, $\Im{z}\ge0$, and $|z|\le\tdelta$ for a sufficiently small $0<\tdelta<\min\{\delta_0,\Im\sigma\}$ because $\Re F(z)$ is increasing with respect to $x$ for $-\delta_0<x<0$ and decreasing with respect to $y$ for $0<y<\Im\sigma$.
Since $\{f_N(z)\}_{N=2,3,\ldots}$ converges to $F(z)$ near $O$ from Lemma~\ref{lem:converge_F}, we may assume that $\Re f_N(z)<0$ in the circular sector $\{z\in\C\mid\Re{z}\le0,\Im{z}\ge0,|z|\le\tdelta\}$ if $N$ is large enough.
\par
Let $\ell$ be the path starting at $-\delta_0$, goes on the real axis to $-\tdelta$, turn left, follows the arc that is the boundary of the circular sector above to $\i\tdelta$, turn left again, and then goes on the imaginary axis to $\i\Im\sigma$.
Then from Cauchy's integral theorem we have
\begin{equation*}
  \int_{-\delta_0}^{0}e^{Nf_N(z)}\,dz-\int_{\i\Im\sigma}^{0}e^{Nf_N(z)}\,dz
  =
  \int_{\ell}e^{Nf_N(z)}\,dz.
\end{equation*}
From the construction, we have $\Re{f_N(z)}<-\varepsilon_0$ for some $\varepsilon_0>0$ when $z$ is on $\ell$ if $N$ is sufficiently large.
This means that the right hand side of the equation above is of order $O\left(e^{-\varepsilon_0N}\right)$, proving the lemma.
\end{proof}
\par
%%%%%%%%%%%%%%%%%%%%%%%%%%%%%%%%%%%%%%%%%%%%%%%%%%%%
%%%%%%%%%%%%%%%%%%%%%%%%%%%%%%%%%%%%%%%%%%%%%%%%%%%%
Next, we will estimate the sum $\sum_{-\delta_0\le(2k+1)/(2N)<0}\exp\left(Nf_N\bigl((2k+1)/(2N)\bigr)\right)$.
%%%%%%%%%%%%%%%%%%%%%%%%%%%%%%%%%%%%%%%%%%%%%%%%%%%%%%%%%%%%%%%%%%%%%%%%%%%%%%%
\par
Put $p:=F'(0)$ and $h_N(z):=f_N(z)-pz$.
Since $f_N(z)$ is an odd function, we can write $f_N(z)=f'_N(0)z+z^3\psi_N(z)$ for holomorphic functions $\psi_N(z)$ ($N=1,2,3,\dots$) in $\Theta$.
See Lemma~\ref{lem:domain_f} for the definition of $\Theta$.
Lemma~\ref{lem:fN_F} implies that $f'_N(0)=p+N^{-2}\chi_N$, where $\chi_N\in\C$ is bounded as $N\to\infty$.
Therefore, we have $h_N(z)=N^{-2}\chi_N z+z^3\psi_N(z)$.
Since $F(z)$ is also an odd function, we can write $F(z)=pz+z^3\Psi(z)$ for a holomorphic function $\Psi(z)$ in $\Theta_{\nu}$.
See Lemma~\ref{lem:converge_F} for the definition of $\Theta_{\nu}$.
\par
We prepare the following lemma.
%%%%%%%%%%%%%%%%%%%%%%%%%%%%%%%%%%%%%%%%%%%%%%%%%%%%%%%%%%%%%%%%%%%%%%%%%%%%%%%
\begin{lem}\label{lem:psiN}
Let $B_{\rho}$ be an open disk in $\Theta_{\nu}$, with radius $0<\rho<\Re\gamma/|\gamma|$ centered at the origin, where $\gamma:=\xi/(2\pi\i)$ as usual.
\par
Then, we have $\psi_N(z)=\Psi(z)+O(N^{-2})$ in $B_{\rho}$ as $N\to\infty$.
In particular, the series of functions $\{\psi_N(z)\}_{N=1,2,3,\dots}$ is bounded in $B_{\rho}$.
\end{lem}
\begin{proof}
From \eqref{eq:F_defn}, \eqref{eq:F'} and \eqref{eq:L1_L2_defn}, we have
\begin{equation*}
\begin{split}
  &z^3\Psi(z)
  \\
  =&
  \frac{1}{\xi}\L_2\bigl(\gamma(1-z)\bigr)
  -
  \frac{1}{\xi}\L_2\bigl(\gamma(1+z)\bigr)
  -
  2\L_1(\gamma)z
  \\
  =&
  \frac{1}{4\gamma}\int_{\Rpath}\frac{e^{\bigl(2\gamma(1-z)-1\bigr)t}}{t^2\sinh{t}}\,dt
  -
  \frac{1}{4\gamma}\int_{\Rpath}\frac{e^{\bigl(2\gamma(1+z)-1\bigr)t}}{t^2\sinh{t}}\,dt
  +
  z\int_{\Rpath}\frac{e^{(2\gamma-1)t}}{t\sinh{t}}\,dt
  \\
  =&
  \frac{1}{2\gamma}
  \int_{\Rpath}
  \frac{e^{(2\gamma-1)t}}{t^2\sinh{t}}\bigl(2\gamma zt-\sinh(2\gamma zt)\bigr)\,dt.
\end{split}
\end{equation*}
The Taylor expansion of $2\gamma zt-\sinh(2\gamma zt)$ around $z=0$ becomes
\begin{equation}\label{eq:Taylor_sinh}
  2\gamma zt-\sinh(2\gamma zt)
  =
  -\sum_{j=1}^{\infty}\frac{(2\gamma t)^{2j+1}}{(2j+1)!}z^{2j+1}.
\end{equation}
Therefore we have
\begin{equation*}
\begin{split}
  \Psi(z)
  =&
  -\frac{1}{2\gamma}
  \int_{\Rpath}
  \frac{e^{(2\gamma-1)t}}{t^2\sinh{t}}
  \left(
    \sum_{j=1}^{\infty}\frac{(2\gamma t)^{2j+1}z^{2j-2}}{(2j+1)!}
  \right)
  \,dt
  \\
  =&
  -\frac{1}{2\gamma}\sum_{j=1}^{\infty}\frac{(2\gamma)^{2j+1}z^{2j-2}}{(2j+1)!}
  \int_{\Rpath}
  \frac{t^{2j-1}e^{(2\gamma-1)t}}{\sinh{t}}
  \,dt.
\end{split}
\end{equation*}
\par
Since $f'_N(0)=-\frac{2\gamma}{N}T'_N(\gamma)-\xi+2\pi\i$ from the proof of Lemma~\ref{lem:fN_F}, \eqref{eq:fN_defn} and \eqref{eq:TN_defn} imply that
\begin{equation*}
\begin{split}
  &z^3\psi_N(z)
  \\
  =&
  \frac{1}{N}T_N\bigl(\gamma(1-z)\bigr)-\frac{1}{N}T_N\bigl(\gamma(1+z)\bigr)
  +\frac{2\gamma z}{N}T'_N(\gamma)
  \\
  =&
  \frac{1}{4N}\int_{\Rpath}\frac{e^{\bigl(2\gamma(1-z)-1\bigr)t}}{t\sinh{t}\sinh(\gamma t/N)}\,dt
  -
  \frac{1}{4N}\int_{\Rpath}\frac{e^{\bigl(2\gamma(1+z)-1\bigr)t}}{t\sinh{t}\sinh(\gamma t/N)}\,dt
  \\
  &+
  \frac{\gamma z}{N}\int_{\Rpath}\frac{e^{(2\gamma-1)t}}{\sinh{t}\sinh(\gamma t/N)}\,dt,
  \\
  =&
  \frac{1}{2N}
  \int_{\Rpath}
  \frac{e^{(2\gamma-1)t}}{t\sinh{t}\sinh(\gamma t/N)}\bigl(2\gamma zt-\sinh(2\gamma zt)\bigr)\,dt,
\end{split}
\end{equation*}
where the second equality follows from the proof of Lemma~\ref{lem:TN'}.
The equality \eqref{eq:Taylor_sinh} implies
\begin{equation*}
\begin{split}
  \psi_N(z)
  =&
  -\frac{1}{2N}
  \int_{\Rpath}
  \frac{e^{(2\gamma-1)t}}{t\sinh{t}\sinh(\gamma t/N)}
  \left(
    \sum_{j=1}^{\infty}
    \frac{(2\gamma t)^{2j+1}z^{2j-2}}{(2j+1)!}
  \right)\,dt
  \\
  =&
  -\frac{1}{2N}
  \sum_{j=1}^{\infty}
  \frac{(2\gamma)^{2j+1}z^{2j-2}}{(2j+1)!}
  \int_{\Rpath}
  \frac{t^{2j}e^{(2\gamma-1)t}}{\sinh{t}\sinh(\gamma t/N)}
  \,dt.
\end{split}
\end{equation*}
Therefore we have
\begin{equation*}
\begin{split}
  &\Psi(z)-\psi_N(z)
  \\
  =&
  -\frac{1}{2\gamma}\sum_{j=1}^{\infty}\frac{(2\gamma)^{2j+1}z^{2j-2}}{(2j+1)!}
  \int_{\Rpath}
  \frac{t^{2j-1}e^{(2\gamma-1)t}}{\sinh{t}}
  \\
  &+
  \frac{1}{2N}
  \sum_{j=1}^{\infty}
  \frac{(2\gamma)^{2j+1}z^{2j-2}}{(2j+1)!}
  \int_{\Rpath}
  \frac{t^{2j}e^{(2\gamma-1)t}}{\sinh{t}\sinh(\gamma t/N)}
  \,dt
  \\
  =&
  \frac{1}{2\gamma}
  \sum_{j=1}^{\infty}\frac{(2\gamma)^{2j+1}z^{2j-2}}{(2j+1)!}
  \int_{\Rpath}
  \frac{t^{2j-1}e^{(2\gamma-1)t}}{\sinh{t}}
  \left(\frac{\gamma t/N}{\sinh(\gamma t/N)}-1\right)
  \,dt.
\end{split}
\end{equation*}
It follows that
\begin{equation*}
\begin{split}
  &|\Psi(z)-\psi_N(z)|
  \\
  <&
  \frac{1}{2\rho^3|\gamma|}
  \sum_{j=1}^{\infty}
  \frac{(2\rho|\gamma|)^{2j+1}}{(2j+1)!}
  \int_{\Rpath}
  \left|\frac{t^{2j-1}e^{(2\gamma-1)t}}{\sinh{t}}\right|
  \left|\frac{\gamma t/N}{\sinh(\gamma t/N)}-1\right|
  \,|dt|
\end{split}
\end{equation*}
since $|z|<\rho$.
Since $\left|\frac{\gamma t/N}{\sinh(\gamma t/N)}-1\right|\le\frac{c|t|^2}{N^2}$ for a positive constant $c$ as in the proof of Lemma~\ref{lem:TN'}, we have
\begin{equation}\label{eq:Psi_psiN}
  |\Psi(z)-\psi_N(z)|
  \le
  \frac{c}{2\rho^3|\gamma|N^2}
  \sum_{j=1}^{\infty}\frac{(2\rho|\gamma|)^{2j+1}}{(2j+1)!}
  \int_{\Rpath}
  \left|\frac{t^{2j+1}e^{(2\gamma-1)t}}{\sinh{t}}\right|
  \,|dt|.
\end{equation}
Put
\begin{align*}
  I_{+,j}
  :=&
  \int_{1}^{\infty}\left|\frac{t^{2j+1}e^{(2\gamma-1)t}}{\sinh{t}}\right|\,dt
  =
  \int_{1}^{\infty}\frac{t^{2j+1}e^{(2\Re\gamma-1)t}}{\sinh{t}}\,dt,
  \\
  I_{-,j}
  :=&
  \int_{-\infty}^{-1}\left|\frac{t^{2j+1}e^{(2\gamma-1)t}}{\sinh{t}}\right|\,dt
  =
  \int_{-\infty}^{-1}\frac{t^{2j+1}e^{(2\Re\gamma-1)t}}{\sinh{t}}\,dt,
  \\
  I_{0,j}
  :=&
  \int_{|t|=1,\Im{t}\ge0}\left|\frac{t^{2j+1}e^{(2\gamma-1)t}}{\sinh{t}}\right|\,|dt|.
\end{align*}
Note that $\int_{\Rpath}\left|\frac{t^{2j+1}e^{(2\gamma-1)t}}{\sinh{t}}\right|\,|dt|=I_{+,j}+I_{-,j}+I_{0,j}$.
We will estimate the integrals $I_{+,j}$, $I_{-,j}$ and $I_{0,j}$.
\par
We have
\begin{equation*}
\begin{split}
  I_{+,j}
  =&
  2\int_{1}^{\infty}\frac{t^{2j+1}e^{2(\Re\gamma-1)t}}{1-e^{-2t}}\,dt
  \\
  \le&
  \frac{2}{1-e^{-2}}\int_{1}^{\infty}t^{2j+1}e^{2(\Re\gamma-1)t}\,dt
  <
  \frac{2}{1-e^{-2}}\int_{0}^{\infty}t^{2j+1}e^{2(\Re\gamma-1)t}\,dt
  \\
  &\text{(put $s:=2(1-\Re\gamma)t$ noting that $\Re\gamma=b/(2\pi)<1/4$)}
  \\
  =&
  \frac{2}{1-e^{-2}}
  \int_{0}^{\infty}\left(\frac{s}{2(1-\Re\gamma)}\right)^{2j+1}
  \frac{e^{-s}}{2(1-\Re\gamma)}\,ds
  \\
  =&
  \frac{\Gamma(2j+2)}{(1-e^{-2})2^{2j+1}(1-\Re\gamma)^{2j+2}}
  =
  \frac{(2j+1)!}{(1-e^{-2})2^{2j+1}(1-\Re\gamma)^{2j+2}}.
\end{split}
\end{equation*}
Similarly, we have
\begin{equation*}
\begin{split}
  I_{-,j}
  =&
  \int_{1}^{\infty}
  \frac{t^{2j+1}e^{(1-2\Re\gamma)t}}{\sinh{t}}\,dt
  =
  2\int_{1}^{\infty}
  \frac{t^{2j+1}e^{-2t\Re\gamma}}{1-e^{-2t}}\,dt
  \\
  \le&
  \frac{2}{1-e^{-2}}
  \int_{1}^{\infty}t^{2j+1}e^{-2t\Re\gamma}\,dt
  <
  \frac{2}{1-e^{-2}}
  \int_{0}^{\infty}t^{2j+1}e^{-2t\Re\gamma}\,dt
  \\
  &\text{(put $s:=2t\Re\gamma$ noting that $\Re\gamma>0$)}
  \\
  =&
  \frac{2}{1-e^{-2}}
  \int_{0}^{\infty}\left(\frac{s}{2\Re\gamma}\right)^{2j+1}\frac{e^{-s}}{2\Re\gamma}\,ds
  \\
  =&
  \frac{\Gamma(2j+2)}{(1-e^{-2})2^{2j+1}(\Re\gamma)^{2j+2}}
  =
  \frac{(2j+1)!}{(1-e^{-2})2^{2j+1}(\Re\gamma)^{2j+2}}.
\end{split}
\end{equation*}
Putting $t:=e^{s\i}$ ($0\le s\le\pi$), we have
\begin{equation*}
\begin{split}
  I_{0,j}
  =&
  \int_{0}^{\pi}
  \left|
    \frac{e^{(2j+1)s\i}e^{(2\gamma-1)e^{s\i}}}{\sinh(e^{s\i})}
  \right|
  \times
  \left|\i e^{s\i}\right|\,ds
  \\
  \le&
  \frac{1}{L}\int_{0}^{\pi}
  e^{(2\Re\gamma-1)\cos{s}-2\Im\gamma\sin{s}}\,ds
  =M,
\end{split}
\end{equation*}
for a positive constant $M$, where we put $L:=\min_{|z|=1,\Im{z}\ge0}|\sinh{z}|$ as in the proof of Lemma~\ref{lem:TN'}.
\par
From \eqref{eq:Psi_psiN}, we have
\begin{equation*}
\begin{split}
  &|\Psi(z)-\psi_N(z)|
  \\
  \le&
  \frac{c}{2\rho^3|\gamma|N^2}
  \sum_{j=1}^{\infty}\frac{(2\rho|\gamma|)^{2j+1}}{(2j+1)!}
  \\
  &\times
  \left(
    M
    +
    \frac{(2j+1)!}{(1-e^{-2})2^{2j+1}(1-\Re\gamma)^{2j+2}}
    +
    \frac{(2j+1)!}{(1-e^{-2})2^{2j+1}(\Re\gamma)^{2j+2}}
  \right)
  \\
  =&
  \frac{cM}{2\rho^3|\gamma|N^2}
  \sum_{j=1}^{\infty}\frac{(2\rho|\gamma|)^{2j+1}}{(2j+1)!}
  +
  \frac{c}{2\rho^3|\gamma|(1-e^{-2})(1-\Re\gamma)N^2}
  \sum_{j=1}^{\infty}
  \left(\frac{\rho|\gamma|}{1-\Re\gamma}\right)^{2j+1}
  \\
  &+
  \frac{c}{2\rho^3|\gamma|(1-e^{-2})\Re\gamma N^2}
  \sum_{j=1}^{\infty}
  \left(\frac{\rho|\gamma|}{\Re\gamma}\right)^{2j+1}.
\end{split}
\end{equation*}
The first summation clearly converges.
Since $\frac{\rho|\gamma|}{\Re\gamma}<1$ from the assumption, the third summation converges.
From $b<\pi/2$, we have $\frac{\rho|\gamma|}{1-\Re\gamma}<\frac{\Re\gamma}{1-\Re\gamma}=\frac{b}{2\pi-b}<1$, and the second summation converges.
It follows that $\psi_N(z)=\Psi(z)+O(N^{-2})$.
\end{proof}
%%%%%%%%%%%%%%%%%%%%%%%%%%%%%%%%%%%%%%%%%%%%%%%%%%%%%%%%%%%%%%%%%%%%%%%%%%%%%%%
%%%%%%%%%%%%%%%%%%%%%%%%%%%%%%%%%%%%%%%%%%%%%%%%%%%%%%%%%%%%%%%%%%%%%%%%%%%%%%%
Note that if $\cosh{a}-\cos{b}>1/2$, then $\Re{p}=\Re F'(0)=\log|2\cosh\xi-2|>0$ from Corollary~\ref{cor:F} and Lemma~\ref{lem:alpha_beta_a_b}.
\par
Now we can prove the following lemma.
\begin{lem}\label{lem:fN_DN}
Suppose that $\cosh{a}-\cos{b}>1/2$.
We assume that $\delta_0<\rho$, where $\rho>0$ is given in Lemma~\ref{lem:psiN}.
\par
Since $f_N(x)=F(x)+O(N^{-2})$ and $p=F'(0)$, we may assume that $\Re f_N(x)<0$ for $0<x\le\delta_0$.
Then, we have
\begin{equation*}
  \sum_{-\delta_0<(2k+1)/(2N)<0}
  e^{Nf_N\bigl((2k+1)/(2N)\bigr)}
  =
  -\frac{2\sinh(\xi/2)}{\Delta(e^{\xi})}+O(N^{-2})
\end{equation*}
as $N\to\infty$, where $\Delta(t):=-t+3-t^{-1}$ is the normalized Alexander polynomial of the figure-eight knot.
\end{lem}
\begin{proof}
First of all, we have
\begin{equation}\label{eq:delta0_plus_minus}
\begin{split}
  \sum_{-\delta_0\le(2k+1)/(2N)<0}e^{Nf_N\bigl((2k+1)/(2N)\bigr)}
  =&
  \sum_{k=-1,-2,\dots,\ceil{-\delta_0N-1/2}}e^{-Nf_N\left(-(2k+1)/(2N)\right)}
  \\
  =&
  \sum_{k=0}^{D_N}e^{-Nf_N\bigl((2k+1)/(2N)\bigr)}
\end{split}
\end{equation}
where $\ceil{x}$ is the least integer greater than or equal to $x$, and we put $D_N:=\floor{\delta_0N-1/2}$, with $\floor{x}$ the greatest integer less than or equal to $x$.
So we will estimate the last sum.
\par
We also put
\begin{align*}
  r(k)
  &:=
  e^{-Npz}\biggm|_{z:=(2k+1)/(2N)}
  =
  e^{-p(k+1/2)},
  \\
  s_N(k)
  &:=
  e^{-Nh_N(z)}\biggm|_{z:=(2k+1)/(2N)}
  =
  e^{-Nh_N\bigl((2k+1)/(2N)\bigr)}
\end{align*}
so that $e^{-Nf_N\bigl((2k+1)/(2N)\bigr)}=r(k)s_N(k)$.
\par
Putting $R(t):=\sum_{i=0}^{\floor{t}}r(i)$, we use Abel's summation formula (see for example \cite[(3.13.2)]{deBruijn:1981}) to obtain
\begin{equation*}
\begin{split}
  \sum_{k=0}^{D_N}e^{-Nf_N\bigl((2k+1)/(2N)\bigr)}
  =&
  \sum_{k=0}^{D_N}r(k)s_N(k)
  \\
  =&
  R(D_N)s_N(D_N)
  -
  \sum_{j=0}^{D_N-1}R(j)\left(s_N(j+1)-s_N(j)\right)
  \\
  =&
  R(D_N)s_N(D_N)
  -
  \int_{0}^{D_N}R(t)s'_N(t)\,dt.
\end{split}
\end{equation*}
Here the last equality follows because
\begin{equation*}
\begin{split}
  \int_{0}^{D_N}R(t)s'_N(t)\,dt
  =&
  \sum_{j=0}^{D_N-1}\int_{j}^{j+1}R(t)s'_N(t)\,dt
  \\
  =&
  \sum_{j=0}^{D_N-1}\int_{j}^{j+1}R(j)s'_N(t)\,dt
  \\
  =&
  \sum_{j=0}^{D_N-1}R(j)\int_{j}^{j+1}s'_N(t)\,dt
  \\
  =&
  \int_{0}^{D_N}R(t)\bigl(s_N(j+1)-s_N(j)\bigr).
\end{split}
\end{equation*}
Since $R(t)=\frac{1-e^{-p(\floor{t}+1)}}{e^{p/2}-e^{-p/2}}$, we have
\begin{equation}\label{eq:sum_eh}
\begin{split}
  &\sum_{k=0}^{D_N}e^{-Nf_N\bigl((2k+1)/(2N)\bigr)}
  \\
  =&
  \frac{1-e^{-p(D_N+1)}}{e^{p/2}-e^{-p/2}}s_N(D_N)
  -
  \int_{0}^{D_N}\frac{1-e^{-p(\floor{t}+1)}}{e^{p/2}-e^{-p/2}}s'_N(t)\,dt
  \\
  =&
  \frac{s_N(D_N)-e^{-p(D_N+1)}s_N(D_N)}{e^{p/2}-e^{-p/2}}
  -
  \frac{1}{e^{p/2}-e^{-p/2}}\int_{0}^{D_N}s'_N(t)\,dt
  \\
  &+
  \frac{1}{e^{p/2}-e^{-p/2}}\int_{0}^{D_N}e^{-p(\floor{t}+1)}s'_N(t)\,dt
  \\
  =&
  \frac{s_N(D_N)-e^{-p(D_N+1)}s_N(D_N)}{e^{p/2}-e^{-p/2}}
  -
  \frac{s_N(D_N)-s_N(0)}{e^{p/2}-e^{-p/2}}
  \\
  &+
  \frac{1}{e^{p/2}-e^{-p/2}}\int_{0}^{D_N}e^{-p(\floor{t}+1)}s'_N(t)\,dt
  \\
  =&
  \frac{s_N(0)-e^{-p(D_N+1)}s_N(D_N)}{e^{p/2}-e^{-p/2}}
  +
  \frac{e^{-p}}{e^{p/2}-e^{-p/2}}\int_{0}^{D_N}e^{-p\floor{t}}s'_N(t)\,dt.
\end{split}
\end{equation}
\par
We will estimate $s_N(0)$, $e^{-p(D_N+1)}s_N(D_N)$, and $\int_{0}^{D_N}e^{-p\floor{t}}s'_N(t)\,dt$.
\par
First, we obtain
\begin{equation*}
  h_N\bigl(1/(2N)\bigr)
  =
  \frac{\chi_N}{2N^3}+\frac{\psi_N\bigl(1/(2N)\bigr)}{(2N)^3},
\end{equation*}
and so
\begin{equation}\label{eq:sN0}
  s_N(0)
  =
  e^{-Nh_N\bigl(1/(2N)\bigr)}
  =
  1+O(N^{-2})
\end{equation}
as $N\to\infty$ since $\psi_N\bigl(1/(2N)\bigr)$ is bounded as $N\to\infty$ from Lemma~\ref{lem:psiN}.
\par
Next, we study $e^{-p(D_N+1)}s_N(D_N)$.
Since $e^{-Nf_N\bigl((2j+1)/(2N)\bigr)}=r(j)s_N(j)$, we have
\begin{equation*}
  \left|e^{-p(D_N+1)}s_N(D_N)\right|
  =
  e^{-\Re{p}/2}e^{-N\Re f_N\bigl((2D_N+1)/(2N)\bigr)}.
\end{equation*}
Now, since $(2D_N+1)/(2N)\le\delta_0$, we see that $-\Re f_N\bigl((2D_N+1)/(2N)\bigr)<0$ from the assumption.
It follows that
\begin{equation}\label{eq:sN_DN}
  e^{-p(D_N+1)}s_N(D_N)
  =
  O(e^{-\varepsilon N})
\end{equation}
for some $\varepsilon>0$ as $N\to\infty$.
\par
Lastly, we consider the integral $\int_{0}^{D_N}e^{-p\floor{t}}s'_N(t)\,dt$.
\par
Since $s_N(t)=e^{-Nh_N\bigl((2t+1)/(2N)\bigr)}$ and $h_N(z)=N^{-2}\chi_Nz+z^3\psi_N(z)$, we have
\begin{equation*}
\begin{split}
  s'_N(t)
  =&
  -h'_N\left(\frac{2t+1}{2N}\right)e^{-Nh_N\bigl((2t+1)/(2N)\bigr)}
  \\
  =&
  -s_N(t)
  \left(
    \frac{\chi_N}{N^2}
    +
    \frac{3(2t+1)^2}{4N^2}\psi_N\left(\frac{2t+1}{2N}\right)
    +
    \frac{(2t+1)^3}{8N^3}\psi'_N\left(\frac{2t+1}{2N}\right)
  \right).
\end{split}
\end{equation*}
Since $\psi_N(x)$ is bounded as $N\to\infty$ from Lemma~\ref{lem:psiN}, we have $|\psi_N(x)|\le C_0$ for some constant $C_0$ if $0\le x\le\delta_0$.
From Cauchy's estimate, we also have $|\psi'_N(x)|\le C_1$ if $0\le x\le\delta_0$ for some constant $C_1$.
So we have
\begin{equation*}
  |s'_N(t)|
  \le
  \frac{M}{N^2}
  \left(
    |\chi_N|
    +
    C_0\frac{3(2t+1)^2}{4}
    +
    C_1\frac{(2t+1)^3}{8N}
  \right),
\end{equation*}
where we put $M:=\max_{0\le x\le\delta_0}|e^{-Nh_N(x)}|$.
\par
Therefore we have
\begin{equation*}
\begin{split}
  &\left|\int_{0}^{D_N}e^{-p\floor{t}}s'_N(t)\,dt\right|
  =
  \left|\sum_{k=0}^{D_N-1}\int_{k}^{k+1}e^{-p\floor{t}}s'_N(t)\,dt\right|
  \\
  \le&
  \sum_{k=0}^{D_N-1}e^{-\Re(p)k}\int_{k}^{k+1}|s'_N(t)|\,dt
  \\
  \le&
  \frac{M}{N^2}
  \sum_{k=0}^{D_N-1}e^{-\Re(p)k}
  \int_{k}^{k+1}
  \left(|\chi_N|+\frac{3C_0}{4}(2t+1)^2+\frac{C_1}{8N}(2t+1)^3\right)
  \,dt
  \\
  <&
  \frac{M}{N^2}
  \sum_{k=0}^{D_N-1}e^{-\Re(p)k}
  \left(|\chi_N|+\frac{3C_0}{4}(2k+3)^2+\frac{C_1}{8N}(2k+3)^3\right)
  \\
  =&
  \frac{M}{N^2}|\chi_N|\sum_{k=0}^{D_N-1}e^{-\Re(p)k}
  +
  \frac{3C_0M}{4N^2}\sum_{k=0}^{D_N-1}e^{-\Re(p)k}(2k+3)^2
  \\
  &+
  \frac{C_1M}{8N^3}\sum_{k=0}^{D_N-1}e^{-\Re(p)k}(2k+3)^3.
\end{split}
\end{equation*}
Since $\Re{p}>0$, each summation clearly converges.
So we conclude that
\begin{equation}\label{eq:int_s'N}
  \int_{0}^{D_N}e^{-p\floor{t}}s'_N(t)\,dt
  =
  O(N^{-2}).
\end{equation}
\par
Therefore, Equations~\eqref{eq:sum_eh}, \eqref{eq:sN0}, \eqref{eq:sN_DN}, and \eqref{eq:int_s'N} imply the following asymptotic formula.
\begin{equation*}
 \sum_{k=0}^{D_N}e^{-Nf_N\bigl((2k+1)/(2N)\bigr)}
 =
 \frac{e^{p/2}}{e^{p}-1}+O(N^{-2})
 =
 \frac{\sqrt{2\cosh\xi-2}}{2\cosh\xi-3}+O(N^{-2}).
\end{equation*}
\par
Since $\cosh\xi=2\sinh^2(\xi/2)+1$ and $\Delta(e^{\xi})=-2\cosh\xi+3$, we have $\sqrt{2\cosh\xi-2}/(2\cosh\xi-3)=-2\sinh(\xi/2)/\Delta(e^{\xi})$.
\par
Therefore, from \eqref{eq:delta0_plus_minus} we finally have
\begin{equation*}
  \sum_{-\delta_0\le(2k+1)/(2N)<0}e^{Nf_N\bigl((2k+1)/(2N)\bigr)}
  =
  -\frac{2\sinh(\xi/2)}{\Delta(e^{\xi})}+O(N^{-2}),
\end{equation*}
completing the proof.
\end{proof}
%%%%%%%%%%%%%%%%%%%%%%%%%%%%%%%%%%%%%%%%%%%%%%%%%%%%%%%%%%%%%%%%%%%%%%%%%%%%%%%%
\par
Now we can prove the main theorem.
\begin{proof}[Proof of Theorem~\ref{thm:main} {\rm(}Main Theorem{\rm)}]
If $\Re F(\sigma)>0$, then from \eqref{eq:saddle_positive} and Corollary~\ref{cor:sum_delta1}, we have
\begin{equation*}
\begin{split}
  &\sum_{k=0}^{N-1}e^{Nf_N\bigl((2k+1)/(2N)\bigr)}
  \\
  =&
  \sum_{0\le(2k+1)/(2N)\le1-\delta_1}e^{Nf_N\bigl((2k+1)/(2N)\bigr)}
  +
  \sum_{1-\delta_1<(2k+1)/(2N)<1}e^{Nf_N\bigl((2k+1)/(2N)\bigr)}
  \\
  =&
  \frac{\sqrt{2\pi N}e^{NF(\sigma)}}{\sqrt{\xi\sqrt{(2\cosh\xi+1)(2\cosh\xi-3)}}}
  \bigl(1+O(N^{-1})\bigr)
  +
  O\left(Ne^{N\bigl(\Re F(\sigma)-\varepsilon_1\bigr)}\right)
  \\
  =&
  \frac{\sqrt{2\pi N}e^{NF(\sigma)}}{\sqrt{\xi\sqrt{(2\cosh\xi+1)(2\cosh\xi-3)}}}
  \left(1+O(N^{-1})+O\bigl(N^{1/2}e^{-\varepsilon_1N}\bigr)\right)
  \\
  =&
  \frac{\sqrt{\pi N}e^{NF(\sigma)}}{\sqrt{\frac{\xi}{2}\sqrt{(2\cosh\xi+1)(2\cosh\xi-3)}}}
  \bigl(1+O(N^{-1})\bigr)
  \\
  =&
  \sqrt{\pi}T(\xi)^{1/2}\left(\frac{N}{\xi}\right)^{1/2}e^{NF(\sigma)}
  \bigl(1+O(N^{-1})\bigr).
\end{split}
\end{equation*}
Now, since $\Re(\xi\sigma)=\Re\varphi=c<a$ from Lemma~\ref{lem:phi_xi}, we have $S(\xi)=\xi F(\sigma)$ from Corollary~\ref{cor:F} and \eqref{eq:S}.
So, we obtain \eqref{eq:main_pos} in Main Theorem from from \eqref{eq:JN_fN}.
\par
If $\Re F(\sigma)\le0$, then from \eqref{eq:saddle_negative}, Lemma~\ref{lem:int_delta_sigma} and Lemma~\ref{lem:fN_DN}, choosing $\delta_1>0$ sufficiently small, we have
\begin{equation*}
\begin{split}
  &\sum_{k=0}^{N-1}e^{Nf_N\bigl((2k+1)/(2N)\bigr)}
  \\
  =&
  \sum_{0\le(2k+1)/(2N)\le1-\delta_1}e^{Nf_N\bigl((2k+1)/(2N)\bigr)}
  +
  \sum_{1-\delta_1<(2k+1)/(2N)<1}e^{Nf_N\bigl((2k+1)/(2N)\bigr)}
  \\
  =&
  \frac{\sqrt{2\pi N}e^{NF(\sigma)}}{\sqrt{\xi\sqrt{(2\cosh\xi+1)(2\cosh\xi-3)}}}
  \bigl(1+O(N^{-1})\bigr)
  +\frac{2\sinh(\xi/2)}{\Delta(e^{\xi})}+O(N^{-2}).
\end{split}
\end{equation*}
Therefore from \eqref{eq:JN_fN}, if $\Re F(\sigma)<0$, we have
\begin{equation*}
  J_N(\FE;e^{\xi/N})
  =
  \frac{1}{\Delta(e^{\xi})}+O(N^{-2}),
\end{equation*}
which implies \eqref{eq:main_neg} in Main Theorem.
If $\Re F(\sigma)=0$, we have
\begin{equation*}
  J_N(\FE;e^{\xi/N})
  =
  \frac{\sqrt{\pi}}{2\sinh(\xi/2)}T(\xi)^{1/2}\left(\frac{N}{\xi}\right)^{1/2}e^{NF(\sigma)}
  +\frac{1}{\Delta(e^{\xi})}+O(N^{-1/2})
\end{equation*}
since $\sqrt{N}e^{N F(\sigma)}$ is of order $N^{1/2}$.
This shows \eqref{eq:main_zero} in Main Theorem since $F(\sigma)=S(\xi)/\xi$.
\par
The proof is complete.
\end{proof}
\par
We can give another proof to \cite[Theorem~1.1]{Murakami:JPJGT2007} when $a>0$ and $0<b<\pi/2$ with $\cosh{a}-\cos{b}<1/2$.
\begin{prop}\label{prop:Alexander}
When $a>0$, $0<b<\pi/2$, and $\cosh{a}-\cos{b}<1/2$, then we have
\begin{equation*}
  J_N(\FE;e^{\xi/N})
  =
  \frac{1}{\Delta(e^{\xi})}+O(N^{-2})
\end{equation*}
as $N\to\infty$.
\end{prop}
\begin{proof}
First of all, from Corollary~\ref{cor:cosha_cosb_1_2}, we know that $\Re F(x)<0$ for $0<x\le1$.
Since $f_N(z)=F(z)+O(N^{-2})$ from Lemma~\ref{lem:fN_F}, we have $\Re f_N(x)<0$ for $0<x\le1$ if $N$ is sufficiently large.
It follows that for any $\delta>0$,
\begin{equation}\label{eq:sum_delta_1_2}
\sum_{\delta<(2k+1)/(2N)<1}e^{Nf_N\bigl((2k+1)/(2N)\bigr)}=O(Ne^{-\varepsilon N})
\end{equation}
for some $\varepsilon>0$.
\par
We will estimate the sum $\sum_{0<(2k+1)/(2N)\le\delta}e^{Nf_N\bigl((2k+1)/(2N)\bigr)}$.
\par
If we put $\tD_N:=\floor{\delta N-1/2}$, then $\sum_{0<(2k+1)/(2N)\le\delta}e^{Nf_N\bigl((2k+1)/(2N)\bigr)}=\sum_{k=0}^{\tD_N}e^{Nf_N\bigl((2k+1)/(2N)\bigr)}$.
We also put $p:=F'(0)$ and $h_N(z):=f_N(z)-pz$ as in the proof of Lemma~\ref{lem:fN_DN}.
Note that $\Re{p}<0$ from Lemma~\ref{lem:alpha_beta_a_b}.
\par
We write $\tr(k):=e^{p(k+1/2)}$, $\ts_N(k):=e^{Nh_N\bigl((2k+1)/(2N)\bigr)}$, and $\tR(t):=\sum_{k=0}^{\floor{t}}\tr(k)$.
Note that $e^{Nf_N\bigl((2k+1)/(2N)\bigr)}=\tr(k)\ts_N(k)$.
\par
We define $\chi_N$ and $\psi_N(z)$ so that $f'_N(0)=p+N^{-2}\chi_N$ and that $h_N(z)=N^{-2}\chi_N z+z^3\psi_N(z)$.
\par
From Abel's summation formula, we have
\begin{equation}\label{eq:Abel}
\begin{split}
  &\sum_{k=0}^{\tD_N}e^{Nf_N\bigl((2k+1)/(2N)\bigr)}
  \\
  =&
  \frac{\ts_N(0)-e^{p(\tD_N+1)}\ts_N(\tD_N)}{e^{-p/2}-e^{p/2}}
  +
  \frac{e^{p}}{e^{-p/2}-e^{p/2}}\int_{0}^{\tD_N}e^{p\floor{t}}\ts'_N(t)\,dt
\end{split}
\end{equation}
in a way similar to the proof of Lemma~\ref{lem:fN_DN}.
\par
By the same reason as in the proof of Lemma~\ref{lem:fN_DN}, we have
\begin{equation}\label{eq:sN0_1_2}
  \ts_N(0)
  =
  e^{Nh_N\bigl(1/(2N)\bigr)}
  =
  1+O(N^{-2})
\end{equation}
as $N\to\infty$.
\par
Since $e^{Nf_N\bigl((2j+1)/(2N)\bigr)}=\tr(j)\ts_N(j)$ and $\Re{p}<0$, we have
\begin{equation*}
  \left|e^{p(\tD_N+1)}\ts_N(\tD_N)\right|
  =
  e^{N\Re f_N\bigl((2\tD_N+1)/(2N)\bigr)}.
\end{equation*}
Since $\Re f_N(x)<0$ for $0<x\le1$, we have
\begin{equation}\label{eq:sN_DN_1_2}
  e^{p(\tD_N+1)}\ts_N(\tD_N)
  =
  O(e^{-\varepsilon N})
\end{equation}
for some $\varepsilon>0$ as $N\to\infty$.
\par
Now, we will estimate the integral $\int_{0}^{\tD_N}e^{p\floor{t}}\ts'_N(t)\,dt$.
Since $\ts_N(t)=e^{Nh_N\bigl((2t+1)/(2N)\bigr)}$ and $h_N(z)=N^{-2}\chi_Nz+z^3\psi_N(z)$, we have
\begin{equation*}
  |\ts'_N(t)|
  \le
  \frac{\tM}{N^2}
  \left(
    |\chi_N|
    +
    \tC_0\frac{3(2t+1)^2}{4N^2}
    +
    \tC_1\frac{(2t+1)^3}{8N}
  \right)
\end{equation*}
in a way similar to the proof of Lemma~\ref{lem:fN_DN}, where $|\psi_N(x)|\le\tC_0$ and $|\psi'_N(x)|\le\tC_1$ for $0\le x\le\delta$, and $\tM:=\max_{0\le x\le\delta}e^{e^{Nh_N(x)}}$.
It follows that
\begin{equation*}
\begin{split}
  &\left|\int_{0}^{\tD_N}e^{p\floor{t}}\ts'_N(t)\,dt\right|
  \\
  \le&
  \frac{\tM}{N^2}|\chi_N|\sum_{k=0}^{\tD_N-1}e^{\Re(p)k}
  +
  \frac{3\tC_0\tM}{4N^2}\sum_{k=0}^{\tD_N-1}e^{\Re(p)k}(2k+3)^2
  \\
  &+
  \frac{\tC_1\tM}{8N^3}\sum_{k=0}^{\tD_N-1}e^{\Re(p)k}(2k+3)^3.
\end{split}
\end{equation*}
\par
Since $\Re{p}<0$, each summation converges, and so
\begin{equation}\label{eq:int_s'N_1_2}
  \int_{0}^{\tD_N}e^{p\floor{t}}\ts'_N(t)\,dt=O(N^{-2}).
\end{equation}
\par
Since $p=\log(2\cosh\xi-2)$, from \eqref{eq:Abel}, \eqref{eq:sN0_1_2}, \eqref{eq:sN_DN_1_2}, \eqref{eq:int_s'N_1_2}, we have
\begin{equation*}
 \sum_{k=0}^{\tD_N}e^{Nf_N\bigl((2k+1)/(2N)\bigr)}
 =
 \frac{1}{e^{-p/2}-e^{p/2}}+O(N^{-2})
 =
 \frac{2\sinh(\xi/2)}{\Delta(e^{\xi})}+O(N^{-2}),
\end{equation*}
by the same reason as Lemma~\ref{lem:fN_DN}.
\par
Therefore, from \eqref{eq:sum_delta_1_2} we have
\begin{equation*}
\begin{split}
  &\sum_{k=0}^{N-1}
  e^{Nf_N\bigl((2k+1)/(2N)\bigr)}
  \\
  =&
  \sum_{0<(2k+1)/(2N)\le\delta}
  e^{Nf_N\bigl((2k+1)/(2N)\bigr)}
  +
  \sum_{\delta<(2k+1)/(2N)<1}
  e^{Nf_N\bigl((2k+1)/(2N)\bigr)}
  \\
  =&
 \frac{2\sinh(\xi/2)}{\Delta(e^{\xi})}+O(N^{-2}).
\end{split}
\end{equation*}
Now the proposition follows from \eqref{eq:JN_fN}.
\end{proof}
\begin{rem}\label{rem:cosha_cosb_1_2}
When $\cosh{a}-\cos{b}=1/2$, we can prove \eqref{eq:Abel}, \eqref{eq:sN0_1_2}, and \eqref{eq:sN_DN_1_2} in the same way.
However, the author does not know how to prove \eqref{eq:int_s'N_1_2}.
See Remark~\ref{rem:main_neg_0}.
\end{rem}
%%%%%%%%%%%%%%%%%%%%%%%%%%%%%%%%%%%%%%%%%%%%%%%%%%%%%%%%%%%%%%%%%%%%%%%%%%%%%%%
%%%%%%%%%%%%%%%%%%%%%%%%%%%%%%%%%%%%%%%%%%%%%%%%%%%%%%%%%%%%%%%%%%%%%%%%%%%%%%%
\section{Chern--Simons invariant and the adjoint Reidemeister torsion}\label{sec:CSR}
In this section we calculate the Chern--Simons invariant and the adjoint Reidemeister torsion associated with the irreducible representation of the fundamental group of the complement of the figure-eight knot to the Lie group $\PSL(2;\C)$ parametrized by a complex number.
We will follow \cite[Chapter~5]{Murakami/Yokota:2018}.
\par
Let $N(\FE)$ be the closed tubular neighborhood of the figure-eight knot $\FE$ in the three-sphere $S^3$.
Put $X:=S^3\setminus\Int N(\FE)$, and let $\Pi:=\pi_1(X)$ be its fundamental group, where $\Int$ means the interior.
%%%%%%%%%%%%%%%%%%%%%%%%%%%%%%%%%%%%%%%%%%%%%%%%%%%%%%%%%%%%%%%%%%%%%%%%%%%%%%%
\subsection{Representation}\label{subsec:rep}
It is well-known that $\Pi$ is presented as
\begin{equation*}
  \Pi
  =
  \langle x,y\mid xy^{-1}x^{-1}yx=yxy^{-1}x^{-1}y\rangle.
\end{equation*}
We choose $x$ as the meridian, which can be presented by an oriented, simple closed curve on $\partial{X}$ that bounds a disk in $N(\FE)$.
The preferred longitude is an oriented, simple closed on $\partial{X}$ that is parallel to $\FE$ in $N(\FE)$ and null-homologous in $X$.
It presents the element $xy^{-1}xyx^{-2}yxy^{-1}x^{-1}$ in $\Pi$.
\par
Following \cite{Riley:QUAJM31984}, for a complex number $\xi$ let $\rho^{\pm}_{\xi}$ be the non-Abelian representation of $\Pi$ to the Lie group $\PSL(2;\C)$ given by
\begin{equation*}
  \rho^{\pm}_{\xi}(x):=\begin{pmatrix}e^{\xi/2}&1\\0&e^{-\xi/2}\end{pmatrix}
  \quad\text{and}\quad
  \rho^{\pm}_{\xi}(y):=\begin{pmatrix}e^{\xi/2}&0\\d_{\pm}(\xi)&e^{-\xi/2}\end{pmatrix},
\end{equation*}
where
\begin{equation*}
  d_{\pm}(\xi)
  :=
  \frac{1}{2}\left(-2\cosh{\xi}+3\pm\sqrt{(2\cosh{\xi}-3)(2\cosh{\xi}+1)}\right).
\end{equation*}
The preferred longitude is sent to
\begin{equation*}
  \begin{pmatrix}
    \ell(\xi)^{\pm1}&\mp2\cosh(\xi/2)\sqrt{(2\cosh{\xi}-3)(2\cosh{\xi}+1)} \\
    0               &\ell(\xi)^{\mp1}
  \end{pmatrix},
\end{equation*}
where
\begin{equation*}
  \ell(\xi)
  :=
  \cosh(2\xi)-\cosh{\xi}-1-\sinh{\xi}\sqrt{(2\cosh{\xi}-3)(2\cosh{\xi}+1)}.
\end{equation*}
\par
We can express $\ell(\xi)$ in terms of the derivative of $S(\xi)$ if $\xi\in\Xi$.
We first prove the following lemma.
\begin{lem}\label{lem:dS}
If $\xi\in\Xi$, then we have $\frac{d\,S(\xi)}{d\,\xi}=\log\left(2\cosh\bigl(\xi+\varphi\bigr)-2\right)$.
\end{lem}
\begin{proof}
We put $a:=\Re\xi$, $b:=\Im\xi$, $c:=\Re\varphi$, and $d:=\Im\varphi$ as usual.
\par
From the definition of $S(\xi)$ (see \eqref{eq:S}), we have
\begin{equation}\label{eq:dS0}
\begin{split}
  \frac{d\,S(\xi)}{d\,\xi}
  =&
  \left(\frac{d\,\varphi}{d\,\xi}+1\right)\log(1-e^{-\xi-\varphi})
  +
  \left(\frac{d\,\varphi}{d\,\xi}-1\right)\log(1-e^{-\xi+\varphi})+\varphi
  +
  \xi\frac{d\,\varphi}{d\,\xi}
  \\
  =&
  \frac{d\,\varphi}{d\,\xi}\left(\log(1-e^{-\xi-\varphi})+\log(1-e^{-\xi+\varphi})+\xi\right)
  \\
  &+
  \log(1-e^{-\xi-\varphi})-\log(1-e^{-\xi+\varphi})+\varphi.
\end{split}
\end{equation}
From Lemma~\ref{lem:phi_xi}, we see that $\Re(-\xi-\varphi)=-(a+c)$ is negative, $\Re(-\xi+\varphi)=c-a$ is negative, $\Im(-\xi-\varphi)=-(b+d)$ is in $(-\pi,0)$, and $\Im(-\xi+\varphi)=-b+d$ is in $(0,\pi/2)$.
So we obtain from Lemma~\ref{lem:triangle}
\begin{equation}\label{eq:arg_xi_phi}
\begin{split}
  0&<\arg(1-e^{-\xi-\varphi})\le(-b-d)/2+\pi/2,
  \\
  (-b+d)/2-\pi/2&\le\arg(1-e^{-\xi+\varphi})<0.
\end{split}
\end{equation}
Adding these inequalities and $b$, we have
\begin{equation*}
  (b+d)/2-\pi/2<\arg(1-e^{-\xi-\varphi})+\arg(1-e^{-\xi+\varphi})+b<(b-d)/2+\pi/2,
\end{equation*}
Since $b+d>0$ and $b-d<0$, we have
\begin{equation}\label{eq:dS1}
\begin{split}
  &\log(1-e^{-\xi-\varphi})+\log(1-e^{-\xi+\varphi})+\xi
  \\
  =&
  \log\left(e^{\xi}+e^{-\xi}-e^{\varphi}-e^{-\varphi}\right)
  \\
  =&
  \log\bigl(2(\cosh\xi-\cosh\varphi)\bigr)
  =
  0
\end{split}
\end{equation}
since $\cosh\varphi=\cosh\xi-1/2$.
From \eqref{eq:dS0}, we obtain
\begin{equation}\label{eq:dS}
\begin{split}
  \frac{d\,S(\xi)}{d\,\xi}
  =&
  \log(1-e^{-\xi-\varphi})-\log(1-e^{-\xi+\varphi})+\varphi
  \\
  =&
  2\log\left(1-e^{-\xi-\varphi}\right)+\varphi+\xi
  \\
  =&
  \log\left(e^{\xi+\varphi}+e^{-\xi-\varphi}-2\right),
\end{split}
\end{equation}
where the second equality follows from \eqref{eq:dS1}, and the last equality follows since $\arg\bigl(2\log(1-e^{-\xi-\varphi})+\varphi+\xi\bigr)$ is between $b+d$ and $\pi$ from the first equality in \eqref{eq:arg_xi_phi}.
\end{proof}
\begin{defn}\label{defn:v}
For $\xi\in\Xi$, put
\begin{align*}
  v(\xi)
  &:=
  2\frac{d\,S(\xi)}{d\,\xi}-2\pi\i,
  \\
  v^{\pm}(\xi)
  &:=
  2\frac{d\,S^{\pm}(\xi)}{d\,\xi}-2\pi\i.
\end{align*}
\end{defn}
Then we have the following corollary.
\begin{cor}\label{cor:S_l}
We have
\begin{equation*}
  \ell(\xi)
  =-e^{-v(\xi)/2}
  =-e^{v^{+}(\xi)/2}
  =-e^{-v^{-}(\xi)/2}
\end{equation*}
if $\xi\in\Xi$.
\end{cor}
\begin{proof}
From Lemma~\ref{lem:dS} and \eqref{eq:def_phi}, for $\xi\in\Xi$ we have
\begin{equation*}
\begin{split}
  &\exp\left(\frac{d\,S(\xi)}{d\,\xi}\right)
  \\
  =&
  e^{\xi+\varphi}+e^{-\xi-\varphi}-2
  \\
  =&
  e^{\xi}\left(\cosh\xi-\frac{1}{2}+\frac{1}{2}\sqrt{(2\cosh\xi-3)(2\cosh\xi+1)}\right)
  \\
  &+
  e^{-\xi}\left(\cosh\xi-\frac{1}{2}-\frac{1}{2}\sqrt{(2\cosh\xi-3)(2\cosh\xi+1)}\right)
  -2
  \\
  =&
  \cosh(2\xi)-\cosh\xi-1+\sinh\xi\sqrt{(2\cosh\xi-3)(2\cosh\xi+1)}
  \\
  =&
  \ell(\xi)^{-1}.
\end{split}
\end{equation*}
So we have $\ell(\xi)=e^{-v(\xi)/2-\pi\i}=-e^{-v(\xi)/2}$.
\par
Since $S^{\pm}(\xi)=\mp S(\xi)+2\xi\pi\i$ (see \eqref{eq:Splus} and \eqref{eq:Sminus}), we have
\begin{equation*}
\begin{split}
  v^{+}(\xi)
  =&
  -v(\xi),
  \\
  v^{-}(\xi)
  =&
  v(\xi)+4\pi\i
\end{split}
\end{equation*}
and the corollary follows.
\end{proof}
%%%%%%%%%%%%%%%%%%%%%%%%%%%%%%%%%%%%%%%%%%%%%%%%%%%%%%%%%%%%%%%%%%%%%%%%%%%%%%%
\subsection{Adjoint Reidemeister torsion}
For a representation $\rho\colon\Pi\to\PSL(2;\C)$, let $C^{\ast}\left(X;\mathfrak{sl}(2;\C)_{\rho}\right):=\Hom_{\Z[\Pi]}\left(C_{\ast}(\tX;\Z),\mathfrak{sl}(2;\C)\right)$ be the cochain complex twisted by the adjoint action of $\rho$, where $\tX$ is the universal cover of $X$, $\Pi$ acts on $\tX$ by the deck transformation, and $\Pi$ acts on $\mathfrak{sl}(2;\C)$ by the adjoint action of $\rho$, that is, $w\cdot g:=\rho(w)^{-1}g\rho(w)$ for $w\in\Pi$ and $g\in\mathfrak{sl}(2;\C)$.
\par
The cohomological adjoint Reidemeister torsion associated with the meridian, denoted by $T(\rho)$ is defined to be the torsion of the cochain complex $C^{\ast}\left(X;\mathfrak{sl}(2;\C)_{\rho}\right)$.
For the figure-eight knot and the representation $\rho^{\pm}_{\xi}$ above, we have
\begin{equation*}
  T(\rho^{\pm}_{\xi})
  =
  \frac{2}{\sqrt{(2\cosh{\xi}-3)(2\cosh{\xi}+1)}}
\end{equation*}
if we choose the sign appropriately.
For more details, see \cite{Porti:MAMCAU1997},\cite{Dubois:CANMB2006}, and \cite[Chapter~5]{Murakami/Yokota:2018}.
%%%%%%%%%%%%%%%%%%%%%%%%%%%%%%%%%%%%%%%%%%%%%%%%%%%%%%%%%%%%%%%%%%%%%
\subsection{Chern--Simons invariant}
Let $X$ be a compact three-manifold with boundary $T$ that is homeomorphic to a torus.
\par
Let $R(X)$ and $R(T)$ be the $\PSL(2;\C)$ character varieties of $\pi_1(X)$ and $\pi_1(T)$, respectively.
So $R(X)$ ($R(T)$, respectively) consists of equivalence classes of characters of representations $\pi_1(X)\to\PSL(2;\C)$ ($\pi_1(T)\to\PSL(2;\C)$, respectively).
We introduce the following equivalence relations in $\Hom\bigl(\pi_1(T),\C\bigr)\times\C^{\times}$:
\begin{equation}\label{eq:equivalence}
\begin{split}
  (\alpha,\beta;z)
  &\sim
  (\alpha+1/2,\beta;ze^{-4\pi\i\beta}),
  \\
  (\alpha,\beta;z)
  &\sim
  (\alpha,\beta+1/2;ze^{4\pi\i\alpha}),
  \\
  (\alpha,\beta;z)
  &\sim
  (-\alpha,-\beta;z),
\end{split}
\end{equation}
where $(\alpha,\beta)$ is the coordinate with respect to a fixed generators of $\Hom(\pi_1(T),\C)\cong\C^2$, and $z\in\C^{\times}$.
Let $E(T)$ be the quotient space $\bigl(\Hom\bigl(\pi_1(T),\C\bigr)\times\C^{\times}\bigr)/\sim$ and denote by $[\alpha,\beta;z]$ the equivalence class of $(\alpha,\beta;z)$ in $E(T)$.
\par
Then, following \cite{Kirk/Klassen:COMMP1993}, one can define the Chern--Simons function $\cs_X\colon R(X)\to E(T)$ such that $p\circ\cs_X=r$, where $p\colon E(T)\to R(T)$ is the projection map $p\colon[\alpha,\beta;z]\mapsto[\alpha,\beta]$, and $r\colon R(X)\to R(T)$ is the restriction map, where $[\alpha,\beta]\in R(T)$ denote the equivalence class.
\begin{equation*}
\begin{tikzcd}
                                            &E(T)\arrow[d,"p"] \\
R(X)\arrow[ru,"\cs_{X}"]\arrow[r,"r"]&R(T)
\end{tikzcd}
\end{equation*}
One can calculate $\cs_X$ by using the following theorem.
\begin{thm}[P.~Kirk and E.~Klassen \cite{Kirk/Klassen:COMMP1993}]\label{thm:KK}
If $\rho_t$ is a path of representations and $\cs_X\left(\left[\rho_t\right]\right)=[\alpha_t,\beta_t;z_t]$, then
\begin{equation*}
  \frac{z_1}{z_0}
  =
  \exp
  \left(
    \frac{\i}{2\pi}
    \int_{0}^{1}\left(\alpha_t\frac{d\,\beta_t}{d\,t}-\frac{d\,\alpha_t}{d\,t}\beta_t\right)\,dt
  \right).
\end{equation*}
\end{thm}
\par
Fixing generators $(m,l)$ of $\pi_1(T)$, the Chern--Simons invariant $\CS_{(\mu,\lambda)}(\rho)\in\C/(\pi^2\Z)$ of $\rho$ associated with $(\mu,\lambda)$ is defined as follows.
\begin{equation}\label{eq:CS_def}
  \cs_{X}([\rho])
  =
  \left[\frac{\mu}{4\pi\i},\frac{\lambda}{4\pi\i};
  \exp\left(\frac{2\CS_{(\mu,\lambda)}(\rho)}{\pi\i}\right)\right]
  \in E(T),
\end{equation}
where $\rho\colon\pi_1(X)\to\PSL(2;\C)$, $[\rho]$ is its equivalence class in $R(X)$, and $\rho(m)=\begin{pmatrix}e^{\mu/2}&\ast\\0&e^{-\mu/2}\end{pmatrix}$ and $\rho(l)=\begin{pmatrix}-e^{\lambda/2}&\ast\\0&-e^{-\lambda/2}\end{pmatrix}$ after a suitable conjugation.
\par
Putting $X:=S^3\setminus\Int{N(\FE)}$ and $T:=\partial{X}$, we calculate the Chern--Simons invariant $\CS_{\bigl(\xi,v(\xi)\bigr)}\left(\rho^{-}_{\xi}\right)$ of $\rho^{-}_{\xi}$ associated with $\bigl(\xi,v(\xi))$ for $\xi\in\Xi$, where we choose the meridian and the preferred longitude as generators of $\pi_1(T)$.
\begin{lem}\label{lem:CS_minus}
If $\xi\in\Xi$, then the Chern--Simons invariant of $\rho^{-}_{\xi}$ associated with $(\xi,v(\xi))$ is given as
\begin{equation*}
  \CS_{\bigl(\xi,v(\xi)\bigr)}\left(\rho^{-}_{\xi}\right)
  =
  S(\xi)-\xi\pi\i-\frac{1}{4}\xi v(\xi).
\end{equation*}
\end{lem}
\begin{proof}
Let $\rho^{\rm{Abel}}_t$ be the Abelian representation sending both $x$ and $y$ to $\begin{pmatrix}e^{t\kappa/2}&0\\0&e^{-t\kappa/2}\end{pmatrix}$.
\par
Since $d_{-}(\kappa)=0$, the representation $\rho^{-}_{\kappa}$ defined in Subsection~\ref{subsec:rep} is upper triangular, and so its trace coincides with that of the Abelian representation $\rho^{\rm{Abel}}_{1}$.
We can assume that $\cs_X\left(\left[\rho^{\rm{Abel}}_t\right]\right)$ is given by $\left[\frac{t\kappa}{4\pi\i},0;z_t\right]$ because $\rho^{\rm{Abel}}_t$ sends the preferred longitude to the identity matrix.
Then from Theorem~\ref{thm:KK}, $z_1=z_0=1$, since $\rho^{\rm{Abel}}_0$ is the trivial representation.
\par
Now, since $\cs_X$ depends only on traces, we have $\cs_{X}\left(\left[\rho^{-}_{\kappa}\right]\right)=\cs_{X}\left(\left[\rho^{\rm{Abel}}_{\kappa}\right]\right)$.
So we have from Theorem~\ref{thm:KK}
\begin{equation}\label{eq:z0_kappa}
  \cs_{X}\left(\left[\rho^{-}_{\kappa}\right]\right)
  =
  \cs_{X}\left(\left[\rho^{\rm{Abel}}_{\kappa}\right]\right)
  =
  \left[\frac{\kappa}{4\pi\i},0;1\right]
  =
  \left[\frac{\kappa}{4\pi\i},-\frac{1}{2};e^{-\kappa}\right],
\end{equation}
where we use \eqref{eq:equivalence} in the last equality.
\par
Since $\varphi(\kappa)=0$ and $e^{\kappa}+e^{-\kappa}=3$, from Definition~\ref{defn:v} and \eqref{eq:dS} we have
\begin{equation}\label{eq:eta_kappa}
\begin{split}
  v(\kappa)
  =&
  2\frac{d\,S(\kappa)}{d\,\xi}-2\pi\i
  \\
  =&
  2\log(e^{\kappa}+e^{-\kappa}-2)-2\pi\i
  \\
  =&
  -2\pi\i.
\end{split}
\end{equation}
Therefore, we conclude
\begin{equation*}
  \cs_{X}\left(\left[\rho^{-}_{\kappa}\right]\right)
  =
  \left[
    \frac{\kappa}{4\pi\i},
    \frac{v(\kappa)}{4\pi\i};
    e^{-\kappa}
  \right].
\end{equation*}
from \eqref{eq:z0_kappa}.
We also have
\begin{equation}\label{eq:CS_kappa}
  \CS_{\bigl(\kappa,v(\kappa)\bigr)}\left(\rho^{-}_{\kappa}\right)
  =
  \frac{-\kappa\pi\i}{2}
\end{equation}
from \eqref{eq:CS_def}.
\par
Let $\gamma(t)$ ($0\le t\le1$) be a smooth path with $\gamma(0)=\kappa$, $\gamma(1)=\xi$, and $\gamma(t)\in\Xi$ except for $\gamma(0)=\kappa$.
Then $\{\rho^{-}_{\gamma(t)}\}_{0\le t\le1}$ is a path of representations connecting $\rho^{-}_{\kappa}$ to $\rho^{-}_{\xi}$.
We can define $z(t)\in\C$ so that
\begin{equation*}
  \cs_{X}\left(\left[\rho^{-}_{\gamma(t)}\right]\right)
  =
  \left[
    \frac{\gamma(t)}{4\pi\i},\frac{v\bigl(\gamma(t)\bigr)}{4\pi\i};z(t)
  \right],
\end{equation*}
since $\rho^{-}_{\gamma(t)}$ sends the meridian and the preferred longitude to
\begin{align*}
  \begin{pmatrix}e^{\gamma(t)/2}&1\\0&e^{-\gamma(t)/2}\end{pmatrix}&,
  \\
  \begin{pmatrix}\ell\bigl(\gamma(t)\bigr)^{-1}&\ast\\0&\ell\bigl(\gamma(t)\bigr)\end{pmatrix}&
  =
  \begin{pmatrix}-e^{v\bigl(\gamma(t)\bigr)/2}&\ast\\0&-e^{-v\bigl(\gamma(t)\bigr)/2}\end{pmatrix},
\end{align*}
respectively from Corollary~\ref{cor:S_l}.
Note that $z(0)=e^{-\kappa}$ from \eqref{eq:CS_kappa}.
Then we have from Theorem~\ref{thm:KK}
\begin{equation}\label{eq:KK_real}
  z(1)
  =
  e^{-\kappa}
  \exp
  \left(
    \frac{\i}{2\pi}
    \int_{0}^{1}
    \left(
      \gamma(t)\frac{d\,v\bigl(\gamma(t)\bigr)}{d\,t}
      -
      v\bigl(\gamma(t)\bigr)\frac{d\,\gamma(t)}{d\,t}
    \right)
    \,dt
  \right).
\end{equation}
Using integration by parts, the integral above becomes
\begin{equation*}
\begin{split}
  &\Bigl[\gamma(t)v\bigl(\gamma(t)\bigr)\Bigr]_{0}^{1}
  -
  2\int_{0}^{1}\frac{d\,\gamma(t)}{d\,t}v\bigl(\gamma(t)\bigr)\,dt
  \\
  =&
  \xi v(\xi)+2\kappa\pi\i
  -
  2\int_{0}^{1}\frac{d\,\gamma(t)}{d\,t}v\bigl(\gamma(t)\bigr)\,dt.
  \\
  =&
  \xi v(\xi)+2\kappa\pi\i
  -
  2\int_{\gamma}v(z)\,dz,
\end{split}
\end{equation*}
where the first equality follows from \eqref{eq:eta_kappa}, and $\gamma$ is the image of $\gamma(t)$ ($0\le t\le1$) in the complex plane.
Since $v(z)=2\frac{d\,S(z)}{d\,z}-2\pi\i$ from Definition~\ref{defn:v}, and $S(\kappa)=0$ from \eqref{eq:Sminus}, we have
\begin{equation*}
  \int_{\gamma}v(z)\,dz
  =
  \Bigl[2S(z)-2\pi\i z\Bigr]_{\kappa}^{\xi}
  =
  2S(\xi)-2\xi\pi\i+2\kappa\pi\i.
\end{equation*}
From \eqref{eq:KK_real} we have
\begin{equation*}
\begin{split}
  z(1)
  =&
  e^{-\kappa}
  \exp
  \left(
    \frac{\i}{2\pi}
    \left(
      \xi v(\xi)+2\kappa\pi\i-2\bigl(2S(\xi)-2\xi\pi\i+2\kappa\pi\i\bigr)
    \right)
  \right)
  \\
  =&
  \exp
  \left(
    -2\xi+\frac{\i}{2\pi}\xi v(\xi)-\frac{2\i}{\pi}S(\xi)
  \right).
\end{split}
\end{equation*}
So we have
\begin{equation*}
  \CS_{\xi,v(\xi)}(\rho^{-}_{\xi})
  =
  S(\xi)-\xi\pi\i-\frac{1}{4}\xi v(\xi)
\end{equation*}
from \eqref{eq:CS_def}.
%%%%%%%%%%%%%%%%%%%%%%%%%%%%%%%%%%%%%%%%%%%%%%%%%%%%%%%%%%%%%%%%%%%%%%%%%%%%%%%
\end{proof}
In Cases (i) and (viii) in Introduction, we proved in \cite{Murakami:JTOP2013}see also \cite[Chapter~6]{Murakami/Yokota:2018}) that the colored Jones polynomial determines the Chern--Simons invariant associated with the irreducible representation $\rho^{+}_{u}$.
In fact we proved the formula $\CS_{\bigl(\xi,v^{+}(u)\bigr)}\left(\rho^{+}_{u}\right)=S^{+}(u)-u\pi\i-\frac{1}{4}uv^{+}(u)$.
Putting $u=0$, the same formula holds in Case (vii).
In Case (ii) the same formula also holds, and it corresponds to the affine representation $\rho^{+}_{\kappa}=\rho^{-}_{\kappa}$ (see \cite[\S~6]{Murakami:AGT2025}).
\par
On the other hand, in Cases (iii) and (vi), the colored Jones polynomial determines the Chern--Simons invariant associated with $\rho^{-}_{u}$ as in Lemma~\ref{lem:CS_minus}.
So the corresponding hyperbolic structure has the reversed orientation.
In particular, $S(0)=\CS_{(0,0)}\left(\rho^{-}_{0}\right)=-\Vol(S^3\setminus\FE)\i$, which is compared with $S^{+}(0)=\CS_{(0,0)}\left(\rho^{+}_{0}\right)=\Vol(S^3\setminus\FE)\i$.
\begin{rem}\label{rem:big_real}
As for (vi), in \cite{Murakami/Tran:Takata2025}, we proved that for $a\in\R$ with $a>\kappa$, then $\CS_{\bigl(a,\tv(a)\bigr)}\left(\rho^{-}_{a}\right)=S(a)-\frac{1}{4}a\tv(a)$, where
\begin{equation*}
\begin{split}
  \tv(a)
  :=&
  \log
  \left(
    \cosh(2a)-\cosh{a}-1+\sinh{a}\sqrt{(2\cosh{a}-3)(2\cosh{a}+1)}
  \right)
  \\
  =&
  v(a)+2\pi\i.
\end{split}
\end{equation*}
From \eqref{eq:CS_def} and \eqref{eq:equivalence}, both $\CS_{\bigl(a,\tv(a)\bigr)}\left(\rho^{-}_{a}\right)$ and $\CS_{\bigl(a,v(a)\bigr)}\left(\rho^{-}_{a}\right)$ determine $\cs_{X}\left(\left[\rho^{-}_{a}\right]\right)$.
\end{rem}
\appendix
\section{Proof of Proposition~\ref{prop:curvature}}\label{sec:curvature}
%%%%%%%%%%%%%%%%%%%%%%%%%%%%%%%%%%%%%%%%%%%%%%%%%%%%%%%%%%%%%%%%%%%%%%%%%%%%%%%
In this appendix, we prove Proposition~\ref{prop:curvature}.
\subsection{Calculation of the curvature}
Recall that our curve $\overline{C}_G$ is given by the equation $\Phi(X,Y)=1/4$ in the region $\{(X,Y)\in\R^2\mid-\pi<Y\le\pi\}$, where $\Phi(X,Y)$ is defined in \eqref{eq:Phi}.
\par
Put
\begin{equation*}
  A:=\cosh{X}\cos{Y},
  \quad
  B:=\sinh{X}\sin{Y}.
\end{equation*}
Recall that we have put
\begin{equation*}
  \alpha
  =
  \cosh{a}\cos{b},
  \quad
  \beta
  =
  \sinh{a}\sin{b}.
\end{equation*}
Thus, from \eqref{eq:Phi} we have $\Phi(X,Y)=(A-\alpha)^2+(B-\beta)^2$.
\par
To prove the proposition, we first prepare a lemma expressing the curvature in terms of $\alpha$, $\beta$, $A$, and $B$ (and $\Phi_X(X,Y)^2+\Phi_Y(X,Y)^2$).
\begin{lem}\label{lem:curvature_AB}
When $(X,Y)\ne(0,0)$, the curvature of $\overline{C}_{G}$ at $(X,Y)$ is given by
\begin{equation*}
  \frac{2\left(\alpha A\left(A^2+B^2-1\right)+\beta B\left(A^2+B^2+1\right)-A^2+B^2+1\right)}
          {(\Phi_X^2+\Phi_Y^2)^{3/2}}
\end{equation*}
if we orient it appropriately, where we write $\Phi_X$ and $\Phi_Y$ for $\Phi_X(X,Y)$ and $\Phi_Y(X,Y)$ respectively.
\end{lem}
\begin{rem}\label{rem:cosh_cos_alpha_beta}
From Lemma~\ref{lem:alpha_beta_a_b}, $\cosh{a}-\cos{b}>1/2$ if and only if $(\alpha-1)^2+\beta^2>1/4$, and $\cosh{a}-\cos{b}=1/2$ if and only if $(\alpha-1)^2+\beta^2=1/4$.
\par
Since the only singularity of $\overline{C}_{G}$, that is, the point $(X,Y)$ with $\Phi_X=\Phi_Y=0$, is $(0,0)$ when $\cosh{a}-\cos{b}=1/2$ from Lemma~\ref{lem:H0G}, the denominator of the curvature formula in Lemma~\ref{lem:curvature_AB} does not vanish unless $(\alpha-1)^2+\beta^2=1/4$ and $(X,Y)=(0,0)$.
\end{rem}
\begin{proof}
It is well-known that the curvature of the curve $\overline{C}_{G}$ at a point $(X,Y)$ is given by
\begin{equation*}
  \pm\frac{\Phi_{XX}\Phi_Y^2+\Phi_{YY}\Phi_X^2-2\Phi_{XY}\Phi_X\Phi_Y}{(\Phi_X^2+\Phi_Y^2)^{3/2}},
\end{equation*}
where the sign depends on the orientation of the curve.
We will show that the numerator with the plus sign equals $2\left(\alpha A\left(A^2+B^2-1\right)+\beta B\left(A^2+B^2+1\right)-A^2+B^2+1\right)$.
\par
In the following, we write $X_c:=\cosh{X}$, $X_s:=\sinh{X}$, $Y_c:=\cos{Y}$, and $Y_s:=\sin{Y}$ for short.
Note that $A=X_cY_c$ and $B=X_sY_s$.
\par
Then from \eqref{eq:Phi_partial}, we have
\begin{align*}
  \frac{1}{2}\Phi_{X}
  &=
  X_sX_c-\beta X_cY_s-\alpha X_sY_c,
  \\
  \frac{1}{2}\Phi_{Y}
  &=
  -Y_sY_c-\beta X_sY_c+\alpha X_cY_s.
\end{align*}
From \eqref{eq:Phi_partial2}, we also have
\begin{align*}
  \frac{1}{2}\Phi_{XX}
  &=
  2X_c^2-1-\beta X_sY_s-\alpha X_cY_c
  =
  2X_c^2-1-\beta B-\alpha A,
  \\
  \frac{1}{2}\Phi_{YY}
  &=
  1-2Y_c^2+\beta X_sY_s+\alpha X_cY_c
  =
  1-2Y_c^2+\beta B+\alpha A,
  \\
  \frac{1}{2}\Phi_{XY}
  &=
  \alpha X_sY_s-\beta X_cY_c.
  =
  \alpha B-\beta A,
\end{align*}
where $\Phi_{XX}$, $\Phi_{XY}$, and $\Phi_{YY}$ mean $\Phi_{XX}(X,Y)$, $\Phi_{XY}(X,Y)$, and $\Phi_{YY}(X,Y)$, respectively.
\par
In the following calculations, we use Mathematica for confirmation.
We have
\begin{equation*}
\begin{split}
  &\frac{1}{8}\Phi_{XY}\Phi_X\Phi_Y
  \\
  =&
  (\alpha B-\beta A)
  \\
  &\times
  \left(
   -AB-\beta AX_s^2+\alpha BX_c^2
    +\beta AY_s^2+\beta^2AB-\alpha\beta X_c^2Y_s^2
  \right.
  \\
  &\quad
  \left.
    +\alpha BY_c^2+\alpha\beta X_s^2Y_c^2-\alpha^2AB
  \right)
  \\
  =&
  (\alpha B-\beta A)
  \\
  &\times
  \left(
    AB(\beta^2-\alpha^2-1)
    +\alpha B(X_c^2+Y_c^2)
    +\beta A(Y_s^2-X_s^2)
    +\alpha\beta(X_s^2Y_c^2-X_c^2Y_s^2)
  \right)
  \\
  =&
  (\alpha B-\beta A)
  \\
  &\times
  \left(
    AB(\beta^2-\alpha^2-1)
    +\alpha B(A^2+B^2+1)
  \right.
  \\
  &\qquad
  \left.
    -\beta A(A^2+B^2-1)
    +\alpha\beta(A^2-B^2-1)
  \right),
\end{split}
\end{equation*}
where the last equality follows from
\begin{align}
  X_c^2+Y_c^2
  =&
  A^2+B^2+1,
  \label{eq:Xc_Yc}
  \\
  X_s^2-Y_s^2
  =&
  A^2+B^2-1,
  \label{eq:Xs_Ys}
  \\
  X_s^2Y_c^2-X_c^2Y_s^2
  =&
  (X_c^2-1)Y_c^2-X_c^2(1-Y_c^2)=A^2-B^2-1.
  \label{eq:XsYc_XcYs}
\end{align}
\par
We also have
\begin{equation*}
\begin{split}
  &\frac{1}{8}\Phi_{XX}\Phi_Y^2
  \\
  =&
  \left(2X_c^2-\alpha A-\beta B-1\right)
  \\
  &\times
  \left(
    Y_s^2Y_c^2+\beta^2X_s^2Y_c^2+\alpha^2X_c^2Y_s^2
    +2\beta BY_c^2-2\alpha\beta AB-2\alpha AY_s^2
  \right)
  \\
  =&
  2A^2Y_s^2+2\beta^2A^2X_s^2+2\alpha^2X_c^4Y_s^2
  +4\beta A^2B-4\alpha\beta ABX_c^2-4\alpha AX_c^2Y_s^2
  \\
  &
  -(\alpha A+\beta B+1)
  \left(
    Y_s^2Y_c^2+\beta^2X_s^2Y_c^2+\alpha^2X_c^2Y_s^2
    +2\beta BY_c^2-2\alpha\beta AB-2\alpha AY_s^2
  \right),
\end{split}
\end{equation*}
and
\begin{equation*}
\begin{split}
  &\frac{1}{8}\Phi_{YY}\Phi_X^2
  \\
  =&
  \left(1-2Y_c^2+\alpha A+\beta B\right)
  \\
  &\times
  \left(
    X_s^2X_c^2+\beta^2X_c^2Y_s^2+\alpha^2X_s^2Y_c^2
    -2\beta BX_c^2+2\alpha\beta AB-2\alpha AX_s^2
  \right)
  \\
  =&
  -2A^2X_s^2-2\beta^2A^2Y_s^2-2\alpha^2X_s^2Y_c^4
  +4\beta A^2B-4\alpha\beta ABY_c^2+4\alpha AX_s^2Y_c^2
  \\
  &+
  (\alpha A+\beta B+1)
  \\
  &\quad\times
  \left(
    X_s^2X_c^2+\beta^2X_c^2Y_s^2+\alpha^2X_s^2Y_c^2
    -2\beta BX_c^2+2\alpha\beta AB-2\alpha AX_s^2
  \right).
\end{split}
\end{equation*}
Therefore we have
\begin{equation*}
\begin{split}
  &\frac{1}{8}\Phi_{XX}\Phi_Y^2+\frac{1}{8}\Phi_{YY}\Phi_X^2
  \\
  =&
  2A^2(Y_s^2-X_s^2)+2\beta^2A^2(X_s^2-Y_s^2)
  +2\alpha^2(X_c^4Y_s^2-X_s^2Y_c^4)
  +8\beta A^2B
  \\
  &-4\alpha\beta AB(X_c^2+Y_c^2)+4\alpha A(X_s^2Y_c^2-X_c^2Y_s^2)
  \\
  &+
  (\alpha A+\beta B+1)
  \Bigl(
    X_s^2X_c^2-Y_s^2Y_c^2
    +\beta^2(X_c^2Y_s^2-X_s^2Y_c^2)
    +\alpha^2(X_s^2Y_c^2-X_c^2Y_s^2)
  \\
  &\qquad
    -2\beta B(X_c^2+Y_c^2)
    -2\alpha A(X_s^2-Y_s^2)
    +4\alpha\beta AB
  \Bigr)
  \\
  =&
  2A\bigl(\beta^2A-A-\alpha(\alpha A+\beta B+1)\bigr)(A^2+B^2-1)
  \\
  &
  +2\alpha^2(B^4+2B^2+A^2B^2-A^2+1)
  +8\beta A^2B
  \\
  &-2\beta B(3\alpha A+\beta B+1)(A^2+B^2+1)
  \\
  &
  +\bigl(4\alpha A+(\alpha^2-\beta^2)(\alpha A+\beta B+1)\bigr)(A^2-B^2-1)
  \\
  &
  +(\alpha A+\beta B+1)(A^4+B^4+2A^2B^2+B^2-A^2)
  +4\alpha\beta AB(\alpha A+\beta B+1),
\end{split}
\end{equation*}
where we use \eqref{eq:Xc_Yc}, \eqref{eq:Xs_Ys}, \eqref{eq:XsYc_XcYs},
\begin{equation*}
\begin{split}
  X_s^2X_c^2-Y_s^2Y_c^2
  =&
  (X_c^2-1)X_c^2-(1-Y_c^2)Y_c^2
  \\
  =&
  (X_c^2+Y_c^2)^2-2X_c^2Y_c^2-(X_c^2+Y_c^2)
  \\
  =&
  A^4+B^4+2A^2B^2+B^2-A^2,
\end{split}
\end{equation*}
and
\begin{equation*}
\begin{split}
  X_c^4Y_s^2-X_s^2Y_c^4
  =&
  X_c^2\bigl(X_s^2Y_c^2-(A^2-B^2-1)\bigr)-Y_c^2(X_c^2Y_s^2+A^2-B^2-1)
  \\
  =&
  A^2(X_s^2-Y_s^2)-(A^2-B^2-1)(X_c^2+Y_c^2)
  \\
  =&
  A^2(A^2+B^2-1)-(A^2-B^2-1)(A^2+B^2+1)
  \\
  =&
  B^4+2B^2+A^2B^2-A^2+1
\end{split}
\end{equation*}
in the second equality.
\par
Now Mathematica gives the following result.
\begin{equation*}
\begin{split}
  &\frac{1}{8}
  \left(\Phi_{XX}\Phi_Y^2+\Phi_{YY}\Phi_X^2-2\Phi_{XY}\Phi_X\Phi_Y\right)
  \\
  =&
  \left((A-\alpha)^2+(B-\beta)^2\right)
  \left(
    \alpha A(A^2+B^2-1)+\beta B(A^2+B^2+1)-A^2+B^2+1
  \right)
  \\
  =&
  \frac{1}{4}
  \left(
    \alpha A(A^2+B^2-1)+\beta B(A^2+B^2+1)-A^2+B^2+1
  \right),
\end{split}
\end{equation*}
where we use $(A-\alpha)^2+(B-\beta)^2=\Phi(X,Y)=1/4$ since $(X,Y)\in\overline{C}_G$.
\par
The proof is complete.
\end{proof}
\begin{rem}\label{rem:XY00_AB10}
The condition $(X,Y)=(0,0)$ is equivalent to $(A,B)=(1,0)$ because of the following reason.
\par
Clearly, the equality $(X,Y)=(0,0)$ implies $(A,B)=(1,0)$.
\par
Conversely, suppose that $(A,B)=(1,0)$.
Then we have $\cosh{X}\cos{Y}=1$ and $\sinh{X}\sin{Y}=0$.
From the second equality, we have either $X=0$, or $Y=0, \pi$, since we are assuming that $-\pi<Y\le\pi$.
\par
If $X=0$, then $\cos{Y}=1$ and so we have $Y=0$.
If $Y=0$, then $\cosh{X}=1$ and so we have $X=0$.
If $Y=\pi$, then we have $\cosh{X}=-1$, which is impossible.
\end{rem}
\par
If $(X,Y)\in\overline{C}_G$, then $(A-\alpha)^2+(B-\beta)^2=1/4$.
So we can put
\begin{align*}
  A&=\alpha+(\cos{s})/2,
  \\
  B&=\beta+(\sin{s})/2
\end{align*}
for a real number $s$ with $0\le s<2\pi$.
We also put
\begin{multline}\label{eq:lambda}
  \lambda(\alpha,\beta,s)
  \\
  :=
  \alpha A(A^2+B^2-1)+\beta B(A^2+B^2+1)-A^2+B^2+1
  \Biggm|_{\substack{A:=\alpha+(\cos{s})/2,\\ \!\!B:=\beta+(\sin{s})/2}}.
\end{multline}
\par
From Lemma~\ref{lem:curvature_AB}, and Remarks~\ref{rem:cosh_cos_alpha_beta} and \ref{rem:XY00_AB10}, to prove Proposition~\ref{prop:curvature}, it is enough to show $\lambda(\alpha,\beta,s)>0$ if $(\alpha-1)^2+\beta^2\ge1/4$ and $(A,B)\ne(1,0)$.
\par
We split the proof into the following three cases recalling that $\alpha>1/2$ and $\beta>0$:
\begin{enumerate}
\item $\beta\ge1/2$,
\item $0<\beta<1/2$ and $\alpha\ge3/2$,
\item $0<\beta<1/2$ and $1/2<\alpha<3/2$.
\end{enumerate}
%%%%%%%%%%%%%%%%%%%%%%%%%%%%%%%%%%%%%%%%%%%%%%%%%%%%%%%%%%%%%%%%%%%%%%%%%%%%%%%
%%%%%%%%%%%%%%%%%% \input{case_i} %%%%%%%%%%%%%%%%%%%%%%%%%%%%%%%%%%%%%%%%%%%%%
%%%%%%%%%%%%%%%%%%%%%%%%%%%%%%%%%%%%%%%%%%%%%%%%%%%%%%%%%%%%%%%%%%%%%%%%%%%%%%%
\subsection{The case where $\beta\ge1/2$}
In this subsection we prove Proposition~\ref{prop:curvature} for Case (i).
\begin{lem}[Case (i)]
If $\beta\ge1/2$, $\alpha>1/2$,  $(\alpha-1)^2+\beta^2\ge1/4$, and $(A,B)\ne(1,0)$, then we have $\lambda(\alpha,\beta,s)>0$.
\end{lem}
\begin{proof}
First of all, if $\beta\ge1/2$, then the inequality $(\alpha-1)^2+\beta^2\ge1/4$ certainly holds.
Therefore in the following proof, we do not use the assumption $(\alpha-1)^2+\beta^2\ge1/4$.
\par
Since $\beta\ge1/2$ and $\alpha>1/2$, we have $B=\beta+(\sin{s})/2\ge0$ and $A=\alpha+(\cos{s})/2>0$.
So we have
\begin{equation*}
\begin{split}
  \lambda(\alpha,\beta,s)
  \ge&
  \alpha A\bigl(A^2-1\bigr)
  +\frac{1}{2}BA^2-A^2+1
  \\
  &\quad\text{(since $B=\beta+(\sin{s})/2\ge1/2+(\sin{s})/2$)}
  \\
  \ge&
  \alpha A\bigl(A^2-1\bigr)+\frac{1}{4}A^2+\frac{1}{4}A^2\sin{s}-A^2+1
  \\
  &\quad\text{(since $\alpha=A-(\cos{s})/2$)}
  \\
  =&
  \bigl(A-(\cos{s})/2\bigr)A\bigl(A^2-1\bigr)-\frac{3}{4}A^2+\frac{1}{4}A^2\sin{s}+1
  \\
  =&
  \frac{1}{4}
  \left(4A^4-7A^2+4+A^2\sin{s}-2A\bigl(A^2-1\bigr)\cos{s}\right)
  \\
  &\quad\text{(composition of simple harmonic motions)}
  \\
  =&
  \frac{1}{4}
  \left(4A^4-7A^2+4+\sqrt{A^4+4A^2\bigl(A^2-1\bigr)^2}\sin(s+\theta)\right)
  \\
  \ge&
  \frac{1}{4}
  \left(4A^4-7A^2+4-A\sqrt{4A^4-7A^2+4}\right)
  \\
  &\quad\text{(since $4A^4-7A^2+4=4(A^2-1)^2+A^2>0$)}
  \\
  =&
  \frac{1}{4}\sqrt{4A^4-7A^2+4}
  \left(\sqrt{4A^4-7A^2+4}-A\right)
  \\
  \ge&
  0,
\end{split}
\end{equation*}
where the last inequality follows since $\left(4A^4-7A^2+4\right)-A^2=4\left(A^2-1\right)^2\ge0$.
Let us consider when all the equalities hold.
The first equality holds when $B=0$, and the second holds when $\beta=1/2$.
So, if the first two equalities hold, then we have $B=0$, $\beta=1/2$, and $s=3\pi/2$.
Note that if $s=3\pi/2$, then $4A^4-7A^2+4+A^2\sin{s}-2A\bigl(A^2-1\bigr)\cos{s}=4(A^2-1)^2$.
The last equality holds when $A=1$ since $A>0$, in which case the equality in the third inequality also holds.
Therefore we conclude that the equality $\lambda(\alpha,\beta,s)=0$ holds only when $(A,B)=(1,0)$.
\par
Therefore we have $\lambda(\alpha,\beta,s)>0$ from the assumption $(A,B)\ne(1,0)$, completing the proof for the case where $\beta\ge1/2$.
\end{proof}
%%%%%%%%%%%%%%%%%%%%%%%%%%%%%%%%%%%%%%%%%%%%%%%%%%%%%%%%%%%%%%%%%%%%%%%%%%%%%%%
%%%%%%%%%%%%%%% \input{case_ii} %%%%%%%%%%%%%%%%%%%%%%%%%%%%%%%%%%%%%%%%%%%%%%%
%%%%%%%%%%%%%%%%%%%%%%%%%%%%%%%%%%%%%%%%%%%%%%%%%%%%%%%%%%%%%%%%%%%%%%%%%%%%%%%
\subsection{The case where $0<\beta<1/2$ and $\alpha\ge3/2$}
In this subsection we prove Proposition~\ref{prop:curvature} for Case (ii).
\begin{lem}[Case (ii)]
If $0<\beta<1/2$, $\alpha\ge3/2$, $(\alpha-1)^2+\beta^2\ge1/4$, and $(A,B)\ne(1,0)$, then we have $\lambda(\alpha,\beta,s)>0$.
\end{lem}
\begin{proof}
Since $A=\alpha+(\cos{s})/2\ge1$, we have
\begin{equation*}
\begin{split}
  &\lambda(\alpha,\beta,s)
  \\
  \ge&
  \frac{3}{2}\left(A^2-1\right)
  +\beta B\left(A^2+B^2+1\right)
  -A^2+1
  \\
  =&
  (\beta B+1/2)\left(\alpha^2+\alpha\cos{s}\right)
  +
  \frac{1}{4}\cos^2{s}(\beta B+1/2)
  +
  \beta B\left(B^2+1\right)-\frac{1}{2}.
\end{split}
\end{equation*}
We regard the last expression as a quadratic function of $\alpha$.
Since $\beta B+1/2=\beta^2+(\beta\sin{s}+1)/2>0$, it is strictly increasing for $\alpha>-(\cos{s})/2$.
Since $-(\cos{s})/2<3/2$, it is greater than the value at $\alpha=3/2$.
Therefore we have
\begin{equation*}
\begin{split}
  &\lambda(\alpha,\beta,s)
  \\
  >&
  (\beta B+1/2)\left(\frac{9}{4}+\frac{3}{2}\cos{s}\right)
  +
  \frac{1}{4}\cos^2{s}(\beta B+1/2)
  +
  \beta B\left(B^2+1\right)-\frac{1}{2}.
  \\
  =&
  \beta^4
  +
  \frac{3}{2}\beta^3\sin{s}
  +
  \frac{1}{2}\beta^2\left(7+3\cos{s}+\sin^2{s}\right)
  \\
  &+
  \frac{1}{4}\beta\sin{s}(3\cos{s}+7)
  +
  \frac{1}{8}(\cos{s}+1)(\cos{s}+5)
  \\
  &\text{(since $7+3\cos{s}+\sin^2{s}\ge4$)}
  \\
  \ge&
  \beta^4+\frac{3}{2}\beta^3\sin{s}
  +
  2\beta^2
  +
  \frac{1}{4}\beta\sin{s}(3\cos{s}+7)
  +
  \frac{1}{8}(\cos{s}+1)(\cos{s}+5).
\end{split}
\end{equation*}
We need to show that the last expression, which we denote by $\mu(\beta,s)$, is non-negative.
\par
When $0\le s\le\pi$, we see that $\mu(\beta,s)>0$ since each term is non-negative.
\par
So it suffices to prove that $\mu(\beta,s)\ge0$ for $\pi<s<2\pi$.
Put
\begin{align*}
  \mu_1(\beta,s)
  &:=
  \beta^4
  +
  \frac{5}{4}\beta^2
  +
  \frac{1}{4}\beta\sin{s}(3\cos{s}+7)
  +
  \frac{1}{8}(\cos{s}+1)(\cos{s}+5),
  \\
  \mu_2(\beta,s)
  &:=
  \beta^2\left(\frac{3}{2}\beta\sin{s}+\frac{3}{4}\right)
\end{align*}
so that $\mu(\beta,s)=\mu_1(\beta,s)+\mu_2(\beta,s)$.
\par
Since $-1\le\sin{s}<0$ for $\pi<s<2\pi$, and $0<\beta<1/2$, we have $-1/2<\beta\sin{s}<0$ and so $\mu_2(\beta,s)>0$.
\par
To prove $\mu_1(\beta,s)>0$, we put $s:=2\pi-\arccos{x}$ for $|x|<1$ and $0<\arccos{x}<\pi$ so that $\cos{s}=x$.
Since $\sin{s}<0$, we have $\sin{s}=-\sqrt{1-x^2}$.
Putting $\tmu_1(\beta,x):=\mu_1(\beta,2\pi-\arccos{x})$ we have
\begin{equation*}
  \tmu_1(\beta,x)
  =
  \beta^4
  +
  \frac{5}{4}\beta^2
  -
  \frac{1}{4}\beta(3x+7)\sqrt{1-x^2}
  +
  \frac{1}{8}(x+1)(x+5)
\end{equation*}
with $0<\beta<1/2$ and $|x|<1$.
Since we have
\begin{align*}
  \frac{\partial}{\partial\,\beta}\tmu_1(\beta,x)
  &=
  4\beta^3+\frac{5}{2}\beta-\frac{1}{4}(3x+7)\sqrt{1-x^2},
  \\
  \frac{\partial}{\partial\,x}\tmu_1(\beta,x)
  &=
  \frac{1}{4}
  \left(
    x+3-3\beta\sqrt{1-x^2}+\frac{\beta x(3x+7)}{\sqrt{1-x^2}}
  \right),
\end{align*}
critical points of $\tmu_1(\beta,x)$ for $(\beta,x)\in(0,1/2)\times(-1,1)$ should satisfy the following system of equations:
\begin{align}
  4\beta^3+\frac{5}{2}\beta-\frac{1}{4}(3x+7)\sqrt{1-x^2}
  &=0,
  \label{eq:d_mu_beta}
  \\
  x+3-3\beta\sqrt{1-x^2}+\frac{\beta x(3x+7)}{\sqrt{1-x^2}}
  &=0.
  \label{eq:d_mu_x}
\end{align}
From \eqref{eq:d_mu_x}, we have
\begin{equation*}
  \beta
  =
  \frac{-(x+3)\sqrt{1-x^2}}{(3x-1)(2x+3)},
\end{equation*}
since the left hand side of \eqref{eq:d_mu_x} becomes $10/3\ne0$ if $x=1/3$.
Substituting $\beta=\frac{-(x+3)\sqrt{1-x^2}}{(3x-1)(2x+3)}$ into \eqref{eq:d_mu_beta}, Mathematica says that it becomes
\begin{multline*}
  \frac{\sqrt{1-x^2}}{4(3x-1)^3(2x+3)^3}
  \\
  \times
  \left(648 x^7+3780 x^6+7310 x^5+4443 x^4-1494 x^3-1704 x^2+504 x+513\right)=0.
\end{multline*}
Mathematica also tells us that the real roots of $648 x^7+3780 x^6+7310 x^5+4443 x^4-1494 x^3-1704 x^2+504 x+513$ are $-2.44837\ldots,-1.66468\ldots$, and $-1.26834\ldots$.
So we conclude that $\tmu_1(\beta,x)$ has no critical point in $(0,1/2)\times(-1,1)$.
\par
Therefore $\tmu_1(\beta,x)$ is greater than its values in the boundary of the region $[0,1/2]\times[-1,1]$.
We calculate
\begin{align*}
  \tmu_1(0,x)
  &=
  \frac{1}{8}(x+1)(x+5),
  \\
  \tmu_1(1/2,x)
  &=
  \frac{3}{8}+\frac{1}{8}(x+1)(x+5)-\frac{1}{8}(3x+7)\sqrt{1-x^2},
  \\
  \tmu_1(\beta,-1)
  &=
  \beta^4+\frac{5}{4}\beta^2,
  \\
  \tmu_1(\beta,1)
  &=
  \beta^4+\frac{5}{4}\beta^2+\frac{3}{2}.
\end{align*}
Except for $\tmu_1(1/2,x)$, it is clear that they are positive for $0<\beta<1/2$ and $|x|<1$.
\par
To prove $\tmu_1(1/2,x)>0$, it suffices to show that $3+(x+1)(x+5)>(3x+7)\sqrt{1-x^2}$, that is, $\bigl(3+(x+1)(x+5)\bigr)^2-(3x+7)^2(1-x^2)>0$ if $|x|<1$.
Mathematica tells that $\bigl(3+(x+1)(x+5)\bigr)^2-(1-x^2)(3x+7)^2=10x^4+54x^3+92x^2+54x+15$ has no real root, so it is positive for any real $x$.
\par
This shows that $\tmu_1(\beta,x)>0$ when $0<\beta<1/2$ and $|x|<1$, completing the proof of the lemma.
\end{proof}
%%%%%%%%%%%%%%%%%%%%%%%%%%%%%%%%%%%%%%%%%%%%%%%%%%%%%%%%%%%%%%%%%%%%%%%%%%%%%%%
%%%  Case iii  %%%%%%%%%%%%%%%%%%%%%%%%%%%%%%%%%%%%%%%%%%%%%%%%%%%%%%%%%%%%%%%%
%%%%%%%%%%%%%%%%%%%%%%%%%%%%%%%%%%%%%%%%%%%%%%%%%%%%%%%%%%%%%%%%%%%%%%%%%%%%%%%
\subsection{The case where $0<\beta<1/2$ and $1/2<\alpha<3/2$}
It remains to prove the case where $\beta<1/2$ and $1/2<\alpha<3/2$ with $(\alpha-1)^2+\beta^2\ge1/4$.
We split this case into four subcases:
\begin{enumerate}
\item[(iii-1)] $0<\beta<1/2$, $1/2<\alpha<3/2$, and $B>0$,
\item[(iii-2)] $0<\beta<1/2$, $1/2<\alpha<3/2$, and $B=0$,
\item[(iii-3)] $0<\beta<1/2$, $1\le\alpha<3/2$, and $B<0$,
\item[(iii-4)] $0<\beta<1/2$, $1/2<\alpha<1$, and $B<0$.
\end{enumerate}
\begin{rem}\label{rem:caseiii}
Putting $s_0:=\arcsin(2\beta)$, we have
\begin{itemize}
\item $B>0$ if and only if $s\in(0,\pi+s_0)\cup(2\pi-s_0,2\pi)$,
\item $B=0$ if and only if $s=\pi+s_0$ or $2\pi-s_0$,
\item $B<0$ if and only if $s\in(\pi+s_0,2\pi-s_0)$
\end{itemize}
since $B=\beta+(\sin{s})/2$ and $0<s_0<\pi/2$.
\end{rem}
%%%%%%%%%%%%%%%%%%%%%%%%%%%%%%%%%%%%%%%%%%%%%%%%%%%%%%%%%%%%%%%%%%%%%%%%%%%%%%%
%%%%%%%%%%%%%%% \input{case_iii1} %%%%%%%%%%%%%%%%%%%%%%%%%%%%%%%%%%%%%%%%%%%%%
%%%%%%%%%%%%%%%%%%%%%%%%%%%%%%%%%%%%%%%%%%%%%%%%%%%%%%%%%%%%%%%%%%%%%%%%%%%%%%%
\begin{lem}[Subcase (iii-1)]
If $0<\beta<1/2$, $1/2<\alpha<3/2$, $(\alpha-1)^2+\beta^2\ge1/4$, and $B>0$, then we have $\lambda(\alpha,\beta,s)>0$.
Note that $(A,B)\ne(1,0)$ in this case.
\end{lem}
\begin{proof}
Since $\alpha>1/2$, we have $A=\alpha+(\cos{s})/2>0$.
Moreover since $\beta>0$ and $B>0$, from \eqref{eq:lambda} we have
\begin{equation*}
  \lambda(\alpha,\beta,s)
  >
  \alpha A(A^2-1)+\beta B(A^2+1)-A^2+1
  =
  (\alpha A-1)(A^2-1)+\beta B(A^2+1).
\end{equation*}
We denote the last expression by $\tlambda(\alpha,\beta,s)$, that is, we put
\begin{equation}\label{eq:tlambda}
  \tlambda(\alpha,\beta,s)
  :=
  (\alpha A-1)(A^2-1)+\beta B(A^2+1).
\end{equation}
\par
Since $B>0$ and $B$ is increasing with respect to $\beta$, we see that $\tlambda(\alpha,\beta,s)$ is increasing with respect to $\beta$.
Since $\beta\ge\sqrt{1/4-(\alpha-1)^2}$, we have
\begin{equation*}
\begin{split}
  &\tlambda(\alpha,\beta,s)
  \\
  >&
  (\alpha A-1)(A^2-1)+\beta B
  \\
  \ge&
  (\alpha A-1)(A^2-1)
  +
  \sqrt{1/4-(\alpha-1)^2}\left(\sqrt{1/4-(\alpha-1)^2}+(\sin{s})/2\right).
\end{split}
\end{equation*}
\par
First assume that $0<s\le\pi$ so that $\sin{s}\ge0$.
Note that $s\in(0,\pi+s_0)\cup(2\pi-s_0,2\pi)$ from Remark~\ref{rem:caseiii}.
\par
Since $\sin{s}\ge0$, we have
\begin{equation*}
  \nu_1(\alpha,s)
  :=
  (\alpha A-1)(A^2-1)+1/4-(\alpha-1)^2.
\end{equation*}
We will find critical points of $\nu_1(\alpha,s)$ in the compact region $\{(\alpha,s)\in\R^2\mid1/2\le\alpha\le3/2,0\le s\le\pi\}$.
Since $\frac{\partial\,A}{\partial\,\alpha}(\alpha,s)=1$ and $\frac{\partial\,A}{\partial\,s}(\alpha,s)=-(\sin{s})/2$, we have
\begin{align*}
  \frac{\partial\,\nu_1}{\partial\,\alpha}(\alpha,s)
  &=
  A^3+3\alpha A^2-3A-3\alpha+2,
  \\
  \frac{\partial\,\nu_1}{\partial\,s}(\alpha,s)
  &=
  -\frac{1}{2}\left(3\alpha A^2-2A-\alpha\right)\sin{s}.
\end{align*}
So a critical point $(\alpha,s)$ satisfies the following system of equations:
\begin{align}
  A^3+3\alpha A^2-3A-3\alpha+2&=0,
  \label{eq:d_lambda_alpha}
  \\
  \left(3\alpha A^2-2A-\alpha\right)\sin{s}&=0.
  \label{eq:d_lambda_s}
\end{align}
From \eqref{eq:d_lambda_s}, we have $\sin{s}=0$ or
\begin{equation*}
  3\alpha A^2-2A-\alpha
  =
  \frac{3}{4}\alpha\cos^2{s}+(3\alpha^2-1)\cos{s}+3\alpha(\alpha^2-1)
  =0.
\end{equation*}
Solving the equation above for $\cos{s}$, we obtain
\begin{equation*}
  \cos{s}
  =
  \frac{2(1-3\alpha^2\pm\sqrt{1+3\alpha^2})}{3\alpha}.
\end{equation*}
Since $\frac{2(1-3\alpha^2-\sqrt{1+3\alpha^2})}{3\alpha}$ is clearly decreasing and when $\alpha=1/2$ its value equals $\frac{1-2\sqrt{7}}{3}=-1.435\ldots<-1$, we see that $\cos{s}=\frac{2(1-3\alpha^2+\sqrt{1+3\alpha^2})}{3\alpha}$.
Therefore we conclude that a solution to the system of equations \eqref{eq:d_lambda_alpha} and \eqref{eq:d_lambda_s} satisfies $\sin{s}=0$ or $\cos{s}=\frac{2(1-3\alpha^2+\sqrt{1+3\alpha^2})}{3\alpha}$.
\par
If $\sin{s}=0$, then $s=0$ or $\pi$.
If $s=0$, then $A=\alpha+1/2$ and so from \eqref{eq:d_lambda_alpha}, we obtain
\begin{equation*}
  4\alpha^3+\frac{9}{2}\alpha^2-\frac{9}{2}\alpha+\frac{5}{8}=0.
\end{equation*}
The left hand side is increasing for $\alpha\ge1/2$ since its derivative is positive, and its value at $\alpha=1/2$ equals $0$.
Therefore the only solution is $\alpha=1/2$.
If $s=\pi$, then $A=\alpha-1/2$ and from \eqref{eq:d_lambda_alpha}, we have
\begin{equation*}
  \frac{1}{8}(2\alpha-3)(16\alpha^2+6\alpha-9)=0.
\end{equation*}
Since $1/2\le\alpha\le3/2$, we have $\alpha=3/2$ or $\frac{3}{16}(\sqrt{17}-1)=0.585582\ldots$.
\par
If $\cos{s}=\frac{2(1-3\alpha^2+\sqrt{1+3\alpha^2})}{3\alpha}$, from \eqref{eq:d_lambda_alpha} we obtain
\begin{equation*}
  \frac{-2}{27\alpha^3}
  \left(
    27\alpha^4-27\alpha^3-2
    +
    (3\alpha^2-2)\sqrt{3\alpha^2+1}
  \right)
  =0.
\end{equation*}
By Mathematica, the real solutions to the equation above is $1$ and $-0.456897\ldots$.
So we have $\alpha=1$ and $\cos{s}=0$, that is, $(\alpha,s)=(1,\pi/2)$.
\par
Therefore, we conclude that the critical points of $\nu_1(\alpha,s)$ for $1/2\le\alpha\le3/2$ and $0\le s\le\pi$ are the following.
\begin{equation*}
  \left(\frac{1}{2},0\right),\quad
  \left(\frac{3}{2},\pi\right),\quad
  \left(\frac{3(\sqrt{17}-1)}{16},\pi\right),\quad
  \left(1,\frac{\pi}{2}\right).
\end{equation*}
Since the second partial derivatives are
\begin{align*}
  \frac{\partial^2\,\nu_1}{\partial\,\alpha^2}(\alpha,s)
  &=
  6\left(A^2+\alpha A-1\right),
  \\
  \frac{\partial^2\,\nu_1}{\partial\,s^2}(\alpha,s)
  &=
  \frac{\cos{s}}{2}(-3\alpha A^2+2A+\alpha)
  +
  \frac{\sin^2{s}}{2}(3\alpha A-1),
  \\
  \frac{\partial^2\,\nu_1}{\partial\,\alpha\,\partial\,s}(\alpha,s)
  &=
  -\frac{3\sin{s}}{2}
  \left(A^2+2\alpha A-1\right),
\end{align*}
the corresponding Hessians become
\begin{equation*}
  \begin{pmatrix}3&0\\0&1/2\end{pmatrix},\quad
  \begin{pmatrix}9&0\\0&1/2\end{pmatrix},\quad
  \begin{pmatrix}
    \frac{9(17-9\sqrt{17})}{32}&0\\0&\frac{189\sqrt{17}-1541}{2048}
  \end{pmatrix},\quad
  \begin{pmatrix}6&-3\\-3&1\end{pmatrix},
\end{equation*}
respectively.
Therefore these critical points are a local minimum, a local minimum, a local maximum, and a saddle point, respectively.
So the points $(1/2,0)$ and $(3/2,\pi)$ are candidates of the minimum points with value $\nu_1(1/2,0)=\nu_1(3/2,\pi)=0$.
\par
For the boundary, we have
\begin{align*}
  \nu_1(1/2,s)
  &=
  \frac{1}{16}(\cos^2{s}-9)(\cos{s}-1)\ge0,
  \\
  \nu_1(3/2,s)
  &=
  \frac{1}{16}(3\cos{s}+5)(\cos{s}+5)(\cos{s}+1)\ge0,
  \\
  \nu_1(\alpha,0)
  &=
  \frac{\alpha}{8}(2\alpha-1)^2(2\alpha+5)\ge0,
  \\
  \nu_1(\alpha,\pi)
  &=
  \frac{\alpha}{8}(2\alpha+3)(2\alpha-3)^2\ge0.
\end{align*}
Therefore we conclude that $\nu_1(\alpha,s)$ takes its minimum $0$ at $(1/2,0)$ and $(3/2,\pi)$, proving that $\lambda(\alpha,\beta,s)>0$ when $1/2<\alpha<3/2$ and $0<s\le\pi$.
\par
Next, we assume $s\in(\pi,\pi+s_0)\cup(2\pi-s_0,2\pi)$, that is, $\sin{s}<0$.
Since $\sin{s}<0$, we have $\sin{s}=-\sqrt{1-\cos^2{s}}$.
So we have
\begin{equation*}
  \tlambda(\alpha,\beta,s)
  >
  \nu_2(\alpha,x),
\end{equation*}
where we put
\begin{multline*}
  \nu_2(\alpha,x)
  \\
  :=
  (\alpha A-1)(A^2-1)
  +
  \frac{1}{4}(A^2+1)\sqrt{1-4(\alpha-1)^2}\left(\sqrt{1-4(\alpha-1)^2}-\sqrt{1-x^2}\right)
\end{multline*}
with $x:=\cos{s}$ since $\beta\ge\sqrt{1/4-(\alpha-1)^2}$.
Note that $|x|<1$ because $\pi<s<2\pi$.
Since $\sqrt{1-4(\alpha-1)^2}+\sqrt{1-x^2}>0$, we have
\begin{equation}\label{eq:square_root}
\begin{split}
  \sqrt{1-4(\alpha-1)^2}-\sqrt{1-x^2}
  =&
  \frac{\left(1-4(\alpha-1)^2\right)-(1-x^2)}
       {\sqrt{1-4(\alpha-1)^2}+\sqrt{1-x^2}}
  \\
  =&
  \frac{(x-2\alpha+2)(x+2\alpha-2)}
       {\sqrt{1-4(\alpha-1)^2}+\sqrt{1-x^2}}.
\end{split}
\end{equation}
Since $A=\alpha+x/2$, we have
\begin{equation*}
\begin{split}
  &\nu_2(\alpha,x)
  \\
  =&
  (x/2+\alpha-1)
  \left(
    (\alpha A-1)(A+1)
    +
    \frac{(A^2+1)\sqrt{1-4(\alpha-1)^2}\bigl(x-2\alpha+2\bigr)}
         {2\left(\sqrt{1-4(\alpha-1)^2}+\sqrt{1-x^2}\right)}
  \right)
  \\
  =&
  \frac{x/2+\alpha-1}{2\left(\sqrt{1-4(\alpha-1)^2}+\sqrt{1-x^2}\right)}
  \nu_3(\alpha,x),
\end{split}
\end{equation*}
where we put
\begin{equation*}
\begin{split}
  \nu_3(\alpha,x)
  :=&
  2(\alpha A-1)(A+1)\left(\sqrt{1-4(\alpha-1)^2}+\sqrt{1-x^2}\right)
  \\
  &+
  (A^2+1)\sqrt{1-4(\alpha-1)^2}\bigl(x-2\alpha+2\bigr).
\end{split}
\end{equation*}
With help from Mathematica, we have
\begin{equation*}
\begin{split}
  &\nu_3(\alpha,x)
  \\
  =&
  (x/2+\alpha-1)
  \left(p(\alpha,x)\sqrt{1-4(\alpha-1)^2}+q(\alpha,x)\sqrt{1-x^2}\right)
  \\
  &+
  4(\alpha-1)\left(\sqrt{1-x^2}-\sqrt{1-4(\alpha-1)^2}\right)
\end{split}
\end{equation*}
with
\begin{align}
  p(\alpha,x)
  &:=
  \frac{1}{2}x^2+(\alpha+2)x+4(\alpha+1)
  =
  \alpha(x+4)+\frac{1}{2}(x+2)^2+2
  >
  4,
  \label{eq:p_positive}
  \\
  q(\alpha,x)
  &:=
  2\alpha^2+(x+4)\alpha-2
  =
  2(\alpha+1)^2+\alpha x-4
  >0.
  \label{eq:q_positive}
\end{align}
From \eqref{eq:square_root}, we have
\begin{equation*}
  \nu_3(\alpha,x)
  =
  \frac{(x/2+\alpha-1)}{\sqrt{1-x^2}+\sqrt{1-4(\alpha-1)^2}}\nu_4(\alpha,x),
\end{equation*}
where we put
\begin{equation*}
\begin{split}
  &\nu_4(\alpha,x)
  \\
  :=&
  \left(p(\alpha,x)\sqrt{1-4(\alpha-1)^2}+q(\alpha,x)\sqrt{1-x^2}\right)
  \left(\sqrt{1-4(\alpha-1)^2}+\sqrt{1-x^2}\right)
  \\
  &+
  8(\alpha-1)(2\alpha-2-x).
\end{split}
\end{equation*}
So we finally have
\begin{equation*}
\begin{split}
  &\nu_2(\alpha,x)
  =
  \frac{(x/2+\alpha-1)^2}{2\left(\sqrt{1-4(\alpha-1)^2}+\sqrt{1-x^2}\right)^2}
  \nu_4(\alpha,x).
\end{split}
\end{equation*}
\par
Note that $\left(\sqrt{1-4(\alpha-1)^2}+\sqrt{1-x^2}\right)^2$ is positive for $|x|<1$ and $1/2<\alpha<3/2$.
Note also that $(x/2+\alpha-1)^2\ge0$ and the equality holds only when $\alpha+x/2=1$, that is, $A=1$.
However if $A=1$, then $\tlambda(\alpha,\beta,s)=2\beta B>0$ from \eqref{eq:tlambda}.
So it is enough to show that $\nu_4(\alpha,x)>0$ for $1/2<\alpha<3/2$ and $|x|<1$.
\par
First, consider the case where $x\ge2\alpha-2$.
If $\alpha\le1$, then since we have $(\alpha-1)(2\alpha-2-x)\ge0$, we see that $\nu_4(\alpha,x)>0$.
If $\alpha>1$, then we have $q(\alpha,x)>8-1-4>3$ from \eqref{eq:q_positive}.
Since we also have $p(\alpha,x)>3$ from \eqref{eq:p_positive}, we obtain
\begin{equation*}
\begin{split}
  \nu_4(\alpha,x)
  >&
  3\left(\sqrt{1-4(\alpha-1)^2}+\sqrt{1-x^2}\right)^2
  +
  8(\alpha-1)(2\alpha-2-x)
  \\
  >&
  3\left(2-4(\alpha-1)^2-x^2\right)+8(\alpha-1)(2\alpha-2-x)
  \\
  =&
  4(\alpha-x-1)^2-7x^2+6.
\end{split}
\end{equation*}
Let $\nu_5(\alpha,x)$ be the last expression.
Sine $x\ge2\alpha-2$ and $\alpha>1$, we have $\alpha-x-1\le-\alpha+1<0$.
Therefore, $\nu_5(\alpha,x)$ is decreasing with respect to $\alpha$, which implies that $\nu_5(\alpha,x)\ge\nu_5(x/2+1,x)=6-6x^2>0$ since $\alpha\ge x/2+1$.
\par
Next, consider the case where $x<2\alpha-2$.
If $\alpha\ge1$, then since we have $(\alpha-1)(2\alpha-2-x)\ge0$ as before, we see that $\nu_4(\alpha,x)>0$.
\par
So it remains to prove the case where $\alpha<1$ and $x<2\alpha-2$.
Since $p(\alpha,x)>4$ from \eqref{eq:p_positive} again, we obtain
\begin{equation*}
\begin{split}
  &\nu_4(\alpha,x)
  \\
  >&
  \left(4\sqrt{1-4(\alpha-1)^2}+q(\alpha,x)\sqrt{1-x^2}\right)
  \left(\sqrt{1-4(\alpha-1)^2}+\sqrt{1-x^2}\right)
  \\
  &+
  8(\alpha-1)(2\alpha-2-x)
  \\
  =&
  r(\alpha,x)+\bigl(q(\alpha,x)+4\bigr)\sqrt{1-x^2}\sqrt{1-4(\alpha-1)^2}
  \\
  >&
  r(\alpha,x)
\end{split}
\end{equation*}
with
\begin{equation*}
\begin{split}
  r(\alpha,x)
  :=&
  4(1-4(\alpha-1)^2)+q(\alpha,x)(1-x^2)+8(\alpha-1)(2\alpha-2-x)
  \\
  =&
  2(1-x^2)\alpha^2+(-x^3-4x^2-7x+4)\alpha+2x^2+8x+2.
\end{split}
\end{equation*}
Since $\alpha>1/2$ and $-1<x<2(\alpha-1)<0$, we have $\frac{\partial}{\partial\,\alpha}r(\alpha,x)=4(1-x^2)\alpha-x^3-4x^2-7x+4>-4x^2+4>0$.
Thus we have
\begin{equation*}
  r(\alpha,x)
  >
  r(1/2,x)
  =
  \frac{1}{2}(9-x^2)(1+x)
  >0.
\end{equation*}
\par
Therefore the proof of the lemma is complete.
\end{proof}
%%%%%%%%%%%%%%%%%%%%%%%%%%%%%%%%%%%%%%%%%%%%%%%%%%%%%%%%%%%%%%%%%%%%%%%%%%%%%%%
%%%%%%%%%%%%%%% \input{case_iii2} %%%%%%%%%%%%%%%%%%%%%%%%%%%%%%%%%%%%%%%%%%%%%
%%%%%%%%%%%%%%%%%%%%%%%%%%%%%%%%%%%%%%%%%%%%%%%%%%%%%%%%%%%%%%%%%%%%%%%%%%%%%%%
\begin{lem}[Subcase (iii-2)]
If $0<\beta<1/2$, $1/2<\alpha<3/2$, $(\alpha-1)^2+\beta^2\ge1/4$, $B=0$, and $A\ne1$, then we have $\lambda(\alpha,\beta,s)>0$.
\end{lem}
\begin{proof}
Note that in this case we have
\begin{equation*}
  \lambda(\alpha,\beta,s)
  =
  (\alpha A-1)(A^2-1).
\end{equation*}
We also note that since $(\alpha-1)^2+\beta^2\ge1/4$ and $0<\beta<1/2$, we see that $\alpha\ne1$.
\par
From the inequality $(\alpha-1)^2+\beta^2\ge1/4$, we have
\begin{equation}\label{eq:alpha-1}
  -|\alpha-1|\le\pm\frac{1}{2}\sqrt{1-4\beta^2}\le|\alpha-1|.
\end{equation}
\par
From Remark~\ref{rem:caseiii}, we have $s=\pi+s_0$ or $2\pi-s_0$.
Since we put $s_0:=\arcsin(2\beta)$, we have $\cos(\pi+s_0)=-\sqrt{1-4\beta^2}$, and $\cos(2\pi-s_0)=\sqrt{1-4\beta^2}$.
\par
First, we consider the case where $s=\pi+s_0$.
In this case we have $A=\alpha-\frac{1}{2}\sqrt{1-4\beta^2}>0$ since $\alpha>1/2$.
\par
If $\alpha>1$, then $A\ge\alpha+(1-\alpha)=1$ from \eqref{eq:alpha-1}.
Since we assume that $A\ne1$, we conclude that $A^2-1>0$.
Moreover, we have $\alpha A-1\ge\alpha-1>0$ since $A>0$.
Thus we see that $\lambda(\alpha,\beta,s)>0$.
\par
If $\alpha<1$, then from \eqref{eq:alpha-1} we have $A\le\alpha+(1-\alpha)=1$, and so $A^2-1<0$ because $A\ne1$.
We also have $\alpha A-1<A-1<0$ since $A>0$.
Thus we conclude that $\lambda(\alpha,\beta,s)>0$ as well.
\par
Next, we consider the case where $s=2\pi-s_0$.
In this case we have $A=\alpha+\frac{1}{2}\sqrt{1-4\beta^2}>0$.
\par
If $\alpha>1$, then $A>1$ in a reason similar to above, and so we have $A^2-1>0$.
We also have $\alpha A-1>A-1>0$.
Thus we conclude that $\lambda(\alpha,\beta,s)>0$.
\par
If $\alpha<1$, then $A<1$ as above, and so we conclude that $A^2-1<0$.
We also have $\alpha A-1<A-1<0$, and so $\lambda(\alpha,\beta,s)<0$ in this case
\par
Now the proof is complete.
\end{proof}
%%%%%%%%%%%%%%%%%%%%%%%%%%%%%%%%%%%%%%%%%%%%%%%%%%%%%%%%%%%%%%%%%%%%%%%%%%%%%%%
%%%%%%%%%%%%%%% \input{case_iii3} %%%%%%%%%%%%%%%%%%%%%%%%%%%%%%%%%%%%%%%%%%%%%
%%%%%%%%%%%%%%%%%%%%%%%%%%%%%%%%%%%%%%%%%%%%%%%%%%%%%%%%%%%%%%%%%%%%%%%%%%%%%%%
Now, we consider Subcases (iii-3) and (iii-4).
Since $B<0$ for both subcases, we have
\begin{equation*}
  \lambda(\alpha,\beta,s)
  >
  \hlambda(\alpha,\beta,s)
  :=
  (\alpha A-1)(A^2-1)+\beta B(A^2+B^2+1)
\end{equation*}
from \eqref{eq:lambda}.
So it is enough to show that $\hlambda(\alpha,\beta,s)\ge0$ in Subcases (iii-3) and (iii-4).
\par
The partial derivatives of $\hlambda(\alpha,\beta,s)$ with respect to $\alpha$ become
\begin{align}
  \frac{\partial}{\partial\,\alpha}\hlambda(\alpha,\beta,s)
  =&
  3\alpha A^2+A^3-3A-\alpha+2\beta AB,
  \label{eq:der_lambda1_alpha}
  \\
  \frac{\partial^2}{\partial\,\alpha^2}\hlambda(\alpha,\beta,s)
  =&
  6A^2+6\alpha A-4+2\beta B.
  \label{eq:der_lambda1_alpha2}
\end{align}
\par
We start with Subcase (iii-3).
\begin{lem}[Subcase (iii-3)]
If $0<\beta<1/2$, $1\le\alpha<3/2$, $(\alpha-1)^2+\beta^2\ge1/4$, and $B<0$, then we have $\hlambda(\alpha,\beta,s)>0$.
Note that $(A,B)\ne(1,0)$ in this case.
\end{lem}
\begin{proof}
In the following calculations we often use Mathematica.
\par
We will show that $\frac{\partial}{\partial\,\alpha}\hlambda(\alpha,\beta,s)>0$.
\par
First, note that $\alpha\ge1+\sqrt{1/4-\beta^2}$ from the inequality $(\alpha-1)^2+\beta^2\ge1/4$ since $\alpha\ge1$.
\par
Since $A=\alpha+\cos{s}/2\ge1/2$ from $\alpha\ge1$, and $B=\beta+\sin{s}/2\ge\beta-1/2$, we have from \eqref{eq:der_lambda1_alpha2}
\begin{equation*}
  \frac{\partial^2}{\partial\,\alpha^2}\hlambda(\alpha,\beta,s)
  \ge
  3/2+3-4+2\beta(\beta-1/2)
  =
  2(\beta-1/4)^2+3/8
  >0.
\end{equation*}
\par
So $\frac{\partial}{\partial\,\alpha}\hlambda(\alpha,\beta,s)$ is increasing with respect to $\alpha$.
Since $\alpha\ge1+\sqrt{1/4-\beta^2}$, we have
\begin{equation}\label{eq:p1_q1}
\begin{split}
  \frac{\partial}{\partial\,\alpha}\hlambda(\alpha,\beta,s)
  \ge&
  \frac{\partial}{\partial\,\alpha}\hlambda\left(1+\sqrt{1/4-\beta^2},\beta,s\right)
  \\
  =&
  \frac{1}{8}p_1(\beta,s)+\frac{1}{8}q_1(\beta,s)\sqrt{1/4-\beta^2},
\end{split}
\end{equation}
where we put
\begin{align*}
  p_1(\beta,s)
  :=&
  -4(7\cos{s}+20)\beta^2+4\sin{s}(\cos{s}+2)\beta
  \\
  &+\cos^2{s}(\cos{s}+12)+33\cos{s}+24,
  \\
  q_1(\beta,s)
  :=&
  -16\beta^2+8\beta\sin{s}+12\bigl(\cos^2{s}+6\cos{s}+6\bigr).
\end{align*}
\par
Both $p_1(\beta,s)$ and $q_1(\beta,s)$ are quadratic functions of $\beta$, with axes of symmetries $\beta=\frac{\sin{s}(\cos{s}+2)}{2(7\cos{s}+20)}$ and $\beta=(\sin{s})/4$, respectively.
Since $\sin{s}=2(B-\beta)<0$, we conclude that $p_1(\beta,s)$ and $q_1(\beta,s)$ are decreasing with respect to $\beta>0$.
Moreover, since we have $q_1(1/2,s)=4\left(3(\cos{s}+3)^2-10+\sin{s}\right)\ge4$, the function $q_1(\beta,s)$ is positive for $0<s<1/2$.
Since the function $\sqrt{1/4-\beta^2}$ is also positive and decreasing for $0<\beta<1/2$, we conclude that $p_1(\beta,s)+q_1(\beta,s)\sqrt{1/4-\beta^2}$ is decreasing with respect to $\beta$.
\par
Since $\beta=B-\sin{s}/2<-(\sin{s})/2$, we have
\begin{equation*}
\begin{split}
  &p_1(\beta,s)+q_1(\beta,s)\sqrt{1/4-\beta^2}
  \\
  >&
  p_1\bigl(-(\sin{s})/2,s\bigr)+\frac{1}{2}q_1(-(\sin{s})/2,s)|\cos{s}|
  \\
  =&
  2\cos{s}(5\cos^2{s}+18\cos{s}+12)+2(5\cos^2{s}+18\cos{s}+16)|\cos{s}|.
  \\
  =&
  \begin{cases}
    -8\cos{s}
    &\quad\text{(if $\cos{s}\le0$)},
    \\
    4\cos{s}(5\cos^2{s}+18\cos{s}+14)
    &\quad\text{(if $\cos{s}>0$)},
  \end{cases}
  \\
  \ge&0.
\end{split}
\end{equation*}
\par
From \eqref{eq:p1_q1} it follows that $\frac{\partial}{\partial\,\alpha}\hlambda(\alpha,\beta,s)>0$.
\par
Now, since $\alpha\ge1+\sqrt{1/4-\beta^2}$, we have
\begin{equation*}
  \hlambda(\alpha,\beta,s)
  \ge
  \clambda(\beta,s)
  :=
  \hlambda(1+\sqrt{1/4-\beta^2},\beta,s)
  =
  p_2(\beta,s)+q_2(\beta,s)\sqrt{1/4-\beta^2},
\end{equation*}
where we put
\begin{align*}
  p_2(\beta,s)
  :=&
  \beta^4+\beta^3\sin{s}-\frac{1}{4}\left(5\cos^2{s}+14\cos{s}+6\right)\beta^2
  +
  \frac{\sin{s}}{4}(2\cos{s}+5)\beta
  \\
  &+
  \frac{1}{16}\left(2\cos^3{s}+11\cos^2{s}+18\cos{s}+17\right)
  \\
  q_2(\beta,s)
  :=&
  -\frac{1}{2}(\cos{s}+4)\beta^2+\frac{\sin{s}}{2}(\cos{s}+2)\beta
  \\
  &+
  \frac{1}{8}\left(\cos^3{s}+12\cos^2{s}+27\cos{s}+8\right).
\end{align*}
\par
From Remark~\ref{rem:caseiii}, we knot that $\pi+s_0<s<2\pi-s_0$.
So we will show that $\clambda(\beta,s)>0$ for $0<\beta<1/2$ and $\pi+s_0<s<2\pi-s_0$.
\par
We have
\begin{equation}\label{eq:q2_der_beta}
  \frac{\partial}{\partial\,\beta}q_2(\beta,s)
  =
  -(\cos{s}+4)\beta+\frac{\sin{s}}{2}(\cos{s}+2)<0,
\end{equation}
since $\sin{s}<0$.
Since $\sin{s}=2(B-\beta)<-2\beta$, we also have
\begin{equation}\label{eq:p2_der_beta}
\begin{split}
  \frac{\partial}{\partial\,\beta}p_2(\beta,s)
  =&
  4\beta^3+3\beta^2\sin{s}-\frac{1}{2}\left(5\cos^2{s}+14\cos{s}+6\right)\beta
  +
  \frac{\sin{s}}{4}(2\cos{s}+5),
  \\
  <&
  -\frac{\beta}{2}
  \left(5\cos^2{s}+16\cos{s}+11\right)
  =
  -\frac{\beta}{2}(5\cos{s}+11)(\cos{s}+1)<0.
\end{split}
\end{equation}
From \eqref{eq:q2_der_beta} and \eqref{eq:p2_der_beta}, we conclude that both $p_2(\beta,s)$ and $q_2(\beta,s)$ are decreasing with respect to $\beta$.
\par
Now, we consider the following two cases: (I) $3\pi/2\le s<2\pi-s_0$ and (II) $\pi+s_0<s<3\pi/2$.
\begin{enumerate}
\item[(I)]
The case where $3\pi/2\le s<2\pi-s_0$.
Note that $\cos{s}\ge0$ in this case.
\par
Since $q_2(\beta,s)$ is decreasing with respect to $\beta$, and $\cos{s}\ge0$, we have
\begin{equation*}
\begin{split}
  q_2(\beta,s)
  >&
  q_2(1/2,s)
  \\
  =&
  \frac{1}{8}
  \left(
    \cos^3{s}+12\cos^2{s}+26\cos{s}+2\cos{s}\sin{s}+4\sin{s}+4
  \right)
  \\
  \ge&
  \frac{1}{8}
  \left(
    \cos^3{s}+12\cos^2{s}+2\cos{s}(13+\sin{s})
  \right)
  \ge0.
\end{split}
\end{equation*}
Since $\sqrt{1/4-\beta^2}$ is also positive and decreasing, we conclude that $q_2(\beta,s)\sqrt{1/4-\beta^2}$ is decreasing with respect to $\beta$.
Therefore $\clambda(\beta,s)$ is decreasing with respect to $\beta$ in this case.
Since $\beta<-(\sin{s})/2$ and $\cos{s}\ge0$, we have
\begin{equation*}
\begin{split}
  \clambda(\beta,s)
  >&
  \clambda\bigl(-(\sin{s})/2,s\bigr)
  \\
  =&
  p_2\bigl(-(\sin{s})/2,s\bigr)+q_2\bigl(-(\sin{s})/2,s\bigr)(\cos{s})/2
  \\
  =&
  \frac{\cos^2{s}}{2}(\cos{s}+2)(\cos{s}+3)
  \ge0.
\end{split}
\end{equation*}
Therefore $\clambda(\beta,s)>0$ if $3\pi/2\le s<2\pi-s_0$.
\item[(II)]
The case where $\pi+s_0<s<3\pi/2$.
Note that $\cos{s}<0$ in this case.
\par
Mathematica calculates
\begin{equation*}
  \clambda(\beta,s)\times\left(\frac{\sqrt{1-4\beta^2}-\cos{s}}{B}\right)^2
  =
  \tp_2(\beta,s)+\tq_2(\beta,2)\sqrt{1-4\beta^2}
\end{equation*}
with
\begin{align*}
  \tp_2(\beta,s)
  :=&
  -4\beta^4+\beta^2(3\cos^2{s}+6\cos{s}+8)
  -4\beta\sin{s}(\cos{s}+3)
  \\
  &-\frac{1}{4}(2\cos^3{s}+9\cos^2{s}-2\cos{s}-17),
  \\
  \tq_2(\beta,s)
  :=&
  -\beta^2(\cos{s}-4)-2\beta\sin{s}(\cos{s}+3)
  -\frac{1}{4}(\cos^3{s}+8\cos^2{s}+7\cos{s}-8).
\end{align*}
Since both $\sin{s}(\cos{s}+3)$ and $2\cos^3{s}+9\cos^2{s}-2\cos{s}-17$ are clearly negative, we have
\begin{equation*}
  \tp_2(\beta,s)
  >
  -4\beta^4+\beta^2(3\cos^2{s}+6\cos{s}+8)
  =
  \beta^2\bigl(3(\cos{s}+1)^2+5-4\beta^2\bigr)
  >0.
\end{equation*}
Since we have $\cos{s}-4<0$ and $\sin{s}(\cos{s}+3)<0$, we also have
\begin{equation*}
  \tq_2(\beta,s)
  >
  -\frac{1}{4}(\cos^3{s}+8\cos^2{s}+7\cos{s}-8)
  >
  -\frac{1}{4}(8\cos^2{s}-8)
  >0,
\end{equation*}
where the second inequality follows since $\cos{s}<0$.
\par
Since $\sqrt{1-4\beta^2}-\cos{s}>0$, this shows that $\clambda(\beta,s)>0$ in this case.
\end{enumerate}
\par
Therefore we finally conclude that $\clambda(\beta,s)>0$ when $\pi+s_0<s<3\pi/2$, which completes the proof of the lemma.
\end{proof}
%%%%%%%%%%%%%%%%%%%%%%%%%%%%%%%%%%%%%%%%%%%%%%%%%%%%%%%%%%%%%%%%%%%%%%%%%%%%%%%
%%%%%%%%%%%%%%% \input{case_iii4} %%%%%%%%%%%%%%%%%%%%%%%%%%%%%%%%%%%%%%%%%%%%%
%%%%%%%%%%%%%%%%%%%%%%%%%%%%%%%%%%%%%%%%%%%%%%%%%%%%%%%%%%%%%%%%%%%%%%%%%%%%%%%
\begin{lem}[Subcase iii-4]
If $0<\beta<1/2$, $1/2<\alpha<1$, $(\alpha-1)^2+\beta^2\ge1/4$, and $B<0$, then we have $\hlambda(\alpha,\beta,s)>0$.
Note that $(A,B)\ne(1,0)$ in this case.
\end{lem}
\begin{proof}
First of all, note that $\alpha<1-\sqrt{1/4-\beta^2}$ since $\alpha<1$ and $0<\beta<1/2$.
Note also that $\pi+s_0<s<2\pi-s_0$ from Remark~\ref{rem:caseiii}.
\par
We will show that $\frac{\partial\,\hlambda}{\partial\,\alpha}(\alpha,\beta,s)<0$.
\par
Since we have
\begin{equation*}
  \frac{\partial^3\,\hlambda}{\partial\,\alpha^3}(\alpha,\beta,s)
  =
  12A+6\alpha+6A
  =
  24\alpha+9\cos{s}
  >0
\end{equation*}
from \eqref{eq:der_lambda1_alpha2}, the function $\frac{\partial\,\hlambda}{\partial\,\alpha}(\alpha,\beta,s)$ of $\alpha$ is convex down in the closed interval $\left[1/2,1-\sqrt{1/4-\beta^2}\right]$.
It means that $\frac{\partial\,\hlambda}{\partial\,\alpha}(\alpha,\beta,s)$ is less than or equal to $\max\left\{\frac{\partial\,\hlambda}{\partial\,\alpha}(1/2,\beta,s),\frac{\partial\,\hlambda}{\partial\,\alpha}(1-\sqrt{1/4-\beta^2},\beta,s)\right\}$.
So, it suffices to show that both $\frac{\partial\,\hlambda}{\partial\,\alpha}(1/2,\beta,s)$ and $\frac{\partial\,\hlambda}{\partial\,\alpha}(1-\sqrt{1/4-\beta^2},\beta,s)$ are negative.
\par
From \eqref{eq:der_lambda1_alpha}, we have
\begin{equation*}
\begin{split}
  &\frac{\partial\,\hlambda}{\partial\,\alpha}(1/2,\beta,s)
  \\
  =&
  (\cos{s}+1)\left(\beta+\frac{\sin{s}}{4}\right)^2
  +\frac{1}{16}\left(3\cos^3{s}+13\cos^2{s}-7\cos{s}-25\right),
\end{split}
\end{equation*}
which is convex down with respect to $\beta$.
Since $0<\beta<1/2$, it is less than
\begin{equation*}
\begin{split}
  &\max
  \left\{
    \frac{\partial\,\hlambda}{\partial\,\alpha}(1/2,0,s),
    \frac{\partial\,\hlambda}{\partial\,\alpha}(1/2,1/2,s)
  \right\}
  \\
  =&
  \frac{1}{8}
  \max
  \{\cos^3{s}+6\cos^2{s}-3\cos{s}-12,
  \\&\phantom{\frac{1}{8}\max}
  \cos^3{s}+6\cos^2{s}-\cos{s}+\sin(2s)+2\sin{s}-10\}
  \\
  <&0,
\end{split}
\end{equation*}
where the inequality follows since $\cos^3{s}+6\cos^2{s}-3\cos{s}-12$ is clearly negative, and $\cos^3{s}+6\cos^2{s}-\cos{s}+\sin(2s)+2\sin{s}-10<1+6+1+1+2\sin{s}-10<0$ from $\sin{s}<0$ (see Remark~\ref{rem:caseiii}).
\par
Next, we show that $\frac{\partial\,\hlambda}{\partial\,\alpha}(1-\sqrt{1/4-\beta^2},\beta,s)<0$.
We write
\begin{equation*}
\begin{split}
  \frac{\partial\,\hlambda}{\partial\,\alpha}(1-\sqrt{1/4-\beta^2},\beta,s)
  =
  p_3(\beta,s)+q_3(\beta,s)\sqrt{1/4-\beta^2}
\end{split}
\end{equation*}
with
\begin{align*}
  p_3(\beta,s)
  :=&
  -\frac{1}{2}(7\cos{s}+20)\beta^2
  +\frac{1}{2}\sin{s}(\cos{s}+2)\beta
  \\
  &+\frac{1}{8}(\cos^3{s}+12\cos^2{s}+33\cos{s}+24),
  \\
  q_3(\beta,s)
  :=&
  2\beta^2-\beta\sin{s}-\frac{1}{2}(3\cos^2{s}+18\cos{s}+18)
\end{align*}
from \eqref{eq:der_lambda1_alpha}.
\par
Since $\sin{s}=2B-2\beta<-2\beta$, we have $\frac{\partial\,q_3}{\partial\,s}(\beta,s)=-\beta\cos{s}+3\sin{s}(\cos{s}+3)<-\beta(7\cos{s}+18)<0$, and so $q_3(\beta,s)$ is decreasing with respect to $s$.
Since $s>\pi+s_0$, it follows that
\begin{equation*}
  q_3(\beta,s)
  <
  q_3(\beta,\pi+s_0)
  =
  9\sqrt{1-4\beta^2}+10\beta^2-\frac{21}{2}.
\end{equation*}
Since $|\beta|<1/2$, we see that $10\beta^2-21/2<0$ and $1-4\beta^2>0$.
Since we calculate
\begin{equation*}
  \left(10\beta^2-21/2\right)^2-\left(9\sqrt{1-4\beta^2}\right)^2
  =
  100\beta^4+114\beta^2+117/4
  >
  0,
\end{equation*}
we conclude that $q_3(\beta,s)<0$, and so we have $q_3(\beta)\sqrt{1/4-\beta^2}<0$.
\par
Note that if we can prove that $\left(q_3(\beta)\sqrt{1/4-\beta^2}\right)^2>\left(p_3(\beta,s)\right)^2$, then we will have $\frac{\partial\,\hlambda}{\partial\,\alpha}(1-\sqrt{1/4-\beta^2},\beta,s)=p_3(\beta,s)+q_3(\beta,s)\sqrt{1/4-\beta^2}<0$.
So we will show that $r_3(\beta,s):=\left(q_3(\beta,s)\right)^2(1/4-\beta^2)-\left(p_3(\beta,s)\right)^2>0$.
\par
By Mathematica, we calculate
\begin{equation*}
\begin{split}
  r_3(\beta,s)
  =&
  -4\beta^6
  +4\beta^5\sin{s}
  -\frac{\beta^4}{4}(21\cos^2{s}+136\cos{s}+256)
  \\
  &+\frac{\beta^3}{2}\sin{s}(\cos^2{s}-2\cos{s}+2)
  \\
  &-\frac{\beta^2}{8}(9\cos^4{s}+104\cos^3{s}+401\cos^2{s}+548\cos{s}+246)
  \\
  &-\frac{\beta}{8}\sin{s}(\cos^4{s}+14\cos^3{s}+51\cos^2{s}+54\cos{s}+12)
  \\
  &-\frac{1}{64}\bigl(\cos^6{s}+24\cos^5{s}+174\cos^4{s}+408\cos^3{s}-63\cos^2{s}
  \\
  &\phantom{-\frac{1}{64}\bigl(}
  -1008\cos{s}-720\bigr),
\end{split}
\end{equation*}
and
\begin{equation*}
\begin{split}
  \frac{\partial^2\,r_3}{\partial\,\beta^2}(\beta,s)
  =&
  -120\beta^4
  +80\beta^3\sin{s}
  -3\beta^2(21\cos^2{s}+136\cos{s}+256)
  \\
  &+3\beta\sin{s}(\cos^2{s}-2\cos{s}+2)
  \\
  &-\frac{1}{4}(9\cos^4{s}+104\cos^3{s}+401\cos^2{s}+548\cos{s}+246).
\end{split}
\end{equation*}
Since we have $21\cos^2{s}+136\cos{s}+256>-136+256>0$, $\cos^2{s}-2\cos{s}+2=(\cos^2{s}-1)^2+1>0$, and $9\cos^4{s}+104\cos^3{s}+401\cos^2{s}+548\cos{s}+246=\upsilon_2(\cos{s})>0$ from Lemma~\ref{lem:inequalities_cos}, we conclude that $\frac{\partial^2\,r_3}{\partial\,\beta^2}(\beta,s)<0$ when $\pi+s_0<s<2\pi-s_0$ and $\beta>0$.
So the function $r_3(\beta,s)$ is convex upward with respect to $\beta$.
As a result, we have $r_3(\beta,s)>\min\left\{r_3(0,s),r_3\bigl(-(\sin{s})/2,s\bigr)\right\}$ since $0<\beta=B-(\sin{s})/2<-(\sin{s})/2$.
Now, we have
\begin{align*}
  r_3(0,s)
  &=
  -\frac{1}{64}(\cos^6{s}+24\cos^5{s}+174\cos^4{s}+408\cos^3{s}-63\cos^2{s}
  \\
  &\phantom{-\frac{1}{64}(}
  -1008\cos{s}-720),
  \\
  r_3\left(-\frac{\sin{s}}{2},s\right)
  &=
  \frac{1}{2}\cos^2{s}(5\cos^2{s}+18\cos{s}+14).
\end{align*}
From Lemma~\ref{lem:inequalities_cos}, $r_3(0,s)=-\upsilon_1(\cos{s})/64>0$.
Since $5x^2+18x+14=5(x+9/5)^2-11/5$, the function $5x^2+18x+14$ for $|x|<1$ is greater than its value at $x=-1$, which is $1$.
So we conclude that $r_3\bigl(-(\sin{s})/2,s\bigr)>0$, which proves $r_3(\beta,s)>0$.
Therefore the inequality $\frac{\partial\,\hlambda}{\partial\,\alpha}(1-\sqrt{1/4-\beta^2},\beta,s)<0$ holds.
\par
As a result, we see that $\hlambda(\alpha,\beta,s)$ is decreasing with respect to $\alpha$, and from $\alpha<1-\sqrt{1/4-\beta^2}$, we have
\begin{equation*}
  \hlambda(\alpha,\beta,s)
  \ge
  \hlambda\left(1-\sqrt{1/4-\beta^2},\beta,s\right).
\end{equation*}
We define $p_4(\beta,s)$ and $q_4(\beta,s)$ so that
\begin{equation*}
  \hlambda\left(1-\sqrt{1/4-\beta^2},\beta,s\right)
  =
  p_4(\beta,s)+q_4(\beta,s)\sqrt{1/4-\beta^2},
\end{equation*}
that is, we put
\begin{align*}
  p_4(\beta,s)
  :=&
  \beta^4+\beta^3\sin{s}
  -\frac{\beta^2}{4}(5\cos^2{s}+14\cos{s}+6)
  +\frac{\beta}{4}\sin{s}(2\cos{s}+5)
  \\
  &+\frac{1}{16}(2\cos^3{s}+11\cos^2{s}+18\cos{s}+17),
  \\
  q_4(\beta,s)
  :=&
  \frac{\beta^2}{2}(\cos{s}+4)
  -\frac{\beta}{2}\sin{s}(\cos{s}+2)
  -\frac{1}{8}(\cos^3{s}+12\cos^2{s}+27\cos{s}+8).
\end{align*}
\par
We will show that $p_4(\beta,s)>0$.
To do that we calculate
\begin{align*}
  \frac{\partial\,p_4}{\partial\,\beta}(\beta,s)
  &=
  4\beta^3+3\beta^2\sin{s}
  -\frac{\beta}{2}(5\cos^2{s}+14\cos{s}+6)
  +\frac{1}{4}\sin{s}(2\cos{s}+5).
\end{align*}
Since $\beta<-(\sin{s})/2$ and $\sin{s}<0$, we have $4\beta^3+3\beta^2\sin{s}=\beta^2(4\beta+3\sin{s})<\beta^2\sin{s}<0$.
Since $2\cos{s}+5>0$ and $\sin{s}<-2\beta$, we also have
\begin{equation*}
\begin{split}
  &
  -\frac{\beta}{2}(5\cos^2{s}+14\cos{s}+6)
  +\frac{1}{4}\sin{s}(2\cos{s}+5)
  \\
  <&
  -\frac{\beta}{2}(5\cos^2{s}+16\cos{s}+11)
  <0,
\end{split}
\end{equation*}
where the last inequality follows since the quadratic function $5x^2+16x+11=5(x+8/5)^2-9/5$ is increasing for $|x|<1$ and equals $0$ when $x=-1$.
Therefore we conclude that $\frac{\partial}{\partial\,\beta}p_4(\beta,s)<0$ for $0<\beta<-(\sin{s})/2$, and so we have
\begin{equation*}
  p_4(\beta,s)
  >
  p_4\bigl(-(\sin{s})/2,s\bigr)
  =
  \frac{1}{4}\cos^2{s}(\cos{s}+2)(\cos{s}+3)\ge0.
\end{equation*}
\par
Now, we will show that $p_4(\beta,s)^2-\left(q_4(\beta,s)\sqrt{1/4-\beta^2}\right)^2>0$.
Note that this implies $\hlambda(1-\sqrt{1/4-\beta^2},\beta,s)>0$ since $p_4(\beta,s)>0$.
\par
Thanks to Mathematica, we have
\begin{equation*}
\begin{split}
  &\frac{\beta-(\sin{s})/2}{B}
  \left(p_4(\beta,s)^2-\left(q_4(\beta,s)\sqrt{1/4-\beta^2}\right)^2\right)
  \\
  =&
  2B\beta^7
  +
  r_4(\beta,s),
\end{split}
\end{equation*}
where, we put
\begin{equation*}
\begin{split}
  &r_4(\beta,s)
  \\
  :=&
  -\beta^8
  -\frac{\beta^6}{4}(7\cos^2{s}+20\cos{s}-2)
  -\frac{\beta^5}{4}\sin{s}(2\cos^2{s}+16\cos{s}+23)
  \\
  &
  -\frac{\beta^4}{16}(9\cos^4{s}+24\cos^3{s}-21\cos^2{s}-100\cos{s}-128)
  \\
  &
  -\frac{\beta^3}{16}\sin{s}(3\cos^4{s}+16\cos^3{s}+28\cos^2{s}+40\cos{s}+65)
  \\
  &
  -\frac{\beta^2}{64}(\cos^6{s}+4\cos^5{s}-8\cos^4{s}-64\cos^3{s}-123\cos^2{s}-100\cos{s}-94)
  \\
  &
  -\frac{\beta}{64}\sin{s}(9\cos^4{s}+56\cos^3{s}+78\cos^2{s}-56\cos{s}-87)
  \\
  &
  -\frac{1}{256}(\cos{s}-3)(\cos{s}-1)(\cos{s}+1)(\cos{s}+3)(\cos{s}+5)(3\cos{s}+5).
\end{split}
\end{equation*}
Since $\beta-(\sin{s})/2<0$ from $0<\beta<-(\sin{s})/2$, and $2B\beta^7<0$, we conclude that $p_4(\beta,s)^2-\left(q_4(\beta,s)\sqrt{1/4-\beta^2}\right)^2>0$ if we can prove $r_4(\beta,s)<0$.
\par
We have
\begin{equation*}
\begin{split}
  &\frac{\partial^2\,r_4}{\partial\,\beta^2}(\beta,s)
  \\
  =&
  -56\beta^6
  -\frac{15\beta^4}{2}(7\cos^2{s}+20\cos{s}-2)
  -5\beta^3\sin{s}(2\cos^2{s}+16\cos{s}+23)
  \\
  &
  -\frac{3\beta^2}{4}(9\cos^4{s}+24\cos^3{s}-21\cos^2{s}-100\cos{s}-128)
  \\
  &
  -\frac{3\beta}{8}\sin{s}(3\cos^4{s}+16\cos^3{s}+28\cos^2{s}+40\cos{s}+65)
  \\
  &
  -\frac{1}{32}(\cos^6{s}+4\cos^5{s}-8\cos^4{s}-64\cos^3{s}-123\cos^2{s}-100\cos{s}-94).
\end{split}
\end{equation*}
Since $2\cos^2{s}+16\cos{s}+23=2(\cos{s}+4)^2-9>0$ and $-\sin{s}>2\beta>0$, we have
\begin{equation*}
\begin{split}
  &-56\beta^6
  -\frac{15\beta^4}{2}(7\cos^2{s}+20\cos{s}-2)
  -5\beta^3\sin{s}(2\cos^2{s}+16\cos{s}+23)
  \\
  >&
  -\beta^4
  \left(
    56\beta^2+\frac{15}{2}(7\cos^2{s}+20\cos{s}-2)-10(2\cos^2{s}+16\cos{s}+23)
  \right)
  \\
  =&
  -\beta^4\left(56\beta^2+\frac{5}{2}(13\cos^2{s}-4\cos{s}-98)\right)
  >
  -\beta^4\left(14-\frac{405}{2}\right)>0
\end{split}
\end{equation*}
since $0<\beta<1/2$, where we use $13\cos^2{s}-4\cos{s}-98<-81$ because $13x^2-4x-98$ is convex and equals $-89$ when $x=1$ and $-81$ when $x=-1$.
\par
From Lemma~\ref{lem:inequalities_cos}, we see that the last three terms in $\frac{\partial^2\,r_4}{\partial\,\beta^2}(\beta,s)$ are all positive, proving that $\frac{\partial^2\,r_4}{\partial\,\beta^2}(\beta,s)>0$ for $0<\beta<1/2$ and $\pi+s_0<s<2\pi-s_0$.
\par
Therefore the function $r_4(\beta,s)$ is convex down with respect to $\beta$ for $0<\beta<-(\sin{s})/2$.
Since we we have
\begin{multline*}
  r_4(0,s)
  \\
  =
  -\frac{1}{256}(\cos{s}-3)(\cos{s}-1)(\cos{s}+1)(\cos{s}+3)(\cos{s}+5)(3\cos{s}+5)
  <0,
\end{multline*}
and
\begin{equation*}
  r_4\left(-\frac{\sin{s}}{2},s\right)
  =
  0,
\end{equation*}
we conclude that $r_4(\beta,s)<0$.
\par
The proof is now complete.
\end{proof}

\section{Technical lemmas}\label{sec:lemmas}
In this appendix, we prove several lemmas used in the paper.
%%%%%%%%%%%%%%%%%%%%%%%%%%%%%%%%%%%%%%%%%%%%%%%%%%%%%%%%%%%%%%%%%%%%%%%%%%%%%%%
%\begin{comment}
%%%%%%%%%%%%%%%%%%%%%%%%%%%%%%%%%%%%%%%%%%%%%%%%%%%%%%%%%%%%%%%%%%%%%%%%%%%%%%%
\begin{lem}\label{lem:T_converge}
The integral $\int_{\Rpath}\frac{e^{(2z-1)t}}{t\sinh{t}\sinh(\gamma t/N)}\,dt$ converges if $-\frac{\Re\gamma}{2N}<\Re{z}<\frac{\Re\gamma}{2N}+1$.
\end{lem}
\begin{proof}
The proof is the same as that of \cite[Lemma~2.2]{Murakami:CANJM2023}.
\par
Since $\Re\gamma/N=b/N>0$, we have $\sinh(\gamma t/N)\underset{t\to\infty}{\sim}\frac{1}{2}\exp(\gamma t/N)$ and $\sinh(\gamma t/N)\underset{t\to-\infty}{\sim}-\frac{1}{2}\exp(\gamma t/N)$.
So we conclude
\begin{align*}
  \frac{e^{(2z-1)t}}{t\sinh{t}\sinh(\gamma t/N)}
  &\underset{t\to\infty}{\sim}
  \frac{4}{t}\exp\bigl((2z-2-\gamma/N)t\bigr),
  \\
  \frac{e^{(2z-1)t}}{t\sinh{t}\sinh(\gamma t/N)}
  &\underset{t\to-\infty}{\sim}
  \frac{4}{t}\exp\bigl((2z+\gamma/N)t\bigr).
\end{align*}
Therefore the integral converges if $\Re(2z-2-\gamma/N)<0$ and $\Re(2z+\gamma/N)>0$, that is, if $-\Re\bigl(\gamma/(2N)\bigr)<\Re{z}<\Re\bigl(\gamma/(2N)\bigr)+1$.
This completes the proof.
\end{proof}
%%%%%%%%%%%%%%%%%%%%%%%%%%%%%%%%%%%%%%%%%%%%%%%%%%%%%%%%%%%%%%%%%%%%%%%%%%%%%%%
\begin{lem}\label{lem:Omega}
If $\xi\in\Xi\cap\Cl(\Omega)=\{\xi\in\C\mid a>0, 0<b<\pi/2, \cosh{a}\cos{b}>1/2, \cosh{a}-\cos{b}\le1/2\}$, then $\Re F(\sigma)<0$.
See \eqref{eq:Omega_def} for the definition of $\Omega$.
\end{lem}
\begin{rem}\label{rem:Omega_appendix}
If $\cosh{a}-\cos{b}\le1/2$, then $\cosh{a}\cos{b}\ge\cosh{a}(\cosh{a}-1/2)>1/2$.
Therefore we have
\begin{equation*}
  \Xi\cap\Cl(\Omega)
  =
  \{\xi\in\C\mid a>0, 0<b<\pi/2, \cosh{a}-\cos{b}\le1/2\}.
\end{equation*}
\end{rem}
Put $\tS(\xi):=a\Re S(\xi)+b\Im S(\xi)$.
From $F(\sigma)=S(\xi)/\xi$, we have $\Re F(\sigma)=\frac{1}{|\xi|^2}S(\xi)(a-b\i)=\frac{1}{|\xi|^2}\tS(\xi)$.
So, it is enough to show that $\tS(\xi)<0$ if $\xi\in\Xi\cap\Cl(\Omega)$.
\par
We split the proof into two sublemmas; Sublemma~\ref{sublem:Omega1} and Sublemma~\ref{sublem:Omega2}.
Recall that we are putting $a:=\Re\xi$ and $b:=\Im\xi$.
\begin{sublem}\label{sublem:Omega1}
If $\xi\in\Gamma$, that is, if $a>0$, $0<b<\pi/2$, and $\cosh{a}\cos{b}>1/2$, then the function $\tS(\xi)$ is monotonically increasing with respect to $a$.
\end{sublem}
\begin{sublem}\label{sublem:Omega2}
Suppose that $\xi\in\Gamma$.
If $\xi$ is on the curve $a=\arcosh(\cos{b}+1/2)$, then $\tS(\xi)<0$.
\end{sublem}
\begin{proof}[Proof of Lemma~\ref{lem:Omega} assuming Sublemmas~\ref{sublem:Omega1} and \ref{sublem:Omega2}]
We will show that if $\xi\in\Xi\cap\Cl(\Omega)$, then $\tS(\xi)<0$.
\par
From Sublemma~\ref{sublem:Omega1}, the region $\{\xi\in\C\mid\tS(\xi)<0\}\cap\Gamma$ is on the left of the curve $\{\xi\in\C\mid\tS(\xi)=0\}\cap\Gamma$.
\par
Since $\cosh{a}-\cos{b}\le1/2$, we have $\cos{b}\ge1/2$, which implies that $b\le\pi/3$.
So the region $\Xi\cap\Cl(\Omega)$ can be expressed as
\begin{equation*}
  \{a+b\i\in\C\mid a>0,0<b<\pi/3,\cosh{a}\cos{b}>1/2,a\le\arcosh(\cos{b}+1/2)\},
\end{equation*}
which is on the left of the curve $a=\arcosh(\cos{b}+1/2)$ in the $ab$-plane.
\par
Since from Sublemma~\ref{sublem:Omega1} the curve above is on the left of the curve $\tS(\xi)=0$, we conclude that $\Xi\cap\Cl(\Omega)$ is in the region $\{\xi\in\C\mid\tS(\xi)<0\}\cap\Gamma$, and the lemma follows.
\end{proof}
Now we prove the sublemmas.
\begin{proof}[Proof of sublemma~\ref{sublem:Omega1}]
We will show that
\begin{equation}\label{eq:dReSImS}
  \frac{\partial}{\partial\,a}\tS(\xi)
  =
  \Re S(\xi)+a\frac{\partial}{\partial\,a}\Re S(\xi)
  +b\frac{\partial}{\partial\,a}\Im S(\xi)
  >0.
\end{equation}
From \eqref{eq:dS}, we have
\begin{equation}\label{eq:d_Re_S}
\begin{split}
  \frac{\partial}{\partial\,a}\Re S(\xi)
  =&
  2\log\left|e^{(\xi+\varphi)/2}-e^{-(\xi+\varphi)/2}\right|,
  \\
  \frac{\partial}{\partial\,a}\Im S(\xi)
  =&
  2\arg\left(e^{(\xi+\varphi)/2}-e^{-(\xi+\varphi)/2}\right).
\end{split}
\end{equation}
\par
In the following, we will show (i) $\frac{\partial}{\partial\,a}\Re S(\xi)>0$ and (ii) $\Re S(\xi)+b\frac{\partial}{\partial\,a}\Im S(\xi)$ to prove the inequality in \eqref{eq:dReSImS}.
\begin{enumerate}
\item
$\frac{\partial}{\partial\,a}\Re S(\xi)>0$.
\par
We first show the inequality $\cosh(a+c)-\cos(b+d)>1/2$, where we put $a:=\Re\xi$, $b:=\Im\xi$, $c:=\Re\varphi$, and $d:=\Im\varphi$ as usual.
We calculate
\begin{equation}\label{eq:cosh-cos}
\begin{split}
  \cosh(a+c)-\cos(b+d)
  =&
  \cosh{a}\cosh{c}+\sinh{a}\sinh{c}-\cos{b}\cos{d}+\sin{b}\sin{d}
  \\
  >&
  \cosh{a}\cosh{c}-\cos{b}\cos{d},
\end{split}
\end{equation}
since $a>0$, $c>0$, $0<b<\pi/2$, and $0<d<\pi/2$ from Lemma~\ref{lem:phi_xi}.
\par
Since $\alpha=\cosh{a}\cos{b}$ and $\beta=\sinh{a}\sin{b}$ from \eqref{eq:alpha} and \eqref{eq:beta}, we have $\left(\frac{\alpha}{\cos{b}}\right)^2-\left(\frac{\beta}{\sin{b}}\right)^2=1$, which implies that $\cos^2{b}$ satisfies the equation
\begin{equation}\label{eq:cosh^2a_cos^2b}
  x^2-(\alpha^2+\beta^2+1)x+\alpha^2=0.
\end{equation}
Similarly, from $\left(\frac{\alpha}{\cosh{a}}\right)^2+\left(\frac{\beta}{\sinh{a}}\right)^2=1$, we conclude that $\cosh{a}$ also satisfies \eqref{eq:cosh^2a_cos^2b}.
Therefore we obtain
\begin{align*}
  \cos^2{b}
  &=
  \frac{\alpha^2+\beta^2+1-\sqrt{(\alpha^2+\beta^2+1)^2-4\alpha^2}}{2}
  \\
  &=
  \frac{\alpha^2+\beta^2+1
        -\sqrt{\bigl((\alpha-1)^2+\beta^2\bigr)\bigl((\alpha+1)^2+\beta^2\bigr)}}{2},
  \\
  \cosh^2{a}
  &=
  \frac{\alpha^2+\beta^2+1
        +\sqrt{\bigl((\alpha-1)^2+\beta^2\bigr)\bigl((\alpha+1)^2+\beta^2\bigr)}}{2},
\end{align*}
because $\cosh^2{a}>1>\cos^2{b}$.
\par
Similarly, from \eqref{eq:Re_phi} and \eqref{eq:Im_phi} we have $\left(\frac{\alpha-1/2}{\cos{d}}\right)^2-\left(\frac{\beta}{\sin{d}}\right)^2=1$ and $\left(\frac{\alpha-1/2}{\cosh{c}}\right)^2+\left(\frac{\beta}{\sinh{c}}\right)^2=1$.
So $\cos^2{d}$ and $\cosh^2{c}$ both satisfy the equation $x^2-\bigl((\alpha-1/2)^2+\beta^2+1\bigr)x+(\alpha-1/2)^2=0$, from which we have
\begin{align*}
  \cos^2{d}
  &=
  \frac{(\alpha-1/2)^2+\beta^2+1-\sqrt{\bigl((\alpha-1/2)^2+\beta^2+1\bigr)^2-4(\alpha-1/2)^2}}{2}
  \\
  &=
  \frac{(\alpha-1/2)^2+\beta^2+1
        -\sqrt{\bigl((\alpha-3/2)^2+\beta^2\bigr)\bigl((\alpha+1/2)^2+\beta^2\bigr)}}{2}
  \\
  \cosh^2{c}
  &=
  \frac{(\alpha-1/2)^2+\beta^2+1
        +\sqrt{\bigl((\alpha-3/2)^2+\beta^2\bigr)\bigl((\alpha+1/2)^2+\beta^2\bigr)}}{2},
\end{align*}
by the same reason as above.
\par
It is clear that $\cosh{a}$ and $\cosh{c}$ are monotonically increasing with respect to $\beta>0$.
Since we have
\begin{align*}
  \frac{\partial}{\partial\,\beta}\cos^2{b}
  &=
  \beta\left(1-\frac{\alpha^2+\beta^2+1}{\sqrt{(\alpha^2+\beta^2+1)^2-4\alpha^2}}\right)
  <0,
  \\
  \frac{\partial}{\partial\,\beta}\cos^2{d}
  &=
  \beta
  \left(
    1-\frac{(\alpha-1/2)^2+\beta^2+1}{\sqrt{\bigl((\alpha-1/2)^2+\beta^2+1\bigr)^2-4\alpha^2}}
  \right)
  <0,
\end{align*}
we see that both $\cos^2{b}$ and $\cos^2{d}$ are monotonically decreasing with respect to $\beta>0$.
\par
Thus we conclude that $\cosh{a}\cosh{c}-\cos{b}\cos{d}$ is monotonically increasing with respect to $\beta>0$, and so from \eqref{eq:cosh-cos} we have
\begin{equation*}
\begin{split}
  &\cosh(a+c)-\cos(b+d)
  \\
  >&
  \cosh{a}\cosh{c}-\cos{b}\cos{d}\Bigr|_{\beta=0}
  \\
  =&
  \sqrt{\frac{\alpha^2+1+|(\alpha-1)(\alpha+1)|}{2}}
  \sqrt{\frac{(\alpha-1/2)^2+1+|(\alpha-3/2)(\alpha+1/2)|}{2}}
  \\
  &-
  \sqrt{\frac{\alpha^2+1-|(\alpha-1)(\alpha+1)|}{2}}
  \sqrt{\frac{(\alpha-1/2)^2+1-|(\alpha-3/2)(\alpha+1/2)|}{2}}
  \\
  =&
  \begin{cases}
    -\alpha^2+\alpha/2+1&\quad\text{($1/2<\alpha\le1$)},
    \\
    1/2&\quad\text{($1<\alpha\le3/2$)},
    \\
    \alpha^2-\alpha/2-1&\quad\text{($\alpha>3/2$)}.
  \end{cases}
\end{split}
\end{equation*}
It follows that $\cosh(a+c)-\cos(b+d)>1/2$.
\par
Next, we will show that if $w$ satisfies the inequality $\cosh\Re{w}-\cos\Im{w}>1/2$, then $\log\left|e^{w/2}-e^{-w/2}\right|>0$.
\par
Writing $w:=x+y\i$, we have
\begin{equation*}
\begin{split}
  2\log\left|e^{w/2}-e^{-w/2}\right|
  =&
  \log\left(2\bigl|\cosh{w}-1\bigr|\right)
  \\
  =&
  \log\left(2\sqrt{(\cosh{x}\cos{y}-1)^2+\sinh^2{x}\sin^2{y}}\right)
  \\
  =&
  \log\bigl(2|\cosh{x}-\cos{y}|\bigr)
  >0.
\end{split}
\end{equation*}
\par
Now, putting $w:=\xi+\varphi$, from \eqref{eq:d_Re_S} we have
\begin{equation*}
  \frac{\partial}{\partial\,a}\Re S(\xi)>0
\end{equation*}
as desired.
\item
$\Re S(\xi)+b\frac{\partial}{\partial\,a}\Im S(\xi)>0$.
\par
Write $R(\xi):=\Re S(\xi)+b\frac{\partial}{\partial\,a}\Im S(\xi)$.
\par
If $b=0$ and $0<a\le\kappa:=\arcosh(3/2)$, $\varphi$ is purely imaginary from Remark~\ref{rem:phi_different} (see also \cite[Lemma~4.2]{Murakami:CANJM2023}.
Therefore from \eqref{eq:Sminus}, we have $\Re S(a)=\Re\Li_2(e^{-a-\varphi})-\Re\Li_2(e^{-a+\varphi})=\Re\left(\Li_2(e^{-a-\varphi})-\overline{\Li_2(e^{-a-\varphi}})\right)=0$, where $\overline{w}$ is the complex conjugate of $w$.
From \cite[P.~186]{Murakami/Tran:Takata2025} we know that $S(a)$ is real and positive when $a>\kappa$.
\par
So we conclude that $R(\xi)\ge0$ if $b=0$.
Thus, if we could prove that $\frac{\partial}{\partial\,b}\Re R(\xi)>0$, we would conclude that $R(\xi)>0$ for $a>0$, $0<b<\pi/2$.
\par
Now, we have
\begin{equation*}
\begin{split}
  \frac{\partial}{\partial\,b}R(\xi)
  =&
  \frac{\partial}{\partial\,b}\Re S(\xi)
  +
  \frac{\partial}{\partial\,a}\Im S(\xi)
  +
  b\frac{\partial^2}{\partial\,a\,\partial\,b}\Im S(\xi)
  \\
  =&
  b\frac{\partial^2}{\partial\,a^2}\Re S(\xi),
\end{split}
\end{equation*}
where the second equality follows from the Cauchy--Riemann equations: $\frac{\partial}{\partial\,a}\Re S(\xi)=\frac{\partial}{\partial\,b}\Im S(\xi)$ and $\frac{\partial}{\partial\,b}\Re S(\xi)=-\frac{\partial}{\partial\,a}\Im S(\xi)$.
\par
We will show that $\frac{\partial^2}{\partial\,a^2}\Re S(\xi)>0$.
\par
From \eqref{eq:dS} we have
\begin{equation}\label{eq:d2S}
  \frac{d^2}{d\,\xi^2}S(\xi)
  =
  (1+\varphi')\coth\bigl((\xi+\varphi)/2\bigr),
\end{equation}
where $\varphi':=\frac{d}{d\,\xi}\varphi$.
From $\cosh\varphi=\cosh\xi-1/2$, we obtain
\begin{equation}\label{eq:phi'}
\begin{split}
  \varphi'
  =&
  \frac{\sinh\xi}{\sinh\varphi}
  \\
  =&
  \frac{
  (\sinh{a}\cos{b}+\i\cosh{a}\sin{b})(\sinh{c}\cos{d}-\i\cosh{c}\sin{d})}{|\sinh\varphi|^2}
  \\
  =&
  \frac{\sinh{a}\sinh{c}\cos{b}\cos{d}+\cosh{a}\cosh{c}\sin{b}\sin{d}}{|\sinh\varphi|^2}
  \\
  &+\i
  \frac{\cosh{a}\sinh{c}\sin{b}\cos{d}-\sinh{a}\cosh{c}\cos{b}\sin{d}}{|\sinh\varphi|^2}.
\end{split}
\end{equation}
\par
From the half-angle formula, one has
\begin{equation}\label{eq:half_angle}
  \coth(z/2)
  =
  \frac{\cosh{z}+1}{\sinh{z}},
\end{equation}
which implies that
\begin{equation*}
\begin{split}
  &\coth\bigl((x+y\i)/2\bigr)
  \\
  =&
  \frac{\cosh(x+y\i)+1}{\sinh(x+y\i)}
  \\
  =&
  \frac{(1+\cosh{x}\cos{y}+\i\sinh{x}\sin{y})(\sinh{x}\cos{y}-\i\cosh{x}\sin{y})}
       {\sinh^2{x}\cos^2{y}+\cosh^2{x}\sin^2{y}}
  \\
  =&
  \frac{\sin{x}-\i\sin{y}}{\cosh{x}-\cos{y}}.
\end{split}
\end{equation*}
So we have
\begin{equation}\label{eq:cosh}
  \coth\bigl((\xi+\varphi)/2\bigr)
  =
  \frac{\sinh(a+c)-\i\sin(b+d)}{\cosh(a+c)-\cos(b+d)}.
\end{equation}
Note that the denominator is positive since $\cosh(a+c)>1>\cos(b+d)$.
\par
Since we have $\Re\varphi'>0$ from \eqref{eq:phi'}, and $\Re\coth\bigl((\xi+\varphi)/2\bigr)>0$ from \eqref{eq:cosh}, \eqref{eq:d2S} implies
\begin{equation*}
\begin{split}
  &\frac{\partial^2}{\partial\,a^2}\Re S(\xi)
  \\
  =&
  \Re\coth\bigl((\xi+\varphi)/2\bigr)
  +
  \Re\coth\bigl((\xi+\varphi)/2\bigr)\times\Re\varphi'
  -
  \Im\coth\bigl((\xi+\varphi)/2\bigr)\times\Im\varphi'
  \\
  >&
  \Re\coth\bigl((\xi+\varphi)/2\bigr)
  -
  \Im\coth\bigl((\xi+\varphi)/2\bigr)\times\Im\varphi'
  \\
  =&
  \frac{\sinh(a+c)+\sin(b+d)\Im\varphi'}{\cosh(a+c)-\cos(b+d)}
  \\
  =&
  \frac{1}{|\sinh\varphi|^2\bigl(\cosh(a+c)-\cos(b+d)\bigr)}
  \\
  &\times
  \Bigl[
    \sinh(a+c)\left(\sinh^2{c}\cos^2{d}+\cosh^2{c}\sin^2{d}\right)
  \\
  &\quad
    +
    \sin(b+d)(\cosh{a}\sinh{c}\sin{b}\cos{d}-\sinh{a}\cosh{c}\cos{b}\sin{d})
  \Bigr],
\end{split}
\end{equation*}
where we use \eqref{eq:phi'} in the last equality.
Now the expression in the square brackets is greater than
\begin{equation*}
\begin{split}
  &\sinh(a+c)\cosh^2{c}\sin^2{d}
  -
  \sin(b+d)\sinh{a}\cosh{c}\cos{b}\sin{d})
  \\
  =&
  \cosh{c}\sin{d}
  \bigl(\sinh(a+c)\cosh{c}\sin{d}-\sin(b+d)\sinh{a}\cos{b}\bigr)
  \\
  &\quad\text{(since $\cosh{c}>1$ and $\cos{b}<1$)}
  \\
  >&
  \cosh{c}\sin{d}
  \bigl(\sinh(a+c)\sin{d}-\sin(b+d)\sinh{a}\bigr)
  \\
  &\quad\text{(since $\sinh(a+c)=\sinh{a}\cosh{c}+\cosh{a}\sinh{c}>\sinh{a}+\sinh{c}$,}
  \\
  &\quad\text{and $\sin(b+d)=\sin{b}\cos{d}+\cos{b}\sin{d}<\sin{b}+\sin{d}$)}
  \\
  >&
  \cosh{c}\sin{d}
  \bigl(\sinh{a}\sin{d}+\sinh{c}\sin{d}-\sin{b}\sinh{a}-\sin{d}\sinh{a}\bigr)
  \\
  =&
  \cosh{c}\sin{d}\bigl(\sinh{c}\sin{d}-\sin{b}\sinh{a}\bigr)
  \\
  =&0,
\end{split}
\end{equation*}
where the last equality follows from \eqref{eq:beta} and \eqref{eq:Im_phi}.
Therefore, we conclude that $\frac{\partial^2}{\partial\,a^2}\Re S(\xi)>0$, and (ii) follows.
\end{enumerate}
\end{proof}
%%%%%%%%%%%%%%%%%%%%%%%%%%%%%%%%%%%%%%%%%%%%%%%%%%%%%%%%%%%%%%%%%%%%%
%%%%%%%%%%%%%%%%%%%%%%%%%%%%%%%%%%%%%%%%%%%%%%%%%%%%%%%%%%%%%%%%%%%%%
\begin{proof}[Proof of Sublemma~\ref{sublem:Omega1}]
Put $\kappa:=\arcosh(3/2)$ as before.
It is clear that $\arcosh(\cos{b}+1/2)<\kappa$, which means that the curve $a=\arcosh(\cos{b}+1/2)$ ($a>0$, $0<b<\pi/3$) is on the left of the line $a=\kappa$.
So, it is enough to show that the vertical line segment $\kappa+b\i$ ($0<b<\pi/3$) is in the region $\{\xi\in\C\mid \Re F(\sigma)<0\}$.
Since $\Re F(\sigma)=\Re\left(\frac{S(\xi)}{\xi}\right)=\frac{1}{|\xi|^2}\Re\bigl(S(\xi)(a-b\i)\bigr)=\frac{1}{|\xi|^2}\bigl(a\Re S(\xi)+b\Im S(\xi)\bigr)$, we will show that $Q(b):=\kappa Q_1(b)+Q_2(b)>0$ for $0<b<\pi/3$, where we put
\begin{align*}
  Q_1(b)
  :=&
  \Re S(\xi_b),
  \\
  Q_2(b)
  :=&
  b\Im S(\xi_b),
\end{align*}
and $\xi_b:=\kappa+b\i$.
We calculate
\begin{align}
  Q_1'(b)
  =&
  \frac{d}{d\,b}\Re S(\xi_b),
  \label{eq:dQ1}
  \\
  Q_2'(b)
  =&
  \Im S(\xi_b)+b\frac{d}{d\,b}\Im S(\xi_b),
  \notag
  \\
  Q_1''(b)
  =&
  \frac{d^2}{d\,b^2}\Re S(\xi_b),
  \notag
  \\
  Q_2''(b)
  =&
  2\frac{d}{d\,b}\Im S(\xi_b)+b\frac{d^2}{d\,b^2}\Im S(\xi_b),
  \notag
  \\
  Q_1^{(3)}(b)
  =&
  \frac{d^3}{d\,b^3}\Re S(\xi_b),
  \label{eq:d3Q1}
  \\
  Q_2^{(3)}(b)
  =&
  3\frac{d^2}{d\,b^2}\Im S(\xi_b)+b\frac{d^3}{d\,b^3}\Im S(\xi_b).
  \label{eq:d3Q2}
\end{align}
\par
We will show that both $Q_1^{(3)}(b)$ and $Q_2^{(3)}(b)$ are positive, which implies that $Q^{(3)}(b)>0$.
\par
We will calculate $\frac{d^k}{d\,b^k}S(\xi_b)$ for $k=1,2,3$.
Since
\begin{equation*}
  \frac{d}{d\,\xi}S(\xi)
  =
  2\log\left(2\sinh\bigl((\xi+\varphi)/2\bigr)\right)
\end{equation*}
from \eqref{eq:dS} and $\varphi'=\frac{\sinh\xi}{\sinh\varphi}$ from \eqref{eq:phi'}, we have
\begin{equation*}
\begin{split}
  &\frac{d^2}{d\,\xi^2}S(\xi)
  \\
  =&
  2\frac{2\cosh\bigl((\xi+\varphi)/2\bigr)\times\bigl((1+\varphi')/2\bigr)}
        {2\sinh\bigl((\xi+\varphi)/2\bigr)}
  \\
  =&
  \frac{\coth\bigl((\xi+\varphi)/2\bigr)(\sinh\xi+\sinh\varphi)}{\sinh\varphi}
  \\
  =&
  \frac{\bigl(\cosh(\xi+\varphi)+1\bigr)(\sinh\xi+\sinh\varphi)}{\sinh(\xi+\varphi)\sinh\varphi}
  \\
  =&
  \frac{1}{\sinh(\xi+\varphi)\sinh\varphi}
  \\
  &\times
  (\cosh\xi\sinh\xi\cosh\varphi+\sinh^2\xi\sinh\varphi
  \\
  &\quad+
  \cosh\xi\cosh\varphi\sinh\varphi+\sinh\xi\sinh^2\varphi+\sinh\xi+\sinh\varphi)
  \\
  =&
  \frac{1}{\sinh(\xi+\varphi)\sinh\varphi}
  \\
  &\times
  (\cosh\xi\sinh\xi\cosh\varphi+\cosh^2\xi\sinh\varphi
  +
  \cosh\xi\cosh\varphi\sinh\varphi+\sinh\xi\cosh^2\varphi)
  \\
  =&
  \frac{1}{\sinh(\xi+\varphi)\sinh\varphi}
  \\
  &\times
  \bigl(\cosh\xi(\sinh\xi\cosh\varphi+\cosh\xi\sinh\varphi)
  +
  \cosh\varphi(\cosh\xi\sinh\varphi+\sinh\xi\cosh\varphi)\bigr)
  \\
  =&
  \frac{\cosh\xi+\cosh\varphi}{\sinh\varphi},
\end{split}
\end{equation*}
where we use \eqref{eq:half_angle} in the second equality.
We also have
\begin{equation*}
\begin{split}
  \frac{d^3}{d\,\xi^3}S(\xi)
  =&
  \frac{(\sinh\xi+\varphi'\sinh\varphi)\sinh\varphi-(\cosh\xi+\cosh\varphi)\varphi'\cosh\varphi}
       {\sinh^2\varphi}
  \\
  =&
  2\frac{\sinh\xi}{\sinh\varphi}
  -
  \frac{\sinh\xi\cosh\varphi(\cosh\xi+\cosh\varphi)}{\sinh^3\varphi}.
\end{split}
\end{equation*}
Putting $\varphi_b:=\varphi\bigm|_{\xi:=\xi_b}$, we have
\begin{align*}
  \frac{d}{d\,b}S(\xi_b)
  =&
  2\i\log\left(e^{(\xi_b+\varphi_b)/2}-e^{-(\xi_b+\varphi_b)/2}\right)
  \\
  \frac{d^2}{d\,b^2}S(\xi_b)
  =&
  -\frac{\cosh\xi_b+\cosh\varphi_b}{\sinh\varphi_b},
  \\
  \frac{d^3}{d\,b^3}S(\xi_b)
  =&
  -\i
  \left(
    2\frac{\sinh\xi_b}{\sinh\varphi_b}
    -
    \frac{\sinh\xi_b\cosh\varphi_b(\cosh\xi_b+\cosh\varphi_b)}{\sinh^3\varphi_b}
  \right)
\end{align*}
since $d\,\varphi_b/d\,b=\i$.
\par
Therefore, we have
\begin{align}
  \frac{d^2}{d\,b^2}\Im S(\xi_b)
  &:=
  -
  \Im
  \left(
    \frac{\cosh\xi_b+\cosh\varphi_b}{\sinh\varphi_b}
  \right),
  \label{eq:d2ImS}
  \\
  \frac{d^3}{d\,b^3}\Re S(\xi_b)
  &:=
  \Im
  \left(
    2\frac{\sinh\xi_b}{\sinh\varphi_b}
    -
    \frac{\sinh\xi_b\cosh\varphi_b(\cosh\xi_b+\cosh\varphi_b)}{\sinh^3\varphi_b}
  \right),
  \label{eq:d3ReS}
  \\
  \frac{d^3}{d\,b^3}\Im S(\xi_b)
  &:=
  -\Re
  \left(
    2\frac{\sinh\xi_b}{\sinh\varphi_b}
    -
    \frac{\sinh\xi_b\cosh\varphi_b(\cosh\xi_b+\cosh\varphi_b)}{\sinh^3\varphi_b}
  \right).
  \label{eq:d3ImS}
\end{align}
\par
We put
\begin{align*}
  X(\xi)
  :=&
  \frac{\sinh\xi}{\sinh\varphi},
  \\
  Y(\xi)
  :=&
  \frac{\cosh\varphi}{\sinh\varphi},
  \\
  Z(\xi)
  :=&
  \frac{\cosh\xi}{\sinh\varphi},
  \\
  W(\xi)
  :=&
  Y(\xi)+Z(\xi)
\end{align*}
so that
\begin{align*}
  Q_1^{(3)}(b)
  =&
  2\Im X(\xi_b)-\Im\bigl(X(\xi_b)Y(\xi_b)W(\xi_b)\bigr),
  \\
  Q_2^{(3)}(b)
  =&
  -3\Im W(\xi_b)
  -
  b\Bigl(2\Re X(\xi_b)-\Re\bigl(X(\xi_b)Y(\xi_b)W(\xi_b)\bigr)\Bigr)
\end{align*}
from \eqref{eq:d3Q1}, \eqref{eq:d3Q2}, \eqref{eq:d2ImS}, \eqref{eq:d3ReS}, and \eqref{eq:d3ImS}.
\par
Since $\cosh\kappa=3/2$, we have $\sinh\kappa=\sqrt{5}/2$.
Putting $t:=\tan(b/2)$, we have
\begin{align*}
  \sin{b}&=\frac{2t}{1+t^2},
  \\
  \cos{b}&=\frac{1-t^2}{1+t^2}.
\end{align*}
from the tangent half-angle formula.
From \eqref{eq:Re_phi} and \eqref{eq:Im_phi}, we have
\begin{align*}
  \cosh{c}\cos{d}&=\frac{3}{2}\times\frac{1-t^2}{1+t^2}-\frac{1}{2},
  \\
  \sinh{c}\sin{d}&=\frac{\sqrt{5}}{2}\times\frac{2t}{1+t^2}.
\end{align*}
So we have
\begin{align*}
  \sinh{c}
  &=
  \frac{\sqrt{3t^4-t^2+t\sqrt{(t^4+t^2+4)(9t^2+5)}}}{\sqrt{2}(t^2+1)},
  \\
  \cosh{c}
  &=
  \frac{\sqrt{5t^4+3t^2+2+t\sqrt{(t^4+t^2+4)(9t^2+5)}}}{\sqrt{2}(t^2+1)},
  \\
  \sin{d}
  &=
  \frac{\sqrt{-3t^4+t^2+t\sqrt{(t^4+t^2+4)(9t^2+5)}}}{\sqrt{2}(t^2+1)},
  \\
  \cos{d}
  &=
  \frac{\sqrt{5t^4+3t^2+2-t\sqrt{(t^4+t^2+4)(9t^2+5)}}}{\sqrt{2}(t^2+1)}.
\end{align*}
\par
Using
\begin{align*}
  X(\xi)
  =&
  \frac{\sinh(\kappa+b\i)\sinh(c-d\i)}{|\sinh\varphi|^2},
  \\
  Y(\xi)
  =&
  \frac{\cosh(c+d\i)\sinh(c-d\i)}{|\sinh\varphi|^2},
  \\
  Z(\xi)
  =&
  \frac{\cosh(\kappa+b\i)\sin(c-d\i)}{|\sinh\varphi|^2},
\end{align*}
Mathematical calculates
\begin{align*}
  &Q_1^{(3)}(2\arctan{t})
  \\
  =&
  \frac{\sqrt{5}(-2t^4-9t^2+5)}{2\sqrt{2}|\sinh\varphi(t)|^6\left(t^2+1\right)^5}
  \left(\frac{t}{\sqrt{(9t^2+5)(t^4+t^2+4)}-3t^3+11t}\right)^{3/2}
  \\
  &\times
  \Bigl(t(27t^8-108 t^6+263t^4+650t^2+140)
  \\
  &\quad
  +(-9t^6+43t^4+70t^2+10)\sqrt{(9t^2+5)(t^4+t^2+4)}\Bigr),
  \\
  \intertext{and}
  &Q_2^{(3)}(2\arctan{t})
  \\
  =&
  \frac{1}{\sqrt{2}(t^2+1)^6|\sinh\varphi|^6}
  \left(\frac{t}{\sqrt{(9t^2+5)(t^4+t^2+4)}-3t^3+11t}\right)^{3/2}
  \\
  &\times
  \bigl(P_1(t)-(\arctan{t})P_2(t)\bigr).
\end{align*}
with
\begin{align*}
  P_1(t)
  :=&
  (9t^2+5)(t^4+t^2+4)
  \biggl[3t^2(21t^4-52t^2+35)\sqrt{(9t^2+5)(t^4+t^2+4)}
  \\
  &+3t(-63t^8+107t^6-197t^4+115t^2+50)\biggr],
  \\
  P_2(t)
  :=&
  (t^2+5)(t^2+1)
  \biggl[t(-117t^6+98t^4+422t^2+95)\sqrt{(9t^2+5)(t^4+t^2+4)}
  \\
  &+351t^{10}-821t^8+532t^6+3467t^4+1015t^2+100\biggr].
\end{align*}
Note that $t^2+1$, $9t^2+5$, $t^4+t^2+4$, and $\sqrt{(9t^2+5)(t^4+t^2+4)}$ are clearly positive.
Note also that Lemma~\ref{lem:inequalities_P} shows that $-2t^4-9t^2+5$, $\sqrt{(9t^2+5)(t^4+t^2+4)}-3t^3+11t$, $27t^8-108 t^6+263t^4+650t^2+140$, $-9t^6+43t^4+70t^2+10$, $21t^4-52t^2+35$, $-63t^8+107t^6-197t^4+115t^2+50$, $-117t^6+98t^4+422t^2+95$, and $351t^{10}-821t^8+532t^6+3467t^4+1015t^2+100$ are all positive for $0<t<1/\sqrt{3}$.
\par
It follows that both $P_1(t)$ and $P_2(t)$ are positive.
Since $\arctan{t}<t$ for $t>0$ we have
\begin{equation*}
\begin{split}
  &P_1(t)-(\arctan{t})P_2(t)
  \\
  >&
  P_1(t)-tP_2(t)
  \\
  =&
  t(-2052 t^{14}-1042 t^{12}-5935 t^{10}+1658 t^8-36108 t^6-19100 t^4+7375 t^2+2500)
  \\
  &+
  (684 t^{10}+82 t^8+919 t^6-6783 t^4-1495 t^2+1625)\sqrt{(9t^2+5)(t^4+t^2+4)}.
\end{split}
\end{equation*}
Since expressions in the brackets above are positive for $0<t<1/\sqrt{3}$ from Lemma~\ref{lem:inequalities_Q}, we conclude that $P_1(t)-(\arctan{t})P_2(t)>0$.
\par
This shows that $Q_1^{(3)}(b)>0$ and $Q_2^{(3)}(b)>0$ for $0<b<\pi/3$, and we conclude that $Q^{(3)}(b)=\kappa Q_1^{(3)}(b)+Q_2^{(3)}(b)>0$ for $0<b<1/\sqrt{3}$.
\par
It follows that $Q''(b)$ is monotonically increasing for $0<b<1/\sqrt{3}$.
By Mathematica, we have $Q''(0.1)=-1.84946\ldots<0$ and $Q''(\pi/3)=3.28977\ldots>0$.
So there exists $0.1<b_0<\pi_1$ (in fact $b_0=0.208854\ldots$ by Mathematica) such that $Q''(b_0)=0$, $Q''(b)<0$ for $0<b<b_0$, and $Q''(b)>0$ for $b_0<b<\pi/3$.
\par
Therefore we see that $Q'(b)$ is decreasing for $0<b<b_0$ and increasing for $b_0<b<\pi/3$.
Since $\xi_b=\kappa$ and $\varphi_b=0$ if $b=0$, $S(\xi_b)$ vanishes when $b=0$ from \eqref{eq:Sminus}.
So from \eqref{eq:dQ1} we conclude that $Q'(0)=0$.
We also have $Q'(\pi/3)=1.28288\ldots>0$ by Mathematica.
Therefore we see that there exists $b_1$ with $b_0<b_1<\pi/3$ such that $Q'(b_1)=0$, $Q'(b)<0$ for $0<b<b_1$ and $Q'(b)>0$ for $b_1<b<\pi/3$.
Note that $b_1=0.648548\ldots$ by Mathematica.
\par
So we conclude that $Q(b)$ is decreasing for $0<b<b_1$, and increasing for $b_1<b<\pi/3$.
Since $Q(0)=0$ and $Q(\pi/3)=-0.0762858\ldots$ by Mathematica, we finally conclude that $Q(b)<0$ for $0<b<\pi/3$.
\end{proof}
%%%%%%%%%%%%%%%%%%%%%%%%%%%%%%%%%%%%%%%%%%%%%%%%%%%%%%%%%%%%%%%%%%%%%%%%%%%%%%%
%\end{comment}
%%%%%%%%%%%%%%%%%%%%%%%%%%%%%%%%%%%%%%%%%%%%%%%%%%%%%%%%%%%%%%%%%%%%%%%%%%%%%%%
To prove Lemma~\ref{lem:arg}, we prepare a lemma.
\begin{lem}\label{lem:triangle}
Let $w\ne1$ be a complex number with $|w|\le1$, and $\arg{w}$ be the argument of $w$ with $-\pi<\arg{w}\le\pi$.
\par
Then we have
\begin{align}
  (\arg{w})/2-\pi/2\le\arg\left(1-w\right)<0&\qquad(\text{if $\Im{w}>0$}),
  \label{eq:X_le_1_Im_pos}
  \\
  0<\arg\left(1-w\right)\le(\arg{w})/2+\pi/2&\qquad(\text{if $\Im{w}<0$}),
  \label{eq:X_le_1_Im_neg}
  \\
  \arg(1-w)=0&\qquad(\text{if $\Im{w}=0$}),
  \label{eq:X_le_1_Im_0}
\end{align}
where the equality in \eqref{eq:X_le_1_Im_pos} and that in \eqref{eq:X_le_1_Im_neg} hold only when $|w|=1$.
\end{lem}
\begin{proof}
It is clear that if $w<1$ is real, then $\arg(1-w)=0$.
So we assume that $\Im{w}\ne0$.
\par
Consider the triangle $\triangle{\rm{OAP}}$ in the complex plane with $\rm{O}$ the origin, $\rm{A}:1$, and ${\rm P}:w\ne1$ with $\Im{w}\ne0$ and $\overline{\rm{OP}}=|w|\le1$.
Then from the left picture of Figure~\ref{fig:angle} we have
\begin{equation}\label{eq:angle}
\begin{split}
  0>\arg(1-w)=-\angle{\rm{OAP}}&\qquad(\text{if $\Im{w}>0$}),
  \\
  0<\arg(1-w)=\phantom{-}\angle{\rm{OAP}}&\qquad(\text{if $\Im{w}<0$})
\end{split}
\end{equation}
since $1-w$ is presented by the vector $\overrightarrow{\rm{PA}}$, and $\arg(1-w)$ is the signed angle of $\overrightarrow{\rm{PA}}$ from the line $l$, where $l$ is parallel to $\rm{OA}$.
\begin{figure}[h]
\pic{0.3}{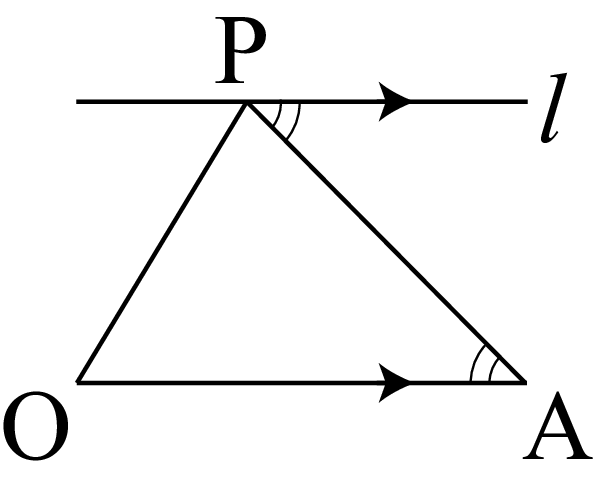}\hspace{20mm}\pic{0.3}{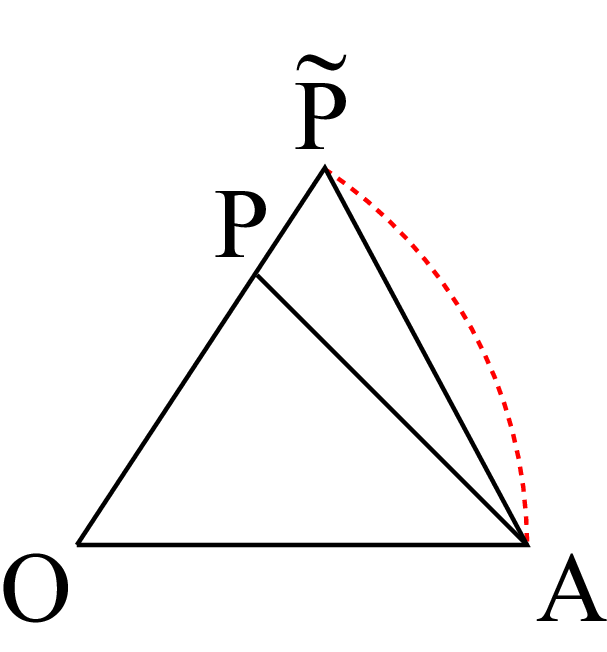}
\caption{$\angle{\rm{OAP}}$ equals the angle between $l$ and $\rm{PA}$, where the line $l$ is parallel to $\rm{OA}$.$\overline{\rm{OA}}=\overline{\rm{O\tP}}=1$, and $\overline{\rm{OP}}\le1$.
The dotted red arc is the unit circle centered at $\rm{O}$.}
\label{fig:angle}
\end{figure}
\par
Now we have $\angle\rm{OAP}=\pi-\angle\rm{AOP}-\angle\rm{OPA}\le\pi-\angle\rm{AOP}-\angle\rm{OAP}$ since $\angle\rm{OPA}\ge\angle\rm{O\tP A}=\angle\rm{OA\tP}\ge\angle\rm{OAP}$, where $\rm{\tP}$ is the intersection between the ray $\rm{OP}$ and the unit circle ($\rm{\tP}=\rm{P}$ if $\overline{\rm{OP}}=1$).
See the right picture of Figure~\ref{fig:angle}.
Note that the equality holds only when $\rm{P}=\rm{\tP}$, that is, $\overline{\rm{OP}}=1$.
So we obtain $\angle\rm{OAP}\le(\pi-\angle\rm{AOP})/2$.
Then from \eqref{eq:angle} we have \eqref{eq:X_le_1_Im_pos} and \eqref{eq:X_le_1_Im_neg}.
\end{proof}
Now we can prove Lemma~\ref{lem:arg} using Lemma~\ref{lem:triangle}.
%%%%%%%%%%%%%%%%%%%%%%%%%%%%%%%%%%%%%%%%%%%%%%%%%%%%%%%%%%%%%%%%%%%%%%%%%%%%%%%
\begin{proof}[Proof of Lemma~\ref{lem:arg}]
We put $Z=X+Y\i$ and $\xi=a+b\i$ as usual.
We will show the following inequalities.
\begin{align}
  -\pi<D_1(Z)<\pi
  &\qquad\text{(if $|X|<a$),}
  \label{eq:X_le_a}
  \\
  D_2(Z)<3\pi/2
  &\qquad\text{(if $X\ge a$ and $Y\le b$).}
  \label{eq:X_ge_a}
\end{align}
Since $D_1\bigl(X+(Y+2\pi)\i\bigr)=D_1(X+Y\i)$ and $D_2\bigl(X+(Y+2\pi)\i\bigr)=D_2(X+Y\i)+2\pi$, it is sufficient to prove \eqref{eq:X_le_a} and \eqref{eq:X_ge_a} for the case $b-2\pi<Y\le b$.
\par
If $X>-a$, then $\Re(-\xi-Z)=-a-X<0$, if $|X|<a$, then $\Re(-\xi+Z)=-a+X<0$, and if $X\ge a$, then $\Re(\xi-Z)=a-X\le0$.
So we can use Lemma~\ref{lem:triangle} with $w=e^{-\xi-Z}$ when $X>-a$, $w=e^{-\xi+Z}$ when $|X|<a$, and $w=e^{\xi-Z}$ when $X\ge a$.
\par
There are four cases to consider: (a) $-b<Y\le b$, (b)  $b-\pi<Y\le-b$, (c) $-b-\pi<Y\le b-\pi$, and (d) $b-2\pi<Y\le-b-\pi$.
\par
\renewcommand{\labelenumi}{(\alph{enumi})}
\begin{enumerate}
\item%(a)
The case where $-b<Y\le b$.
Since $-2b\le-b-Y<0$, $-2b<-b+Y\le0$, and $0\le b-Y<2b$, we have $\Im(e^{-\xi-Z})=e^{-a-X}\sin(-b-Y)<0$, $\Im(e^{-\xi+Z})=e^{-a+X}\sin(-b+Y)\le0$, and $\Im(e^{\xi-Z})=e^{a-X}\sin(b-Y)\ge0$.
We also have  $\arg(e^{-\xi-Z})=-b-Y$, $\arg(e^{-\xi+Z})=-b+Y$, and $\arg(e^{\xi-Z})=b-Y$.
So we can use \eqref{eq:X_le_1_Im_neg}, \eqref{eq:X_le_1_Im_pos}, and \eqref{eq:X_le_1_Im_0} to obtain
\begin{alignat}{4}
  &0&<\arg(1-e^{-\xi-Z})&<(-b-Y)/2+\pi/2
  &&\quad\text{($X>-a$)},
  \label{eq:arg_a1}
  \\
  &0&\le\arg(1-e^{-\xi+Z})&<(-b+Y)/2+\pi/2
  &&\quad\text{($|X|<a$)},
  \label{eq:arg_a2}
  \\
  &&\arg(1-e^{\xi-Z})&\le0
  &&\quad\text{($X\ge a$)}.
  \label{eq:arg_a3}
\end{alignat}
From \eqref{eq:arg_a1} and \eqref{eq:arg_a2}, we have $b<D_1(Z)<\pi$ when $|X|<a$, and \eqref{eq:X_le_a} follows.
From \eqref{eq:arg_a1} and \eqref{eq:arg_a3}, we have $D_2(Z)<(Y-b)/2+3\pi/2\le3\pi/2$ when $X\ge a$ since $Y\le b$, which implies \eqref{eq:X_ge_a}.
\item%(b)
The case where $b-\pi<Y\le-b$.
We have $0\le-b-Y<-2b+\pi$, $-\pi<-b+Y\le-2b$, and $2b\le b-Y<\pi$.
It follows that $\Im(e^{-\xi-X})\ge0$, $\Im(e^{-\xi+Z})<0$, and $\Im(e^{\xi-Z})>0$.
Since $\arg(e^{-b-Y})=-b-Y$, $\arg(e^{-b+Y})=-b+Y$, and $\arg(e^{b-Y})=b-Y$, we can use \eqref{eq:X_le_1_Im_pos} and \eqref{eq:X_le_1_Im_neg}, and \eqref{eq:X_le_1_Im_0} to obtain
\begin{alignat}{4}
  &(-b-Y)/2-\pi/2&<\arg(1-e^{-\xi-Z})&\le0
  &&\quad\text{($X>-a$)},
  \label{eq:arg_b1}
  \\
  &\hspace{25mm}0&<\arg(1-e^{-\xi+Z})&<(-b+Y)/2+\pi/2
  &&\quad\text{($|X|<a$)},
  \label{eq:arg_b2}
  \\
  &&\arg(1-e^{\xi-Z})&<0
  &&\quad\text{($X\ge a$)}.
  \label{eq:arg_b3}
\end{alignat}
So we have $(b-Y)/2-\pi/2<D_1(Z)<(b+Y)/2+\pi/2$ when $|X|<a$ from \eqref{eq:arg_b1} and \eqref{eq:arg_b2}.
Since $(b-Y)/2-\pi/2\ge b-\pi/2$, and $(b+Y)/2+\pi/2\le\pi/2$, \eqref{eq:X_le_a} follows.
From \eqref{eq:arg_b1} and \eqref{eq:arg_b3}, we have $D_2(Z)<Y+\pi$ when $X\ge a$.
From $Y\le-b$, \eqref{eq:X_ge_a} follows.
\item%(c)
The case where  $-b-\pi<Y\le b-\pi$.
From $-2b+\pi\le-b-Y<\pi$, $-2b-\pi<-b+Y\le-\pi$, and $\pi\le b-Y<2b+\pi$, we have $\Im(e^{-\xi-Z})>0$, $\Im(e^{-\xi+Z})\ge0$, and $\Im(e^{\xi-Z})\le0$.
Since $\arg(e^{-\xi-Z})=-b-Y$, $\arg(e^{-\xi+Z})=-b+Y+2\pi$, and $\arg(e^{\xi-Z})=b-Y-2\pi$ (except for the case $Y=b-\pi$; in that case $\Im(e^{\xi-Z})=0$ and $\arg(e^{\xi-Z})=b-Y=\pi$), we use \eqref{eq:X_le_1_Im_pos}, \eqref{eq:X_le_1_Im_neg}, and \eqref{eq:X_le_1_Im_0} to obtain
\begin{alignat}{4}
  &(-b-Y)/2-\pi/2&<\arg(1-e^{-\xi-Z})&<0
  &&\quad\text{($X>-a$)},
  \label{eq:arg_c1}
  \\
  &(-b+Y+2\pi)/2-\pi/2&<\arg(1-e^{-\xi+Z})&\le0
  &&\quad\text{($|X|<a$)},
  \label{eq:arg_c2}
\end{alignat}
and
\begin{equation}\label{eq:arg_c3}
  \arg(1-e^{\xi-Z})
  \begin{cases}
    <(b-Y-2\pi)/2+\pi/2
    &\quad\text{($X\ge a$ and $Y\ne b-\pi$)},
    \\
    =0
    &\quad\text{($X\ge a$ and $Y=b-\pi$)}.
  \end{cases}
\end{equation}
From \eqref{eq:arg_c1} and \eqref{eq:arg_c2} we have $0<D_1(Z)<b<\pi/2$ when $|X|<a$, and \eqref{eq:X_le_a} follows.
If $X\ge a$, from \eqref{eq:arg_c1} and \eqref{eq:arg_c3}, we have $D_2(Z)<(Y+b)/2+\pi/2\le b<\pi/2$ when $Y\ne b-\pi$ since $Y+b\le2b-\pi$, and $D_2(Z)<Y+\pi=b<\pi/2$ when $Y=b-\pi$.
\item%(d)
The case where  $b-2\pi<Y\le-b-\pi$.
From $\pi\le-b-Y<-2b+2\pi$, $-2\pi<-b+Y\le-2b-\pi$, and $2b+\pi\le b-Y<2\pi$, we have $\Im(e^{-\xi-Z})\le0$, $\Im(e^{-\xi+Z})>0$, and $\Im(e^{\xi-Z})<0$.
Since $\arg(e^{-\xi-Z})=-b-Y-2\pi$ (except for the case $Y=-b-\pi$; in that case $\Im(e^{-\xi-Z})=0$ and $\arg(e^{-\xi-Z})=-b-Y=\pi$), $\arg(e^{-\xi+Z})=-b+Y+2\pi$, and $\arg(e^{\xi-Z})=b-Y-2\pi$, from \eqref{eq:X_le_1_Im_neg} and \eqref{eq:X_le_1_Im_pos} we have
\begin{equation}\label{eq:arg_d1}
\begin{cases}
  0<\arg(1-e^{-\xi-Z})<(-b-Y-2\pi)/2+\pi/2
  &\quad\text{($X>-a$ and $Y\ne-b-\pi$),}
  \\
  \phantom{0<}\arg(1-e^{-\xi-Z})=0
  &\quad\text{($X>-a$ and $Y=-b-\pi$),}
\end{cases}
\end{equation}
\begin{align}
  &(-b+Y+2\pi)/2-\pi/2&<\arg(1-e^{-\xi+Z})&<0
  &&\quad\text{($|X|<a$)},
  \label{eq:arg_d2}
  \\
  &&\arg(1-e^{\xi-Z})&<(b-Y-2\pi)/2+\pi/2
  &&\quad\text{($X\ge a$)}.
  \label{eq:arg_d3}
\end{align}
If $|X|<a$, then from \eqref{eq:arg_d1} and \eqref{eq:arg_d2}, we have $(Y+b)/2+\pi/2<D_1(Z)<(b-Y)/2-\pi/2$ when $Y\ne-b-\pi$, and $(Y+b)/2+\pi/2<D_1(Z)<b$ when $Y=-b-\pi$.
Since $(Y+b)/2+\pi/2>b-\pi/2$ and $(b-Y)/2-\pi/2<\pi/2$ if $Y\ne-b-\pi$, and $(Y+b)/2+\pi/2=0$ and $(b-Y)/2-\pi/2=b$ if $Y=-b-\pi$, we obtain \eqref{eq:X_le_a}.
If $X\ge a$, then from \eqref{eq:arg_d1} and \eqref{eq:arg_d3}, we have $D_2(Z)<0$ when $Y\ne-b-Y$, and $D_2(Z)<(b+Y)/2+\pi/2=0$ when $Y=-b-\pi$, which implies \eqref{eq:X_ge_a}.
\end{enumerate}
The proof is now complete.
\end{proof}
%%%%%%%%%%%%%%%%%%%%%%%%%%%%%%%%%%%%%%%%%%%%%%%%%%%%%%%%%%%%%%%%%%%%%%%%%%%%%%%
%%%%%%%%%%%%%%%%%%%%%%%%%%%%%%%%%%%%%%%%%%%%%%%%%%%%%%%%%%%%%%%%%%%%%%%%%%%%%%%%%%%%%%%%%
\begin{lem}\label{lem:inequalities_cos}
Put
\begin{align*}
  \upsilon_1(x)
  &:=
  x^6+24x^5+174x^4+408x^3-63x^2-1008x-720,
  \\
  \upsilon_2(x)
  &:=
  9x^4+104x^3+401x^2+548x+246,
  \\
  \upsilon_3(x)
  &:=
  9x^4+24x^3-21x^2-100x-128,
  \\
  \upsilon_4(x)
  &:=
  3x^4+16x^3+28x^2+40x+65,
  \\
  \upsilon_5(x)
  &:=
  x^6+4x^5-8x^4-64x^3-123x^2-100x-94.
\end{align*}
Then we have $\upsilon_1(x)<0$, $\upsilon_2(x)>0$, $\upsilon_3(x)<0$, $\upsilon_4(x)>0$, and $\upsilon_5(x)<0$ for $|x|<1$.
\end{lem}
\begin{proof}
The first derivative $\upsilon_1'(x)$ of $\upsilon_1(x)$ equals $6(x^5+20x^4+116x^3+204x^2-21x-168)$, and it vanishes at $x=-11.3961\ldots$, $-5.29209\ldots$, $-2.80804\ldots$, $-1.27925$, and $0.775472\ldots$ due to Mathematica.
Since the coefficient of $x^6$ in $\upsilon_1(x)$ is positive, $\upsilon_1(x)$ is decreasing for $(-1,0.775472\ldots)$ and increasing in $(0.775472\ldots,1)$.
Since when $\upsilon_1(-1)=-32$ and $\upsilon_1(1)=-1184$, we conclude that $\upsilon_1(x)<0$ for $|x|<1$.
\par
Since $\upsilon_2'(x)=36x^3+312x^2+802x+548$, it vanishes ad $x=-4.34891\ldots$, $-3.23616\ldots$, and $-1.0816\ldots$ by Mathematica.
Since the coefficient of $x^4$ in $\upsilon_2(x)$ is positive, it is increasing in $(-1,1)$.
Since $\upsilon_2(-1)=4$, we conclude that $\upsilon_2(x)>0$ for $|x|<1$.
\par
Since $\upsilon_3'(x)=36x^3+72x^2-42x-100$, it vanishes at $x=-1.77663\ldots$, $-1.36706\ldots$, and $1.1437\ldots$ by Mathematica.
Since the coefficient of $x^4$ in $\upsilon_3(x)$ is positive, it is decreasing for $|x|<1$.
Since $\upsilon_3(-1)=-64<0$, we have $\upsilon_3(x)<0$ for $|x|<1$.
\par
As for $\upsilon_4(x)$, since $|x|<1$, we see that $x^3$ and $x$ are greater than $-1$ and that $x^4$ and $x^2$ are greater than or equal to $0$.
So we conclude that $\upsilon_4(x)>-16-40+65=9>0$.
\par
To prove $\upsilon_5(x)<0$, we first observe that $x^6+4x^5-8x^4=x^4(x^2+4x-8)<1+4-8<0$ for $|x|<1$.
Since $\frac{d}{d\,x}(-64x^3-123x^2-100x-94)=-192x^2-246x-100=-192\left(x+\frac{41}{64}\right)^2-\frac{1357}{64}<0$, it follows that $-64x^3-123x^2-100x-94<\left(-64x^3-123x^2-100x-94\right)\Bigr|_{x=-1}=64-123+100-94=-53<0$.
So we conclude that $\upsilon_5(x)<0$ for $|x|<1$.
\par
The proof is complete.
\end{proof}
%%%%%%%%%%%%%%%%%%%%%%%%%%%%%%%%%%%%%%%%%%%%%%%%%%%%%%%%%%%%%%%%%%%%%%%%%%%%%%%
\begin{lem}\label{lem:inequalities_P}
Put
\begin{align*}
  p_1(t)
  :=&
  -2t^4-9t^2+5,
  \\
  p_2(t)
  :=&
  21t^4-52t^2+35,
  \\
  p_3(t)
  :=&
  -9t^6+43t^4+70t^2+10,
  \\
  p_4(t)
  :=&
  -117t^6+98t^4+422t^2+95,
  \\
  p_5(t)
  :=&
  27t^8-108t^6+263t^4+650t^2+140,
  \\
  p_6(t)
  :=&
  -63t^8+107t^6-197t^4+115t^2+50,
  \\
  p_7(t)
  :=&
  351t^{10}-821t^8+532t^6+3467t^4+1015t^2+100,
  \\
  p_8(t)
  :=&
  \sqrt{(9t^2+5)(t^4+t^2+4)}-3t^3+11t.
\end{align*}
Then we have $p_i(x)>0$ {\rm(}$i=1,2,\dots,8${\rm)} for $0<t<1/\sqrt{3}$.
\end{lem}
\begin{proof}
From $0<t<1/\sqrt{3}$, we have
\begin{align*}
  p_1(t)
  >&
  -2/(\sqrt{3})^4-9/(\sqrt{3})^2+5
  =
  16/9>0,
  \\
  p_2(t)
  >&
  -52/(\sqrt{3})^2+35
  =
  53/3>0,
  \\
  p_3(t)
  >&
  -9/(\sqrt{3})^6+10
  =29/3>0,
  \\
  p_4(t)
  >&
  -117/(\sqrt{3})^6+95
  =272/3>0,
  \\
  p_5(t)
  >&
  -108/(\sqrt{3})^6+140
  =136>0,
  \\
  p_6(t)
  >&
  -63/(\sqrt{3})^8-197/(\sqrt{3})^4+50
  =82/3>0,
  \\
  p_7(t)
  >&
  -821/(\sqrt{3})^8+100
  =7279/81>0,
  \\
  p_8(5)
  >&
  \sqrt{5\times4}-3/(\sqrt{3})^3
  =
  2\sqrt{5}-1/\sqrt{3}>0.
\end{align*}
This completes the proof.
\end{proof}
\begin{lem}\label{lem:inequalities_Q}
Put
\begin{align*}
  q_1(t)
  :=&
  -2052 t^{14}-1042 t^{12}-5935 t^{10}+1658 t^8-36108 t^6-19100 t^4+7375 t^2
  \\
  &+2500
  \\
  q_2(t)
  :=&
  684 t^{10}+82 t^8+919 t^6-6783 t^4-1495 t^2+1625
\end{align*}
Then we have $q_i(t)>0$ {\rm(}$i=1,2${\rm)} for $0<t<1/\sqrt{3}$.
\end{lem}
\begin{proof}
In a similar way as above, we have
\begin{equation*}
  q_2(t)
  >
  -6783/(\sqrt{3})^4-1495/(\sqrt{3})^2+1625
  =373>0.
\end{equation*}
\par
As for $q_1(t)$, we will show that $\tq_1(x):=q_1(\sqrt{x})>0$ for $0<x<1/3$.
Since we have
\begin{equation*}
\begin{split}
  \tq_1''(x)
  =&
  -4(21546x^5+7815x^4+29675x^3-4974x^2+54162x+9550)
  \\
  <&
  -4(-4974/3^2+9550)
  <0,
\end{split}
\end{equation*}
the function $\tq_1(x)$ is convex down for $0<x<1/3$.
Therefore we conclude that $\tq_1(x)<\min\{\tq_1(0),\tq_1(1/3)\}=\min\{2500,1088000/729\}>0$.
\par
The proof is complete.
\end{proof}
%%%%%%%%%%%%%%%%%%%%%%%%%%%%%%%%%%%%%%%%%%%%%%%%%%%%%%%%%%%%%%%%%%%%%%%%%%%%%%%
%\end{comment}
%%%%%%%%%%%%%%%%%%%%%%%%%%%%%%%%%%%%%%%%%%%%%%%%%%%%%%%%%%%%%%%%%%%%%%%%
\bibliography{mrabbrev,hitoshi}
\bibliographystyle{amsplain}
\end{document}